\documentclass[12pt]{amsart}
\usepackage{latexsym, amscd, amsfonts, eucal, mathrsfs, amsmath, amssymb, amsthm, xypic,xr, stmaryrd, amsxtra, MnSymbol}

\usepackage{enumitem}
\usepackage[hyperindex]{hyperref}

\newtheorem{theorem}{Theorem}[subsection]
\newtheorem{lemma}[theorem]{Lemma}
\newtheorem{proposition}[theorem]{Proposition}

\newtheorem{corollary}[theorem]{Corollary}

\theoremstyle{definition}

\newtheorem{definition}[theorem]{Definition}
\newtheorem{construction}[theorem]{Construction}
\newtheorem{example}[theorem]{Example}

\newtheorem{notation}[theorem]{Notation}

\newtheorem{warning}[theorem]{Warning}
\newtheorem{remark}[theorem]{Remark}
\newtheorem{variant}[theorem]{Variant}

\newcommand{\Etale}{\'{E}tale}
\newcommand{\etale}{\'{e}tale}

\def\Z{\mathbf{Z}}
\def\C{\mathbf{C}}
\def\F{\mathbf{F}}

\DeclareMathOperator{\sheafE}{{\mathscr E}}
\DeclareMathOperator{\sheafF}{{\mathscr F}}
\DeclareMathOperator{\sheafG}{{\mathscr G}}
\DeclareMathOperator{\sheafH}{{\mathscr H}}
\DeclareMathOperator{\calO}{{\mathcal O}}

\DeclareMathOperator{\calC}{{\mathcal C} }
\DeclareMathOperator{\calD}{{\mathcal D}}
\DeclareMathOperator{\calE}{{\mathcal E}}

\DeclareMathOperator{\RH}{RH}
\DeclareMathOperator{\op}{op}
\DeclareMathOperator{\alg}{alg}

\DeclareMathOperator{\Sol}{Sol}
\DeclareMathOperator{\RSol}{RSol}
\DeclareMathOperator{\Ext}{Ext}
\DeclareMathOperator{\id}{id}
\DeclareMathOperator{\im}{im}
\DeclareMathOperator{\Gal}{Gal}

\DeclareMathOperator{\Shv}{Shv}

\DeclareMathOperator{\Spec}{Spec}
\DeclareMathOperator{\Set}{Set}

\DeclareMathOperator{\Frob}{Fr}
\DeclareMathOperator{\QCoh}{QCoh}

\DeclareMathOperator{\Mod}{Mod}

\DeclareMathOperator{\Tor}{Tor}

\DeclareMathOperator{\Hom}{Hom}

\DeclareMathOperator{\perf}{perf}
\DeclareMathOperator{\perfection}{1/p^{\infty}} 
\DeclareMathOperator{\EK}{EK}

\DeclareMathOperator{\supp}{supp}
\DeclareMathOperator{\CAlg}{CAlg}

\DeclareMathOperator{\Ind}{Ind}
\DeclareMathOperator{\coker}{coker}

\DeclareMathOperator{\cn}{cn}
\DeclareMathOperator{\hol}{hol} 
\newcommand{\et}{\'{e}t}

\newcommand{\mathet}{\text{\et}}

\DeclareMathOperator{\fgu}{fgu}

\DeclareMathOperator{\qE}{ \mathcal{E} } 
\DeclareMathOperator{\qF}{ \mathcal{F} } 

\newcommand{\Adjoint}[4]{\xymatrix@1{#2 \ar@<.4ex>[r]^-{#1} & #3 \ar@<.4ex>[l]^-{#4}}}

\title{A Riemann-Hilbert Correspondence in Positive Characteristic}
\author{Bhargav Bhatt and Jacob Lurie}
\begin{document}

\maketitle

\setcounter{tocdepth}{2}

\tableofcontents

\newpage \section{Introduction}
\setcounter{subsection}{0}
\setcounter{theorem}{0}

Let $p$ be a prime number, which we regard as fixed throughout this paper. Our starting point is the following theorem of Katz (see \cite[Proposition 4.1.1]{KatzPadic}):

\begin{theorem}[Katz]\label{theorem.folk}
Let $k$ be a perfect field of characteristic $p$ and let $\overline{k}$ be an algebraic closure of $k$. Then the construction
$V \mapsto ( V \otimes_{ \F_p } \overline{k} )^{ \Gal( \overline{k} / k)}$ induces an equivalence from the
category of finite-dimensional $\F_p$-vector spaces $V$ with a continuous action of $\Gal( \overline{k} / k )$ to
the category of finite-dimensional $k$-vector spaces $M$ equipped with a Frobenius-semilinear automorphism $\varphi_{M}$.
\end{theorem}

The equivalence of Theorem \ref{theorem.folk} can be extended to infinite-dimensional vector spaces; in this case, we must add the requirement that
$M$ is {\em locally finite} in the sense that every element $x \in M$ belongs to a finite-dimensional $\varphi_{M}$-stable subspace. Our primary goal in this paper is to prove the following more general result:

\begin{theorem}\label{maintheoXXX}
Let $R$ be a commutative $\F_p$-algebra. Then there is a fully faithful embedding of abelian categories
$$ \xymatrix{ \{ \text{$p$-torsion {\etale} sheaves on $\Spec(R)$}\} \ar[d]^{ \RH } \\
\{ \text{$R$-modules with a Frobenius-semilinear automorphism $\varphi_M$} \}. }$$
Moreover, the essential image of $\RH$ consists of those $R$-modules $M$ equipped with
a Frobenius-semilinear automorphism $\varphi_M: M \rightarrow M$ which satisfies the following condition:
every element $x \in M$ satisfies an equation of the form $$\varphi^{n}_M x + a_{1} \varphi^{n-1}_M x + \cdots + a_n x = 0$$
for some coefficients $a_1, \ldots, a_n \in R$.
\end{theorem}

We also establish various extensions of Theorem \ref{maintheoXXX}, where we replace $p$-torsion sheaves by $p^{n}$-torsion sheaves (Theorem \ref{theorem.generalizedRH}),
the affine scheme $\Spec(R)$ by an arbitrary $\F_p$-scheme (Theorem \ref{globalRH}), and the abelian category of {\etale} sheaves by its derived category (Theorem \ref{RHDerived}).

\subsection{Outline}

The first half of this paper is devoted to the proof of Theorem \ref{maintheoXXX}. Note that Theorem \ref{maintheoXXX} supplies a description of the category of ($p$-torsion) {\etale} sheaves on $\Spec(R)$ as
quasi-coherent sheaves on $\Spec(R)$ with additional structure, and can therefore be viewed as a positive-characteristic analogue of the Riemann-Hilbert correspondence. We will emphasize this perspective
by referring to the functor $\RH$ appearing in Theorem \ref{maintheoXXX} as the {\it Riemann-Hilbert functor}. It is not so easy to describe this functor directly. Instead, we begin in \S \ref{section.overview}
by constructing a functor in the opposite direction. Let $\Mod_{R}^{\Frob}$ denote the category whose objects are pairs $(M, \varphi_M)$, where $M$ is an $R$-module and
$\varphi_{M}$ is a Frobenius semilinear endomorphism of $M$; we will refer to such pairs as {\it Frobenius modules over $R$}. If $M$ is an $R$-module and $\widetilde{M}$ denotes the associated quasi-coherent sheaf on the {\etale} site of $\Spec(R)$, then every Frobenius-semilinear automorphism $\varphi_{M}$ of $M$ determines an automorphism of $\widetilde{M}$, which we will also denote by $\varphi_{M}$. We let $\Sol(M)$ denote subsheaf of $\varphi_{M}$-fixed points in $\widetilde{M}$: that is, the kernel of the map $\id - \varphi_{M}: \widetilde{M} \rightarrow \widetilde{M}$, formed in the category $\Shv_{\mathet}( \Spec(R), \F_p)$ of $p$-torsion {\etale} sheaves on $\Spec(R)$. The construction $(M, \varphi_M) \mapsto \Sol(M)$ determines a functor
$$ \Sol: \Mod_{R}^{\Frob} \rightarrow \Shv_{\mathet}( \Spec(R), \F_p ),$$
which we will refer to as the {\it solution functor} (Construction \ref{construction.solsheaf}). 

We will say that a Frobenius module $(M, \varphi_M)$ is {\it perfect} if the map $\varphi_M$ is invertible. In \S \ref{section.RH}, we will show that, when restricted to category
$\Mod_{R}^{\perf}$ of perfect Frobenius modules, the solution functor $\Sol$ has a left adjoint (Theorem \ref{theorem.RHexist}). This left adjoint is the Riemann-Hilbert functor
$\RH: \Shv_{\mathet}( \Spec(R), \F_p ) \rightarrow \Mod_{R}^{\perf}$ appearing in the statement Theorem \ref{maintheoXXX}. The existence of the functor $\RH$ is not evident: to construct it,
we will need to develop a theory of {\em compactly supported direct images} in the setting of (perfect) Frobenius modules; this is the subject of \S \ref{section.compactimage}. 
We also prove in \S \ref{section.RH} that the functor $\RH$ is exact (Proposition \ref{proposition.corX33}). This is easy to see in the case where $R$ is an algebraically closed field: in this case,
the category $\Shv_{\mathet}( \Spec(R), \F_p)$ is equivalent to the category of vector spaces over $\F_p$, where every exact sequence is split. We handle the general case by
reducing to the case of an algebraically closed field, using a theory of base change for perfect Frobenius modules (which we study in \S \ref{section.frobenius-module}) and its compatibiity
with the Riemann-Hilbert correspondence (which we prove as Proposition \ref{prop70}).

We will complete the proof of Theorem \ref{maintheoXXX} in \S \ref{section.mainproof} by showing that the Riemann-Hilbert functor $\RH$ is fully faithful and characterizing its essential image. 
The full-faithfulness is actually fairly easy, once we know that the functor $\RH$ is exact: it essentially follows from the exactness of the Artin-Schreier sequence
$0 \rightarrow \underline{\F_p} \rightarrow \widetilde{R} \rightarrow \widetilde{R} \rightarrow 0$ in the category of {\etale} sheaves $\Shv_{\mathet}( \Spec(R), \F_p )$ (see Proposition \ref{prop75}).
To understand the essential image of the Riemann-Hilbert functor, it will be convenient to consider first the functor $\RH^{c}$ obtained by restricting
$\RH$ to the subcategory $\Shv_{\mathet}^{c}( \Spec(R), \F_p) \subseteq \Shv_{\mathet}( \Spec(R), \F_p)$ of {\em constructible} {\etale} sheaves on $\Spec(R)$.
The functor $\RH^{c}$ takes values in the subcategory $\Mod_{R}^{\hol} \subseteq \Mod_{R}^{\perf}$ of {\it holonomic} Frobenius modules (Definition \ref{defhol2}), which we study in \S \ref{section.holonomic}.
Theorem \ref{maintheoXXX} will then follow by combining the following two assertions:
\begin{itemize}
\item The functor $\RH^{c}: \Shv_{\mathet}^{c}( \Spec(R), \F_p ) \rightarrow \Mod_{R}^{\hol}$ is an equivalence of categories (Theorem \ref{companion}): that is, every holonomic Frobenius module $M$
the form $\RH(\sheafF)$, for some constructible {\etale} sheaf $\sheafF$ on $\Spec(R)$. We will prove this using formal arguments to reduce to the case where $R$ is a field,
in which case the desired result follows from Theorem \ref{theorem.folk} (which we reprove here as Proposition \ref{proposition.lemX51}). 

\item A perfect Frobenius module $M$ can be written as a filtered colimit of holonomic Frobenius modules if and only if every element $x \in M$
satisfies an equation $\varphi^{n}_M x + a_{1} \varphi^{n-1}_M x + \cdots + a_n x = 0$ for some coefficients $a_1, \ldots, a_n \in R$
(Theorem \ref{theoX54}). We will refer to such Frobenius modules as {\it algebraic}.
\end{itemize}

In the second half of this paper, we consider several refinements of Theorem \ref{maintheoXXX}:

\begin{itemize}
\item Let $(M, \varphi_M)$ and $(N, \varphi_N)$ be Frobenius modules over $R$. If $M$ and $N$ are perfect, then they can also
be regarded as Frobenius modules over the perfection $R^{\perfection}$ (Proposition \ref{prop5}). In this case, we can
regard the tensor product $M \otimes_{ R^{\perfection} } N$ as a (perfect) Frobenius module over $R$, with Frobenius endomorphism given by
$x \otimes y \mapsto \varphi_M(x) \otimes \varphi_N(y)$. In \S \ref{section.tensor}, we show that the Riemann-Hilbert functor $\RH$ of Theorem \ref{maintheoXXX}
is compatible with tensor products, in the sense that there are canonical isomorphisms
$$ \RH( \sheafF \otimes_{\F_p} \sheafG ) \simeq \RH(\sheafF) \otimes_{ R^{\perfection} } \RH(\sheafG)$$
(Corollary \ref{tensoraffine}). Our proof relies on the vanishing of the $\Tor$-groups $\Tor_{n}^{R^{\perfection}}(M, N)$ for $n > 0$ when
$M$ and $N$ are algebraic Frobenius modules, which we establish as Theorem \ref{tenso10}.

\item In \S \ref{section.moretorsion}, we prove a generalization of Theorem \ref{maintheoXXX} where the category 
$\Shv_{\mathet}( \Spec(R), \F_p)$ of $p$-torsion {\etale} sheaves is replaced by the larger category
$\Shv_{\mathet}( \Spec(R), \Z / p^{n} \Z)$ of $p^{n}$-torsion {\etale} sheaves, for some integer $n \geq 0$. In this case,
we must also replace the category $\Mod_{R}^{\Frob}$ of Frobenius modules over $R$ by the larger category
$\Mod_{ W_n(R)}^{\Frob}$ of {\it Frobenius modules over $W_n(R)$}; here $W_n(R)$ denotes the ring of $n$-truncated Witt vectors of $R$
(see Theorem \ref{theorem.generalizedRH}).

\item In \S \ref{section.global}, we prove a generalization of Theorem \ref{maintheoXXX} where the affine scheme $\Spec(R)$ is replaced by
an arbitrary $\F_p$-scheme $X$ (Theorem \ref{globalRH}). We also show that the Riemann-Hilbert correspondence is compatible with the formation of
(higher) direct images along proper morphisms $f: X \rightarrow Y$ of finite presentation (Theorem \ref{properRH1}). As an application,
we reprove a special case of the proper base change theorem in {\etale} cohomology (namely, the case of $p$-torsion sheaves on $\F_p$-schemes;
see Corollary \ref{corollary.basechange}).

\item In \S \ref{section.duality}, we study the derived category $D( R[F] )$ of Frobenius modules over $R$. The equivalence
of abelian categories $\RH^{c}: \Shv_{\mathet}^{c}( \Spec(R), \F_p) \simeq \Mod_{R}^{\hol}$ extends to an equivalence of triangulated categories
$$D^{b}_{c}( \Spec(R), \F_p) \simeq D^{b}_{\hol}( R[F] ),$$ where $D^{b}_{\hol}( R[F] ) \subseteq D(R[F])$ denotes the full subcategory spanned by the
cohomologically bounded chain complexes with holonomic cohomology and $D^{b}_{c}( \Spec(R), \F_p)$ is the constructible derived category of $\Spec(R)$;
see Corollary \ref{corollary.RHDerived}. We also construct a duality functor $\mathbb{D}: D^{b}_{\hol}( R[F] ) \rightarrow D(R[F] )^{\op}$, and show that it is a fully
faithful embedding (Theorem \ref{theorem.dual-image}). Combining this duality functor with our Riemann-Hilbert correspondence, we obtain a second
embedding from the constructible derived category $D^{b}_{c}( \Spec(R), \F_p)$ to the derived category $D(R[F] )$. Using this construction, we recover the contravariant Riemann-Hilbert correspondence of
\cite{EK} (in a strong form, which does not require the $F_p$-algebra $R$ to be regular or even Noetherian; see Theorem \ref{strongEK}), whose statement we review in \S \ref{sec:ContravariantRH} (see also \S \ref{EKR}).
\end{itemize}

These sections are more or less independent of one another, and can be read in any order (except that \S \ref{section.duality} depends on \S \ref{sec:ContravariantRH}). One can also develop a theory which incorporates several of these refinements simultaneously (for example, one can compare derived categories of $\Z / p^{n} \Z$-sheaves on an arbitrary $\F_p$-scheme $X$ with derived categories of
quasi-coherent Frobenius modules over the Witt sheaf $W_n( \calO_X)$); we leave such extensions to the reader.

\subsection{The Work of B\"{o}ckle-Pink}\label{rmk:BP}

In the case where the $R$ is Noetherian, Theorem \ref{maintheoXXX} is essentially due to B\"{o}ckle and Pink.
Let us briefly summarize some of their work. Assume that $R$ is Noetherian, and let $\Mod_{R, \mathrm{fg}}^{\Frob}$
denote the full category of Frobenius modules $(M, \varphi_M)$ which are {\em finitely generated} as modules over $R$. 
In \cite{BP}, B\"{o}ckle and Pink construct an equivalence of abelian categories
$$\Shv^c_{\mathet}(\Spec(R),\F_p) \simeq \Mod_{R, \mathrm{fg}}^{\Frob} / \mathrm{Nil},$$
where $\mathrm{Nil}$ is the full subcategory of $\Mod_{R, \mathrm{fg}}^{\Frob}$ spanned by those Frobenius modules $(M, \varphi_M)$ where
$\varphi_M$ is nilpotent, and $\Mod_{R, \mathrm{fg}}^{\Frob} / \mathrm{Nil}$ denotes the Serre quotient (B\"{o}ckle and Pink denote this Serre quotient by
$\mathrm{Crys}(R)$ and refer to it as the category of {\it crystals} on $\Spec(R)$). From the perspective of \cite{BP}, the main innovation of this paper
is to realize the category $\mathrm{Crys}(R)$ concretely as a full subcategory $\Mod_{R}^{\Frob}$, rather than abstractly as a Serre quotient. To achieve this,
we note that every Frobenius module $(M, \varphi_M)$ admits a {\em perfection} $M^{\perfection}$, given as an abelian group by the direct limit
$$ \varinjlim( M \xrightarrow{ \varphi_M} M \xrightarrow{ \varphi_M } M \xrightarrow{ \varphi_M} M \rightarrow \cdots );$$
we will study this construction in detail in \S \ref{sec3sub2}. This construction annihilates every Frobenius module $(M, \varphi_M)$ for which $\varphi_M$ is nilpotent, and
therefore determines a functor
$$ \mathrm{Crys}(R) = \Mod_{R, \mathrm{fg}}^{\Frob} / \mathrm{Nil} \rightarrow \Mod_{R}^{\Frob}.$$
This functor is fully faithful, and its essential image is the subcategory $\Mod_{R}^{\hol} \subseteq \Mod_{R}^{\Frob}$ of {\it holonomic} Frobenius modules that
we study in \S \ref{section.holonomic}. The resulting identification of $\mathrm{Crys}(R)$ with $\Mod_{R}^{\hol}$ carries the equivalence
$\Shv^c_{\mathet}(\Spec(R),\F_p) \simeq \mathrm{Crys}(R)$ to the Riemann-Hilbert equivalence $\RH^{c}: \Shv^c_{\mathet}(\Spec(R),\F_p) \simeq \Mod_{R}^{\hol}$
of Theorem \ref{companion}.

One advantage of our approach is that it does not require the ring $R$ to be Noetherian. Beware that if $R$ is not Noetherian, then the subcategory
$\Mod_{R, \mathrm{fg}}^{\Frob} \subseteq \Mod_{R}^{\Frob}$ is not abelian, and the formalism of Serre quotients is not available. Nevertheless,
we will see that the category $\Mod_{R}^{\hol}$ of holonomic Frobenius modules is still a well-behaved abelian subcategory of $\Mod_{R}^{\Frob}$
(Corollary \ref{corX30}). Note that the extra generality afforded by allowing non-Noetherian rings can quite useful in practice: one of the main themes of the present paper is
that the theory is often clarified by replacing an $\F_p$-algebra $R$ by its perfection $R^{\perfection}$, which is almost never Noetherian.

Our realization of crystals as holonomic Frobenius modules also has the advantage of essentially trivializing the passage to derived categories in \S \ref{section.derivedRH}
(see Theorem \ref{RHDerived}). The corresponding statement in \cite{BP} requires more categorical preliminaries (largely to deal with the derived category of the Serre quotient category $\mathrm{Crys}(R)$ in a useful fashion), and does not describe the constructible derived category $D^b_c(\Spec(R), \F_p)$ as explicitly as Corollary~\ref{corollary.RHDerived}.

\begin{remark}
When $R$ is not Noetherian, we cannot realize the category $\Mod_{R}^{\hol}$ as a Serre quotient of the category $\Mod_{R, \mathrm{fg}}^{\Frob}$. Nevertheless,
it is still possible to realize the holonomic derived category $D^{b}_{\hol}( R[F] )$ as a {\em Verdier} quotient of the triangulated category of complexes
of Frobenius modules which are finitely generated and projective over $R$; see Remark \ref{remark.Verdier}.
\end{remark}

\subsection{The Work of Emerton-Kisin}\label{EKR}

The Riemann-Hilbert correspondence of Theorem \ref{maintheoXXX} is also closely related to the work of Emerton and Kisin (see \cite{EK}). In the case where $R$ is a smooth algebra over a field $k$ of characteristic $p$,
Emerton and Kisin construct an equivalence of triangulated categories
$$ \RSol_{\EK}: D^{b}_{\fgu}( R[F] )^{\op} \simeq D^{b}_{c}( \Spec(R), \F_p),$$
where $D^{b}_{\fgu}( R[F] )$ denotes the full subcategory of $D( R[F] )$ spanned by the cohomologically bounded chain complexes whose cohomology groups
{\it finitely generated unit} Frobenius modules (see Definition \ref{definition.fgu}). This differs from our Riemann-Hilbert equivalence
$\Sol: \Mod_{R}^{\hol} \simeq \Shv_{\mathet}^{c}( \Spec(R), \F_p)$ in two important respects: 
\begin{itemize}
\item The functor $\Sol: \Mod_{R}^{\hol} \simeq \Shv_{\mathet}^{c}( \Spec(R), \F_p)$ is an equivalence of abelian categories (though it can be extended to an equivalence of suitable derived categories,
see Corollary \ref{corollary.RHDerived}). However, the functor $\RSol_{\EK}$ is well-defined only at the level of derived categories: in other words, it is not t-exact. Gabber has identified the image of the abelian category of finitely generated unit Frobenius modules under the equivalence $\RSol_{\EK}$ with a certain category of {\em perverse} $\mathbf{F}_p$-sheaves inside $D^b_c(\Spec(R),\F_p)$ (see \cite{Gabber}).

\item The equivalence $\Sol: \Mod_{R}^{\hol} \simeq \Shv_{\mathet}^{c}( \Spec(R), \F_p)$ is a covariant functor, while $\RSol_{\EK}$ is contravariant.
\end{itemize}

\begin{example}
To illustrate the contrast between the Riemann-Hilbert correspondence of Theorem \ref{maintheoXXX} and the Riemann-Hilbert correspondence of \cite{EK} in more concrete terms,
let us consider an arbitrary $\F_p$-algebra $R$. A choice of non-zero divisor $t \in R$ determines closed and open immersions
$$ i: \Spec( R/tR) \hookrightarrow \Spec(R) \quad \quad j: \Spec( R[t^{-1}] ) \hookrightarrow \Spec(R),$$
so that the constant sheaf $\underline{\F_p}$ on $\Spec(R)$ fits into an exact sequence
\begin{equation}
\label{eq:Sol1}
 0 \rightarrow j_{!} \underline{\F_p} \rightarrow \underline{\F_p} \rightarrow i_{\ast} \underline{\F_p} \rightarrow 0.
 \end{equation}
Then the functor $\RH$ of Theorem \ref{maintheoXXX} carries the {\etale} sheaves $\underline{\F_p}$, $j_{!} \underline{\F_p}$, and $i_{\ast} \underline{\F_p}$
to the Frobenius modules $R^{\perfection}$, $(tR)^{\perfection}$, and $(R/tR)^{\perfection}$, respectively, and the exact sequence \eqref{eq:Sol1} to the short exact sequence 
$$ 0 \rightarrow (tR)^{\perfection} \rightarrow R^{\perfection} \rightarrow (R/tR)^{\perfection} \rightarrow 0.$$
of Frobenius modules.

On the other hand, if $R$ is a smooth algebra over a field $k$ of characteristic $p$, then the contravariant functor $\RSol_{\EK}$ from \cite{EK} carries the {\etale} sheaves $\underline{\F_p}$, $j_{!} \underline{\F_p}$, and $i_{\ast} \underline{\F_p}$ to the chain complexes of (finitely generated unit) Frobenius modules $R$, $R[t^{-1}]$, and $(R[t^{-1}] / R)[-1]$, respectively (up to a cohomological shift of $\dim(R)$: see Remark \ref{remark.normalization}). In other words, $\RSol_{\EK}$ carries exact sequence \eqref{eq:Sol1} into the the distinguished triangle
$(R[t^{-1}]/R)[-1] \to R \to R[t^{-1}]$ in the derived category $D^b_{\fgu}( \Mod_{R}^{\Frob} )$, obtained by ``rotating'' the exact sequence
$$0 \rightarrow R \rightarrow R[t^{-1}] \rightarrow R[t^{-1}] / R \rightarrow 0$$
in the category of Frobenius modules.
\end{example}

%

In \S \ref{section.duality}, we will show that the functor $\RSol_{\EK}$ fits into a commutative diagram of triangulated categories
$$ \xymatrix{ & D^{b}_{\hol}( R[F] ) \ar[dl]_{\mathbb{D}} \ar[dr]^{ \RSol} & \\
D^{b}_{\fgu}( R[F] )^{\op} \ar[rr]^{ \RSol_{\EK} } & & D^{b}_{c}( \Spec(R); \F_p ), }$$
where $\RSol$ is a derived version of our solution functor $\Sol$, and $\mathbb{D}$ denotes a certain duality functor on the derived category of Frobenius modules.
Using the fact that $\RSol$ is an equivalence of categories (which follows easily from Theorem \ref{maintheoXXX} and its proof) and that $\mathbb{D}$ is
an equivalence of categories (which we prove as Theorem \ref{theorem.dual-image}), we give a new proof of the assertion that $\RSol_{\EK}$ is an equivalence of categories. Moreover, our argument
does not require the assumption that $R$ is a smooth algebra over a field: we allow $R$ to be an arbitrary $\F_p$-algebra, with the caveat that the triangulated category $D^{b}_{\fgu}( R[F] )$
must be suitably defined (if $R$ is not a regular Noetherian ring, then the criteria for membership in the subcategory $D^{b}_{\fgu}( R[F] ) \subseteq D( R[F] )$ must be imposed at the derived level,
rather than at the level of individual cohomology groups; see Definition \ref{definition.derived-fgu}).

\begin{remark}
Let $R$ be a commutative $\F_p$-algebra, let $X = \Spec(R)$ denote the associated affine scheme, and let 
$\varphi_X: X \rightarrow X$ denote the absolute Frobenius map.
A Frobenius module over $R$ can be identified with a quasi-coherent sheaf $\calE$ on $X$ equipped with a map $\varphi_{\calE}: \calE \rightarrow \varphi_{X \ast} \calE$,
or equivalently with a map $\psi_{\calE}: \varphi_{X}^{\ast} \calE \rightarrow \calE$. In this paper, we will be primarily concerned with the class of {\em perfect} Frobenius modules, characterized by the requirement that the map $\varphi_{\calE}$ is an isomorphism. By contrast, the book \cite{EK} is mainly concerned with the class of {\it unit} Frobenius modules, characterized by the requirement that the map $\psi_{\calE}$ is an isomorphism.
Note that the direct image functor $\varphi_{X \ast}$ is always exact (since $\varphi_X$ is affine morphism), but the exactness of the pullback functor $\varphi_{X}^{\ast}$ requires some strong hypotheses on $R$
(for example, that $R$ is a regular Noetherian ring). Consequently, the category $\Mod_{R}^{\perf}$ of perfect Frobenius modules is always abelian, but the category of finitely generated unit
Frobenius modules is well-behaved only in special cases.
\end{remark}

\subsection{Other Related Works}

Extensions of the contravariant Riemann-Hilbert correspondence of \cite{EK} to singular schemes have also been explored in the papers of Blickle-B\"{o}ckle \cite{BB1,BB2}, Schedlmeier \cite{Sched}, and Ohkawa \cite{Ohkawa}.
Under mild finiteness conditions on $R$, the papers \cite{BB1,BB2} develop a theory of ``Cartier" modules: these are coherent sheaves $\calE$ on $X = \Spec(R)$ equipped with a map
$C_{\calE}: \varphi_{X \ast} \calE \rightarrow \calE$. Passing to a quotient by a naturally defined subcategory of nilpotent objects yields the category $\mathrm{Crys}_{\mathrm{Cart}}(R)$ of {\em Cartier crystals} on $R$. For $X$ smooth over a perfect field, the category $\mathrm{Crys}_{\mathrm{Cart}}(R)$ is identified in \cite{BB1} with the category of finitely generated unit Frobenius modules (and thus with the category of perverse $\F_p$-sheaves on $\Spec(R)_{\mathet}$ via \cite{EK} and \cite{Gabber}). Roughly speaking, the smoothness of $R$ ensures that $\varphi^! \calE \simeq \varphi^* \calE \otimes \omega_{X}^{1-p}$, so a Cartier module $C_{\calE}: \varphi_{X \ast} \calE \to \calE$ gives by adjunction a map $\alpha: \calE \otimes \omega_X^{-1} \to \varphi_{X}^*(\calE \otimes \omega_X^{-1})$ whose unitalization (see Construction~\ref{unitalize}) yields the desired finitely generated unit Frobenius module. For $R$ not necessarily regular, even though finitely generated unit Frobenius modules may be badly behaved, the abelian category $\mathrm{Crys}_{\mathrm{Cart}}(R)$ is shown to have good behaviour in \cite{BB1,BB2}; see also \cite{Gabber}. Using this, the paper \cite{Sched} shows that, given an embedding $X \hookrightarrow \Spec(P)$ with $P$ smooth over a perfect field, a suitably defined derived category of $\mathrm{Crys}_{\mathrm{Cart}}(R)$ is equivalent to the full subcategory of $D^b_{\fgu}(P[F])$ spanned by complexes supported on $X$. Combining this with the Riemann-Hilbert correspondence from \cite{EK} for $P$ and a suitable analogue of Kashiwara's theorem, one obtains a description of
the constructible derived category $D^b_c(\Spec(R),\F_p)$ in terms of Cartier crystals on $R$ (\cite{Ohkawa,Sched}). 

\begin{remark}
The most important example of a Cartier module is the coherent dualizing sheaf $\omega_{X}$ for $X = \Spec(R)$, where we take $C_{ \omega_{X} }: \varphi_{X \ast} \omega_{X} \rightarrow \omega_{X}$
to be the Grothendieck dual of the unit map $\calO_X \rightarrow \varphi_{X \ast} \calO_X$. In the case where $X$ is a smooth scheme over a perfect field $k$, this map can also be described in terms of
the Cartier operator on the de Rham complex of $X$, which motivates the terminology.
\end{remark}

In summary, the papers \cite{BB1,BB2,Ohkawa,Sched} extend the Riemann-Hilbert correspondence of \cite{EK} to algebras of finite type over a field by developing the theory of Cartier modules
and reducing to the smooth case. In contrast, the discussion in \S \ref{sec:ContravariantRH} gives an intrinsic extension of the Riemann-Hilbert correspondence of \cite{EK} to all $\F_p$-algebras $R$, which
is described using Frobenius modules (that is, quasi-coherent sheaves with a map $\calE \rightarrow \varphi_{X \ast} \calE$). The presentation via Cartier crystals gives a module-theoretic description of perverse $\F_p$-sheaves on $\Spec(R)_{\mathet}$ (which has an important precursor in \cite{Gabber}), while the presentation in \S \ref{sec:ContravariantRH} works entirely in the derived category and thus avoids discussion of the abelian category of perverse sheaves. 

\subsection{Acknowledgements}

During the period in which this work was carried out, the first author was supported by National Science Foundation under grant number 1501461 and a Packard fellowship, and the second author was supported by the National Science Foundation under grant number 1510417. We thank Manuel Blickle and Johan de Jong for useful conversations.

\newpage



\newpage \section{Overview}\label{section.overview}
\setcounter{subsection}{0}
\setcounter{theorem}{0}

Our goal in this section is to give a precise formulation of Theorem \ref{maintheoXXX}. We begin by introducing two of the principal objects of interest in this paper:
the category $\Mod_{R}^{\Frob}$ of {\it Frobenius modules} over a commutative $\F_p$-algebra $R$ (Definition \ref{definition.frobenius-module}) and the category $\Shv_{\mathet}( \Spec(R), \F_p )$ of {\it $p$-torsion {\etale} sheaves}
on $\Spec(R)$ (Definition \ref{definition.etalesheaf}). The Riemann-Hilbert correspondence of Theorem \ref{maintheoXXX} is a fully faithful embedding of categories
$$ \RH: \Shv_{\mathet}( \Spec(R), \F_p) \rightarrow \Mod_{R}^{\Frob}.$$
It will take a bit of work to construct this functor (this is the main objective of \S \ref{section.RH}). In this section, we consider instead the {\it solution sheaf} functor
$\Sol: \Mod_{R}^{\Frob} \rightarrow \Shv_{\mathet}( \Spec(R), \F_p)$, which is left inverse to the Riemann-Hilbert functor $\RH$ and admits a very simple description
(Definition \ref{defSol}). We then formulate a variant of Theorem \ref{maintheoXXX}, which asserts that the functor $\Sol$ becomes an equivalence
when restricted to a certain full subcategory $\Mod_{R}^{\alg} \subseteq \Mod_{R}^{\Frob}$ (Theorem \ref{theoX50}).

\subsection{Frobenius Modules}
\label{DefFmod}

We begin by introducing some definitions.

\begin{definition}\label{definition.frobenius-module}
Let $R$ be a commutative $\F_p$-algebra. A {\it Frobenius module over $R$} is an $R$-module $M$ equipped with an additive map
$\varphi_M: M \rightarrow M$ satisfying the identity $\varphi_M( \lambda x) = \lambda^{p} \varphi_M(x)$ for $x \in M$, $\lambda \in R$.

Let $(M, \varphi_M)$ and $(N, \varphi_N)$ be Frobenius modules over $R$. A {\it morphism of Frobenius modules} from $(M, \varphi_M)$ to $(N, \varphi_N)$ is an
$R$-module homomorphism $\rho: M \rightarrow N$ for which the diagram
$$ \xymatrix{ M \ar[r]^{\rho} \ar[d]^{\varphi_M} & N \ar[d]^{\varphi_N} \\
M \ar[r]^{\rho} & M' }$$
commutes. We let $\Mod_{R}^{\Frob}$ denote the category whose objects are Frobenius modules $(M, \varphi_M)$ and whose morphisms are morphisms of Frobenius modules.
We will refer to $\Mod_{R}^{\Frob}$ as the {\it category of Frobenius modules over $R$}. 
\end{definition}

\begin{notation}
Let $M$ and $N$ be Frobenius modules over a commutative $\F_p$-algebra $R$. We let $\Hom_{R}^{\Frob}(M, N)$ denote the set of Frobenius module morphisms from $(M, \varphi_M)$ to $(N, \varphi_N)$.
\end{notation}

\begin{remark}
Let $(M, \varphi_M)$ be a Frobenius module over a commutative $\F_p$-algebra $R$. We will often abuse terminology by simply referring to $M$ as a Frobenius module over $R$: in this case, we are implicitly asserting that $M$ is equipped with a Frobenius-semilinear map $\varphi_M: M \rightarrow M$. 
\end{remark}

\begin{example}\label{exX3}
Let $R$ be a commutative $\F_p$-algebra. Then we can regard $R$ as a Frobenius module over itself, via the Frobenius map
$$\varphi_{R}: R \rightarrow R \quad \quad \varphi_{R}(\lambda) = \lambda^{p}.$$
More generally, the same comment applies to an ideal $I \subseteq R$.
\end{example}

It will sometimes be helpful to identify Frobenius modules over a commutative $\F_p$-algebra $R$ with modules over a certain (noncommutative) enlargement of $R$.

\begin{notation}\label{not2}
Let $R$ be an $\F_p$-algebra. We let $R[F]$ denote the noncommutative ring whose elements are finite sums
$\sum_{n \geq 0} c_n F^{n}$, with multiplication given by
$$ ( \sum_{m \geq 0} c_m F^{m} ) ( \sum_{ n \geq 0 } c'_{n} F^{n} ) = \sum_{ k \geq 0 } ( \sum_{ i + j = k } c_{i} c_{j}^{ p^{i} } ) F^{k}.$$
We will identify $R$ with the subring of $R[F]$ consisting of those sums $\sum_{n \geq 0} c_n F^{n}$ for which the coefficients $c_i$ vanish for $i > 0$.
Unwinding the definitions, we see that the category $\Mod_{R}^{\Frob}$ is equivalent to the category of left $R[F]$-modules. In particular, $\Mod_{R}^{\Frob}$ is an abelian category
which has enough projective objects and enough injective objects. Given objects $M, N \in \Mod_{R}^{\Frob}$, we let
$\Ext^{n}_{R[F]}( M, N)$ denote the $n$th $\Ext$-group of $M$ and $N$, computed in the abelian category $\Mod_{R}^{\Frob}$.
\end{notation}

We now consider the behavior of the category $\Mod_{R}^{\Frob}$ as the commutative $\F_p$-algebra $R$ varies. 

\begin{construction}[Extension of Scalars]\label{construction.extension}
Let $f: A \rightarrow B$ be a homomorphism of commutative $\F_p$-algebras. If $M$ is a Frobenius module over $A$,
then we can regard the tensor product $B \otimes_{A} M$ as a Frobenius module over $B$, with Frobenius map given by
$$ \varphi_{B \otimes_{A} M}( b \otimes x) = b^{p} \otimes \varphi_{M}(x).$$
The construction $M \mapsto B \otimes_{A} M$ determines a functor from $\Mod_{A}^{\Frob}$ to $\Mod_{B}^{\Frob}$, which
we will denote by $f^{\ast}_{ \Frob}$ and refer to as {\it extension of scalars along $f$}.
\end{construction}

\begin{remark}[Restriction of Scalars]\label{remark.restriction}
In the situation of Construction \ref{construction.extension}, the extension of scalars $f^{\ast}_{\Frob}(M) = B \otimes_{A} M$
is characterized by the following universal property: for every Frobenius module $N$ over $B$, composition with
the map $M \rightarrow B \otimes_{A} M$ induces a bijection
$$ \Hom_{ B[F] }( B \otimes_{A} M, N) \rightarrow \Hom_{ A[F] }( M, N ).$$
In other words, we can regard the functor $f^{\ast}_{\Frob}$ as a left adjoint to the forgetful functor
$\Mod_{B}^{\Frob} \rightarrow \Mod_{A}^{\Frob}$. We will denote this forgetful functor by $f_{\ast}$ and
refer to it as {\it restriction of scalars along $f$}.
\end{remark}

\begin{warning}
Let $f: A \rightarrow B$ be a homomorphism of commutative $\F_p$-algebras. Then $f$ induces extends to homomorphism of noncommutative
rings $f^{+}: A[F] \rightarrow B[F]$, where $A[F]$ and $B[F]$ are defined as in Notation \ref{not2}. The content of Construction \ref{construction.extension} and
Remark \ref{remark.restriction} is that, for every Frobenius module $M$ over $A$, the canonical map 
$$B \otimes_{A} M \rightarrow B[F] \otimes_{A[F] } M$$
is an isomorphism. This relies on the fact that $A[F]$ is freely generated as a {\em left} $A$-module by the elements
$\{ F^n \}_{n \geq 0}$. Beware that $A[F]$ is usually not free when regarded as a {\em right} $A$-module (so the analogous
compatibility would fail if we were to study right modules over $A[F]$, rather than left modules). In fact, following Notation~\ref{not:FrobRestriction}, the ring $A[F]$ identifies with $\oplus_{n \geq 0} A \cdot F^i$ as a left $A$-module, and with $\oplus_{n \geq 0} F^i \cdot A^{1/p^n}$ as a right $A$-module. In particular, the latter is free over $A$ only under strong conditions (such as regularity). 
\end{warning}

\begin{remark}\label{remark.cor20}
Let $f: A \rightarrow B$ be a homomorphism of commutative $\F_p$-algebras, and suppose that the multiplication map $m: B \otimes_{A} B \rightarrow B$
is an isomorphism (this condition is satisfied, for example, if $f$ is surjective, or if $f$ exhibits $B$ as a localization of $A$). Then, for any Frobenius module $M$
over $B$, the counit map
$v: f^{\ast}_{\Frob} f_{\ast}(M) \rightarrow M$ is also an isomorphism: note that the domain of $v$ can be identified with the tensor product
$B \otimes_{A} M \simeq (B \otimes_{A} B) \otimes_{B} M$. It follows that the restriction of scalars functor 
$f_{\ast}: \Mod_{B}^{\Frob} \rightarrow \Mod_{A}^{\Frob}$ is fully faithful.
\end{remark}

\subsection{{\Etale} Sheaves}

For the reader's convenience, we briefly review the theory of {\etale} sheaves. We consider here only the case of {\em affine} schemes (we will discuss sheaves on more general schemes in
\S \ref{section.global}).

\begin{notation}
Let $R$ be a commutative ring. We let $\CAlg_{R}^{\mathet}$ denote the category whose objects are {\etale} $R$-algebras, and whose morphisms are $R$-algebra homomorphisms.
\end{notation}

\begin{definition}\label{definition.etalesheaf}
Let $R$ and $\Lambda$ be commutative rings, and let $\Mod_{\Lambda}$ denote the category of $\Lambda$-modules. An 
{\it {\etale} sheaf of $\Lambda$-modules on $\Spec(R)$} is a functor
$$ \sheafF: \CAlg_{R}^{\mathet} \rightarrow \Mod_{\Lambda}$$
which satisfies the following pair of conditions:
\begin{itemize}
\item For every faithfully flat map $u: A \rightarrow B$ in $\CAlg_{R}^{\mathet}$, the sequence
$$ 0 \rightarrow \sheafF(A) \xrightarrow{ \sheafF(u)} \sheafF(B) \xrightarrow{ \sheafF(u \otimes \id) - \sheafF(\id \otimes u)} \sheafF( B \otimes_{A} B )$$
is exact.
\item For every finite collection of {\etale} $R$-algebras $\{ A_i \}_{i \in I}$, the map $$\sheafF( \prod_{i \in I} A_i ) \rightarrow \prod_{i \in I} \sheafF(A_i)$$ is an isomorphism.
\end{itemize}
We let $\Shv_{\mathet}( \Spec(R), \Lambda)$ denote the category whose objects are {\etale} sheaves of $\Lambda$-modules on $\Spec(R)$ (where morphisms are given by natural transformations).
\end{definition}

\begin{remark}
In this paper, we will be concerned almost exclusively with the case where $\Lambda$ is the finite field $\F_p$. The only exception is in \S \ref{section.moretorsion}, where we take
$\Lambda = \Z / p^{n} \Z$ for some integer $n \geq 0$.
\end{remark}

\begin{example}[Constant Sheaves]
Let $R$ be a commutative ring and let $M$ be a module over a commutative ring $\Lambda$. We let $\underline{M} \in \Shv_{\mathet}( \Spec(R), \Lambda)$
denote the functor which associates to each {\etale} $R$-algebra $A$ the set $\underline{M}(A)$ of locally constant $M$-valued functions on $\Spec(A)$.
We will refer to $\underline{M}$ as the {\it constant sheaf with value $M$}.
\end{example}

\begin{example}[Quasi-Coherent Sheaves]\label{exX70}
Let $R$ be an $\F_p$-algebra. For every $R$-module $M$, the construction $(A \in \CAlg_{R}^{\mathet}) \mapsto A \otimes_{R} M$
determines an {\etale} sheaf of $\F_p$-modules on $M$, which we denote by $\widetilde{M}$ (see \cite[Tag 03DX]{Stacks}). Note that if $M$ is a Frobenius module over $R$, then $\varphi_{M}$
determines a map of {\etale} sheaves $\widetilde{\varphi}_{M}: \widetilde{M} \rightarrow \widetilde{M}$; this map is $\mathbf{F}_p$-linear, but not $R$-linear in general. 
\end{example}

\begin{notation}
Let $R$ and $\Lambda$ be commutative rings. If $\sheafF$ and $\sheafG$ are {\etale} sheaves of $\Lambda$-modules on $\Spec(R)$, we let
$\Hom_{ \underline{\Lambda} }( \sheafF, \sheafG )$ denote the abelian group of morphisms from $\sheafF$ to $\sheafG$ in the
category $\Shv_{\mathet}( \Spec(R), \Lambda)$ (emphasizing the idea that $\sheafF$ and $\sheafG$ can be regarded as modules over the constant sheaf $\underline{\Lambda}$). 

Note that $\Shv_{\mathet}( \Spec(R), \Lambda)$ is an abelian category with enough injective objects, so that we can consider
$\Ext$-groups in $\Shv_{\mathet}( \Spec(R), \Lambda)$. We denote these $\Ext$-groups by $\Ext^{n}_{ \underline{\Lambda} }( \sheafF, \sheafG)$
for $n \geq 0$.
\end{notation}

\begin{remark}[Functoriality]
Let $f: A \rightarrow B$ be a homomorphism of commutative rings. Then $f$ induces a base change functor
$$\CAlg_{A}^{\mathet} \rightarrow \CAlg_{B}^{\mathet} \quad \quad A' \mapsto A' \otimes_{A} B.$$
Precomposition with this functor determines a pushforward functor $$f_{\ast}: \Shv_{\mathet}( \Spec(B), \Lambda) \rightarrow \Shv_{\mathet}( \Spec(A), \Lambda),$$
given concretely by the formula $( f_{\ast} \sheafF)(A') = \sheafF(A' \times_{A} B)$.
The functor $f_{\ast}$ admits a left adjoint $f^{\ast}: \Shv_{\mathet}( \Spec(A), \Lambda) \rightarrow \Shv_{\mathet}( \Spec(B), \Lambda)$, which we
refer to as {\it pullback along $f$}. 
\end{remark}

\subsection{The Solution Functor}

\begin{construction}\label{construction.solsheaf}
Let $R$ be a commutative $\F_p$-algebra and let $M$ be a Frobenius module over $R$.
For every {\etale} $R$-algebra $A$, we regard the tensor product $A \otimes_{R} M$ as a Frobenius module over $A$ (see Construction \ref{construction.extension}).
We define
$$ \Sol(M)(A) = \{ x \in (A \otimes_{R} M): \varphi_{A \otimes_{R} M}(x) = x \}.$$
The construction $A \mapsto \Sol(M)$ determines a functor $\CAlg_{R}^{\mathet} \rightarrow \Mod_{\F_p}$, which we will
refer to as {\it the solution sheaf} of $M$.
\end{construction}

\begin{proposition}\label{propX74weak}
Let $R$ be a commutative $\F_p$-algebra and let $M$ be a Frobenius module over $R$. Then $\Sol(M)$ is an {\etale} sheaf of $\F_p$-modules on $\Spec(R)$.
\end{proposition}

\begin{proof}
Let $\widetilde{M}$ denote the quasi-coherent sheaf associated to $M$ (Example \ref{exX70}). It follows immediately from the definition that
$\Sol(M)$ can be described as the kernel of the map
$$ (\id - \widetilde{\varphi}_M): \widetilde{M} \rightarrow \widetilde{M}.$$
Since $\widetilde{M}$ is an {\etale} sheaf of $\F_p$-modules on $\Spec(R)$, the functor $\Sol(M)$ has the same property.
\end{proof}

\begin{definition}\label{defSol}
Let $R$ be a commutative $\F_p$-algebra. We will regard the construction
$$ (M \in \Mod_{R}^{\Frob}) \mapsto ( \Sol(M) \in \Shv_{\mathet}( \Spec(R), \F_p) )$$
as a functor $\Sol: \Mod_{R}^{\Frob} \rightarrow \Shv_{\mathet}( \Spec(R), \F_p )$. We will refer to
$\Sol$ as {\it the solution functor}.
\end{definition}

\begin{remark}\label{remark.leftexact}
The solution functor $\Sol: \Mod_{R}^{\Frob} \rightarrow \Shv_{\mathet}( \Spec(R), \F_p )$ is left exact. However, it is usually not right exact.
\end{remark}

\subsection{The Riemann-Hilbert Correspondence}

We now introduce a class of Frobenius modules on which the solution functor of Definition \ref{defSol} is particularly well-behaved.

\begin{definition}\label{defhol1}
Let $R$ be a commutative $\F_p$-algebra and let $M$ be a Frobenius module over $R$. We will say that $M$ is
{\it algebraic} if it satisfies the following conditions:
\begin{itemize}
\item[$(a)$] The map $\varphi_{M}: M \rightarrow M$ is an isomorphism of abelian groups.
\item[$(b)$] Every element $x \in M$ satisfies an equation of the form $$\varphi_{M}^{n}(x) + a_1 \varphi_{M}^{n-1}(x) + \cdots + a_n x = 0$$ for some coefficients $a_i \in R$.
\end{itemize}
We let $\Mod_{R}^{\alg}$ denote the full subcategory of $\Mod_{R}^{\Frob}$ spanned by the algebraic Frobenius modules over $R$.
\end{definition}

\begin{remark}
\label{rmk:IndHolCriterion}
Let $R$ be a commutative $\F_p$-algebra and let $M$ be a Frobenius module over $R$. Then condition $(b)$ of Definition \ref{defhol1} is equivalent to the following:
\begin{itemize}
\item[$(b')$] For every finitely generated $R$-submodule $M_0 \subseteq M$, the Frobenius submodule of $M$ generated by $M_0$ is also finitely generated as an $R$-module.
\end{itemize}
To see that $(b)$ implies $(b')$, suppose that $M_0 \subseteq M$ is the $R$-submodule generated by finitely many elements $\{ x_i \}_{i \in I}$. Then condition $(b)$ guarantees that
the $R$-submodule generated by $\{ \varphi_{M}^{k}( x_i) \}_{i \in I, k \leq n}$ is stable under the action of $\varphi_{M}$ for some $n \gg 0$; this is clearly the smallest $\varphi_{M}$-stable
submodule of $M$ which contains $M_0$.

Conversely, suppose that $(b')$ is satisfied and let $x$ be an element of $M$. Applying condition $(b')$ to the submodule $M_0 = Rx \subseteq M$, we see that the sum
$\sum_{k \geq 0} R \varphi_{M}^{k}(x) \subseteq M$ is generated by finitely many elements, and is therefore contained in $\sum_{0 \leq k < n} R \varphi^{k}_{M}(x)$
for some integer $n$. It follows that $\varphi^{n}_{M}(x)$ can be written as a linear combination $a_1 \varphi^{n-1}_M(x) + \cdots + a_n x$ for some coefficients
$a_1, a_2, \ldots, a_n \in R$.
\end{remark}

We can now formulate the main result of this paper:

\begin{theorem}\label{theoX50}
Let $R$ be a commutative $\F_p$-algebra. Then the solution sheaf functor
$\Sol: \Mod_{R}^{\alg} \rightarrow \Shv_{\mathet}( \Spec(R), \F_p )$
is an equivalence of categories.
\end{theorem}

Note that Theorem \ref{theoX50} immediately implies Theorem \ref{maintheoXXX}: assuming Theorem \ref{theoX50}, the Riemann-Hilbert functor
$\RH$ can be defined as the composition
$$ \Shv_{\mathet}( \Spec(R), \F_p) \xrightarrow{ \Sol^{-1} } \Mod_{R}^{\alg} \hookrightarrow \Mod_{R}^{\Frob}.$$
We will give a different (but ultimately equivalent) definition of the functor $\RH$ in \S \ref{section.RH}: the construction of this functor
is one of the key ingredients in our proof of Theorem \ref{theoX50}, which we present in \S \ref{section.mainproof}.

\newpage \section{The Category of Frobenius Modules}\label{section.frobenius-module}
\setcounter{subsection}{0}
\setcounter{theorem}{0}

Let $R$ be a commutative $\F_p$-algebra. Our goal in this section is to establish some elementary properties of the abelian category
$\Mod_{R}^{\Frob}$ of Frobenius modules over $R$. We begin in \S \ref{sec3sub1} by studying the forgetful functor from $\Mod_{R}^{\Frob}$ to the category of $R$-modules.
The main observation is that the ring $R$ admits a very simple resolution 
$$0 \rightarrow R[F] \xrightarrow{F-1} R[F] \rightarrow R \rightarrow 0$$
in the category of left modules over the noncommutative ring $R[F]$ appearing in Notation \ref{not2}. This allows us to reduce calculations
of $\Ext$-groups in the category $\Mod_{R}^{\Frob}$ to calculations of $\Ext$-groups in
the category $\Mod_{R}$: see Construction \ref{conX7}.

In \S \ref{sec3sub2}, we restrict our attention to the class of {\it perfect} Frobenius modules: that is, Frobenius modules $M$ for which
the map $\varphi_{M}: M \rightarrow M$ is bijective (Definition \ref{definition.perfect}). The collection of Frobenius modules with this property form a category which
we denote by $\Mod_{R}^{\perf}$. Our main result is that the inclusion functor $\Mod_{R}^{\perf} \hookrightarrow \Mod_{R}^{\Frob}$ admits an exact left adjoint
$M \rightarrow M^{\perfection}$, given informally by ``inverting the Frobenius'' (Proposition \ref{proposition.perfection}).

Let $f: A \rightarrow B$ be a homomorphism of commutative $\F_p$-algebras. In \S \ref{DefFmod}, we observed that extension of scalars
along $f$ determines a functor $f^{\ast}_{\Frob}: \Mod_{A}^{\Frob} \rightarrow \Mod_{B}^{\Frob}$. Beware that this construction does {\em not}
carry perfect Frobenius modules to perfect Frobenius modules. To remedy this, we introduce in \S \ref{sec3sub3} another functor
$f^{\diamond}: \Mod_{A}^{\perf} \rightarrow \Mod_{B}^{\perf}$, given concretely by the formula $f^{\diamond}(M) = (f^{\ast}_{\Frob} M)^{\perfection}$.
The functors $f^{\diamond}$ and $f^{\ast}_{\Frob}$ are generally different, but they agree when the ring homomorphism $f$ is {\etale} (Corollary \ref{corX14}).
The proof of this fact will require some elementary facts about perfect rings of characteristic $p$, which we review in \S \ref{sec3sub4}.


\subsection{Comparison of $R[F]$-Modules with $R$-Modules}\label{sec3sub1}

Throughout this section, we fix a commutative $\F_p$-algebra $R$. 

\begin{notation}
\label{not:FrobRestriction}
Let $\Mod_{R}$ denote the abelian category of $R$-modules. For each $n \geq 0$, restriction of scalars along the $n$th power of the Frobenius map $\varphi_{R}: R \rightarrow R$ determines a forgetful functor
from $\Mod_{R}$ to itself, which we will denote by $M \mapsto M^{ 1/ p^{n} }$. 

Let $M$ be an $R$-module. Then there is a canonical isomorphism of abelian groups $M \simeq M^{1/p^{n} }$.
For each element $x \in M$, we will denote the image of $x$ under this isomorphism by $x^{1/p^{n} }$. The action of $R$ on $M^{1/p^{n}}$ can then be described by the formula $\lambda(x^{1/p^{n}}) = ( \lambda^{p^{n}} x)^{1/p^{n}}$, for $\lambda \in R$ and $x \in M$.
\end{notation}

\begin{construction}\label{conX72}
Let $N$ be an $R$-module. We let $N^{\dagger}$ denote the $R$-module given by the product $\prod_{n \geq 0} N^{1/p^{n} }$.
We identify elements of $N^{\dagger}$ with the collection of all sequences $(x_0, x_1, x_2, \ldots )$ in $N$, where the action of $R$ is given by
$$ \lambda ( x_0, x_1, x_2, \ldots ) = ( \lambda x_0, \lambda^{p} x_1, \lambda^{p^{2}} x_2, \ldots ).$$
We regard $N^{\dagger}$ as an Frobenius module over $R$, with endomorphism
$\varphi_{N^{\dagger} }: N^{\dagger} \rightarrow N^{\dagger}$ given by
$$ \varphi_{N^{\dagger}}( x_0, x_1, x_2, \cdots ) = ( x_1, x_2, x_3, \cdots).$$ 
\end{construction}

\begin{lemma}\label{lemX1}
Let $M$ be a Frobenius module over $R$, let $N$ be an arbitrary $R$-module, and let $v: N^{\dagger} \rightarrow N$ be the $R$-module homomorphism
given by $$v(x_0, x_1, x_2, \ldots) = x_0.$$ Then composition with $v$ induces a bijection
$\Hom_{R}^{\Frob}( M, N^{\dagger} ) \rightarrow \Hom_{R}( M, N )$.
\end{lemma}

\begin{proof}
For every $R$-module homomorphism $f: M \rightarrow N$, define $f^{+}: M \rightarrow N^{\dagger}$ by the formula
$f^{+}( x ) = ( f(x), f( \varphi_M(x) ), f( \varphi_M^2(x) ), \ldots )$. An elementary calculation shows that the construction $f \mapsto f^{+}$
determines an inverse to the map $\Hom_{R}^{\Frob}( M, N^{\dagger} ) \rightarrow \Hom_{R}( M, N )$ given by composition with $v$.
\end{proof}

\begin{remark}\label{remX2}
It follows from Lemma \ref{lemX1} that we can regard the construction $N \mapsto N^{\dagger}$ as a right adjoint to the forgetful functor
$\Mod_{R}^{\Frob} \rightarrow \Mod_{R}$. 
\end{remark}

\begin{remark}
Using the equivalence of Notation \ref{not2}, we can identify the forgetful map $\Mod_{R}^{\Frob} \rightarrow \Mod_{R}$ with the functor given by
restriction of scalars along the ring homomorphism $R \rightarrow R[F]$. The right adjoint to this restriction of scalars functor is given by
$M \mapsto \Hom_{ R}( R[F], M )$, which we can identify with $M^{\dagger}$ using the canonical left $R$-module basis of $R[F]$ given by
$\{ F^{n} \}_{n \geq 0}$. 
\end{remark}

\begin{remark}\label{remX3}
Let $M$ be a Frobenius module over $R$ and let $N$ be an arbitrary $R$-module. Then the identification
$\Hom_{R[F]}(M, N^{\dagger}) \simeq \Hom_{R}(M, N)$ of Lemma \ref{lemX1} extends to an isomorphism of graded abelian groups
$$\Ext^{\ast}_{R[F]}( M, N^{\dagger} ) \simeq \Ext^{\ast}_{R}(M, N).$$ This isomorphism can be described explicitly by choosing a projective resolution
$$ \cdots \rightarrow P_2 \rightarrow P_1 \rightarrow P_0 \rightarrow M$$
in the abelian category $\Mod_{R}^{\Frob}$. Note that each $P_{n}$ is also projective when regarded as an $R$-module (this follows from the observation
that the algebra $R[F]$ of Notation \ref{not2} is free as a left $R$-module), so that both
$\Ext^{\ast}_{R[F]}( M, N^{\dagger} )$ and $\Ext^{\ast}_{R}(M, N)$ can be computed as the cohomology of the cochain complex of abelian groups
$\Hom_{R[F]}( P_{\ast}, N^{\dagger}) \simeq \Hom_{R}( P_{\ast}, N)$.
\end{remark}

\begin{construction}\label{conX7}
Let $N$ be a Frobenius module over $R$. Then the construction $(x \in N) \mapsto ( x, \varphi_N(x), \varphi_{N}^2(x), \cdots)$
determines a morphism of Frobenius modules $u: N \rightarrow N^{\dagger}$ (which is a unit map for the adjunction of Remark \ref{remX2}). 
Note that $u$ is a monomorphism which fits into a short exact sequence of Frobenius modules
$$ 0 \rightarrow N \xrightarrow{u} N^{\dagger} \xrightarrow{\alpha} (N^{1/p})^{\dagger} \rightarrow 0,$$
where $\alpha$ is given by the formula $$\alpha(x_0, x_1, x_2, \ldots) = ( \varphi_N(x_0) - x_1, \varphi_{N}(x_1) - x_2, \varphi_{N}(x_2) - x_3, \ldots ).$$ 
It follows that for any other Frobenius module $M$ over $R$, we have a short exact sequence of abelian groups
$$ 0 \rightarrow \Hom_{R}^{\Frob}(M, N) \rightarrow \Hom_{R}( M, N) \xrightarrow{\beta} \Hom_{R}(M, N^{1/p}),$$
where $\beta$ is given by the formula $\beta(f)(x) = \varphi_{N}(f(x))^{1/p} - f( \varphi_M(x))^{1/p}$. Moreover, if $M$ is a projective object of
$\Mod_{R}^{\Frob}$, then $\beta$ is surjective. More generally, Remark \ref{remX3} supplies a long exact sequence of abelian groups
$$ \cdots \rightarrow \Ext^{\ast-1}_{R}( M, N^{1/p}) \rightarrow \Ext^{\ast}_{R[F]}( M, N) \rightarrow \Ext^{\ast}_{R}(M, N) \rightarrow
\Ext^{\ast}_{R}(M, N^{1/p} ) \rightarrow \cdots.$$
\end{construction}

\begin{remark}\label{remark.findimension}
Let $R$ be a commutative $\F_p$-algebra and let $M$ be a Frobenius module over $R$. It follows from Construction \ref{conX7} that if
$M$ has projective (injective) dimension $\leq n$ as a module over $R$, then it has projective (injective) dimension $\leq n+1$ as a module over $R[F]$.
\end{remark}

\subsection{Perfect Frobenius Modules}\label{sec3sub2}

Let $R$ be a commutative $\F_p$-algebra, which we regard as fixed throughout this section.

\begin{definition}\label{definition.perfect}
Let $M$ be a Frobenius module over $R$. We will say that $M$ is {\it perfect} if
the map $\varphi_{M}: M \rightarrow M$ is an isomorphism of abelian groups. We let
$\Mod_{R}^{\perf}$ denote the full subcategory of $\Mod_{R}^{\Frob}$ spanned by the {perfect} Frobenius modules over $R$.
\end{definition}

\begin{remark}\label{remark.perfectextension}
The full subcategory $\Mod_{R}^{\perf} \subseteq \Mod_{R}^{\Frob}$ is closed under limits, colimits, and extensions. In particular,
$\Mod_{R}^{\perf}$ is an abelian category, and the inclusion functor $\Mod_{R}^{\perf} \hookrightarrow \Mod_{R}^{\Frob}$ is exact.
\end{remark}

\begin{notation}\label{biskar}
Let $M$ be a Frobenius module over $R$. We let $M^{\perfection}$ denote the colimit of the sequence
$$ M \xrightarrow{ \varphi_{M} } M^{1/p} \xrightarrow{ \varphi_{M} } 
M^{1/p^2} \rightarrow \cdots$$
We will refer to $M^{\perfection}$ as the {\it perfection} of $M$.
\end{notation}

\begin{example}\label{exX6}
Let us regard $R$ as a Frobenius module over itself as in Example \ref{exX3}. Then $R^{\perfection}$ is the perfection of $R$ in the sense of commutative algebra: that is, it is an initial object in the category of $R$-algebras in which every element admits a unique $p$th root.
\end{example}

\begin{example}\label{skriff}
Let $R[F]$ be as in Notation \ref{not2}, which we regard as a Frobenius module over $R$. We will denote the perfection of $R[F]$ by $R^{\perfection}[ F^{\pm 1} ]$.
Unwinding the definitions, we can identify elements of $R^{\perfection}[ F^{\pm 1} ]$ with expressions of the form $\sum_{n \in \Z} c_n F^{n}$
where the coefficients $c_{n} \in R^{\perfection}$ vanish for all but finitely many integers $n$. 
\end{example}

\begin{remark}\label{remX4}
The set $R^{\perfection}[ F^{\pm 1} ]$ has the structure of an associative ring,
with multiplication given by the formula
$$ ( \sum_{m \in \Z} c_m F^{m} ) ( \sum_{ n \in \Z } c'_{n} F^{n} ) = \sum_{ k \in \Z } ( \sum_{ i + j = k } c_{i} c_{j}^{ p^{i} } ) F^{k}.$$
This ring can be obtained from the associative ring $R[F]$ by formally adjoining an inverse of the element $F$. It follows that
the equivalence of $\Mod_{R}^{\Frob}$ with the category of left $R[F]$-modules restricts to an equivalence of
$\Mod_{R}^{\perf}$ with the category of left $R^{\perfection}[ F^{\pm 1} ]$-modules. In particular, $\Mod_{R}^{\perf}$ is an abelian category which
has enough projective objects and enough injective objects.
\end{remark}

In the situation of Notation \ref{biskar}, the perfection $M^{\perfection}$ inherits the structure of a Frobenius module. Moreover, it enjoys the following universal property:

\begin{proposition}\label{proposition.perfection}
The inclusion functor $\iota: \Mod_{R}^{\perf} \hookrightarrow \Mod_{R}^{\Frob}$ admits a left adjoint, which carries
a Frobenius module $M$ to its perfection $M^{\perfection}$.
\end{proposition}

\begin{proof}
Under the equivalence of Remark \ref{remX4}, a left adjoint to $\iota$ corresponds to the functor of extension of scalars along
the map $R[F] \rightarrow R^{\perfection}[ F^{\pm 1} ]$, which is given by $M \mapsto M^{\perfection}$.
\end{proof}

\begin{remark}\label{prebb}
The perfection functor $M \mapsto M^{\perfection}$ is exact. It follows that the inclusion functor $\Mod_{R}^{\perf} \hookrightarrow \Mod_{R}^{\Frob}$ carries injective objects to injective objects.
In particular, if $M$ and $N$ are perfect Frobenius module over $R$, then the canonical map $\Ext^{\ast}_{ \Mod_{R}^{\perf} }( M, N) \rightarrow
\Ext^{\ast}_{ \Mod_{R}^{\Frob} }( M, N)$ is an isomorphism. We will denote either of these $\Ext$-groups by $\Ext^{\ast}_{R[F]}(M, N)$.
\end{remark}

For the purpose of comparing Frobenius modules with {\etale} sheaves, there is no harm in restricting our considerations to perfect Frobenius modules:

\begin{proposition}\label{proposition.solperfect}
Let $f: M \rightarrow N$ be a morphism of Frobenius modules over $R$. If the induced map $M^{\perfection} \rightarrow N^{\perfection}$ is an isomorphism of perfect Frobenius modules,
then the induced map $\Sol(M) \rightarrow \Sol(N)$ is an isomorphism of {\etale} sheaves. 
\end{proposition}

\begin{proof}
Factoring $f$ as a composition $M \rightarrow \im(f) \rightarrow N$, we can reduce to proving Proposition \ref{proposition.solperfect} in the special case where $f$ is assumed
to be either surjective or injective. Suppose first that $f$ is injective. Our hypothesis that $f$ induces an
isomorphism $M^{\perfection} \rightarrow N^{\perfection}$ guarantees that the perfection $(N/M)^{\perfection}$ vanishes: that is, the Frobenius map $\varphi_{N/M}$ is locally nilpotent.
It follows that for any {\etale} $R$-algebra $A$, the Frobenius map $\varphi_{ A \otimes_{R} (N/M) }$ is also locally nilpotent, and therefore has no nonzero fixed points. 
It follows that the {\etale} sheaf $\Sol(N/M)$ vanishes. Since the solution functor is left exact (Remark \ref{remark.leftexact}), we have an exact sequence of {\etale} sheaves
$0 \rightarrow \Sol(M) \rightarrow \Sol(N) \rightarrow \Sol(N/M)$, which proves that $\Sol(f)$ is an isomorphism.

We now treat the case where $f$ is surjective. We wish to prove that $f$ induces an isomorphism $\Sol(M)(A) \rightarrow \Sol(N)(A)$ for every {\etale} $R$-algebra $A$.
Replacing $M$ and $N$ by $A \otimes_{R} M$ and $A \otimes_{R} N$ (which does not injure our assumption that $f$ induces an equivalence of perfections: see Proposition \ref{mallow}),
we can reduce to the case $A = R$. We have a commutative diagram of short exact sequence
$$ \xymatrix{ 0 \ar[r] & \ker(f) \ar[d]^{\id - \varphi_{\ker(f) }} \ar[r] & M \ar[r]^-{f} \ar[d]^{\id - \varphi_M} & N \ar[d]^{\id - \varphi_{N}}  \ar[r] & 0 \\
0 \ar[r] & \ker(f) \ar[r] & M \ar[r]^-{f} & N \ar[r] & 0 }$$
which gives a short exact sequence
$$ \ker(  \id - \varphi_{\ker(f)} ) \rightarrow \Sol(M)(R) \rightarrow \Sol(N)(R) \rightarrow \coker( \id - \varphi_{\ker(f)} ).$$ 
It will therefore suffice to show that the map $\id - \varphi_{\ker(f) }$ is an isomorphism. This is clear: our assumption that $f$ induces an equivalence of
perfections guarantees that $\varphi_{\ker(f)}$ is locally nilpotent, so that $\id - \varphi_{\ker(f) }$ has an inverse given by the infinite sum
$\sum_{n \geq 0} \varphi_{\ker(f)}^{n}$.
\end{proof}

\subsection{Restriction and Extension of Scalars}\label{sec3sub3}

Let $f: A \rightarrow B$ be a homomorphism of commutative $\F_p$-algebras and let $M$ be a Frobenius module over $B$. Then $M$ is perfect as a Frobenius module over $B$ if and only if it is perfect when regarded as a Frobenius module over $A$. In particular, the restriction of scalars functor $f_{\ast}: \Mod_{B}^{\Frob} \rightarrow \Mod_{A}^{\Frob}$ carries $\Mod_{B}^{\perf}$ into $\Mod_{A}^{\perf}$.
We will abuse notation by denoting the induced map $\Mod_{B}^{\perf} \rightarrow \Mod_{A}^{\perf}$ also by $f_{\ast}$, so that we have a commutative diagram $\sigma:$
$$ \xymatrix{ \Mod_{B}^{\perf} \ar[r] \ar[d]^{f_{\ast}} & \Mod_{B}^{\Frob} \ar[d]^{f_{\ast}} \\
\Mod_{A}^{\perf} \ar[r] & \Mod_{A}^{\Frob}. }$$

\begin{remark}\label{remX5}
Let $f: A \rightarrow B$ be a homomorphism of commutative $\F_p$-algebras and let $M$ be a Frobenius module over $B$. Then the canonical map
$u: M \rightarrow M^{\perfection}$ induces a map $f_{\ast}(u): f_{\ast}(M) \rightarrow f_{\ast}( M^{\perfection} )$ of Frobenius modules over $A$ whose target
is perfect. It follows that $f_{\ast}(u)$ extends uniquely to a map $v: f_{\ast}(M)^{\perfection} \rightarrow f_{\ast}( M^{\perfection} )$. Moreover, the map $v$
is an isomorphism: this follows by inspecting the construction of the perfection given in \S \ref{sec3sub2}. Put more informally, the formation of the perfection
$M^{\perfection}$ does not depend on whether we regard $M$ as a Frobenius module over $B$ or over $A$ (or over $\F_p$).
\end{remark}

In the situation of Remark \ref{remX5}, the extension of scalars functor
$f^{\ast}_{\Frob}: \Mod_{A}^{\Frob} \rightarrow \Mod_{B}^{\Frob}$ usually does not carry perfect Frobenius modules
to perfect Frobenius modules. However, we do have the following:

\begin{proposition}\label{mallow}
Let $f: A \rightarrow B$ be a homomorphism of commutative $\F_p$-algebras. Then the forgetful functor
$\Mod_{B}^{\perf} \rightarrow \Mod_{A}^{\perf}$ admits a left adjoint $f^{\diamond}$. Moreover, the diagram of categories
$$ \xymatrix@C=50pt{ \Mod_{A}^{\Frob} \ar[r]^-{(-)^{\perfection}} \ar[d]^{ f^{\ast}_{\Frob} } & \Mod_{A}^{\perf} \ar[d]^{ f^{\diamond} } \\
\Mod_{B}^{\Frob} \ar[r]^-{(-)^{\perfection}} & \Mod_{B}^{\perf} }$$
commutes up to canonical isomorphism. More precisely, for every object $M \in \Mod_{A}^{\Frob}$, the canonical map
$f^{\diamond}(M^{\perfection}) \rightarrow ( f^{\ast}_{\Frob} M)^{\perfection}$ is an equivalence.
\end{proposition}

\begin{proof}
Defining $f^{\diamond}$ by the formula $f^{\diamond}(M) = (f^{\ast}_{\Frob} M)^{\perfection}$, it follows immediately
from the definitions that $f^{\diamond}$ is left adjoint to the forgetful functor $f_{\ast}: \Mod_{B}^{\perf} \rightarrow \Mod_{A}^{\perf}$.
The desired commutativity follows from the commutativity of the diagram $\sigma$ above (by passing to left adjoints). 
\end{proof}

In the situation of Proposition \ref{mallow}, the functors
$$ f^{\ast}_{\Frob}: \Mod_{A}^{\Frob} \rightarrow \Mod_{B}^{\Frob} \quad \quad f^{\diamond}: \Mod_{A}^{\perf} \rightarrow \Mod_{B}^{\perf}$$
are right exact, but generally not left exact (unless $B$ is flat over $A$). We can therefore consider their left derived functors.

\begin{construction}\label{sevos}
Let $f: A \rightarrow B$ be a homomorphism of $\F_p$-algebras. The abelian category $\Mod_{A}^{\Frob}$ has enough projective objects,
so that the extension of scalars functor $f^{\ast}_{\Frob}: \Mod_{A}^{\Frob} \rightarrow \Mod_{B}^{\Frob}$ has left derived functors
$L_{n} f^{\ast}_{\Frob}: \Mod_{A}^{\Frob} \rightarrow \Mod_{B}^{\Frob}$ for $n \geq 0$. More concretely, for $M \in \Mod_{A}^{\Frob}$, we can describe
$L_n f^{\ast}_{\Frob} M$ as the $n$th homology of the chain complex $f^{\ast}_{\Frob}( P_{\ast} )$, where $P_{\ast}$ is a projective resolution of
$M$ in the abelian category $\Mod_{A}^{\Frob}$. Note that $P_{\ast}$ is then also a projective resolution of $M$ in the category $\Mod_{A}$ of
$A$-modules, and that the chain complex $f^{\ast}_{\Frob}( P_{\ast} )$ can be identified with $B \otimes_{A} P_{\ast}$.
It follows that for $n \geq 0$, we have canonical $B$-module isomorphisms $L_n f^{\ast}_{\Frob} M = \Tor_{n}^{A}( M, B)$. 
We can summarize the situation as follows:
\begin{itemize}
\item[$(\ast)$] If $f: A \rightarrow B$ is a homomorphism of $\F_p$-algebras and $M$ is a Frobenius module over $A$, then
the $\Tor$-groups $\Tor_{n}^{A}(M, B)$ can be regarded as Frobenius modules over $B$. Moreover, the construction
$$ \Mod_{A}^{\Frob} \rightarrow \Mod_{B}^{\Frob} \quad \quad M \mapsto \Tor_{n}^{A}(M, B)$$
can be identified with the $n$th left derived functor of the construction $M \mapsto f^{\ast}_{\Frob} B$.
\end{itemize}
\end{construction}

\begin{variant}\label{variant.derived}
In the situation of Construction, the abelian category $\Mod_{A}^{\perf}$ also has enough projective objects, so we can consider
the left derived functors $L_{n} f^{\diamond}$ of the functor $f^{\diamond}: \Mod_{A}^{\perf} \rightarrow \Mod_{B}^{\perf}$.
Note that if $M$ is a perfect Frobenius module over $A$ and $P_{\ast}$ is a projective resolution of $M$ in the category
$\Mod_{A}^{\Frob}$ of {\em all} Frobenius modules over $A$, then $P_{\ast}^{\perfection}$ is a projective resolution of
$M^{\perfection} \simeq M$ in the category $\Mod_{A}^{\perf}$ of perfect Frobenius modules over $A$.
We can therefore identify $(L_{n} f^{\diamond})(M)$ with the $n$th homology group of the chain complex
$f^{\diamond} P_{\ast}^{\perfection} \simeq ( f^{\ast}_{\Frob} P_{\ast})^{\perfection}$.
Using the exactness of the functor $N \mapsto N^{\perfection}$, we obtain isomorphisms
$$ (L_n f^{\diamond})(M) \simeq (L_n f^{\ast}_{\Frob} )(M)^{\perfection} \simeq \Tor_{n}^{A}(M, B)^{\perfection}.$$
\end{variant}

\subsection{Perfect Rings}\label{sec3sub4}

Let $R$ be a commutative $\F_p$-algebra. Recall that $R$ is said to be {\it perfect} if the Frobenius homomorphism $\varphi_{R}: R \rightarrow R$ is an isomorphism.

\begin{remark}
A commutative $\F_p$-algebra $R$ is perfect if and only if it is perfect when regarded as a Frobenius module over itself, in the sense of
Definition \ref{definition.perfect}.
\end{remark}

\begin{example}
Let $R$ be any commutative $\F_p$-algebra. Then the perfection $R^{\perfection}$ of Example \ref{exX6} is a perfect $\F_p$-algebra. 
\end{example}

Let $R$ be an $\F_p$-algebra. If $M$ is a {perfect} Frobenius module over $R$, then $M$ admits the structure of a module over $R^{\perfection}$. More precisely, we have the following result, whose
proof is left to the reader:

\begin{proposition}\label{prop5}
Let $R$ be an algebra over $\F_p$. Then the restriction of scalars functor
$\Mod_{R^{\perfection}}^{\perf} \rightarrow \Mod_{R}^{\perf}$ is an equivalence of categories.
\end{proposition}

\begin{remark}
Let $R$ be a Noetherian $\F_p$-algebra of finite Krull dimension $d$, and suppose that the Frobenius map $\varphi_{R}: R \rightarrow R^{1/p}$ exhibits $R^{1/p}$ as a finite module over $R$. Then the abelian category of all $R^{\perfection}$-modules has global dimension $\leq 2d+1$ (see \cite[Remark 11.33]{BS}). It follows from Proposition \ref{prop5} and Construction~\ref{conX7} that the category $\Mod_R^{\perf}$ has global dimension $\leq 2d+2$. 
\end{remark}

\begin{proposition}\label{prop5prime}
Let $f: A \rightarrow B$ be a homomorphism of perfect $\F_p$-algebras.
Then the extension of scalars functor $f^{\ast}_{\Frob}: \Mod_{A}^{\Frob} \rightarrow \Mod_{B}^{\Frob}$
carries $\Mod_{A}^{\perf}$ into $\Mod_{B}^{\perf}$.
\end{proposition}

\begin{proof}
Let $M$ be a {perfect} Frobenius module over $A$. Then the maps
$$ \varphi_{B}: B \rightarrow B \quad \quad \varphi_{A}: A \rightarrow A \quad \quad \varphi_{M}: M \rightarrow M$$
are isomorphisms, so the induced map $\varphi_{B \otimes_A M}: B \otimes_{A} M \rightarrow B \otimes_{A} M$ is also an isomorphism.
\end{proof}

\begin{warning}
In the proof of Proposition \ref{prop5prime}, it is not enough to assume that $M$ and $B$ are perfect. For example, the
tensor product $\F_p[x]^{\perfection} \otimes_{ \F_p[x] } \F_p[x]^{\perfection}$ is not a perfect ring.
\end{warning}

\begin{corollary}\label{corX14}
Let $f: A \rightarrow B$ be an {\etale} morphism of $\F_p$-algebras. Then the extension of scalars functor $f^{\ast}_{\Frob}: \Mod_{A}^{\Frob} \rightarrow \Mod_{B}^{\Frob}$
carries $\Mod_{A}^{\perf}$ into $\Mod_{B}^{\perf}$.
\end{corollary}

\begin{proof}
Let $M$ be a {perfect} Frobenius module over $A$. Then we can also regard $M$ as a Frobenius module over $A^{\perfection}$. Since $f$ is {\etale}, the diagram
of commutative rings
$$ \xymatrix{ A \ar[r] \ar[d] & B \ar[d] \\
A^{\perfection} \ar[r] & B^{\perfection} }$$
is a pushout square. It follows that we can identify $f^{\ast}_{\Frob} M$ with the tensor product $B^{\perfection} \otimes_{A^{\perfection}} M$, which is perfect
by Proposition \ref{prop5prime}.
\end{proof}

\subsection{Exactness Properties of $f^{\diamond}$}

Our final goal in this section is to establish the following fundamental exactness property for pullbacks of algebraic Frobenius modules:

\begin{theorem}\label{theoX9}
Let $f: A \rightarrow B$ be a homomorphism of commutative $\F_p$-algebras and let $M$ be an algebraic Frobenius module over $A$. Then
the abelian groups $\Tor_{n}^{A}(M, B)^{\perfection}$ vanish for $n > 0$.
\end{theorem}

Before giving the proof of Theorem \ref{theoX9}, let us collect some consequences:

\begin{corollary}\label{lushin}
Let $f: A \rightarrow B$ be a homomorphism of commutative $\F_p$-algebras and suppose we are given an exact sequence
$0 \rightarrow M' \rightarrow M \rightarrow M'' \rightarrow 0$ in $\Mod_{A}^{\perf}$. If $M''$ is algebraic, then
the sequence $0 \rightarrow f^{\diamond} M' \rightarrow f^{\diamond} M \rightarrow f^{\diamond} M'' \rightarrow 0$ is also exact.
\end{corollary}

\begin{proof}
Variant \ref{variant.derived} supplies an exact sequence
$$ \Tor^{A}_{1}(M'', B)^{\perfection} \rightarrow f^{\diamond} M' \rightarrow f^{\diamond} M \rightarrow f^{\diamond} M'' \rightarrow 0,$$
where the first term vanishes by virtue of Theorem \ref{theoX9}.
\end{proof}

\begin{corollary}\label{corX11}
Let $f: A \rightarrow B$ be a morphism of $\F_p$-algebras and suppose we are given objects $M \in \Mod_{A}^{\Frob}$, $N \in \Mod_{B}^{\Frob}$.
If $M$ is algebraic and $N$ is perfect, then the canonical map $\Ext^{\ast}_{ A[F] }( M, N ) \rightarrow \Ext^{\ast}_{ B[F] }( f^{\diamond} M, N )$
is an isomorphism.
\end{corollary}

\begin{proof}
Let $P_{\ast}$ be a projective resolution of $M$ in the category $\Mod_{A}^{\perf}$. Then Theorem \ref{theoX9} guarantees that
$f^{\diamond} P_{\ast}$ is a projective resolution of $f^{\diamond} M$ in the category $\Mod_{B}^{\perf}$, so that
both $\Ext^{\ast}_{A[F]}(M, N)$ and $\Ext^{\ast}_{ B[F] }( f^{\diamond} M, N)$ can be identified with the cohomology
of the cochain complex $\Hom_{ A[F] }( P_{\ast}, N ) \simeq \Hom_{ B[F] }( f^{\diamond} P_{\ast}, N )$.
\end{proof}

We now turn to the proof of Theorem \ref{theoX9}. The main ingredient is the following observation from \cite{BS}:

\begin{lemma}\label{lemma.bs}
Let $A$ be a perfect $\F_p$-algebra containing an element $a$, and let $I = (a, a^{1/p}, a^{1/p^2}, \ldots )$ denote the kernel of
the map $A \rightarrow (A / (a) )^{\perfection}$. Then the elements $a^{1/p^n} \in I$ determine an $A$-module isomorphism of $I$
with the direct limit of the diagram
$$ A \xrightarrow{ a^{ 1 - 1/p} } A \xrightarrow{ a^{1/p - 1/p^2} } A \xrightarrow{ a^{ 1/p^2 - 1/p^3} } A \xrightarrow{ a^{ 1/p^3 - 1/p^4 } } \cdots$$
\end{lemma}

\begin{proof}
Unwinding the definitions, we must show that if an element $x \in A$ satisfies the equation $x a^{ 1/p^m} = 0$ for some $m \geq 0$, then we have
$x a^{ 1/p^m - 1/p^n} = 0$ for some $n \gg m$. We now compute
\begin{eqnarray*}
x a^{ 1/p^m - 1/p^{m+1} } & = & x^{1/p} x^{(p-1)/p} a^{1/p^{m+1} } a^{ (p-2) / p^{m+1} } \\
& = & (x a^{1/p^m} )^{1/p} x^{(p-1)/p} a^{(p-2)/p^{m+1} } \\
& = & 0. \end{eqnarray*}
\end{proof}

\begin{proof}[Proof of Theorem \ref{theoX9}]
Let $f: A \rightarrow B$ be a homomorphism of commutative $\F_p$-algebras and let $M$ be an algebraic Frobenius module over $A$;
we wish to show that the groups $\Tor_{n}^{A}( M, B)^{\perfection}$ vanish for $n > 0$. Writing $B$ as a filtered direct limit of finitely generated
$A$-algebras, we can assume that $B$ is finitely presented over $A$: that is, we can write $B \simeq A[x_1, \ldots, x_k] / I$ for some finitely generated ideal $I \subseteq A[x_1, \ldots, x_k]$.
Set $N = (A[x_1, \ldots, x_k] \otimes_{A} M)^{\perfection}$. Using the flatness of $A[x_1, \ldots, x_k]$ over $A$, we obtain isomorphisms
$$ \Tor_{\ast}^{ A }(M, B)^{\perfection} \simeq \Tor_{\ast}^{ A[x_1, \ldots, x_k] }( N, B )^{\perfection}.$$
Moreover, $N$ is an algebraic Frobenius module over $A[x_1, \ldots, x_k]$ (see Corollary \ref{corX62}). We may therefore replace $A$ 
by $A[x_1, \ldots, x_k]$ (and $M$ by $N$) and thereby reduce to the case where $f$ is surjective.

Proceeding by induction on the number of generators of $I$, we can reduce to the case where $I = (a)$ is a principal ideal. Since $M$ is perfect,
we can regard $M$ as a module over the perfection $A^{\perfection}$, so that we have canonical isomorphisms
$\Tor^{A}_{\ast}(M, B)^{\perfection} \simeq \Tor^{ A^{\perfection} }_{\ast}(M, B^{\perfection} )$. We have an exact sequence
$$0 \rightarrow I^{\perfection} \rightarrow A^{\perfection} \rightarrow B^{\perfection} \rightarrow 0$$
in the category of modules over $A^{\perfection}$, where $I^{\perfection}$ is flat over $A^{\perfection}$ by virtue of Lemma \ref{lemma.bs}.
It follows that the groups $\Tor^{A}_{n}(M, B)^{\perfection}$ vanish for $n \geq 2$, and $\Tor^{A}_{1}(M, B)^{\perfection}$ can be identified
with the kernel of the map
$$ \rho: M \otimes_{A^{\perfection} } I^{\perfection} \rightarrow M \otimes_{ A^{\perfection} } A^{\perfection} \simeq M.$$
We will complete the proof by showing that $\rho$ is injective Using the description of $I^{\perfection}$ supplied by
Lemma \ref{lemma.bs}, we see that the injectivity of $\rho$ can be reformulated as follows:
\begin{itemize}
\item[$(\ast)$] Let $x$ be an element of $M$ which satisfies the equation $a^{1/p^m} x = 0$, for some integer $m$. Then
$a^{ 1/p^m - 1/p^{m'}} x = 0$ for some $m' \gg m$.
\end{itemize}
To prove $(\ast)$, we use our assumption that $M$ is algebraic to write
$$ \varphi_{M}^{n}(x) = c_1 \varphi_{M}^{n-1}(x) + \cdots + c_{n} x$$
for some coefficients $c_1, \ldots, c_n \in A$. We then compute
\begin{eqnarray*}
\varphi_M^{n}( a^{ 1/p^{m+1}} x) & = & a^{p^n / p^{m+1} } \varphi_M^n(x) \\
& = & \sum_{i = 1}^{n} a^{p^n / p^{m+1} } c_i \varphi_{M}^{n-i}(x) \\
& = & \sum_{i = 1}^{n} c_i \varphi_M^{n-i}( a^{ p^{i-1} / p^{m}} x) \\
& = & 0.
\end{eqnarray*}
Using the bijectivity of $\varphi_M$, it follows that $a^{ 1/p^{m+1} } x =0$, which immediately implies $(\ast)$.
\end{proof}

\begin{remark}
\label{TorVanishingRings}
The reasoning used to prove Theorem~\ref{theoX9} can also be used to show the following result (see \cite{BS}): if $B \gets A \to C$ is a diagram of perfect rings, then $\Tor^i_A(B,C) = 0$ for $i > 0$. Indeed, as in the proof of Theorem~\ref{theoX9}, one reduces to the case $B = A/I$, where $I = \bigcup_n f^{1/p^n} A$ is the radical of an ideal generated by a single element $f \in A$. In this case, the presentation given in Lemma~\ref{lemma.bs} and the perfectness of $C$ imply that $I$ is a flat $A$-module, and that $I \otimes_A C \simeq J$, where $J = \bigcup_n f^{1/p^n} C \subseteq C$ is also an ideal. The desired claim follows immediately.
\end{remark}

\newpage \section{Holonomic Frobenius Modules}\label{section.holonomic}
\setcounter{subsection}{0}
\setcounter{theorem}{0}

Let $R$ be a commutative $\F_p$-algebra and let $\Mod_{R}^{\Frob}$ denote the category of Frobenius modules over $R$. In this section, we consider a full subcategory
$\Mod_{R}^{\hol} \subseteq \Mod_{R}^{\Frob}$ whose objects we refer to as {\it holonomic Frobenius modules}. Roughly speaking, the class of
holonomic Frobenius modules can be regarded as a characteristic $p$ analogue of the class of (regular) holonomic $\calD$-modules on complex analytic varieties. 
We will later show that the category $\Mod_{R}^{\hol}$ can be characterized as the essential image of the category 
$\Shv_{\mathet}^{c}( \Spec(R), \F_p)$ of {\em constructible} {\etale} sheaves under the Riemann-Hilbert equivalence
$$ \RH: \Shv_{\mathet}( \Spec(R), \F_p ) \simeq \Mod_{R}^{\alg} \subseteq \Mod_{R}^{\Frob}$$
of Theorem \ref{maintheoXXX} (see Theorem \ref{companion}). Our goal in this section is to lay the groundwork by establishing the basic formal properties of
$\Mod_{R}^{\hol}$.

We begin in \S \ref{sec4sub1} by defining the class of holonomic Frobenius modules (Definition \ref{defhol2}) and verifying some elementary closure properties.
In \S \ref{sec4sub2}, we show that every holonomic Frobenius module is algebraic (Proposition \ref{propX58}) and that, conversely, every
algebraic Frobenius module can be realized as a filtered colimit of holonomic Frobenius modules (Theorem \ref{theoX54}). This result will allow us to reduce certain questions
about algebraic modules to the case of holonomic modules, which enjoy good finiteness properties. For example, we prove in
\S \ref{sec4sub3} that if $R$ is Noetherian, then the category $\Mod_{R}^{\hol}$ is also Noetherian (Proposition \ref{propX29}). In \S \ref{sec4sub4} we associate
to each holonomic Frobenius module $M$ a constructible subset $\supp(M) \subseteq \Spec(R)$ which we refer to as {\it the support of $M$}.
We will see later that the support $\supp(M)$ exerts strong control over the behavior of $M$: for example, it is empty if and only if $M \simeq 0$ (Proposition \ref{theo78}). 

\subsection{Holonomicity}\label{sec4sub1}

\begin{definition}\label{defhol2}
Let $R$ be a commutative $\F_p$-algebra and let $M$ be a Frobenius module over $R$. We will say that $M$ is {\it holonomic} if there exists an isomorphism $M \simeq M_0^{\perfection}$, where $M_0 \in \Mod_{R}^{\Frob}$ is
finitely presented as an $R$-module. We let $\Mod_{R}^{\hol}$ denote the full subcategory of $\Mod_{R}^{\perf}$ spanned by the holonomic Frobenius modules over $R$.
\end{definition}

\begin{proposition}\label{propX51}
Let $f: A \rightarrow B$ be a homomorphism of $\F_p$-algebras. If $M \in \Mod_{R}^{\perf}$ is holonomic,
then $f^{\diamond} M \in \Mod_{B}^{\perf}$ is holonomic.
\end{proposition}

\begin{proof}
Without loss of generality we may assume that $M = M_0^{\perfection}$ for some $M_0 \in \Mod_{A}^{\Frob}$ which is finitely presented as an
$A$-module. Then $f^{\diamond} M \simeq (B \otimes_{A} M_0)^{\perfection}$, and $B \otimes_{A} M_0$ is finitely presented as a $B$-module.
\end{proof}

We also have the following converse of Proposition \ref{propX51}, whose proof we leave to the reader:

\begin{proposition}\label{propX52}
Let $R$ be an $\F_p$-algebra and let $M$ be a holonomic Frobenius module over $R$. Then there exists an inclusion
$\iota: R' \hookrightarrow R$ where $R'$ is finitely generated over $\F_p$ and an isomorphism $M \simeq \iota^{\diamond} M'$, where
$M'$ is a holonomic Frobenius module over $R'$.
\end{proposition}

\begin{remark}\label{remX8}
Let $R$ be a commutative $\F_p$-algebra and suppose we are given objects $M, N \in \Mod_{R}^{\perf}$, where $M$ is holonomic.
Then we can choose an isomorphism $M \simeq M_0^{\perfection}$ for some $M_0 \in \Mod_{R}^{\Frob}$ which is finitely presented as an $R$-module. 
Using Construction \ref{conX7} (and the observation that $\varphi_{N}: N \rightarrow N^{1/p }$ is an isomorphism), we obtain a long exact sequence
$$ \Ext_{R[F]}^{\ast}(M, N) \rightarrow \Ext^{\ast}_{R}(M_0, N) \xrightarrow{\gamma} \Ext^{\ast}_{R}(M_0, N)
\rightarrow \Ext^{\ast+1}_{R}( M_0, N),$$
where $\gamma$ is the map given by $\gamma(f) = f - \varphi_{N}^{-1} \circ f \circ \varphi_{M_0}$.
\end{remark}

\begin{proposition}\label{propX55}
Let $R$ be an $\F_p$-algebra and let $M$ be a Frobenius module over $R$ which is holonomic. Then
$M$ is a compact object of the category $\Mod_{R}^{\perf}$: that is, the functor
$N \mapsto \Hom_{R}^{\Frob}(M, N)$ commutes with filtered colimits.
\end{proposition}

\begin{proof}
Choose an isomorphism $M \simeq M_0^{\perfection}$ for some $M_0 \in \Mod_{R}^{\Frob}$ which is finitely presented as an $R$-module,
and observe that the exact sequence of Remark \ref{remX8} depends functorially on $N$.
\end{proof}

\subsection{Comparison with Algebraic Frobenius Modules}\label{sec4sub2}

Our next goal is to compare the theory of holonomic Frobenius modules (introduced in Definition \ref{defhol2}) with
the theory of algebraic Frobenius (introduced in Definition \ref{defhol1}). Our starting point is the following:

\begin{proposition}\label{propX58}
Let $R$ be a commutative $\F_p$-algebra and let $M$ be a holonomic Frobenius module over $R$. Then $M$ is algebraic.
\end{proposition}

\begin{proof}
Let $x$ be an element of $M$. Using Proposition \ref{propX52}, we can choose a finitely generated subalgebra
$R' \subseteq R$ and an isomorphism $M \simeq (R \otimes_{R'} M')^{\perfection}$ for some $M' \in \Mod_{R'}^{\hol}$.
Enlarging $R'$ if necessary, we may assume that $x$ is the image of some element $x' \in M'$.
Since $M'$ is holonomic, we can write $M' = M'^{\perfection}_0$ for some $M'_0 \in \Mod_{R'}^{\Frob}$ which is finitely presented 
as an $R'$-module. We can then write $x' = \varphi_{M'}^{-k}(y)$ for some $y \in M'$ which lifts to an element $y_0 \in M'_0$. Since $M'_0$ is a Noetherian
$R'$-module, the submodule generated by the elements $\{ \varphi_{M'_0}^{n}( y_0) \}_{n \geq 0}$ is finitely generated. It follows
that $y_0$ satisfies an equation $\varphi_{M'_0}^{n}(y) + a_1 \varphi_{M'_0}^{n-1}(y) + \cdots + a_n y = 0$ for some elements $a_{i} \in R'$,
so that $x$ satisfies the equation $\varphi_{M}^{n}(x) + a_1^{p^{k}} \varphi_{M}^{n-1}(x) + \cdots + a_n^{p^{k}} x=0$.
\end{proof}

\begin{corollary}\label{corollary.propX60}
Let $R$ be a commutative $\F_p$-algebra. Then the collection of holonomic Frobenius modules over $R$ is closed under finite direct sums and cokernels. 
\end{corollary}

\begin{remark}
We will see later that the collection of holonomic Frobenius modules is also closed under the formation of kernels and extensions; in particular, it is 
an abelian subcategory of $\Mod_{R}^{\Frob}$ (Corollary \ref{corX30}).
\end{remark}

\begin{proof}[Proof of Corollary \ref{corollary.propX60}]
Closure under finite direct sums is obvious. Let $u: M \rightarrow N$ be a morphism in $\Mod_{R}^{\hol}$. 
Then we can choose isomorphisms $M \simeq M_0^{\perfection}$ and $N \simeq N_0^{\perfection}$ for some objects $M_0, N_0 \in \Mod_{R}^{\Frob}$ which
are finitely presented as $R$-modules. Let $x_1, \ldots, x_k$ be a set of generators for $M_0$ as an $R$-module.
Let us abuse notation by writing $u(x_i)$ for the image of $x_i$ under the composite map
$M_0 \rightarrow M \xrightarrow{u} N$. Then we can choose some integer $n \gg 0$ for which each
$u(x_i)$ has the form $\varphi_{N}^{-n}( y_i)$ for some $y_i$ which lifts to an element $\overline{y}_i \in N_0$.
Since $N$ is holonomic, it is algebraic (Proposition \ref{propX58}). It follows that each $y_i$ satisfies
some equation $\varphi_{N}^{d_i}(y_i) + c_{1,i} \varphi_{N}^{d_i-1}(y_i) + \cdots + c_{d_i,i} y_i = 0$ in
$N \simeq N_0^{\perfection}$, so that $\varphi_{N_0}^{d_i+e}(\overline{y}_i) + c_{1,i}^{p^{e}} \varphi_{N_0}^{d_i-1+e}(\overline{y}_i) + \cdots + c_{d_i,i}^{p^{e}} \varphi_{N_0}(\overline{y}_i) = 0$ in $N_0$ for $e \gg 0$. In particular, the elements $\{ \varphi_{N_0}^{a}(\overline{y}_i) \}$ generate a Frobenius submodule $N'_0 \subseteq N_0$ which is finitely
generated as an $R$-module. Then $N_0 / N'_0$ is finitely presented as an $R$-module. Using evident isomorphism $\coker(u) \simeq (N_0 / N'_0)^{\perfection}$, we conclude
that $\coker(u)$ is holonomic.
\end{proof}

Our next goal is to establish a converse to Proposition \ref{propX58}, which asserts that every algebraic Frobenius module can be ``built from'' holonomic Frobenius modules
(Theorem \ref{theoX54}). First, we need some general facts about algebraicity.

\begin{proposition}\label{propX53}
Let $R$ be an $\F_p$-algebra. Then $\Mod_{R}^{\alg}$ is a localizing subcategory of $\Mod_{R}^{\perf}$. That is:
\begin{itemize}
\item[$(a)$] Given a short exact sequence $0 \rightarrow M' \rightarrow M \rightarrow M'' \rightarrow 0$ of {perfect} Frobenius modules over $R$,
$M$ is algebraic if and only if $M'$ and $M''$ are algebraic.

\item[$(b)$] The collection of algebraic Frobenius modules is closed under (possibly infinite) direct sums.
\end{itemize}
\end{proposition}

\begin{proof}
The ``only if'' direction of $(a)$ follows immediately from the definitions. To prove the reverse direction,
suppose we are given an exact sequence $0 \rightarrow M' \xrightarrow{\alpha} M \rightarrow M'' \rightarrow 0$ where $M'$ and $M''$ are algebraic.
Let $x$ be an element of $M$. Since $M''$ is algebraic, we deduce that
there is an equation of the form $\varphi_{M}^{m}(x) + a_1 \varphi_{M}^{n-1}(x) + \cdots + a_m x = \alpha(y)$ for some $a_i \in R$ and some
$y \in M'$. Since $M'$ is algebraic, we obtain an equation of the form
$\varphi_{M'}^{n}(y) + b_1 \varphi_{M'}^{m-1}(y) + \cdots + b_n y = 0$ for some $b_j \in R$. It follows that 
$$\sum_{ 0 \leq i \leq m, 0 \leq j \leq n } a_{i} b_{j}^{p^{i}} \varphi_{M}^{i+j}(x) = 0$$
with the convention that $a_0 = b_0 = 1$. Allowing $x$ to vary, we deduce that $M$ is algebraic.

To prove $(b)$, we observe that the general case immediately reduces to the case of a finite direct sum, which follows from $(a)$.
\end{proof}

\begin{proposition}\label{preX59}
Let $R$ be a commutative $\F_p$-algebra, let $M$ be a perfect Frobenius module over $R$, and let $x \in M$ be an element.
The following conditions are equivalent:
\begin{itemize}
\item[$(1)$] There exists a map of Frobenius modules $f: M' \rightarrow M$, where $M'$ is holonomic, and an element
$x' \in M'$ satisfying $f(x') = x$.
\item[$(2)$] There exists an algebraic Frobenius submodule $M_0 \subseteq M$ which contains $x$.
\item[$(3)$] The element $x$ satisfies an equation
$$ \varphi_M^{n}(x) + a_1 \varphi_M^{n-1}(x) + \cdots + a_n x$$
for some coefficients $a_1, \ldots, a_n \in R$.
\end{itemize}
\end{proposition}

\begin{proof}
We first show that $(1)$ implies $(2)$. Let $f: M' \rightarrow M$ be a morphism of Frobenius modules. If $M'$ is holonomic, then it is also algebraic
(Proposition \ref{propX58}). Consequently, if there exists an element $x' \in M'$ satisfying $f(x') = x$, then $x$ belongs to the submodule $\im(f) \subseteq M$,
which is algebraic by virtue of Proposition \ref{propX53}. 

The implication $(2) \Rightarrow (3)$ is immediate from the definitions. We will complete the proof by showing that $(3) \Rightarrow (1)$.
Assume that $x$ satisfies an equation $\varphi_{M}^{n}(x) + a_1 \varphi_{M}^{n-1}(x) + \cdots + a_{n} x = 0$. Let $N$ denote the free $R$-module on a basis
$\{ y_0, \ldots, y_{n-1} \}$, which we regard as a Frobenius module over $R$ by setting
$$\varphi_{N}( y_i ) = \begin{cases} y_{i+1} & \text{ if } i < n-1 \\
- a_1 y_{n-1} - a_2 y_{n-2} - \cdots - a_{n} y_{0} & \text{ if } i = n-1. \end{cases}$$
Then the construction $y_i \mapsto \varphi_M^{i}( x )$ determines a map of Frobenius modules $f_0: N \rightarrow M$. Since $M$ is perfect, we can extend
$f_0$ to a map $f: N^{\perfection} \rightarrow M$. It follows immediately from the construction that $N^{\perfection}$ is holonomic and that
$x$ belongs to the image of $f$.
\end{proof}

\begin{corollary}\label{corX59}
Let $R$ be a commutative $\F_p$-algebra and let $M$ be an algebraic Frobenius module over $R$. Then there
exists an epimorphism of Frobenius modules $\bigoplus M_{\alpha} \rightarrow M$, where each $M_{\alpha}$
is holonomic.
\end{corollary}

\begin{corollary}\label{biggest}
Let $R$ be a commutative $\F_p$-algebra and let $M$ be a perfect Frobenius module over $R$. Then there exists a
largest Frobenius submodule $M_0 \subseteq M$ which is algebraic. Moreover, an element
$x \in M$ belongs to $M_0$ if and only if it satisfies an equation 
$$ \varphi_M^{n}(x) + a_1 \varphi_M^{n-1}(x) + \cdots + a_n x$$
for some coefficients $a_1, \ldots, a_n \in R$.
\end{corollary}

\begin{proof}
Let $M_0$ be the sum of all algebraic Frobenius submodules of $M$. It follows from Proposition \ref{propX53} that $M_0$ is also algebraic,
so that each element $x \in M_0$ satisfies an equation $\varphi_M^{n}(x) + a_1 \varphi_M^{n-1}(x) + \cdots + a_n x$ for some coefficients
$a_1, \ldots, a_n \in R$. The reverse implication follows from Proposition \ref{preX59}.
\end{proof}

\begin{corollary}\label{corX62}
Let $f: A \rightarrow B$ be a homomorphism of commutative $\F_p$-algebras and let
$M$ be an algebraic Frobenius module over $A$. Then $f^{\diamond}(M)$
is an algebraic Frobenius module over $B$.
\end{corollary}

\begin{proof}
Applying Corollary \ref{biggest}, we deduce that
there is a largest algebraic submodule $N \subseteq f^{\diamond}(M)$,
and that $N$ contains the image of the map
$$ M \rightarrow f^{\diamond}(M) \simeq (B \otimes_{A} M)^{\perfection}.$$
Since $N$ is a $B$-submodule of $f^{\diamond}(M)$ which is stable
under the automorphism $\varphi_{ f^{\diamond}(M)}^{-1}$, it follows
that $N = f^{\diamond}(M)$.
\end{proof}

\begin{theorem}\label{theoX54}
Let $R$ be a commutative $\F_p$-algebra. Then the inclusion functor $\Mod_{R}^{\hol} \hookrightarrow \Mod_{R}^{\alg}$
extends to an equivalence of categories $\Ind( \Mod_{R}^{\hol} ) \simeq \Mod_{R}^{\alg}$.
\end{theorem}

\begin{proof}
It follows from Proposition \ref{propX55} that the inclusion $\Mod_{R}^{\hol} \hookrightarrow \Mod_{R}^{\perf}$
extends to a fully faithful embedding $\iota: \Ind( \Mod_{R}^{\hol} ) \rightarrow \Mod_{R}^{\perf}$. Since
every holonomic Frobenius module is algebraic (Proposition \ref{propX58}) and the collection of
algebraic Frobenius modules is closed under filtered colimits, the essential image of
$\iota$ is contained in the full subcategory $\Mod_{R}^{\alg} \subseteq \Mod_{R}^{\perf}$.
To complete the proof, it will suffice to verify the reverse inclusion. Let $M$ be an algebraic Frobenius module;
we wish to show that $M$ can be written as a filtered colimit $\varinjlim M_{\alpha}$, where each $M_{\alpha}$
is holonomic. Using Corollary \ref{corX59}, we can choose an epimorphism $\rho: \bigoplus_{\alpha \in I} M_{\alpha} \rightarrow M$
for some set $I$, where each $M_{\alpha}$ is holonomic. The kernel $\ker(\rho)$ is then
algebraic (Proposition \ref{propX53}), so we can apply Corollary \ref{corX59} again to choose an epimorphism
$\rho': \bigoplus_{\beta \in J} M'_{\beta} \rightarrow \ker(\rho)$, where each $M'_{\beta}$ is holonomic. We can identify
$\rho'$ with a system of maps $\{ \rho'_{\alpha, \beta}: M'_{\beta} \rightarrow M_{\alpha} \}_{\alpha \in I, \beta \in J}$.
Using Proposition \ref{propX55}, we see that for each $\beta \in J$ there are only finitely many $\alpha \in I$ for which
$\rho'_{\alpha, \beta}$ is nonzero. It follows that we can write $M$ as a filtered colimit of Frobenius modules of the form
$\coker( \bigoplus_{ \beta \in J_0} M'_{\beta} \rightarrow \bigoplus_{\alpha \in I_0} M_{\alpha} )$ where
$I_0 \subseteq I$ and $J_0 \subseteq J_0$ are finite. Each of these Frobenius modules is holonomic by virtue of
Corollary \ref{corollary.propX60}.
\end{proof}

Let $f: A \rightarrow B$ be a morphism of commutative $\F_p$-algebras. In general, the restriction of scalars functor
$f_{\ast}: \Mod_{B}^{\Frob} \rightarrow \Mod_{A}^{\Frob}$ does not preserve holonomicity. However, we do have the following:

\begin{proposition}\label{propX45}
Let $f: A \rightarrow B$ be an $\F_p$-algebra homomorphism which is finite and of finite presentation.
Then the restriction of scalars functor $\Mod_{B}^{\Frob} \rightarrow \Mod_{A}^{\Frob}$ carries $\Mod_{B}^{\hol}$ into $\Mod_{A}^{\hol}$
and $\Mod_{B}^{\alg}$ into $\Mod_{A}^{\alg}$.
\end{proposition}

\begin{proof}
Because restriction of scalars commutes with filtered colimits, it will suffice to show that if $M \in \Mod_{B}^{\Frob}$ is a holonomic Frobenius module over $B$,
then it is also a holonomic Frobenius module over $A$ (Theorem \ref{theoX54}). Write $M = M_0^{\perfection}$ for some $M_0 \in \Mod_{B}^{\Frob}$ which is finitely presented as a $B$-module.
We now complete the proof by observing that our assumption on $f$ guarantees that $M_0$ is also finitely presented as an $A$-module.
\end{proof}

\begin{warning}\label{warning.X45}
The finite presentation hypothesis in Proposition~\ref{propX45} cannot be relaxed to finite generation. For example, take $A = \F_p[x_1,x_2,x_3,...]$ to be a polynomial ring on countably many generators, and let $B = A/(x_1,x_2,x_3,...) = \F_p$ be its residue field at the origin. Then $B$ is holonomic when regarded as a Frobenius module over itself, but not when regarded as a Frobenius module over $A$; one can see this directly, but a quick proof is provided by Theorem~\ref{theo74} (note that $\Spec(B)$ is not a constructible subset of $\Spec(A)$).
\end{warning}

\subsection{The Noetherian Case}\label{sec4sub3}

Recall that an object $X$ of an abelian category $\mathcal{A}$ is said to be {\it Noetherian} if the collection of subobjects of $X$
satisfies the ascending chain condition. A Grothendieck abelian category $\mathcal{A}$ is said to be {\it locally Noetherian} if every
object of $\mathcal{A}$ can be written as a union of Noetherian subobjects.

\begin{proposition}\label{propX29}
Let $R$ be a Noetherian $\F_p$-algebra. Then the abelian category $\Mod_{R}^{\alg}$ is locally Noetherian. Moreover, an object of
$\Mod_{R}^{\alg}$ is Noetherian if and and only if it is holonomic.
\end{proposition}

\begin{proof}
We first show that every holonomic $R$-module $M$ is a Noetherian object of the abelian category $\Mod_{R}^{\perf}$.
Write $M = M_0^{\perfection}$ for some $M_0 \in \Mod_{R}^{\Frob}$ which is finitely generated as an $R$-module.
Replacing $M_0$ by its image in $M$, we can assume without loss of generality that $M_0$ is a submodule of $M$. For any
subobject $M' \subseteq M$ in the abelian category $\Mod_{R}^{\alg}$, let $M'_0 = M' \cap M_0$. Note that for
any $x \in M'$, we have $\varphi_{M}^{n}(x) \in M_0$ for $n \gg 0$. It follows that $M' = \{ x \in M: (\exists n) [ \varphi_{M}^{n}(x) \in M'_0] \} \simeq M'^{\perfection}_0$.
Consequently, the construction $M' \mapsto M'_0$ determines an monomorphism from the partially ordered set of subobjects of $M$ (in the abelian category $\Mod_{R}^{\perf}$)
to the partially ordered set of subobjects of $M_0$ (in the abelian category $\Mod_{R}$).
Since $M_0$ is a Noetherian $R$-module, the latter partially ordered set satisfies the ascending chain condition, so the former does as well.

Now suppose that $M$ is an arbitrary algebraic Frobenius module over $R$. Using Corollary \ref{corX59}, we can choose
an epimorphism of Frobenius modules $\bigoplus_{\alpha \in I} M_{\alpha} \rightarrow M$, where each $M_{\alpha}$
is holonomic. For every finite subset $I_0 \subseteq I$, let $M_{I_0}$ denote the image of the composite map
$$ \bigoplus_{\alpha \in I_0} M_{\alpha} \hookrightarrow \bigoplus_{\alpha \in I_0} M_{\alpha} \rightarrow M.$$
Then $M = \bigcup_{I_0} M_{I_0}$, and the first part of the proof shows that each $M_{I_0}$ is Noetherian.
This proves that the category $\Mod_{R}^{\alg}$ is locally Noetherian. It follows that every Noetherian object
of $\Mod_{R}^{\alg}$ is compact, and therefore arises as a direct summand of a holonomic
Frobenius module by virtue of Theorem \ref{theoX54}. Since the collection of holonomic Frobenius modules
is closed under passage to direct summands, it follows that $\Mod_{R}^{\hol}$ is precisely the collection
of Noetherian objects of $\Mod_{R}^{\alg}$.
\end{proof}

\begin{remark}\label{subtil}
Let $R$ be a Noetherian $\F_p$-algebra and let $M$ be a holonomic Frobenius module over $R$.
Then, for every integer $k$, the construction $N \mapsto \Ext^{k}_{ R[F] }( M, N )$ commutes with
filtered direct limits when restricted to {\em perfect} Frobenius modules over $R$. This follows
from the exact sequence of Remark \ref{remX8} (together with the fact that the construction
$N \mapsto \Ext^{k}_{R}(M_0, N)$ commutes with filtered colimits, whenever $M_0$ is a finitely generated $R$-module).
\end{remark}

\begin{corollary}\label{corX30}
Let $R$ be a commutative $\F_p$-algebra. Then $\Mod_{R}^{\hol}$ is an abelian subcategory of $\Mod_{R}^{\perf}$ which is closed under the formation of kernels, cokernels, and extensions.
\end{corollary}

\begin{proof}
Closure under the formation of cokernels was established in Corollary \ref{corollary.propX60}. We next show that it is closed under kernels. 
Let $u: M \rightarrow N$ be a morphism of holonomic Frobenius modules over $R$; we wish to show that the kernel $\ker(u)$ is holonomic.
Using a direct limit argument, we can write $u = f^{\diamond}(u_0)$ for some map $f: R_0 \rightarrow R$ where $R_0$ is a finitely generated $\F_p$-algebra and some $u_0: M_0 \rightarrow N_0$ in $\Mod_{R_0}^{\hol}$. 
Since $R_0$ is Noetherian, Proposition \ref{propX29} guarantees that $M_0$ is a Noetherian object of $\Mod_{R_0}^{\hol}$. It follows that
$\ker(u_0) \subseteq M_0$ is also a Noetherian object of $\Mod_{R_0}^{\hol}$. Corollary \ref{lushin} supplies an isomorphism
$\ker(u) \simeq f^{\diamond} \ker(u_0)$, so that $\ker(u)$ is a holonomic Frobenius module over $R$ by virtue of Proposition \ref{propX51}.

We now prove closure under extensions. Suppose that we are given a short exact sequence of Frobenius modules over $R$,
$$ 0 \rightarrow M' \rightarrow M \rightarrow M'' \rightarrow 0,$$
where $M'$ and $M''$ are holonomic; we wish to show that $M$ is also holonomic. Using Proposition \ref{propX52}, we can
write $M' = f^{\diamond} M'_0$ and $M'' = f^{\diamond} M''_0$, where $f: R_0 \hookrightarrow R$ is the inclusion of a finitely generated subring
and the Frobenius modules $M'_0$ and $M''_0$ are holonomic over $R_0$. The preceding exact sequence is then classified by
an element $$\eta \in \Ext^{1}_{ R[F] }( M'', M' ) \simeq \Ext^{1}_{ R_0[F] }( M''_0, R \otimes_{R_0} M'_0 ).$$
Applying Remark \ref{subtil}, we can arrange (after enlarging $R_0$ if necessary) that $\eta$ can be lifted
to an element $\eta_0 \in \Ext^{1}_{ R_0[F] }( M''_0, M'_0 )$, which classifies a short exact sequence
$$ 0 \rightarrow M'_0 \rightarrow M_0 \rightarrow M''_0 \rightarrow 0,$$
of Frobenius modules over $R_0$. Since $R_0$ is Noetherian, we can regard $M'_0$ and $M''_0$
as Noetherian objects of the category $\Mod_{R_0}^{\perf}$ (Proposition \ref{propX29}). It follows that $M_0$ is also a Noetherian object
of the abelian category $\Mod_{R_0}^{\perf}$, and is therefore a holonomic Frobenius module over $R_0$ (Proposition \ref{propX29}).
Applying Proposition \ref{propX51}, we deduce that $M \simeq f^{\diamond} M_0$ is a holonomic Frobenius module over $R$. 
\end{proof}

\subsection{The Support of a Holonomic Frobenius Module}\label{sec4sub4}

\begin{definition}\label{defsupp}
Let $R$ be a commutative $\F_p$-algebra and let $M$ be a perfect Frobenius module over $R$.
For each point $x \in \Spec(R)$, we let $\kappa(x)$ denote the residue field
of $R$ at $x$, and we let $f_{x}: R \rightarrow \kappa(x)$ denote the canonical map
We let $\supp(M)$ denote the set $\{ x \in \Spec(R): f_{x}^{\diamond}(M) \neq 0 \}$.
We will refer to $\supp(M)$ as the {\it support of $M$}.
\end{definition}

\begin{remark}\label{rem70}
Let $R$ be an $\F_p$-algebra, let $M \in \Mod_{R}^{\perf}$, and let $x$ be a point of $\Spec(R)$, corresponding
to a prime ideal $\mathfrak{p} \subseteq R$. The following conditions are equivalent:
\begin{itemize}
\item[$(1)$] The point $x$ belongs to the support of $M$.
\item[$(2)$] There exists a field $\kappa$ and an $R$-algebra homomorphism $f: R \rightarrow \kappa$
such that $\ker(f) = \mathfrak{p}$ and $f^{\diamond} M \neq 0$.
\item[$(3)$] For every field $\kappa$ and every map $f: R \rightarrow \kappa$ with $\ker(f) = \mathfrak{p}$,
we have $f^{\diamond} M \neq 0$.
\end{itemize}
To see this, we note that any map $f: R \rightarrow \kappa$ with $\ker(f) = \mathfrak{p}$ factors uniquely
as a composition $R \xrightarrow{ f_{x } } \kappa(x) \xrightarrow{\iota} \kappa$, so that we have a canonical isomorphism
$$f^{\diamond} M \simeq \kappa^{\perfection} \otimes_{ \kappa(x)^{\perfection} } f_{x}^{\diamond}(M).$$
\end{remark}

\begin{remark}\label{rem71}
Let $f: A \rightarrow B$ be an $\F_p$-algebra homomorphism and let $M \in \Mod_{A}^{\perf}$. Then
$\supp( f^{\diamond} M)$ is the inverse image of $\supp(M)$ under the map $\Spec(B) \rightarrow \Spec(A)$ determined by $f$ (this follows immediately
from Remark \ref{rem70}). 
\end{remark}

The support $\supp(M)$ of Definition \ref{defsupp} is well-behaved when $M$ is holonomic: 

\begin{theorem}\label{theo74}
Let $R$ be a commutative $\F_p$-algebra and let be a holonomic Frobenius module over $R$.
Then $\supp(M)$ is a constructible subset of $\Spec(R)$.
\end{theorem}

\begin{proof}
Using a direct limit argument, we can choose a finitely
generated subring $R' \subseteq R$ and an isomorphism $M \simeq (R \otimes_{R'} M')^{\perfection}$ for some 
$M' \in \Mod_{R'}^{\hol}$. Replacing $R$ by $R'$ and $M$ by $M'$, we can reduce to the case where $R$ is Noetherian.
By general nonsense, it will suffice to prove the following:
\begin{itemize}
\item[$(a)$] If $x \in \Spec(R)$ belongs to the support of $M$, then there exists an open subset $U \subseteq \overline{ \{x \} }$ which is
contained in $\supp(M)$.
\item[$(b)$] If $x \in \Spec(R)$ does not belong to the support $M$, then there exists an open subset $U \subseteq \overline{ \{ x \} }$
which is disjoint from $\supp(M)$.
\end{itemize}
Let $\mathfrak{p} \subseteq R$ be the prime ideal determined by the point $x$.
Using Remark \ref{rem71}, we can replace $R$ by $R / \mathfrak{p}$ and thereby reduce to the case where $R$ is an integral domain and $x$ is the generic point of $\Spec(R)$. 

Write $M = M_0^{\perfection}$, where $M_0$ is finitely generated as an $R$-module. Let $K$ denote the fraction field of $R$. Then
$V = K^{\perfection} \otimes_{R} M_0$ is a finite-dimensional vector space over $K^{\perfection}$, equipped with a Frobenius-semilinear endomorphism
$\varphi_{V}: V \rightarrow V$. Then $\bigcup_{n} \ker( \varphi_{V}^{n} )$ is a $K^{\perfection}$-subspace of $V$, which admits a basis
$\{ v_i \}_{1 \leq i \leq k}$. 
Replacing $R$ by a localization $R[t^{-1}]$ for some nonzero element $t \in R$, we can assume that each $v_{i}$ can be lifted to an element
of $R^{\perfection} \otimes_{R} M_0$, and therefore also to an element $\overline{v}_i \in R' \otimes_{R} M_0$ for some subalgebra $R' \subseteq R^{\perfection}$ which is finitely generated over
$R$. Note that the inclusion $R \hookrightarrow R'$ induces a homeomorphism $\Spec(R') \rightarrow \Spec(R)$. We may therefore replace $R$ by $R'$ (and
$M_0$ by $R' \otimes_{R} M_0$) and thereby reduce to the case where each $\overline{v}_i$ belongs to $M_0$. Replacing $R$ by a localization if necessary, 
we may further assume that each $\overline{v}_i$ is annihilated by some power of $\varphi_{M_0}$. Then the set $\{ \overline{v}_i \}_{1 \leq i \leq k}$ generates a Frobenius submodule
$M'_0 \subseteq M_0$ whose perfection vanishes. We may therefore replace $M_0$ by $M_0 / M'_0$ and thereby reduce to the case where the map $\varphi_{V}$ is injective.

Let us identify $\varphi_{M_0}$ with an $R$-linear map $\beta: M_0^{(1)} \rightarrow M_0$, where $M_0^{(1)}$ is obtained from
$M_0$ by extension of scalars along the Frobenius map $\varphi_{R}: R \rightarrow R$. Note that the induced map
$\beta_{K^{\perfection}}: K^{\perfection} \otimes_{ R } M_0^{(1)} \rightarrow K^{\perfection} \otimes_{R} M_0$ can be identified with $\varphi_{V}$ and
is therefore a monomorphism. Since the domain and codomain of $\beta_{ K^{\perfection} }$ are vector spaces of the same dimension over
$K^{\perfection}$, it follows that $\beta_{ K^{\perfection} }$ is an isomorphism. Replacing $R$ by a localization $R[t^{-1}]$ if necessary,
we can assume that $M_0$ is a free module of finite rank $r$ and that $\beta$ is an isomorphism.
In this case, it is easy to see that $\supp(M) = \begin{cases} \Spec(R) & \text{ if } r > 0 \\
\emptyset & \text{ if } r = 0, \end{cases}$ from which assertions $(a)$ and $(b)$ follow immediately.
\end{proof}

\newpage \section{Compactly Supported Direct Images}\label{section.compactimage}
\setcounter{subsection}{0}
\setcounter{theorem}{0}

Let $f: A \rightarrow B$ be a morphism of commutative rings. Then the direct image
functor $f_{\ast}: \Shv_{\mathet}( \Spec(B), \F_p) \rightarrow \Shv_{\mathet}( \Spec(A), \F_p)$
admits a left adjoint, which we denote by $f^{\ast}: \Shv_{\mathet}( \Spec(A), \F_p) \rightarrow \Shv_{\mathet}( \Spec(B), \F_p )$
and refer to as {\it pullback along $f$}. In the special case where $f$ is {\etale}, the pullback functor can be described concretely
by the formula $(f^{\ast} \sheafF)(B') = \sheafF(B')$. In particular, $f^{\ast}$ preserves inverse limits. It follows (either by the adjoint functor theorem, or by direct construction)
that the functor $f^{\ast}$ admits a further left adjoint, which we denote by $f_{!}: \Shv_{\mathet}( \Spec(B), \Lambda) \rightarrow \Shv_{\mathet}( \Spec(A), \Lambda)$
and refer to as the {\it compactly supported direct image functor}. 

Now suppose that $f: A \rightarrow B$ is a homomorphism of commutative $\F_p$-algebras. Under the Riemann-Hilbert correspondence of Theorem \ref{maintheoXXX}, the pullback functor $f^{\ast}: \Shv_{\mathet}( \Spec(A), \F_p) \rightarrow \Shv_{\mathet}( \Spec(B), \F_p )$ on {\etale} sheaves corresponds to the extension of scalars functor $f^{\diamond}: \Mod_{A}^{\alg} \rightarrow \Mod_{B}^{\alg}$ on algebraic Frobenius modules
(see Proposition \ref{prop70}). One consequence of this is that, if the morphism $f$ is {\etale}, the functor $f^{\diamond}: \Mod_{A}^{\alg} \rightarrow \Mod_{B}^{\alg}$ must also admit a left adjoint.
Our goal in this section is to give a direct proof of this statement, which does not appeal to the Riemann-Hilbert correspondence (in fact, the work of this section will be needed in
\S \ref{section.RH} to {\em construct} the Riemann-Hilbert functor). 

We begin in \S \ref{sec5sub1} by introducing a notion of {\it compactly supported direct image} in the setting of Frobenius modules (Definition \ref{def51}).
From the definition, it will be immediately clear that if $f: A \rightarrow B$ is an {\etale} morphism of commutative $\F_p$-algebras, then
the formation of compactly supported direct images supplies a {\em partially defined} functor $f_{!}: \Mod_{A}^{\alg} \rightarrow \Mod_{B}^{\alg}$.
Our main result, which we prove in \S \ref{sec5sub4}, is that this functor is actually total: that is, compactly supported direct images of algebraic modules always exist (Theorem \ref{theoX15}).
The strategy of proof is to use the structure theory of {\etale} morphisms to reduce to the case where $B$ is a localization $A[t^{-1}]$, which we handle in \S \ref{sec5sub2}. In this case, 
compactly supported direct images of holonomic Frobenius modules admit a very simple characterization (Proposition \ref{propX12}) which makes them easy to construct explicitly.
In \S \ref{sec5sub3}, we apply this characterization prove an analogue of Kashiwara's theorem for for Frobenius modules: the datum of a holonomic Frobenius module over a quotient ring $R/(t)$ is equivalent to the datum of a holonomic Frobenius module over $R$ whose support is contained in the vanishing locus of $t$ (Theorem \ref{theorem.kashiwara}).

\subsection{Definitions}\label{sec5sub1}

We begin by introducing some terminology.

\begin{definition}\label{definition.compactlysupported}
Let $f: A \rightarrow B$ be an {\etale} morphism of $\F_p$-algebras.
Let $M$ be a perfect Frobenius module over $B$ and let $\overline{M}$ be a perfect Frobenius module over $A$.
We will say that a morphism $u: M \rightarrow f^{\diamond} \overline{M}$ {\it exhibits $\overline{M}$ as a weak compactly supported direct image of $M$}
if, for every perfect Frobenius module $N$ over $A$, the composite map
$$ \Hom_{ A[F] }( \overline{M}, N ) \rightarrow \Hom_{ B[F] }( f^{\diamond} \overline{M}, f^{\diamond} N) \xrightarrow{ \circ u} \Hom_{ B[F] }( M, f^{\diamond} N)$$
is a bijection.
\end{definition}

Let $f: A \rightarrow B$ be an {\etale} morphism of $\F_p$-algebras and let $M$ be a perfect Frobenius module over $B$. It follows immediately from
the definition that if there exists a morphism $u: M \rightarrow f^{\diamond} \overline{M}$ which exhibits $\overline{M}$ as a weak compactly supported direct image of $M$,
then the Frobenius module $\overline{M}$ (and the morphism $u$) are determined up to canonical isomorphism. In general, such a module need not exist.
The main result of this section (Theorem \ref{theoX15}) asserts that every {\em algebraic} Frobenius module $M$ over $B$ admits a weak compactly supported direct image $\overline{M}$. Moreover, we will have the following additional properties:
\begin{itemize}
\item[$(a)$] The module $\overline{M}$ is also algebraic (as a Frobenius module over $A$). 
\item[$(b)$] For every perfect Frobenius module $N$ over $A$, the canonical map
$$ \Ext^{n}_{ A[F] }( \overline{M}, N ) \rightarrow \Ext^{n}_{ B[F] }( M, f^{\diamond} N)$$
is an isomorphism for all integers $n$, rather than merely for $n=0$.
\item[$(c)$] The Frobenius module $\overline{M}$ remains a weak compactly supported direct image of $M$ after any extension of scalars. More precisely, for any
pushout diagram of commutative rings
$$ \xymatrix{ A \ar[r]^{f} \ar[d]^{g'} & B \ar[d]^{g} \\
A' \ar[r]^{f'} & B', }$$
the induced map $g'^{\diamond}(u): g'^{\diamond}(M) \rightarrow g'^{\diamond}( f^{\diamond}( \overline{M} ) ) \simeq f'^{\diamond}( g^{\diamond}( \overline{M} ) )$
exhibits $g^{\diamond}( \overline{M} )$ as a weak compactly supported direct image of $g'^{\diamond}(M)$.
\end{itemize}

Our proof for the existence of weak compactly supported direct images will proceed by a somewhat complicated induction on the structure of the {\etale} morphism
$f: A \rightarrow B$. In order to carry out the details, it will be important to strengthen our inductive hypothesis: that is, we need to show not only that
weak compactly supported direct images exist, but also that they have the properties listed above. For this reason, it will be convenient to introduce a more complicated version
of Definition \ref{definition.compactlysupported} which incorporates properties $(a)$, $(b)$, and $(c)$ automatically.

\begin{definition}\label{def51}
Let $f: A \rightarrow B$ be an {\etale} morphism of $\F_p$-algebras. Suppose we are given algebraic Frobenius modules $M \in \Mod^{\alg}_{B}$ and $\overline{M} \in \Mod^{\alg}_{A}$.
We will say that a morphism $u: M \rightarrow f^{\diamond} \overline{M}$ {\it exhibits $\overline{M}$ as a compactly supported direct image of $M$}
if the following conditions is satisfied:
\begin{itemize}
\item[$(\ast)$] For every pushout diagram of commutative rings
$$ \xymatrix{ A \ar[r]^{f} \ar[d]^{g'} & B \ar[d]^{g} \\
A' \ar[r]^{f'} & B' }$$
and every object $N \in \Mod^{\perf}_{A'}$, the composite map
\begin{eqnarray*}
 \Ext_{ A'[F] }^{\ast}( g'^{\diamond} \overline{M}, N )  & \rightarrow & \Ext^{\ast}_{ B'[F] }( f'^{\diamond} g'^{\diamond} \overline{M},  f'^{\diamond} N ) \\
& \simeq & \Ext^{\ast}_{ B'[F] }( g^{\diamond} f^{\diamond} \overline{M}, f'^{\diamond} N) \\
& \rightarrow & \Ext^{\ast}_{B'[F]}( g^{\diamond} M, f'^{\diamond} N ). \end{eqnarray*}
is an isomorphism.
\end{itemize}
\end{definition}

\begin{remark}
Let $f: A \rightarrow B$ and $u: M \rightarrow f^{\diamond} \overline{M}$ be as in Definition \ref{def51}. If $f$ exhibits
$\overline{M}$ as a compactly supported direct image of $M$, then it also exhibits $\overline{M}$ as a weak compactly supported direct image of $M$.
In fact, the converse holds as well (assuming that $M$ is algebraic): this follows from the uniqueness of weak compactly supported direct images,
once we have shown that compactly supported direct images exist (Theorem \ref{theoX15}).
\end{remark}

\begin{notation}\label{notation.csupport}
Let $f: A \rightarrow B$ be an {\etale} morphism of $\F_p$-algebras and let $M \in \Mod^{\alg}_{B}$. If
there exists an object $\overline{M} \in \Mod_{A}^{\alg}$ and a morphism $u: M \rightarrow f^{\diamond} \overline{M}$ which exhibits
$\overline{M}$ as a compactly supported direct image of $M$, then we will denote $\overline{M}$ by $f_{!} M$. 
In this case, we will say that {\it $f_{!} M$ exists}. Note that, in this event, the Frobenius module $f_{!} M$ depends functorially on $M$.
\end{notation}

\subsection{Extension by Zero}\label{sec5sub2}

Our next goal is to prove the existence of compactly supported direct images in the case of an elementary open immersion $$\Spec( A[t^{-1}] ) \hookrightarrow \Spec(A)$$
(Proposition \ref{prop22}). In this case, Definition \ref{def51} can be formulated more simply, at least for holonomic Frobenius modules.

\begin{definition}\label{definition.extendbyzero}
Let $A$ be a commutative $\F_p$-algebra containing an element $t$ and let $M$ be a holonomic Frobenius module over $A[t^{-1}]$.
An {\it extension by zero of $M$} is a holonomic Frobenius module $\overline{M}$ over $A$ such that
$\overline{M}[t^{-1}]$ is isomorphic to $M$ and $(\overline{M} / t \overline{M})^{\perfection} \simeq 0$.
\end{definition}

\begin{proposition}\label{propX12}
Let $A$ be an $\F_p$-algebra containing an element $t$, and let $f: A \rightarrow A[t^{-1}]$ be the canonical map. Suppose
we are given holonomic Frobenius modules $M \in \Mod_{ A[t^{-1}] }^{\hol}$ and $\overline{M} \in \Mod_{A}^{\hol}$
together with a map $u: M \rightarrow f^{\diamond} \overline{M} \simeq \overline{M}[t^{-1}]$. The
following conditions are equivalent:
\begin{itemize}
\item[$(a)$] The morphism $u$ exhibits $\overline{M}$ as a compactly supported direct image of $M$, in the sense of Definition \ref{def51}.

\item[$(b)$] The morphism $u$ exhibits $\overline{M}$ as an extension by zero of $M$: that is, $u$ is an isomorphism and
the Frobenius module $( \overline{M} / t \overline{M} )^{\perfection}$ vanishes.
\end{itemize}
\end{proposition}

The proof of Proposition \ref{propX12} will require an elementary fact from commutative algebra:

\begin{lemma}\label{esbo}
Let $M$ and $N$ be modules over a commutative ring $A$, and let
$\gamma$ be an element of $\Ext^{n}_{A}(M, N)$ for some $n \geq 0$. Suppose that $M$ is Noetherian. If
$\gamma$ is annihilated by some power of an element $t \in A$, then there exists $d \geq 0$ such that
the image of $\gamma$ vanishes in $\Ext^{n}_{A}( t^{d} M, N)$.
\end{lemma}

\begin{proof}
Let $M_0 \subseteq M$ be the submodule consisting of those elements which
are annihilated by some power of $t$. Since $M$ is Noetherian, we can choose an integer $k \gg 0$ such
that $M_0$ is annihilated by $t^{k}$. For each $d \geq k$, the kernel of the surjection $M \xrightarrow{t^{d}} t^{d} M$
is annihilated by $t^{k}$, so there exists a dotted arrow as indicated in the diagram
$$ \xymatrix{ M \ar[r]^-{t^{k}} \ar[d]^{ t^{d} } & M \ar[d]^{ t^{d-k} } \\
t^{d} M \ar[r] \ar@{-->}[ur] & M. }$$
It follows that the restriction map $\Ext^{n}_{A}( M, N) \rightarrow \Ext^{n}_{A}( t^{d} M, N)$
factors through the map $t^{d-k}: \Ext^{n}_{A}(M,N) \rightarrow \Ext^{n}_{A}(M,N)$.
It therefore suffices to choose $d$ large enough that $t^{d-k} \gamma = 0$.
\end{proof}

\begin{proof}[Proof of Proposition \ref{propX12}]
Assume first that $(a)$ is satisfied. Applying condition $(\ast)$ of Definition \ref{def51} in the case $A' = A[t^{-1}]$, we deduce that
$u$ is an isomorphism. Applying condition $(\ast)$ of Definition \ref{def51} in the case $A' = A/(t)$, we deduce that $(\overline{M}/ t \overline{M})^{\perfection} \simeq 0$.

We now prove the converse. Assume that $u$ is an isomorphism and that $( \overline{M} /t \overline{M} )^{\perfection} \simeq 0$, and suppose we are given a pushout diagram of
$\F_p$-algebras $$ \xymatrix{ A \ar[r]^-{f} \ar[d]^{g'} & A[t^{-1}] \ar[d]^{g} \\
A' \ar[r]^-{f'} & A'[t^{-1}] }$$
and an object $N \in \Mod^{\perf}_{A'}$. To verify condition $(\ast)$ of Definition \ref{def51}, it will suffice to show that the canonical map
$$ \theta: \Ext_{ A'[F] }^{\ast}( g'^{\diamond} \overline{M}, N ) \rightarrow \Ext^{\ast}_{ A'[t^{-1}][F] }( f'^{\diamond} g'^{\diamond} \overline{M},  f'^{\diamond} N )$$
is an isomorphism. Replacing $A$ by $A'$ and $\overline{M}$ by $g'^{\diamond} \overline{M}$, we can reduce to the case $A = A'$. Let $Q \in \Mod_{A}^{\perf}$ denote
the image of the unit map $N \rightarrow f^{\diamond} N \simeq N[t^{-1}]$, so that we have short exact sequences
$$ 0 \rightarrow K \rightarrow N \rightarrow Q \rightarrow 0 \quad \quad 0 \rightarrow Q \rightarrow N[t^{-1}] \rightarrow K' \rightarrow 0.$$
In order to show that the composite map $$\Ext_{ A[F] }^{\ast}( \overline{M}, N ) \rightarrow \Ext_{ A[F]^{\ast} }( \overline{M}, Q ) \rightarrow \Ext_{ A[F] }^{\ast}( \overline{M}, N[t^{-1} ] )$$
is an isomorphism, it will suffice to show that the groups $\Ext_{A[F]}^{\ast}( \overline{M}, K)$ and $\Ext_{ A[F]}^{\ast}( \overline{M}, K')$ vanish. This is a special case of the following:
\begin{itemize}
\item[$(\ast)$] If $\overline{M} \in \Mod_{A}^{\hol}$ satisfies $( \overline{M} / t \overline{M} )^{\perfection} \simeq 0$ and
$N \in \Mod_{A}^{\perf}$ satisfies $N[t^{-1}] \simeq 0$, then $\Ext^{\ast}_{ A[F] }( \overline{M}, N )$ vanishes.
\end{itemize}
Write $\overline{M} = T^{\perfection}$ for some Frobenius module $T \in \Mod_{A}^{\Frob}$ in $A$ which is finitely presented as an $A$-module.
Using our assumption that $(\overline{M} / t \overline{M})^{\perfection} \simeq 0$, we deduce that 
$\varphi_{ T}^{m} T \subseteq t T$ for $m \gg 0$.
By a direct limit argument, we can assume that $T \simeq A \otimes_{A_0} T_0$ for some finitely generated subalgebra
$A_0 \subseteq A$ which contains $t$ and some $T_0 \in \Mod_{A_0}^{\Frob}$ which is finitely presented as an $A_0$-module, and that $T_0$ satisfies
$\varphi_{ T_0}^{m} T_0 \subseteq t T_0$. Using Corollary \ref{corX11}, we can replace $A$ by $A_0$ and
$T$ by $T_0$, and thereby reduce to proving $(\ast)$ in the special case where $A$ is Noetherian.

Using Remark \ref{remX8}, we obtain a long exact sequence
$$ \Ext_{A[F]}^{\ast}(\overline{M}, N) \rightarrow \Ext^{\ast}_{A}(T, N) \xrightarrow{\id - U} \Ext^{\ast}_{A}(T, N) \rightarrow \Ext^{\ast+1}_{A[F]}( \overline{M}, N),$$
where $U$ is the endomorphism of $\Ext^{\ast}_{A}( T, N)$ given by $U(\gamma) = \varphi_{N}^{-1} \circ \gamma \circ \varphi_{T}$.
To prove $(\ast)$, it will suffice to show that the map $1 - U$ is an isomorphism of $\Ext^{\ast}_{A}( T, N)$ with itself.
In fact, we will show that $U$ is locally nilpotent (so that $1-U$ has an inverse given by the formal infinite sum $1 + U + U^2 + \cdots$). 

Fix an element $\gamma \in \Ext^{k}_{A}( T, N)$; we wish to show that $U^{m}(\gamma) \simeq 0$ for $m \gg 0$.
Since $A$ is Noetherian and $T$ is a finitely generated $A$-module, the construction $S \mapsto \Ext^{k}_{A}( T, S)$ commutes with filtered colimits.
In particular, there exists a finitely generated $A$-submodule $N_0 \subseteq N$ such that $\gamma$ can be lifted to an element
$\gamma_0 \in \Ext^{k}_{A}( T, N_0)$. Using our assumption that $N[t^{-1}] \simeq 0$, we deduce that $N_0$ is annihilated by
$t^{c}$ for $c \gg 0$. It follows that the image of $\gamma_0$ in $\Ext^{k}_{A}( t^{c'} T, N_0)$ vanishes for $c' \gg c$ (Lemma \ref{esbo}).
We now observe that for $m \gg 0$, the map $\varphi_{ T' }^{m}$ factors through $t^{c'} T$, so that
$\gamma \circ \varphi_{ T}^{m} = 0$ and therefore $U^m(\gamma) = 0$ as desired.
\end{proof}

\begin{proposition}\label{prop22}
Let $A$ be an $\F_p$-algebra containing an element $t$, and let $f: A \rightarrow A[t^{-1}]$ be the canonical map. Then every algebraic
Frobenius module $M$ over $A[t^{-1}]$ admits a compactly supported direct image along $f$.
\end{proposition}

\begin{proof}
Using Theorem \ref{theoX54}, we can reduce to the case where $M$ is holonomic. Write $M = M_0^{\perfection}$, where $M_0 \in \Mod_{A[t^{-1}]}^{\Frob}$ is finitely
presented as an $A[t^{-1}]$-module. Then we can choose an isomorphism $\alpha: M_0 \simeq \overline{M}_0[t^{-1}]$ for some finitely presented object $\overline{M}_0 \in \Mod_{ A }$. 
Then $\varphi_{M_0}$ determines an $A$-module homomorphism $\rho: \overline{M}_0 \rightarrow \overline{M}_0[t^{-1}]^{1/p}$. 
Since $\overline{M}_0$ is finitely presented as an $A$-module, we can assume that $\rho$ factors as a composition
$$ \overline{M}_0 \rightarrow \overline{M}_0^{1/p} \xrightarrow{ t^{n} } \overline{M}_0[t^{-1}]^{1/p}$$
for some integer $n$. Multiplying the isomorphism $\alpha$ by a suitable power of $t$, we can arrange that $n > 0$.
Set $\overline{M} = \overline{M}_0^{\perfection}$. Then $\alpha$ induces an isomorphism $M \simeq \overline{M}[t^{-1}]$ and
$\varphi_{\overline{M}}$ is locally nilpotent on $\overline{M} / t \overline{M}$, so that $\alpha$ exhibits $\overline{M}$
as an extension by zero of $M$, which is also a compactly supported direct image of $M$ by virtue of Proposition \ref{propX12}.
\end{proof}

\subsection{Kashiwara's Theorem}\label{sec5sub3}

Let $X$ be a smooth algebraic variety over the field $\C$ of complex numbers, and let $Y \subseteq X$ be a smooth subvariety of $X$.
A theorem of Kashiwara (see \cite[\S 1.6]{HottaEtAl}) asserts that the category of algebraic $\mathcal{D}$-modules on $Y$ is equivalent to the category of algebraic
$\mathcal{D}$-modules on $X$ which vanish over the open set $X-Y$. In this section, we prove the following analogue
for (holonomic) Frobenius modules:

\begin{theorem}\label{theorem.kashiwara}
Let $M$ be a holonomic Frobenius module over a commutative $\F_p$-algebra $A$, and let $I \subseteq A$ be an ideal.
The following conditions are equivalent:
\begin{itemize}
\item The support $\supp(M)$ is contained in the vanishing locus $\Spec(R/I) \subseteq \Spec(R)$.
\item The submodule $IM \subseteq M$ vanishes: that is, $M$ has the structure of a Frobenius module over $R/I$.
\end{itemize}
\end{theorem}

\begin{remark}\label{elos}
If the equivalent conditions of Theorem \ref{theorem.kashiwara} are satisfied, then
$M$ is also holonomic when regarded as a Frobenius module over $R/I$. Conversely,
if the ideal $I$ is finitely generated, then any holonomic Frobenius module over $R/I$ is a holonomic
Frobenius module over $R$ which satisfies the conditions of Theorem \ref{theorem.kashiwara} (see Proposition \ref{propX45}). In other words, the category $\Mod_{R/I}^{\hol}$ can be identified with the full subcategory of $\Mod_R^{\hol}$ spanned by objects set-theoretically supported on $\Spec(R/I) \subseteq \Spec(R)$. Beware that this is generally not true if $I$ is not finitely generated (Warning \ref{warning.X45}).
\end{remark}

We begin by treating the following special case of Theorem \ref{theorem.kashiwara} (which is the only case we will actually need):

\begin{proposition}\label{theo78}
Let $M$ be a holonomic Frobenius module over a commutative $\F_p$-algebra $A$. Then $M \simeq 0$ if and only if
the support $\supp(M)$ is empty.
\end{proposition}

\begin{proof}
The ``only if'' direction is obvious. To prove the converse, let us assume that $\supp(M) = \emptyset$; we wish to prove that
$M \simeq 0$. Using Proposition \ref{propX52}, we can choose a finitely generated subring $A' \subseteq A$ and an
equivalence $M \simeq (A \otimes_{A'} M')^{\perfection}$ for some $M' \in \Mod_{A'}^{\hol}$. 
Set $K = \supp(M') \subseteq \Spec(A')$. Then $K$ is constructible (Theorem \ref{theo74}). 
Using Remark \ref{rem71}, we deduce that the image of the map $\Spec(A) \rightarrow \Spec(A')$ is disjoint from $K$. 
Enlarging $R'$ if necessary, we can arrange that $K = \emptyset$. We may therefore replace $A$ by $A'$ and
$M$ by $M'$, and thereby reduce to the case where $A$ is Noetherian. 

Proceeding by Noetherian induction, we may assume that for every nonzero ideal $I \subseteq A$, we have $( M / IM)^{\perfection} \simeq 0$.
We may assume that $A \neq 0$ (otherwise, there is nothing to prove). If $A$ is not reduced, then taking $I$ to be the nilradical of $A$ we deduce that $M = M^{\perfection} \simeq (M/IM)^{\perfection} \simeq 0$. 
We may therefore assume that $A$ is reduced. Using Proposition \ref{propX12}, we deduce that $M$ is the compactly supported direct image of $M[x^{-1}]$ for every nonzero element $x \in A$. It will therefore suffice to show that we can choose a nonzero element $x \in A$ such that $M[x^{-1}] \simeq 0$. Since $A$ is reduced and Noetherian, we can choose a non-zero divisor $t \in A$ such that $A[t^{-1}]$ is an integral domain. Replacing
$A$ by $A[t^{-1}]$, we can assume that $A$ is an integral domain. Write $M = M_0^{\perfection}$ for some
$M_0 \in \Mod_{A}^{\Frob}$ which is finitely presented as an $A$-module. Let $K$ be the fraction field of $A$. Since the support $\supp(M)$ does not contain the generic
point of $\Spec(A)$, the Frobenius module $(K \otimes_{A} M_0)^{\perfection}$ vanishes. Using the finite generation of $M_0$, we conclude that
the Frobenius endomorphism of $K \otimes_{A} M_0$ is nilpotent. It follows that there exists a nonzero element $x \in A$ for which
the Frobenius map $\varphi_{ M_0[x^{-1}] }$ is nilpotent, so that $M[x^{-1}] \simeq M_0[x^{-1}]^{\perfection} \simeq 0$ as desired.
\end{proof}

\begin{proof}[Proof of Theorem \ref{theorem.kashiwara}]
Let $M$ be a holonomic Frobenius module over a commutative $\F_p$-algebra $A$ and let $I \subseteq A$ be an ideal. 
It follows immediately from the definitions that if $M$ is annihilated by $I$, then the support $\supp(M)$ is contained in the vanishing locus of $I$.
Conversely, suppose that $\supp(M) \subseteq \Spec(A/I)$; we wish to show that $M$ is annihilated by each element $x \in I$.
Note that the inclusion $\supp(M) \subseteq \Spec(A/I)$ guarantees that the support of $M[x^{-1}]$ is empty, where we regard
$M[x^{-1}]$ as a holonomic Frobenius module over $A[x^{-1}]$. Applying Proposition \ref{theo78}, we conclude that
$M[x^{-1}] \simeq 0$. 

Choose an isomorphism $M \simeq M_0^{\perfection}$, where $M_0$ is a Frobenius module over $A$ which is finitely presented as an $A$-module.
Let $N$ denote the image of the map $M_0 \rightarrow M$, so that $N \subseteq M$ is a Frobenius submodule which is finitely generated over $A$. The vanishing of $M[ x^{-1} ]$ guarantees
that $x^{k} N = 0$ for some $k \gg 0$. Applying $\varphi_{M}^{-n}$, we conclude that $x^{\frac{k}{p^n}} N^{1/p^{n}} = 0$ for all $n \geq 0$. As $M = \varinjlim_n N^{1/p^{n}}$, it follows that $x^{\frac{k}{p^n}} M = 0$ for all $n \geq 0$
(here we regard $M$ as a Frobenius module over the perfection $A^{\perfection}$). For $n \gg 0$, this implies $xM = 0$, as desired.
\end{proof}

\subsection{Existence of Compactly Supported Direct Images}\label{sec5sub4}

Our goal in this section is to prove the following:

\begin{theorem}\label{theoX15}
Let $f: A \rightarrow B$ be an {\etale} morphism of $\F_p$-algebras. Then, for every object $M \in \Mod_{B}^{\alg}$,
there exists a compactly supported direct image $f_{!} M \in \Mod_{A}^{\alg}$ (see Notation \ref{notation.csupport}). Moreover, the functor $f_{!}: \Mod_{B}^{\alg} \rightarrow \Mod_{A}^{\alg}$ is exact.
\end{theorem}

\begin{remark}\label{remX22}
In the situation of Theorem \ref{theoX15}, the right exactness of the functor $f_{!}: \Mod_{B}^{\alg} \rightarrow \Mod_{A}^{\alg}$ is automatic
(since $f_{!}$ is left adjoint to the functor $f^{\diamond}: \Mod_{A}^{\alg} \rightarrow \Mod_{B}^{\alg}$, which is exact by virtue of Corollary \ref{lushin}).
Moreover, since the functor $f^{\diamond}: \Mod_{A}^{\alg} \rightarrow \Mod_{B}^{\alg}$ preserves filtered colimits, the functor
$f_{!}$ preserves compact objects: that is, it carries $\Mod_{B}^{\hol}$ into $\Mod_{A}^{\hol}$ (see Theorem \ref{theoX54}).
\end{remark}

The proof of Theorem \ref{theoX15} will require some preliminaries. We begin with some elementary remarks, whose proofs follow immediately from our definitions.

\begin{lemma}\label{lemX19}
Suppose we are given a pushout diagram of $\F_p$-algebras
$$ \xymatrix{ A \ar[r]^{f} \ar[d]^{g'} & B \ar[d]^{g} \\
A' \ar[r]^{f'} & B' }$$
where $f$ is {\etale}. If $\overline{M} \in \Mod_B^{\alg}$ and $u: M \rightarrow f^{\diamond} \overline{M}$ is a morphism in $\Mod_{A}^{\alg}$
which exhibits $\overline{M}$ as a compactly supported direct image of $M$, then the induced map
$g^{\diamond} M \rightarrow g^{\diamond} f^{\diamond} \overline{M} \simeq f'^{\diamond} g'^{\diamond} \overline{M}$ exhibits
$g'^{\diamond} \overline{M}$ as a compactly supported direct image of $g^{\diamond} M$.

In particular, if $f_{!} M$ exists, then $f'_{!} (g^{\diamond} M)$ exists (and is canonically isomorphic to $g'^{\diamond}( j_{!} M)$. 
\end{lemma}

\begin{lemma}\label{lemX16}
Let $f: A \rightarrow B$ and $g: B \rightarrow C$ be {\etale} morphisms of $\F_p$-algebras.
Suppose we are given an objects $M_{C} \in \Mod_{C}^{\alg}$, a morphism $u: M_{B} \rightarrow g^{\diamond} M_{C}$ in
$\Mod_{B}^{\alg}$, and a morphism $v: M_{A} \rightarrow f^{\diamond} M_{B}$ in $\Mod_{A}^{\alg}$.
Assume that $u$ exhibits $M_{B}$ as a compactly supported direct image of $M_{C}$. Then
$v$ exhibits $M_{A}$ as a compactly supported direct image of $M_{B}$ if and only if
the composite map $$M_{A} \xrightarrow{v} f^{\diamond} M_{B} \xrightarrow{ g^{\diamond}(u) } g^{\diamond} f^{\diamond} M_{C}$$
exhibits $M_{A}$ as a compactly supported direct image of $M_{C}$.

In particular, if $g_! M$ exists, then $f_{!} (g_! M)$ exists if and only if $(g \circ f)_{!} M$ exists (and, in this case, 
they are canonically isomorphic).
\end{lemma}

\begin{lemma}\label{stona}
Let $f: A \rightarrow B$ be an {\etale} morphism of $\F_p$-algebras and suppose we are given an exact sequence
$0 \rightarrow M' \xrightarrow{u} M \rightarrow M'' \rightarrow 0$ in the abelian category $\Mod_{B}^{\alg}$.
Suppose that $f_{!} M'$ and $f_{!} M$ exist, and that the canonical map $f_{!}(u): f_{!} M' \rightarrow f_{!} M$ is
a monomorphism. Then $f_{!} M''$ exists, and is given by $\coker( f_{!}(u) )$.
\end{lemma}

\begin{lemma}\label{ulroc}
Let $f: A \rightarrow B$ be a faithfully flat {\etale} morphism of $\F_p$-algebras and let $M \in \Mod_{A}^{\perf}$. Then
$M$ is algebraic if and only if $f^{\diamond} M$ is algebraic. 
\end{lemma}

\begin{proof}
The ``only if'' direction follows from Corollary \ref{corX62}. Conversely, suppose that $f^{\diamond} M$ is algebraic.
Choose an element $x \in M$. For each $n \geq 0$, let $M(n)$ denote the $A$-submodule of $M$ generated by
the elements $\{ \varphi_{M}^{k}(x) \}_{k < n}$, so we have inclusions of $A$-submodules 
$$ \{ 0 \} = M(0) \subseteq M(1) \subseteq M(2) \subseteq \cdots \subseteq M.$$
Using Corollary \ref{corX14}, we can identify $f^{\diamond} M$ with $B \otimes_{A} M$, so that each $B \otimes_{A} M(n)$ can be identified with the $B$-submodule $M'(n) \subseteq f^{\diamond} M$ generated by $\{ \varphi^k_{B \otimes_A M}(1 \otimes x)\}_{k < n}$. Since $f^{\diamond} M$ is algebraic, there exists an integer $n$ such that $M'(n) = M'(n+1)$. The faithful flatness of $B$ over $A$ then
guarantees that $M(n) = M(n+1)$, so that $x$ satisfies an equation of the form
$\varphi_{M}^{n}(x) + a_1 \varphi_{M}^{n-1}(x) + \cdots + a_n x = 0$.
\end{proof}

\begin{lemma}\label{lemX23}
Suppose we are given a pushout square of {\etale} morphisms between $\F_p$-algebras
$$ \xymatrix{ A \ar[r]^{f} \ar[d] & B \ar[d]^{g} \\
A' \ar[r]^{f'} & B', }$$
where the vertical maps are faithfully flat. Let $M \in \Mod_{B}^{\alg}$. If 
$f'_{!} ( g^{\diamond} M )$ exists, then $f_{!} M$ exists.
\end{lemma}

\begin{proof}
Use faithfully flat descent together with Lemma \ref{ulroc}. 
\end{proof}

\begin{lemma}\label{lemX21}
Let $A$ be an $\F_p$-algebra containing an element $t$, let $f: A \rightarrow A[t^{-1}]$ be the canonical map,
and suppose we are given objects $M \in \Mod_{A[t^{-1}]}^{\alg}$ and a morphism $u: f_{!} M \rightarrow N$
in $\Mod_{A}^{\alg}$. If $f^{\diamond}(u)$ is a monomorphism, then $u$ is a monomorphism.
\end{lemma}

\begin{proof}
Set $K = \ker( f_{!}(u) )$. Then $f_{!} M / K$ is algebraic (Proposition \ref{propX53}), so Corollary
\ref{lushin} implies that the map $(K / tK)^{\perfection} \rightarrow ((f_! M) / t (f_! M))^{\perfection}$ is a monomorphism. Invoking
Proposition \ref{propX12}, we deduce that the natural map $f_{!} f^{\diamond} K \rightarrow K$ is an equivalence.
Since $f^{\diamond} K \simeq \ker( f^{\diamond} u ) \simeq 0$, we conclude that $K \simeq 0$ so that $u$ is a monomorphism as desired.
\end{proof}

\begin{proof}[Proof of Theorem \ref{theoX15}]
Let us say that an {\etale} ring homomorphism $f: A \rightarrow B$ is {\it good} if the functor $f_{!}: \Mod_{B}^{\alg} \rightarrow \Mod_{A}^{\alg}$ is well-defined and exact. 
Our proof now proceeds in several steps:

\begin{itemize}
\item[$(i)$] Every localization $f: A \rightarrow A[t^{-1}]$ is good. The existence of
$f_{!}$ follows from Proposition \ref{prop22}, and the exactness of $f_{!}$ follows
from Remark \ref{remX22} and Lemma \ref{lemX21}.

\item[$(ii)$] Let $f: A \rightarrow B$ and $g: B \rightarrow C$ be {\etale} ring homorphisms. If $f$ and $g$ are good,
then $(g \circ f): A \rightarrow C$ is good. This follows immediately from Lemma \ref{lemX16}.

\item[$(iii)$] Let $f: A \rightarrow B$ be an {\etale} $\F_p$-algebra homomorphism and suppose we are given elements
$t_0, t_1 \in B$ which generate the unit ideal. Set $B_0 = B[ t_0^{-1}] $, $B_1 = B[ t_{1}^{-1} ]$, and $B_{01} = B[ t_0^{-1}, t_{1}^{-1} ]$.
If the induced maps $f_0: A \rightarrow B_0$ and $f_1: A \rightarrow B_1$ are good, then $f_{!}$ exists. To prove this,
choose any object $M \in \Mod_{B}^{\alg}$, and define 
$$M_{0} = M[ t_0^{-1}] \in \Mod_{B_0}^{\alg} \quad  \quad
M_{1} = M[ t_1^{-1}] \in \Mod_{ B_1}^{\alg}$$
$$ M_{01} = M[ t_0^{-1}, t_1^{-1}] \in \Mod_{B_{01}}^{\alg}. $$
We have a commutative diagram 
$$ \xymatrix{ B \ar[r]^{g_0} \ar[d]^{g_1} \ar[dr]^{ g_{01} } & B_0 \ar[d]^{h} \\
B_1 \ar[r] & B_{01} }$$
which yields a short exact sequence
$$ 0 \rightarrow g_{01!} M_{01} \xrightarrow{u} g_{0!} M_0 \oplus g_{1!} M_1 \rightarrow M \rightarrow 0$$
in $\Mod_{B}^{\alg}$. 
Using our assumptions that $f_0$ and $f_1$ are good (which also implies that the induced map $f_{01}: A \rightarrow B_{01}$ is good,
using $(ii)$ and $(iii)$) together with Lemma \ref{lemX16}, we deduce that
$f_{!} ( g_{01!} M_{01} )$ and $f_{!} ( g_{0!} M_0 \oplus g_{1!} M_1)$ exist. By virtue of Lemma \ref{stona}, to prove
the existence of $f_{!} M$, it will suffice to show that $f_{!} u$ is a monomorphism. In fact, we claim that the composite map
$$f_{!} g_{01!} M_{01} \xrightarrow{ f_{!} u} f_{!}( g_{0!} M_0 \oplus g_{1!} M_1) \rightarrow f_{!} g_{0!} M_0$$
is a monomorphism. Using Lemma \ref{lemX16} and our assumption that $f_{0!}$ is exact, we are reduced
to showing that the map $h_{!} M_{01} \rightarrow M_0$ is a monomorphism in $\Mod_{B_0}^{\alg}$, which
is a special case of Lemma \ref{lemX21}.

\item[$(iv)$] Let $f: A \rightarrow B$ be as in $(iii)$. Then $f$ is good. To prove this, we must show that for
every short exact sequence $0 \rightarrow M' \rightarrow M \rightarrow M'' \rightarrow 0$ in $\Mod_{B}^{\alg}$, the induced map $f_{!} M' \rightarrow f_{!} M$ is a monomorphism.
Define $M_0$, $M_1$, and $M_{01}$ as above, and define $M'_0$, $M'_1$, $M'_{01}$, $M''_{0}$, $M''_{1}$, and $M''_{01}$ similarly. We then have a diagram of short exact sequences
$$ \xymatrix{ 0 \ar[r] & f_{01!} M'_{01} \ar[d]^{\alpha} \ar[r]  & f_{0!} M'_{0} \oplus f_{1!} M'_1 \ar[d]^{\beta} \ar[r] & f_{!} M' \ar[d]^{\gamma} \ar[r] & 0 \\
0 \ar[r] & f_{01!} M_{01} \ar[r] & f_{0!} M_{0} \oplus f_{1!} M_1 \ar[r] & f_{!} M \ar[r] & 0 }$$
Using the exactness of the functors $f_{0!}$, $f_{1!}$, and $f_{01!}$, the snake lemma yields an exact sequence
$$ 0 \rightarrow \ker(\gamma) \rightarrow f_{01!} M''_{01} \xrightarrow{\rho} f_{0!} M''_{0} \oplus f_{1!} M''_{1}.$$
It will therefore suffice to show that $\rho$ is a monomorphism, which was established in the proof of $(iii)$.

\item[$(v)$] Let $f: A \rightarrow B$ be an {\etale} $\F_p$-algebra homomorphism, and suppose that there exist elements $\{ t_i \in B \}_{1 \leq i \leq n}$ such that
each of the induced maps $A \rightarrow B[ t_i^{-1} ]$ is good. Then $f$ is good. This follows from $(iii)$ and $(iv)$, using induction on $n$.


\item[$(vi)$] Suppose we are given a pushout square of {\etale} maps 
$$ \xymatrix{ A \ar[r]^{f} \ar[d] & B \ar[d]^{g} \\
A' \ar[r]^{f'} & B', }$$
where the vertical maps are faithfully flat. If $f'$ is good, then $f$ is good. This follows from Lemma \ref{lemX23}.
\end{itemize}

We now wish to prove that every {\etale} morphism $f: A \rightarrow B$ is good. 
For each point $x \in \Spec(A)$, let $\kappa(x)$ denote the residue field of $A$ at $x$ and let $d(x)$ denote the dimension
$\dim_{ \kappa(x) }( \kappa(x) \otimes_{A} B)$. Set $d(B) = \sup_{x \in \Spec(A)} d(x)$. We proceed by induction on $d(B)$.
If $d = 0$, then $B \simeq 0$ and there is nothing to prove. To carry out the inductive step, we note that since $f$ is {\etale}, the induced map
$\Spec(B) \rightarrow \Spec(A)$ has open image. The complement of this image can be written as the vanishing locus of an ideal
$I = (a_1, \ldots, a_n) \subseteq A$. Then $I$ generates the unit ideal of $B$. By virtue of $(v)$, to prove that $f$ is good, it will suffice
to show that each of the composite maps $A \rightarrow A[ a_{i}^{-1}] \xrightarrow{f_i} B[ a_{i}^{-1} ]$ is good. Using $(i)$ and $(ii)$,
we are reduced to showing that the maps $f_{i}: A[ a_{i}^{-1} ] \rightarrow B[ a_{i}^{-1}]$ are good. Replacing $f$ by $f_{i}$, we
may reduce to the case where $f$ is faithfully flat. Form a pushout square
$$ \xymatrix{ A \ar[r]^{f} \ar[d]^{f} & B \ar[d] \\
B \ar[r]^-{f'} & B \otimes_{A} B. }$$
By virtue of $(vi)$, we can replace $f$ by $f'$ and thereby reduce to the case where $B$ splits as a direct product $A \times B'$.
We then have $d(B') = d(B) - 1 < d(B)$, so our inductive hypothesis implies that the map $A \rightarrow B'$ is good. From this, we immediately deduce that
$f$ is also good.
\end{proof}

\newpage \section{The Riemann-Hilbert Functor}\label{section.RH}

\setcounter{subsection}{0}
\setcounter{theorem}{0}

Let $R$ be a commutative $\F_p$-algebra. In \S \ref{section.overview}, we defined the solution functor
$$ \Sol: \Mod_{R}^{\Frob} \rightarrow \Shv_{\mathet}( \Spec(R), \F_p )$$
and asserted that it becomes an equivalence of categories when restricted to the category
$\Mod_{R}^{\alg} \subseteq \Mod_{R}^{\Frob}$ of algebraic Frobenius modules (Theorem \ref{theoX50}). 
We will prove this by defining a functor $\RH: \Shv_{\mathet}( \Spec(R), \F_p) \rightarrow \Mod_{R}^{\alg}$,
which we will refer to as the {\it Riemann-Hilbert functor}, and then showing that it is an inverse to the solution functor. 
Our goal in this section is to construct the Riemann-Hilbert functor and to establish its basic properties. Our principal results can be summarized as follows:

\begin{itemize}
\item[$(a)$] When restricted to perfect Frobenius modules, the solution functor $\Sol: \Mod_{R}^{\perf} \rightarrow \Shv_{\mathet}( \Spec(R), \F_p )$ admits a
left adjoint (Theorem \ref{theorem.RHexist}). We will take this left adjoint as a definition of the Riemann-Hilbert functor.

\item[$(b)$] The Riemann-Hilbert functor $\RH: \Shv_{\mathet}( \Spec(R), \F_p) \rightarrow \Mod_{R}^{\perf}$ depends functorially on $R$, in the sense
that it is compatible with pullback (Proposition \ref{prop70}). We also show that it compatible with compactly supported direct images along {\etale} morphisms
(Proposition \ref{prop68}), and direct images along morphisms which are finite and of finite presentation (Theorem \ref{theo76}).

\item[$(c)$] The Riemann-Hilbert functor $\RH: \Shv_{\mathet}( \Spec(R), \F_p) \rightarrow \Mod_{R}^{\perf}$ carries constructible {\etale} sheaves on
$\Spec(R)$ to holonomic Frobenius modules over $R$ (Theorem \ref{theo77}). 

\item[$(d)$] The Riemann-Hilbert functor $\RH: \Shv_{\mathet}( \Spec(R), \F_p ) \rightarrow \Mod_{R}^{\perf}$ is exact
(Proposition \ref{proposition.corX33}).
\end{itemize}

In \S \ref{section.mainproof}, we will apply these results to show that $\RH$ is an inverse of the solution functor (once we restrict our attention to
algebraic Frobenius modules), and thereby obtain a proof of Theorem \ref{maintheoXXX}.

\subsection{Existence of the Riemann-Hilbert Functor}

Our starting point is the following:

\begin{theorem}\label{theorem.RHexist}
Let $R$ be a commutative $\F_p$-algebra. Then the solution functor
$\Sol: \Mod_{R}^{\perf} \rightarrow \Shv_{\mathet}( \Spec(R), \F_p )$ admits a left
adjoint $$\RH: \Shv_{\mathet}( \Spec(R), \F_p) \rightarrow \Mod_{R}^{\perf}.$$
Moreover, for every $p$-torsion {\etale} sheaf $\sheafF \in \Shv_{\mathet}( \Spec(R), \F_p )$, the
Frobenius module $\RH( \sheafF )$ is algebraic.
\end{theorem}

\begin{warning}
In the statement of Theorem \ref{theorem.RHexist}, it is important to restriction the solution functor $\Sol$
to the category of {\em perfect} Frobenius modules. The defining property of the Riemann-Hilbert functor $\RH$
is that we have bijections
$$ \Hom_{ \underline{\F_p} }( \sheafF, \Sol(M) ) \simeq \Hom_{ R[F] }( \RH(\sheafF), M )$$
for $\sheafF \in \Shv_{\mathet}( \Spec(R), \F_p )$ and $M$ a perfect Frobenius module over $R$. One does not generally
have such a bijection when $M$ is not perfect.
\end{warning}

To prove Theorem \ref{theorem.RHexist}, it will be convenient to introduce a temporary bit of terminology.
Let $R$ be a commutative $\F_p$-algebra, and suppose we are given a $p$-torsion {\etale} sheaf $\sheafF$.
A {\it Riemann-Hilbert associate} of $\sheafF$ is an object of $\Mod_{R}^{\perf}$ which
corepresents the functor
$$ \Mod_{R}^{\perf} \rightarrow \Set \quad \quad  M \mapsto \Hom_{ \underline{\F_p} }( \sheafF, \Sol(M) ).$$
If $\sheafF$ is a perfect Frobenius module over $R$ which admits a Riemann-Hilbert associate, we will denote that associate by $\RH(\sheafF)$; note that it
is well-defined up to unique isomorphism and depends functorially on $\sheafF$. Theorem \ref{theorem.RHexist} can then be reformulated as the
statement that every {\etale} sheaf $\sheafF \in \Shv_{\mathet}( \Spec(R), \F_p )$ admits an algebraic Riemann-Hilbert associate.
The proof of this assertion is based on three simple observations:

\begin{proposition}\label{proposition.ex69}
Let $R$ be a commutative $\F_p$-algebra. Then the perfection $R^{\perf}$ is a Riemann-Hilbert associate
of the constant sheaf $\underline{\F_p}$.
\end{proposition}

\begin{proof}
For every perfect Frobenius module $M$ over $R$, we have canonical bijections
\begin{eqnarray*}
\Hom_{R[F]}( R^{\perf}, M) & \simeq & \Hom_{ R[F] }( R, M) \\
& \simeq & \{ x \in M: \varphi_M(x) = x \} \\
& \simeq & \Sol(M)(R) \\
& \simeq & \Hom_{ \underline{\F_p} }( \underline{\F_p}, \Sol(M) ). \end{eqnarray*}
\end{proof}

\begin{proposition}\label{proposition.easy1}
Let $R$ be a commutative $\F_p$-algebra, and suppose we are given some diagram of
{\etale} sheaves $\{ \sheafF_{\alpha} \}$ having a colimit $\sheafF = \varinjlim \sheafF_{\alpha}$ in the 
category $\Shv_{\mathet}( \Spec(R), \F_p )$. Suppose that each $\sheafF_{\alpha}$ admits a Riemann-Hilbert
associate $\RH( \sheafF_{\alpha} )$. Then $\sheafF$ admits a Riemann-Hilbert associate,
given by $\varinjlim \RH( \sheafF_{\alpha} )$ (where the colimit is formed in the category $\Mod_{R}^{\perf}$). 
\end{proposition}

\begin{proof}
For any perfect Frobenius module $M$, we have canonical bijections
\begin{eqnarray*}
\Hom_{ R[F] }( \varinjlim_{\alpha} \RH( \sheafF_{\alpha} ), M) & \simeq &  \varprojlim_{\alpha} \Hom_{ R[F] }( \RH( \sheafF_{\alpha} ), M) \\
& \simeq &  \varprojlim_{\alpha} \Hom_{ \underline{\F_p} }( \sheafF_{\alpha}, \Sol(M) ) \\
& \simeq & \Hom_{ \underline{\F_p} }( \sheafF, \Sol(M) ). 
\end{eqnarray*}
\end{proof}

\begin{proposition}\label{proposition.easy2}
Let $f: A \rightarrow B$ be an {\etale} morphism of commutative $\F_p$-algebras and let
$\sheafF$ be a $p$-torsion {\etale} sheaf on $\Spec(B)$. Suppose that
$\sheafF$ admits a Riemann-Hilbert associate $\RH( \sheafF )$ which is algebraic.
Then the compactly supported direct image $f_{!} \RH(\sheafF)$ is a Riemann-Hilbert associate of
$f_{!} \sheafF \in \Shv_{\mathet}( \Spec(A), \F_p)$.
\end{proposition}

\begin{proof}
Let $M$ be a perfect Frobenius module over $A$. It follows immediately from the definitions that
the solution sheaf $\Sol( f^{\diamond} M) \in \Shv_{\mathet}( \Spec(B), \F_p )$ can be identified
with the pullback $f^{\ast} \Sol(M)$. We therefore obtain canonical bijections
\begin{eqnarray*}
\Hom_{ A[F] }( f_{!} \RH( \sheafF), M ) & \simeq & \Hom_{ B[F] }( \RH( \sheafF), f^{\diamond} M ) \\
& \simeq & \Hom_{ \underline{\F_p} }( \sheafF, \Sol( f^{\diamond} M) ) \\
& \simeq &  \Hom_{ \underline{\F_p} }( \sheafF, f^{\ast} \Sol(M) ) \\
& \simeq &  \Hom_{ \underline{\F_p} }( f_{!} \sheafF, \Sol(M) ).
\end{eqnarray*}
\end{proof}

\begin{proof}[Proof of Theorem \ref{theorem.RHexist}]
Let $R$ be a commutative $\F_p$-algebra and let $\sheafF$ be a $p$-torsion {\etale} sheaf on $\Spec(R)$;
we wish to show that $\sheafF$ admits an algebraic Riemann-Hilbert associate.
For every {\etale} ring homomorphism $j: R \rightarrow R'$ and element $\eta \in \sheafF(R')$, we can identify
$\eta$ with a map of {\etale} sheaves $u_{\eta}: j_{!} \underline{\F_p} \rightarrow \sheafF$.
Amalgamating these, we obtain an epimorphism $u: \sheafF' \rightarrow \sheafF$ in the category 
$\Shv_{\mathet}( \Spec(R), \F_p)$, where $\sheafF'$ is a direct sum of {\etale} sheaves
of the form $j_{!} \underline{\F_p}$ (where $j$ varies over {\etale} morphisms $R \rightarrow R'$).
Repeating this argument for $\ker(u)$, we can construct an exact sequence
$$ \sheafF'' \xrightarrow{v} \sheafF' \xrightarrow{u} \sheafF \rightarrow 0,$$
where $\sheafF''$ is also a direct sum of sheaves of the form $j_{!} \underline{\F_p}$. 
By virtue of Proposition \ref{proposition.easy1}, it will suffice to show that each of the sheaves
$j_{!} \underline{\F_p}$ admits an algebraic Riemann-Hilbert associate. Using Proposition \ref{proposition.easy2}, we are reduced
to showing that if $R'$ is an {\etale} $R$-algebra, then the constant sheaf $\underline{\F_p} \in \Shv_{\mathet}( \Spec(R'), \F_p )$
admits an algebraic Riemann-Hilbert associate. This follows from Proposition \ref{proposition.ex69}.
\end{proof}

\subsection{Functoriality}

We now study the behavior of the Riemann-Hilbert functor
$$ \RH: \Shv_{\mathet}( \Spec(R), \F_p) \rightarrow \Mod_{R}^{\perf}$$
as the commutative $\F_p$-algebra $R$ varies. We begin with a simple observation:

\begin{proposition}\label{prop64}
Let $f: A \rightarrow B$ be a homomorphism of commutative $\F_p$-algebras. Then the diagram of categories 
$$ \xymatrix{ \Mod_{B}^{\perf} \ar[r]^{ f_{\ast} } \ar[d]^{ \Sol } & \Mod_{A}^{\perf} \ar[d]^{ \Sol} \\
\Shv_{\mathet}( \Spec(B), \F_p) \ar[r]^{ f_{\ast} }  & \Shv_{\mathet}( \Spec(A), \F_p) }$$
commutes up to canonical isomorphism.
\end{proposition}

\begin{proof}
Let $M$ be a perfect Frobenius module over $B$, let $A'$ be an {\etale} $A$-algebra, and set
$B' = A' \otimes_{A} B$. We then have canonical
bijections
\begin{eqnarray*}
(f_{\ast} \Sol(M))(A') & \simeq & \Sol(M)(B') \\
& \simeq & \{ x \in B' \otimes_{B} M: \varphi_{B' \otimes_{B} M}(x) = x \} \\
& \simeq & \{ x \in A' \otimes_{A} M: \varphi_{A' \otimes_A M}(x) = x \} \\
& \simeq & \Sol( f_{\ast} M)(A').
\end{eqnarray*}
\end{proof}

\begin{proposition}\label{prop70}
Let $f: A \rightarrow B$ be an $\F_p$-algebra homomorphism. Then the diagram of categories
$$ \xymatrix{ \Shv_{\mathet}( \Spec(A), \F_p) \ar[r]^{ f^{\ast} } \ar[d]^{\RH} & \Shv_{\mathet}( \Spec(B), \F_p) \ar[d]^{\RH} \\
\Mod_{A}^{\perf} \ar[r]^{ f^{\diamond} } & \Mod_{B}^{\perf} }$$
commutes up to canonical isomorphism.
\end{proposition}

\begin{proof}
This follows immediately from Proposition \ref{prop64} by passing to left adjoints.
\end{proof}

In the situation of Proposition \ref{prop70}, the vertical maps carry {\etale} sheaves to algebraic Frobenius modules,
so we also have a commutative diagram
$$ \xymatrix{ \Shv_{\mathet}( \Spec(A), \F_p) \ar[r]^{ f^{\ast} } \ar[d]^{\RH} & \Shv_{\mathet}( \Spec(B), \F_p) \ar[d]^{\RH} \\
\Mod_{A}^{\alg} \ar[r]^{ f^{\diamond} } & \Mod_{B}^{\alg}. }$$
In the case where $f: A \rightarrow B$ is {\etale}, the horizontal maps in this diagram admit left adjoints (Theorem \ref{theoX15}).
We therefore obtain a natural transformation $f_{!} \circ \RH \rightarrow \RH \circ f_{!}$
in the category of functors from $\Shv_{\mathet}( \Spec(B), \F_p)$ to $\Mod_{A}^{\alg}$. 

\begin{proposition}\label{prop68}
Let $f: A \rightarrow B$ be an {\etale} morphism of $\F_p$-algebras.
Then the Beck-Chevalley transformation $f_{!} \circ \RH \rightarrow \RH \circ f_{!}$ described above
is an isomorphism. In other words, the diagram of categories
$$ \xymatrix{ \Shv_{\mathet}( \Spec(A), \F_p) \ar[d]^{\RH} & \Shv_{\mathet}( \Spec(B), \F_p) \ar[d]^{\RH} \ar[l]_{f_{!}} \\
\Mod_{A}^{\alg}  & \Mod_{B}^{\alg} \ar[l]_{f_{!}} }$$
commutes up to canonical isomorphism.
\end{proposition}

\begin{proof}
This is a translation of Proposition \ref{proposition.easy2} (or, more precisely, of its proof).
\end{proof}

\begin{remark}\label{old64}
We can also formulate Proposition \ref{prop68} in terms of solution sheaves: it follows from
the commutativity of the diagram
$$ \xymatrix{ \Mod_{B}^{\perf} \ar[r]^{ f_{\ast} } \ar[d]^{ \Sol } & \Mod_{A}^{\perf} \ar[d]^{ \Sol} \\
\Shv_{\mathet}( \Spec(B), \F_p) \ar[r]^{ f_{\ast} }  & \Shv_{\mathet}( \Spec(A), \F_p) }$$
when $f: A \rightarrow B$ is an {\etale} morphism of $\F_p$-algebras, which follows immediately
from the definitions (and was invoked in the proof of Proposition \ref{proposition.easy2}).
\end{remark}

\subsection{Constructible Sheaves}

Let $R$ be a commutative ring. Recall (see \cite[Tag 05BE]{Stacks}) that a sheaf $\sheafF \in \Shv_{\mathet}( \Spec(R), \F_p)$ is said to be {\it constructible} if there is finite stratification 
$$ \emptyset = X_0 \subseteq X_1 \subseteq \cdots \subseteq X_n = \Spec(R),$$
where each open stratum $X_m - X_{m-1}$ is a constructible subset of $\Spec(R)$ and admits an {\etale} surjection $U_m \rightarrow (X_m - X_{m-1} )$
such that the restriction $\sheafF|_{U_m}$ is isomorphic to a constant sheaf $\underline{V}$, for some finite-dimensional vector space $V$ over $\F_p$.
We let $\Shv_{\mathet}^{c}( \Spec(R), \F_p)$ denote the full subcategory of $\Shv_{\mathet}( \Spec(R), \F_p)$
spanned by the constructible sheaves. Note that $\Shv_{\mathet}^{c}( \Spec(R), \F_p)$ is closed under the formation of kernels and cokernels
in $\Shv_{\mathet}( \Spec(R), \F_p)$; in particular, it is an abelian category.

Theorem \ref{theorem.RHexist} admits the following refinement:

\begin{theorem}\label{theo77}
Let $R$ be a commutative $\F_p$-algebra. Then the Riemann-Hilbert functor
$\RH: \Shv_{\mathet}( \Spec(R), \F_p) \rightarrow \Mod_{R}^{\perf}$ carries constructible {\etale} sheaves to holonomic
Frobenius modules over $R$.
\end{theorem}

\begin{notation}\label{notation.RHc}
For every $\F_p$-algebra $R$, we let $\RH^{c}: \Shv_{\mathet}^{c}( \Spec(R), \F_p) \rightarrow \Mod_{R}^{\hol}$ denote the restriction of the
Riemann-Hilbert functor $\RH$ to constructible sheaves.
\end{notation}

For the proof of Theorem \ref{theo77}, we will need a few standard facts about constructible sheaves, which we assert here without proof:

\begin{proposition}\label{propX41}
Let $f: A \rightarrow B$ be an {\etale} ring homomorphism. Then the functor $f_{!}: \Shv_{\mathet}( \Spec(B), \F_p) \rightarrow \Shv_{\mathet}( \Spec(A), \F_p)$ carries constructible sheaves to constructible sheaves.
\end{proposition}

\begin{proof} 
This is contained in \cite[Tag 03S8]{Stacks}.
\end{proof}

\begin{proposition}\label{propX42}
Let $A$ be a commutative ring and let $\sheafF \in \Shv_{\mathet}^{c}( \Spec(A), \F_p)$. Then there exists an {\etale} morphism $f: A \rightarrow B$
and an epimorphism $f_{!} \underline{ \F_p } \rightarrow \sheafF$ in the abelian category $\Shv_{\mathet}( \Spec(A), \F_p)$.
\end{proposition}
\begin{proof}
This follows from \cite[Tag 09YT]{Stacks}.
\end{proof}

\begin{proof}[Proof of Theorem \ref{theo77}]
Let $R$ be a commutative $\F_p$-algebra and let $\sheafF$ be a constructible $p$-torsion {\etale} sheaf on $\Spec(R)$. We wish to show
that the Frobenius module $\RH(\sheafF)$ is holonomic. We first apply Proposition \ref{propX42} to choose
an epimorphism $u: \sheafF' \rightarrow \sheafF$, where $\sheafF'$ has the form $f_{!} \underline{\F_p}$ for some
{\etale} morphism $f: A \rightarrow B$. Then $\sheafF'$ is constructible, so $\ker(u)$ is also constructible.
Applying Proposition \ref{propX42} again, we can choose an epimorphism $v: \sheafF'' \rightarrow \ker(u)$, where
$\sheafF''$ has the form $g_{!} \underline{\F_p}$ for some {\etale} morphism $g: A \rightarrow C$. We then
have an exact sequence
$$ \sheafF'' \xrightarrow{v} \sheafF' \xrightarrow{ u} \sheafF \rightarrow 0.$$
The Riemann-Hilbert functor $\RH$ is right exact (since it is a left adjoint), so we obtain an exact sequence of
Frobenius modules 
$$ \RH(\sheafF'') \rightarrow \RH( \sheafF') \rightarrow \RH( \sheafF) \rightarrow 0.$$
Since the collection of holonomic Frobenius modules over $R$ is closed under the formation of cokernels (Corollary \ref{corollary.propX60}),
it will suffice to show that $\RH( \sheafF'')$ and $\RH(\sheafF')$ are holonomic. 
Using Propositions \ref{prop68} and \ref{proposition.ex69},
we obtain isomorphisms
$\RH( \sheafF' ) \simeq f_{!} B^{\perf}$ and $\RH( \sheafF'' ) \simeq g_{!} C^{\perf}$.
The desired holonomicity now follows from Remark \ref{remX22}.
\end{proof}

We close this section by recording (without proof) a few more elementary facts about constructible sheaves which will be needed in the proof of Theorem \ref{maintheoXXX}.
First, we have the following duals to Propositions \ref{propX41} and \ref{propX42}:

\begin{proposition}\label{propX43}
Let $f: A \rightarrow B$ be a ring homomorphism which is finite and of finite presentation. Then the direct image functor $f_{\ast}: \Shv_{\mathet}( \Spec(B), \F_p) \rightarrow \Shv_{\mathet}( \Spec(A), \F_p)$
carries constructible sheaves to constructible sheaves.
\end{proposition}
\begin{proof}
See \cite[\S 1, Lemma 4.11]{FK}.
\end{proof}

\begin{proposition}\label{propX44}
Let $A$ be a commutative ring and let $\sheafF \in \Shv_{\mathet}^{c}( \Spec(A), \F_p)$. Then there exists a ring homomorphism $f: A \rightarrow B$ which is finite and of finite presentation
and a monomorphism $\sheafF \rightarrow f_{\ast} \underline{ \F_p }$ in the category $\Shv_{\mathet}( \Spec(A), \F_p)$.
\end{proposition}
\begin{proof}
See \cite[\S 1, Proposition 4.12]{FK}.
\end{proof}

\begin{proposition}\label{propX40}
Let $R$ be a commutative ring. Then the inclusion functor $\Shv_{\mathet}^{c}( \Spec(R), \F_p) \hookrightarrow \Shv_{\mathet}( \Spec(R), \F_p)$ extends to an equivalence of
categories $\Ind( \Shv_{\mathet}^{c}( \Spec(R), \F_p) ) \simeq \Shv_{\mathet}( \Spec(R), \F_p)$.
\end{proposition}

\begin{proof}
See \cite[Tag 03SA]{Stacks}.
\end{proof}

\subsection{Exactness of the Riemann-Hilbert Functor}

\begin{proposition}\label{proposition.corX33}
Let $R$ be a commutative $\F_p$-algebra. Then the Riemann-Hilbert functor $\RH: \Shv_{\mathet}( \Spec(R), \F_p) \rightarrow \Mod_R^{\perf}$ is exact.
\end{proposition}

\begin{proof}
Since the Riemann-Hilbert functor $\RH$ is defined as the left adjoint to the solution functor, it is automatically right exact.
It will therefore suffice to show that if $u: \sheafF \rightarrow \sheafG$ is a monomorphism of $p$-torsion {\etale} sheaves on $\Spec(R)$, then
the induced map $\RH(u): \RH(\sheafF) \rightarrow \RH(\sheafG)$ is also a monomorphism. Using Proposition \ref{propX40}, we can reduce
to the case where $\sheafF$ and $\sheafG$ are constructible, so that the Frobenius modules
$\RH( \sheafF)$ and $\RH(\sheafG)$ are holonomic (Theorem \ref{theo77}). It follows that the kernel of the map $\RH(u)$ is
also holonomic (Corollary \ref{corX30}). By virtue of Proposition \ref{theo78}, to show that $\RH(u)$ is a monomorphism,
it will suffice to show that the support $\supp( \ker(\RH(u) ) )$ is empty. Fix a point $x \in \Spec(R)$; we will show that
$x \notin \supp( \ker(\RH(u) ) ) \subseteq \Spec(R)$. Let $\kappa$ be an algebraic closure of the residue field of $R$ at the point $x$ and
let $f: R \rightarrow \kappa$ be the canonical map; we wish to show that $f^{\diamond}( \ker( \RH(u) ) )$ vanishes. 
Since the functor $f^{\diamond}$ is exact on algebraic Frobenius modules (Corollary \ref{lushin}) and compatible with the Riemann-Hilbert functor (Proposition
\ref{prop70}), we have
$$ f^{\diamond}( \ker( \RH(u) ) ) \simeq \ker( f^{\diamond}( \RH(u) ) ) \simeq \ker( \RH( f^{\ast}(u) ) ).$$
We can therefore replace $R$ by $\kappa$ and thereby reduce to the case where $R$ is an algebraically closed field. In this case,
the category $\Shv_{\mathet}( \Spec(R), \F_p)$ is equivalent to the category of vector spaces over $\F_p$. It follows
that every exact sequence in the category $\Shv_{\mathet}( \Spec(R), \F_p)$ is split, so the exactness of the Riemann-Hilbert functor
$\RH$ is automatic.
\end{proof}

\begin{corollary}\label{corX75}
Let $R$ be a commutative $\F_p$-algebra. Then the solution functor $\Sol: \Mod_{R}^{\perf} \rightarrow \Shv_{\mathet}( \Spec(R), \F_p)$
carries injective objects of the abelian category $\Mod_{R}^{\perf}$ to injective objects of the abelian category $\Shv_{\mathet}( \Spec(R), \F_p)$.
\end{corollary}

\subsection{Comparison of Finite Direct Images}\label{finimage}

Let $f: A \rightarrow B$ be a homomorphism of commutative $\F_p$-algebras, so that Proposition \ref{prop70} supplies
a commutative diagram of categories
$$ \xymatrix{ \Shv_{\mathet}( \Spec(A), \F_p) \ar[r]^{ f^{\ast} } \ar[d]^{\RH} & \Shv_{\mathet}( \Spec(B), \F_p) \ar[d]^{\RH} \\
\Mod_{A}^{\perf} \ar[r]^{ f^{\diamond} } & \Mod_{B}^{\perf}. }$$
Note that the horizontal maps in this diagram admit right adjoints
$$ f_{\ast}: \Mod_{B}^{\perf} \rightarrow \Mod_{A}^{\perf} \quad \quad f_{\ast}:  \Shv_{\mathet}( \Spec(B), \F_p)
\rightarrow \Shv_{\mathet}( \Spec(A), \F_p).$$
By general nonsense, we obtain a Beck-Chevalley transformation $\RH \circ f_{\ast} \rightarrow f_{\ast} \circ \RH$
in the category of functors from $\Shv_{\mathet}( \Spec(B), \F_p)$ to $\Mod_{A}^{\perf}$. In general, this map need not be an isomorphism:
for example, if $B^{\perf} = (f_{\ast} \circ \RH)( \underline{\F_p} )$ is not algebraic when regarded as a Frobenius module over $A$, then it cannot belong to the essential
image of the Riemann-Hilbert functor $\RH: \Shv_{\mathet}( \Spec(A), \F_p) \rightarrow \Mod_{A}^{\perf}$. However, under some mild finiteness hypotheses, this phenomenon does not arise:

\begin{theorem}\label{theo76}
Let $f: A \rightarrow B$ be a morphism of commutative $\F_p$-algebras which is finite and of finite presentation. Then, for every
$p$-torsion {\etale} sheaf $\sheafF$ on $\Spec(B)$, the canonical map $\epsilon_{\sheafF}: \RH( f_{\ast} \sheafF) \rightarrow f_{\ast}( \RH(\sheafF) )$ is
an isomorphism of Frobenius modules over $A$. Consequently, the diagram of categories
$$ \xymatrix{ \Shv_{\mathet}( \Spec(A), \F_p) \ar[d]^{\RH} & \Shv_{\mathet}( \Spec(B), \F_p) \ar[d]^{\RH} \ar[l]_{ f_{\ast}} \\
\Mod_{A}^{\perf} & \Mod_{B}^{\perf} \ar[l]_{ f_{\ast} } }$$
commutes (up to canonical isomorphism). 
\end{theorem}

\begin{remark}
In \S \ref{section.global}, we will prove a more general version of Theorem \ref{theo76}, which applies in the situation of a morphism
$f: X \rightarrow Y$ of $\F_p$-schemes which is proper and of finite presentation (Corollary \ref{properindh2}).
\end{remark}

\begin{proof}[Proof of Theorem \ref{theo76}]
The functors 
$$ \RH: \Shv_{\mathet}( \Spec(A), \F_p) \rightarrow \Mod_{A}^{\perf} \quad \quad \RH: \Shv_{\mathet}( \Spec(B), \F_p) \rightarrow \Mod_{B}^{\perf}$$
$$ f_{\ast}: \Shv_{\mathet}( \Spec(B), \F_p) \rightarrow \Shv_{\mathet}( \Spec(A), \F_p) \quad \quad f_{\ast}: \Mod_{B}^{\perf} \rightarrow \Mod_{A}^{\perf}$$
all commute with filtered colimits. Consequently, to show that the map $\epsilon_{\sheafF}$ is an equivalence for
every object $\sheafF \in \Shv_{\mathet}( \Spec(B), \F_p)$, it will suffice to show that $\epsilon_{\sheafF}$ is an equivalence when the sheaf $\sheafF$ is
constructible (Proposition \ref{propX40}). In this case, the direct image $f_{\ast} \sheafF \in \Shv_{\mathet}( \Spec(A), \F_p)$ is constructible
(Proposition \ref{propX43}). Applying Theorem \ref{theo77}, we deduce that $\RH(\sheafF)$ and $\RH( f_{\ast} \sheafF)$ are holonomic.
It follows from Proposition \ref{propX45} that
$f_{\ast} \RH( \sheafF )$ is also holonomic. Applying Corollary \ref{corX30}, we deduce that $\ker( \epsilon_{\sheafF} )$ and
$\coker( \epsilon_{\sheafF} )$ are holonomic. By virtue of Propositions \ref{theo78} and \ref{propX53}, to show that
$\epsilon_{\sheafF}$ is an isomorphism, it will suffice to show that $g^{\diamond}( \epsilon_{\sheafF} )$ is an isomorphism
for every map $g: A \rightarrow \kappa$ where $\kappa$ is an algebraically closed field. Using Proposition \ref{prop70}
(and the fact that pushforward of {\etale} sheaves along finite morphisms commutes with base change), we can replace
$A$ by $\kappa$ and thereby reduce to the case where $A$ is an algebraically closed field. In this case,
$B$ is a finite-dimensional algebra over $\kappa$. Writing $B$ as a product of local rings, we can assume that
$B$ is local with residue field $\kappa$. Then the constructible sheaf $\sheafF \in \Shv_{\mathet}^{c}( \Spec(B), \F_p)$ has the form
$\underline{V}$ for some finite-dimensional vector space $V$ over $\F_p$. Choosing a basis for $V$,
we can reduce to the case where $V = \F_p$. Using Proposition \ref{proposition.ex69}, we see that $\epsilon_{\sheafF}$
can be identified with the canonical map $A^{\perfection} \rightarrow B^{\perfection}$, which is an isomorphism since the
radical of $B$ is nilpotent.
\end{proof}

\newpage \section{The Riemann-Hilbert Correspondence}\label{section.mainproof}
\setcounter{subsection}{0}
\setcounter{theorem}{0}

Let $R$ be a commutative $\F_p$-algebra. Our goal in this section is to prove Theorem \ref{maintheoXXX} by showing that the Riemann-Hilbert functor
$$ \RH: \Shv_{\mathet}( \Spec(R), \F_p ) \rightarrow \Mod_{R}^{\alg}$$
is an equivalence of categories. Let us outline the strategy we will use. Our first objective (which is achieved in \S \ref{sec7sub2}) is to show
that the Riemann-Hilbert functor is fully faithful: that is, that the unit map $\sheafF \rightarrow \Sol( \RH(\sheafF ) )$ is an isomorphism for
any $p$-torsion {\etale} sheaf $\sheafF$ on $\Spec(R)$ (see Proposition \ref{prop75}). One obstacle to proving this is that the
solution functor $\Sol: \Mod_{R}^{\perf} \rightarrow \Shv_{\mathet}( \Spec(R), \F_p )$ is not exact. However, we show in
\S \ref{sec7sub1} that it is {\em almost} exact: more precisely, it has only one derived functor, which can be explicitly described (Proposition \ref{proposition.corX77}).

The rest of this section is devoted to showing that every algebraic Frobenius module $M$ over $R$ belongs to the essential image
of the Riemann-Hilbert functor. To prove this, we may assume without loss of generality that $M$ is holonomic. In this case,
we prove something stronger: the Frobenius module $M$ can be realized as $\RH( \sheafF )$, where $\sheafF$ is a {\em constructible}
$p$-torsion {\etale} sheaf on $\Spec(R)$ (Theorem \ref{companion}). In the case where $R$ is a field, this assertion is classical; we give
a proof in \S \ref{sec7sub3} for the reader's convenience (Proposition \ref{proposition.lemX51}). The general case is treated in \S \ref{sec7sub4}, using
a d\'{e}vissage which reduces to the case where $R$ is a field.

\subsection{Derived Solution Functors}\label{sec7sub1}

Let $R$ be a commutative $\F_p$-algebra. The solution functor
$$ \Sol: \Mod_{R}^{\perf} \rightarrow \Shv_{\mathet}( \Spec(R), \F_p)$$
of Construction \ref{construction.solsheaf} is left exact (Remark \ref{remark.leftexact}), but is usually not exact. Since the category
$\Sol_{R}^{\perf}$ has enough injective objects (Remark \ref{remX4}), we can consider its right derived functors.
For each $n \geq 0$, we let $\Sol^{n}: \Mod_{R}^{\perf} \rightarrow \Shv_{\mathet}( \Spec(R), \F_p)$ denote the $n$th right derived functor of $\Sol$.
These derived functors admit a simple explicit description:

\begin{proposition}\label{proposition.corX77}
Let $R$ be a commutative $\F_p$-algebra and let $M$ be a perfect Frobenius module over $R$. Then we have a canonical short exact sequence
$$ 0 \rightarrow \Sol(M) \rightarrow \widetilde{M} \xrightarrow{ \id - \widetilde{\varphi}_{M}  } \widetilde{M} \rightarrow \Sol^{1}(M) \rightarrow 0$$
and the sheaves $\Sol^{n}(M)$ vanish for $n \geq 2$.
\end{proposition}

Here $\widetilde{M}$ denotes the quasi-coherent sheaf associated to the $R$-module $M$ (see Example \ref{exX70}).
We will deduce Proposition \ref{proposition.corX77} from the following:

\begin{lemma}\label{lemma.propX74}
Let $R$ be a commutative $\F_p$-algebra. If $M$ is an injective object of $\Mod_{R}^{\perf}$, then the sequence
$0 \rightarrow \Sol(M) \rightarrow \widetilde{M} \xrightarrow{\id - \widetilde{\varphi}_{M} } \widetilde{M} \rightarrow 0$
is exact in the category of abelian presheaves on the category $\CAlg_{R}^{\mathet}$.
\end{lemma}

\begin{proof}
Choose an {\etale} morphism $f: R \rightarrow A$. We then have a diagram of exact sequences
$$ \xymatrix{ 0 \ar[d] & 0 \ar[d] \\
\Sol(M)(A) \ar[r]^-{\sim} \ar[d] &  \Ext^{0}_{A[F]}( A, A \otimes_{R} M) \ar[d] \\
\widetilde{M} \ar[r]^-{\sim} \ar[d]^{\id - \varphi} & \Ext^{0}_{ A[F] }( A[F], A \otimes_{R} M) \ar[d] \\
\widetilde{M} \ar[r]^-{\sim} &  \Ext^{0}_{ A[F]}( A[F], A \otimes_{R} M) \ar[d] \\
& \Ext^{1}_{A[F]}( A, A \otimes_{R} M). }$$
To complete the proof, it will suffice to show that the group $\Ext^{1}_{A[F]}( A, A \otimes_{R} M) \simeq \Ext^{1}_{A[F]}( A^{\perfection}, A \otimes_{R} M)$ vanishes.
Using Theorem \ref{theoX15}, we obtain an isomorphism $$\Ext^{1}_{A[F]}( A^{\perfection}, A \otimes_{R} M) \simeq \Ext^{1}_{R[F]}( f_{!} A^{\perfection}, M),$$ where
the right hand side vanishes by virtue of our assumption that $M$ is injective.
\end{proof}

\begin{proof}[Proof of Proposition \ref{proposition.corX77}]
Let $M$ be a perfect Frobenius module over a commutative $\F_p$-algebra $R$, and choose an
injective resolution $0 \rightarrow M \rightarrow Q^{0} \rightarrow Q^{1} \rightarrow \cdots$ in the abelian category
$\Mod_{R}^{\perf}$. Using Lemma \ref{lemma.propX74}, we obtain a short exact sequence of cochain complexes
$$0 \rightarrow \Sol( Q^{\ast} ) \rightarrow \widetilde{Q}^{\ast} \xrightarrow{ \id - \varphi } \widetilde{Q}^{\ast} \rightarrow 0.$$
Since the construction $N \mapsto \widetilde{N}$ is exact, the chain complex $\widetilde{Q}^{\ast}$ is an acyclic resolution of $\widetilde{M}$.
The associated long exact sequence now supplies the desired isomorphisms.
\end{proof}

\begin{corollary}\label{corX78}
Let $R$ be a commutative $\F_p$-algebra. Then the functor $\Sol^{n}: \Mod_{R}^{\perf} \rightarrow \Shv_{\mathet}( \Spec(R), \F_p)$ commutes with filtered colimits for each
$n \geq 0$.
\end{corollary}

\begin{corollary}\label{corX76}
Let $R$ be a commutative $\F_p$-algebra, let $\sheafF$ be a $p$-torsion {\etale} sheaf on $\Spec(R)$, and let
let $M$ be a perfect Frobenius module over $R$. 
Then we have canonical short exact sequences
$$0 \rightarrow \Ext^{n}_{ \underline{\F_p}}( \sheafF, \Sol(M) ) \rightarrow  \Ext^{n}_{R[F]}( \RH( \sheafF ), M ) \rightarrow
\Ext^{n-1}_{ \underline{\F_p}} ( \sheafF, \Sol^{1}(M) ) \rightarrow 0.$$
\end{corollary}

\begin{proof}
Since the solution functor $\Sol: \Mod_{R}^{\perf} \rightarrow \Shv_{\mathet}( \Spec(R), \F_p)$ carries
injective objects of $\Mod_{R}^{\perf}$ to injective objects of $\Shv_{\mathet}( \Spec(R), \F_p)$ (Corollary \ref{corX75}),
we have a Grothendieck spectral sequence 
$$ \Ext^{s}_{ \underline{\F_p} }( \sheafF, \Sol^{t}(M) ) \Rightarrow \Ext^{s+t}_{ R[F] }( \RH(\sheafF), M).$$
The existence of the desired short exact sequences now follows from the vanishing of the groups
$\Sol^{t}(M)$ for $t \geq 2$ (Proposition \ref{proposition.corX77}).
\end{proof}

\subsection{Full Faithfulness of the Riemann-Hilbert Functor}\label{sec7sub2}

We are now ready to prove a weak version of Theorem \ref{maintheoXXX}.

\begin{proposition}\label{prop75}
Let $R$ be a commutative $\F_p$-algebra and let $\sheafF$ be a $p$-torsion {\etale} sheaf on $\Spec(R)$.
Then the unit map $u_{\sheafF}: \sheafF \rightarrow \Sol( \RH(\sheafF))$ is an isomorphism and the sheaf $\Sol^{1}( \RH(\sheafF) )$ vanishes.
\end{proposition}

We first treat a special case of Proposition \ref{prop75}:

\begin{lemma}\label{tuster}
Let $R$ be a commutative $\F_p$-algebra. Then the unit map $u: \underline{\F_p} \rightarrow \Sol( \RH( \underline{\F_p} ) )$ is an isomorphism and
the sheaf $\Sol^{1}( \RH( \underline{\F_p} ) )$ vanishes.
\end{lemma}

\begin{proof}
Without loss of generality, we may assume that $R$ is perfect. Using Propositions \ref{proposition.ex69} and \ref{proposition.corX77}, we see that Lemma \ref{tuster} is equivalent to the exactness of the Artin-Schreier sequence
\[ 0 \rightarrow \underline{\F_p} \rightarrow \widetilde{R} \xrightarrow{ \id - \varphi } \widetilde{R} \rightarrow 0\]
in the category $\Shv_{\mathet}( \Spec(R), \F_p)$.
\end{proof}

\begin{proof}[Proof of Proposition \ref{prop75}]
Using Corollary \ref{corX78} and Proposition \ref{propX40}, we can reduce to the case where the sheaf $\sheafF$ is constructible.
Using Proposition \ref{propX44}, we can choose an exact sequence of constructible sheaves 
$$0 \rightarrow \sheafF \rightarrow \sheafG \rightarrow \sheafH \rightarrow 0,$$
where $\sheafG = f_{\ast} \underline{\F_p}$ for some $\F_p$-algebra homomorphism $f: R \rightarrow A$ which is
finite and of finite presentation. Using Proposition \ref{proposition.corX33} and Proposition \ref{proposition.corX77}, we obtain a commutative diagram 
$$ \xymatrix{ 0 \ar[d] & 0 \ar[d] \\
\sheafF \ar[r]^-{ u_{\sheafF } } \ar[d] & \Sol( \RH(\sheafF) ) \ar[d] \\
\sheafG \ar[r]^-{ u_{\sheafG} } \ar[d] & \Sol( \RH(\sheafG) ) \ar[d] \\
\sheafH \ar[r]^-{ u_{\sheafH} } \ar[d] & \Sol( \RH(\sheafH) ) \ar[d]^{\delta} \\
0 \ar[r] & \Sol^{1}( \RH(\sheafF) ) \ar[d] \\
& \Sol^{1}( \RH(\sheafG) ) }$$
whose columns are exact. It follows from Lemma \ref{tuster}, Theorem \ref{theo76}, and Proposition \ref{prop64} that $u_{\sheafG}$ is an isomorphism and
$\Sol^{1}( \RH(\sheafG) )$ vanishes. Inspecting the diagram, we deduce that $u_{\sheafF}$ is a monomorphism. 
Applying the same argument to $\sheafH$, we see that $u_{\sheafH}$ is also monomorphism, so a diagram chase shows that
$u_{\sheafF}$ is an epimorphism. Applying the same argument to $\sheafH$, we conclude that $u_{\sheafH}$ is also an epimorphism.
The commutativity of the diagram shows that $\delta \circ u_{\sheafH}$ vanishes, so that $\delta = 0$. Since $\delta$ is an epimorphism,
we conclude that $\Sol^{1}( \RH(\sheafF) ) \simeq 0$.
\end{proof}

It follows from Proposition \ref{prop75} that the Riemann-Hilbert functor is fully faithful, even at the ``derived'' level:

\begin{corollary}\label{corX80}
Let $R$ be a commutative $\F_p$-algebra. Then the Riemann-Hilbert functor $\RH: \Shv_{\mathet}( \Spec(R), \F_p) \rightarrow \Mod_{R}^{\perf}$ is fully faithful.
Moreover, for every pair of objects $\sheafF, \sheafG \in \Shv_{\mathet}( \Spec(R), \F_p)$, the induced map
$\Ext^{\ast}_{\underline{\F_p}}( \sheafF, \sheafG ) \rightarrow \Ext^{\ast}_{R[F]}( \RH(\sheafF), \RH(\sheafG) )$ is an isomorphism.
\end{corollary}

\begin{proof}
Combine Proposition \ref{prop75} with Corollary \ref{corX76}.
\end{proof}

\subsection{The Case of a Field}\label{sec7sub3}

It follows from Corollary \ref{corX80} that, for any commutative $\F_p$-algebra $R$, the functor
$$ \RH^{c}: \Shv_{\mathet}^{c}( \Spec(R), \F_p ) \rightarrow \Mod_{R}^{\hol}$$
of Notation \ref{notation.RHc} is fully faithful. We now show that it is an equivalence in the special case where
$R$ is a field, which is essentially a restatement of Theorem \ref{theorem-folk}: 

\begin{proposition}\cite[Proposition 4.1.1]{KatzPadic}\label{proposition.lemX51}
Let $\kappa$ be a field of characteristic $p$. Then the functor
$\RH^{c}: \Shv_{\mathet}^{c}( \Spec(\kappa), \F_p) \rightarrow \Mod_{\kappa}^{\hol}$ is an equivalence of categories.
\end{proposition}

We begin by treating the case where $\kappa$ is algebraically closed (compare \cite[\S III, Lemma 3.3]{CL}):

\begin{lemma}\label{lem81}
Let $\kappa$ be an algebraically closed field of characteristic $p$. Then the functor
$\RH^{c}: \Shv_{\mathet}^{c}( \Spec(\kappa), \F_p) \rightarrow \Mod_{\kappa}^{\hol}$ is an equivalence of categories.
\end{lemma}

\begin{proof}
Using Corollary \ref{corX80} and Proposition \ref{proposition.corX33}, we see that $\RH^{c}$ is a fully faithful embedding whose essential image $\calC \subseteq \Mod_{\kappa}^{\hol}$ is an abelian subcategory which is closed under extensions. We wish to show that $\calC$ contains every object $M \in \Mod_{\kappa}^{\hol}$.
Applying Proposition \ref{propX29}, we see that $M$ is a Noetherian object of the abelian category $\Mod_{\kappa}^{\alg}$. Consequently, there exists
a subobject $M' \subseteq M$ (in the abelian category $\Mod_{\kappa}^{\alg}$) which is maximal among those subobjects which belong to $\calC$. 
It follows from the maximality of $M'$ (and the stability of $\calC$ under extensions) that the quotient $M / M'$ does not contain any nonzero subobjects which belong to
$\calC$. Replacing $M$ by $M / M'$, we can reduce to the case where $M$ does not have any nonzero subobjects which belong to $\calC$.

Suppose that $M$ is nonzero. Choose a nonzero element $x \in M$. Since $M$ is algebraic, the element $x$ satisfies an equation
$$ \varphi_{M}^{n}(x) + \lambda_1 \varphi_M^{n-1}(x) + \cdots + \lambda_n x = 0$$
for some coefficients $\lambda_1, \lambda_2, \ldots, \lambda_n \in \kappa$.
We may assume that $x$ has been chosen so that $n$ is as small as possible; this guarantees that the set
$\{ x, \varphi_{M}(x), \ldots, \varphi_{M}^{n-1}(x) \}$ is linearly independent over $\kappa$, and therefore $\lambda_n \neq 0$.
Since $x \neq 0$, we must have $n > 0$.

Note that $$ f(t) = t^{p^{n}} + \lambda_1^{p^{n-1}} t^{p^{n-1}} + \lambda_2^{p^{n-2}} t^{p^{n-2}}
+ \cdots + \lambda_n x$$
is a separable polynomial of degree $p^{n} > 1$, and therefore has
$p^{n}$ distinct roots in the field $\kappa$. Consequently, there exists a nonzero
element $a \in \kappa$ such that $f(a) = 0$.
Let $$y = a x + (a^p + a \lambda_1) \varphi_{M}(x) +
\cdots + ( a^{p^{n-1}} + a^{p^{n-2}} \lambda_1^{p^{n-2}} + \cdots
+ a \lambda_{n-1} ) \varphi_{M}^{n-1}(x).$$
Since the elements $\{ \varphi_{M}^{i}(x) \}_{0 \leq i < n}$ are linearly independent
and $a \neq 0$, $y$ is a nonzero element of $M$. An explicit calculation gives
\begin{eqnarray*}
y-\varphi_{M}(y) & = &  ax + \sum_{0 < i < n} a \lambda_{i} \varphi_{M}^{i}(x) + (a \lambda_n - f(a) ) \varphi_{M}^{n}(x) \\
& = & a( x + \lambda_1 \varphi_M(x) + \cdots + \lambda_n \varphi_{M}^{n}(x) ) \\
& = & 0.\end{eqnarray*}
It follows that $y$ generates a nonzero Frobenius submodule of $M$ which is isomorphic to
$\kappa \simeq \RH( \underline{ \F_p } )$, contradicting our assumption that $M$ does not contain any nonzero subobjects which belong to $\calC$.
\end{proof}

\begin{proof}[Proof of Proposition \ref{proposition.lemX51}]
Let $\kappa$ be an arbitrary field of characteristic $p$. As in the proof of Lemma \ref{lem81}, we see that the functor
$\RH^{c}: \Shv_{\mathet}^{c}( \Spec(\kappa), \F_p) \rightarrow \Mod_{\kappa}^{\hol}$ is a fully faithful embedding whose essential image 
$\calC \subseteq \Mod_{\kappa}^{\hol}$ is an abelian category which is closed under extensions. We wish to show that $\calC$ contains every object $M \in \Mod_{\kappa}^{\hol}$. Let $\overline{\kappa}$ be an algebraic closure of
$\kappa$. Lemma \ref{lem81} shows that $(\overline{\kappa} \otimes_{\kappa} M)^{\perfection} \in \Mod_{\kappa}^{\hol}$ belongs to the essential image
of the functor $\RH^{c}: \Shv_{ \overline{\kappa} }^{c} \rightarrow \Mod_{ \overline{\kappa} }^{\hol}$. Using a direct limit argument, we see that
there exists a finite algebraic extension $\kappa'$ of $\kappa$ such that $M' = (\kappa' \otimes_{\kappa} M)^{\perfection}$ belongs to the essential image of
the functor $\RH^{c}: \Shv_{ \kappa'}^{c} \rightarrow \Mod_{\kappa'}^{\hol}$. By restriction of scalars, we can regard $M'$ as an object of
$\Mod_{\kappa}^{\hol}$ (Proposition \ref{propX45}), and the resulting object belongs to the subcategory $\calC$ (Theorem \ref{theo76}).
We have an evident monomorphism $M \rightarrow M'$ in the abelian category $\Mod_{\kappa}^{\hol}$. Applying the same argument
to the quotient $M/ M'$, we can choose a monomorphism $M / M' \hookrightarrow M''$ for some $M'' \in \calC$. It follows that
$M$ can be identified with the kernel of the composite map $M' \rightarrow M/ M' \hookrightarrow M''$, and therefore belongs to $\calC$
(since $\calC$ is an abelian subcategory of $\Mod_{\kappa}^{\hol}$).
\end{proof}

\subsection{Proof of the Main Theorem}\label{sec7sub4}

We now generalize Proposition \ref{proposition.lemX51} to the case of an arbitrary $\F_p$-algebra:

\begin{theorem}\label{companion}
Let $R$ be a commutative $\F_p$-algebra. Then the Riemann-Hilbert functor $\RH^{c}: \Shv_{\mathet}^{c}( \Spec(R), \F_p) \rightarrow \Mod_{R}^{\hol}$ 
(see Notation \ref{notation.RHc}) is an equivalence of categories.
\end{theorem}

Before giving the proof of Theorem \ref{companion}, let us collect some of its consequences. First, we note that it immediately implies the results of this paper:

\begin{proof}[Proof of Theorem \ref{maintheoXXX} from Theorem \ref{companion}]
Let $R$ be an $\F_p$-algebra. It follows from Theorem \ref{companion} that the functor $\RH^{c} = \RH|_{ \Shv_{\mathet}^{c}( \Spec(R), \F_p) }$ is a fully faithful embedding, those essential
image consists of compact objects of $\Mod_{R}^{\perf}$ (see Proposition \ref{propX55}).  Moreover, the functor $\RH$ preserves filtered colimits (by virtue of the fact that it is left adjoint to the solution functor).
Using Proposition \ref{propX40}, we deduce that $\RH$ is a fully faithful embedding whose essential image consists of those perfect Frobenius modules which can be realized as filtered colimits
of holonomic Frobenius modules. By virtue of Theorem \ref{theoX54}, this essential image is exactly $\Mod_{R}^{\alg}$.
\end{proof}

\begin{proof}[Proof of Theorem \ref{theoX50} from Theorem \ref{companion}]
Let $R$ be an $\F_p$-algebra. Then $\Sol: \Mod_{R}^{\alg} \rightarrow \Shv_{\mathet}( \Spec(R), \F_p)$ is right
adjoint to the Riemann-Hilbert functor $\RH: \Shv_{\mathet}( \Spec(R), \F_p)  \rightarrow \Mod_{R}^{\alg}$. Since the latter is an equivalence of categories,
the former must also be an equivalence of categories.
\end{proof}

\begin{corollary}\label{sendi}
Let $f: A \rightarrow B$ be a homomorphism of $\F_p$-algebras and let $M$ be an algebraic $A$-module.
Then the comparison map $f^{\ast}( \Sol(M) ) \rightarrow \Sol( f^{\diamond} M)$ is an isomorphism in $\Shv_{\mathet}( \Spec(B), \F_p)$.
\end{corollary}

\begin{proof}
Combine Theorem \ref{theoX50} with Proposition \ref{prop70}. 
\end{proof}

\begin{corollary}\label{corlocal}
Let $A \rightarrow B$ be a homomorphism of commutative $\F_p$-algebras which is {\etale} and faithfully flat, and let $M$ be a perfect Frobenius module over $A$.
If $f^{\diamond}(M) = B \otimes_{A} M$ is a holonomic Frobenius module over $B$, then $M$ is a holonomic Frobenius module over $A$.
\end{corollary}

\begin{proof}
It follows from Lemma \ref{ulroc} that $M$ is algebraic. Consequently, to show that $M$ is holonomic, it will suffice (by virtue of Theorems \ref{maintheoXXX} and \ref{companion})
to show that $\Sol(M)$ is a constructible sheaf. This follows from Remark \ref{old64}, since constructibility of {\etale} sheaves can be tested locally with respect to the {\etale} topology. 
\end{proof}

\begin{corollary}\label{corollary.derived-vanishing}
Let $R$ be a commutative $\F_p$-algebra and let $M$ be an algebraic Frobenius module over $R$. Then $\Sol^i(M) \simeq 0$ for $i > 0$.
\end{corollary}

\begin{proof}
By virtue of Theorem \ref{maintheoXXX} we can write $M = \RH( \sheafF )$ for some $\sheafF \in \Shv_{\mathet}( \Spec(R), \F_p)$. In this case,
the desired result follows from Proposition \ref{prop75}.
\end{proof}

\begin{proof}[Proof of Theorem \ref{companion}]
Let $R$ be a commutative $\F_p$-algebra. As in the proof of Lemma \ref{lem81}, we see that the functor
$\RH^{c}$ is a fully faithful embedding whose essential image $\calC \subseteq \Mod_{R}^{\hol}$ is an abelian subcategory which is closed under extensions.
We wish to show that $\calC$ contains every object $M \in \Mod_{R}^{\hol}$. Using a direct limit argument, we can choose an $\F_p$-algebra homomorphism
$\iota: R_0 \rightarrow R$ and an equivalence $M \simeq \iota^{\diamond} M_0$ for some $M_0 \in \Mod_{R_0}^{\hol}$, where $R_0$ is finitely generated over $\F_p$.
By virtue of Proposition \ref{prop70}, it will suffice to show that $M_0$ belongs to the essential image of the functor $\RH^{c}: \Shv_{R_0}^{c} \rightarrow \Mod_{R_0}^{\hol}$.
We may therefore replace $R$ by $R_0$ (and $M$ by $M_0$) and thereby reduce to the case where $R$ is Noetherian. 

Applying Proposition \ref{propX29}, we see that $M$ is a Noetherian object of the abelian category $\Mod_{R}^{\alg}$. Consequently, there exists
a subobject $M' \subseteq M$ (in the abelian category $\Mod_{R}^{\alg}$) which is maximal among those subobjects which belong to $\calC$. 
It follows from the maximality of $M'$ (and the stability of $\calC$ under extensions) that the quotient $M / M'$ does not contain any nonzero subobjects which belong to
$\calC$. Replacing $M$ by $M / M'$, we can reduce to the case where $M$ does not have any nonzero subobjects which belong to $\calC$.

Let $K \subseteq \Spec(R)$ be the closure of the support $\supp(M)$. Then $K$ is the vanishing locus of a radical ideal $I \subseteq \Spec(R)$.
Using Theorem \ref{theorem.kashiwara}, we see that $M$ can be regarded as a holonomic Frobenius module over the quotient ring $R/I$. 
Using Theorem \ref{theo76}, we can replace $A$ by $A/I$ and thereby reduce to the case where $R$ is reduced and
$K = \Spec(R)$.

If $R \simeq 0$, then $M \simeq 0$ and there is nothing to prove. Otherwise, $R$ contains a minimal prime ideal $\mathfrak{p}$.
Since $R$ is reduced, the localization $R_{\mathfrak{p}}$ is a field. Applying Proposition \ref{proposition.lemX51}, we deduce
that $M_{\mathfrak{p}}$ belongs to the essential image of the functor $\RH^{c}: \Shv_{ R_{\mathfrak{p}}}^{c} \rightarrow \Mod_{ R_{\mathfrak{p}} }^{\hol}$.
It follows by a direct limit argument that there exists some element $t \in R - \mathfrak{p}$ for which
the localization $M[t^{-1}]$ belongs to the essential image of the functor $\RH^{c}: \Shv_{ R[t^{-1}]}^{c} \rightarrow \Mod_{R[t^{-1}]}^{\hol}$.
Let $f: R \rightarrow R[t^{-1}]$ be the localization map, and set $M' = f_{!} M[t^{-1}]$; using Proposition \ref{prop68}, we deduce
that $M'$ belongs to the essential image of the Riemann-Hilbert functor $\RH^{c}: \Shv_{R}^{c} \rightarrow \Mod_{R}^{\hol}$.
Note that Lemma \ref{lemX21} guarantees that that the counit map $M' = f_{!} f^{\ast} M \rightarrow M$ is a monomorphism,
so we must have $M' \simeq 0$. It follows that the localization $M[t^{-1}]$ vanishes, so that the prime ideal $\mathfrak{p}$
cannot belong to the support of $M$. Using the constructibility of $\supp(M)$ (Theorem \ref{theo74}), we deduce that there
exists an open neighborhood of $\mathfrak{p}$ which does not intersection $\supp(M)$, contradicting the equality $K = \Spec(R)$.
\end{proof}

\newpage \section{Tensor Products}\label{section.tensor}
\setcounter{subsection}{0}
\setcounter{theorem}{0}

Let $A$ be a commutative ring and let $\Shv_{\mathet}( \Spec(A), \F_p)$ denote the category of
$p$-torsion {\etale} sheaves on $\Spec(A)$. This category is equipped with a tensor product functor
$$ \otimes_{ \F_p} : \Shv_{\mathet}( \Spec(A), \F_p) \times \Shv_{\mathet}( \Spec(A), \F_p) \rightarrow \Shv_{\mathet}( \Spec(A), \F_p)$$
which carries a pair of {\etale} sheaves $(\sheafF, \sheafG)$ to the sheafification of the presheaf
$$ (B \in \CAlg_{A}^{\mathet}) \mapsto \sheafF(B) \otimes_{ \F_p } \sheafG(B).$$
In the case where $A$ is an $\F_p$-algebra, Theorem \ref{theoX50} supplies an equivalence of categories
$$ \Sol: \Mod_{A}^{\alg} \rightarrow \Shv_{\mathet}( \Spec(A), \F_p)$$
Our goal this section is to promote the solution functor $\Sol$ to an equivalence of {\em symmetric monoidal} categories:
that is, to show that it is compatible with tensor products.

We begin in \S \ref{sec8sub1} by studying an analogous tensor product operation on the category
$\Mod_{A}^{\Frob}$ of Frobenius modules over $A$. In fact, there are two such operations (which are closely related):
\begin{itemize}
\item If $M$ and $N$ are Frobenius modules over $A$, then the tensor product $M \otimes_{A} N$ inherits the
structure of a Frobenius module over $A$ (Construction \ref{tenso0}).
\item If $M$ and $N$ are perfect Frobenius modules over $A$, then they can also be regarded as modules of the perfection
$A^{\perfection}$; in this case, the tensor product $M \otimes_{ A^{\perfection} } N$ inherits the structure of a perfect
Frobenius module over $A$ (Remark \ref{tenso2}).
\end{itemize}
Like the usual tensor product on the category of $A$-modules, the tensor product on Frobenius modules is right exact but
generally not left exact. One can partially remedy this failure of exactness by studying left derived functors of the tensor product:
in \S \ref{sec8sub2}, we show that these agree with the usual $\Tor$-functors of commutative algebra (Proposition \ref{snappums}).
The central result of this section asserts that if we restrict our attention to {\em algebraic} Frobenius modules,
then these $\Tor$-groups automatically vanish (when computed relative to the perfection $A^{\perfection}$; see Theorem \ref{tenso10}).
We prove this statement in \S \ref{sec8sub3}, and apply it in \S \ref{sec8sub4} to show that the Riemann-Hilbert correspondence
is compatible with tensor products (Theorem \ref{bijus}).

\subsection{Tensor Products of Frobenius Modules}\label{sec8sub1}

We begin with some general remarks.

\begin{construction}\label{tenso0}
Let $A$ be a commutative $\F_p$-algebra. If $M$ and $N$ are Frobenius modules over $A$, then we regard
the tensor product $M \otimes_{A} N$ as a Frobenius module over $A$, with Frobenius map
$$ \varphi_{M \otimes_{A} N}: M \otimes_{A} N \rightarrow M \otimes_{A} N$$
given by the formula $\varphi_{ M \otimes_{A} N}( x \otimes y) = \varphi_{M}(x) \otimes \varphi_{N}(y)$.
Note that the commutativity and associativity isomorphisms
$$M \otimes_{A} N \simeq N \otimes_{A} M \quad \quad (M \otimes_{A} N) \otimes_{A} P \simeq M \otimes_{A} (N \otimes_{A} P)$$
are isomorphisms of Frobenius modules, and therefore endow $\Mod_{A}^{\Frob}$ with the structure of a symmetric monoidal category.
\end{construction}

\begin{example}[Tensor Products of Free Modules]\label{example.tensorfree}
Let $A$ be a commutative $\F_p$-algebra and let $M$ and $N$ be Frobenius modules over $A$ which are freely generated (as left $A[F]$-modules)
by elements $x \in M$ and $y \in N$. Then the tensor product $M \otimes_{A} N$ is freely generated by the elements
$F^{n} x \otimes y$ and $x \otimes F^{n} y$ (which coincide when $n =0$). 
\end{example}

\begin{remark}[Compatibility with Extension of Scalars]\label{tenso3}
Let $f: A \rightarrow B$ be a homomorphism of commutative $\F_p$-algebras, and let $f^{\ast}_{\Frob}: \Mod_{A}^{\Frob} \rightarrow \Mod_{B}^{\Frob}$
be the functor of extension of scalars along $f$ (given by $M \mapsto B \otimes_{A} M$). Then $f^{\ast}_{\Frob}$ is a symmetric monoidal functor:
in particular, we have canonical isomorphisms $f^{\ast}_{\Frob}( M \otimes_{A} N) \simeq (f^{\ast}_{\Frob} M) \otimes_{B} (f^{\ast}_{\Frob} N)$.
\end{remark}

\begin{remark}\label{tenso1}
Let $A$ be an $\F_p$-algebra. If $M$ and $N$ are Frobenius modules over $A$, then we have a canonical isomorphism
$$ (M \otimes_{A} N)^{\perfection} \simeq M^{\perfection} \otimes_{ A^{\perfection} } N^{\perfection}.$$
In particular, if $A$, $M$, and $N$ are perfect, then the tensor product $M \otimes_{A} N$ is also perfect.
\end{remark}

\begin{remark}\label{tenso2}
Let $A$ be a perfect $\F_p$-algebra. It follows from Remark \ref{tenso1} that the full subcategory $\Mod_{A}^{\perf} \subseteq \Mod_{A}^{\Frob}$ is
closed under tensor products, and therefore inherits the structure of a symmetric monoidal category (with tensor product $\otimes_{A}$).

More generally, if $A$ is an arbitrary $\F_p$-algebra, then the restriction-of-scalars functor $\theta: \Mod_{A^{\perfection}}^{\perf} \rightarrow \Mod_{ A}^{\perf}$ is an
equivalence of categories (Proposition \ref{prop5}). It follows that there is an essentially unique symmetric monoidal structure on the category $\Mod_{A}^{\perf}$ for
which the functor $\theta$ is symmetric monoidal. We will denote the underlying tensor product by $(M, N) \mapsto M \otimes_{ A^{\perfection} } N$ (note that if
$M$ and $N$ are perfect Frobenius modules over $A$, then they can be regarded as modules over $A^{\perfection}$ in an essentially unique way).
\end{remark}

\begin{warning}
Let $A$ be a commutative $\F_p$-algebra. Then the inclusion functor $\Mod_{A}^{\perf} \hookrightarrow \Mod_{A}^{\Frob}$ is usually not a symmetric monoidal
functor, if we regard $\Mod_{A}^{\Frob}$ as equipped with the symmetric monoidal structure of Construction \ref{tenso0} (given by tensor product over $A$) and $\Mod_{A}^{\perf}$ with
the symmetric monoidal structure of Remark \ref{tenso2} (given by tensor product over $A^{\perfection}$). However, it has a symmetric monoidal left adjoint, given by the perfection construction
$M \mapsto M^{\perfection}$ (note that Remark \ref{tenso1} supplies an isomorphism $(M \otimes_{A} N)^{\perfection} \simeq M \otimes_{ A^{\perfection} } N$ in the case where $M$ and $N$ are perfect).
\end{warning}

\begin{remark}[Compatibility with Extension of Scalars]
Let $f: A \rightarrow B$ be a homomorphism of commutative $\F_p$-algebras, and let $f^{\diamond}: \Mod_{A}^{\perf} \rightarrow \Mod_{B}^{\perf}$ be the
functor of Proposition \ref{mallow}. Then $f^{\diamond}$ is symmetric monoidal with respect to the tensor products described in Remark \ref{tenso2}:
in particular, if $M$ and $N$ are perfect Frobenius modules over $A$, then we have a canonical isomorphism $f^{\diamond}( M \otimes_{ A^{\perfection} } N ) \simeq (f^{\diamond} M) \otimes_{ B^{\perfection} } (f^{\diamond} N)$. This follows from Remark \ref{tenso3}, applied to the map $f^{\perfection}: A^{\perfection} \rightarrow B^{\perfection}$. 
\end{remark}

\subsection{Derived Tensor Products}\label{sec8sub2}

Let $A$ be an $\F_p$-algebra and let $M$ be a Frobenius module over $A$. Then the construction $N \mapsto M \otimes_{A} N$
determines a right exact functor from the abelian category $\Mod_{A}^{\Frob}$ to itself. Since the abelian category $\Mod_{A}^{\Frob}$ has enough projective objects (it is equivalent to the category of
left modules over the noncommutative ring $A[F]$ of Notation \ref{not2}), the construction $N \mapsto M \otimes_{A} N$ admits left derived functors, which we will temporarily denote by
$N \mapsto T_{\ast}(M, N)$. More concretely, we define $T_{k}(M, N)$ to be the $k$th homology group of the chain complex
$$ \cdots \rightarrow M \otimes_{A} P_2 \rightarrow M \otimes_{A} P_{1} \rightarrow M \otimes_{A} P_{0} \rightarrow 0,$$
where $\cdots \rightarrow P_2 \rightarrow P_1 \rightarrow P_0 \rightarrow N \rightarrow 0$ is a projective resolution of $N$ in the category $\Mod_{A}^{\Frob}$
(it follows from elementary homological algebra that the Frobenius modules $T_{\ast}(M, N)$ are independent of the choice of resolution, up to canonical isomorphism).

\begin{proposition}\label{snappums}
Let $A$ be an $\F_p$-algebra. For every pair of Frobenius modules $M$ and $N$ over $A$, we have canonical $A$-module isomorphisms
$\Tor_{\ast}^{A}( M, N) \xrightarrow{\sim} T_{\ast}(M, N)$.
\end{proposition}

\begin{proof}
Since $A[F]$ is free as a left $A$-module, every projective left $A[F]$-module is also projective when viewed as a left $A$-module. Consequently, if $P_{\ast}$ is a resolution of
$N$ by projective left $A[F]$-modules, then it is also a resolution of $N$ by projective $A$-modules, so the homology groups of the chain complex $M \otimes_{A} P_{\ast}$
can be identified with $\Tor_{\ast}^{A}(M, N)$.
\end{proof}

We can formulate Proposition \ref{snappums} more informally as follows: if $M$ and $N$ are Frobenius modules over $A$, then the $\Tor$-groups $\Tor_{\ast}^{A}(M, N)$ inherit
the structure of Frobenius modules over $A$.

\begin{remark}\label{tenso5}
Our description of the Frobenius structure on the $\Tor$-groups $\Tor_{\ast}^{A}(M, N)$ is {\it a priori} asymmetric in $M$ and $N$, since it depends on taking
the left derived functors of the construction $M \otimes_{A} \bullet$. However, one can give a more symmetric description as follows. Let $\calC$ denote
the category whose objects are triples $(A, M, N)$, where $A$ is an associative ring, $M$ is a right module over $A$, and $N$ is a left module over $A$. For
each integer $k$, the construction $(A, M, N) \mapsto \Tor_{k}^{A}(M, N)$ can be regarded as a functor from $\calC$ to the category of abelian groups.
In the special case where $A$ is a commutative $\F_p$-algebra and $M, N \in \Mod^{\Frob}_{A}$, we can regard the triple $( \varphi_A, \varphi_M, \varphi_N)$
as a morphism from $(A, M, N)$ to itself in the category $\calC$, and therefore induces a map of abelian groups $\varphi: \Tor_{k}^{A}(M, N) \rightarrow \Tor_{k}^{A}(M, N)$.
It is easy to check that $\varphi$ corresponds to the Frobenius map on $T_{k}(M, N)$ under the isomorphism of Proposition \ref{snappums}.
\end{remark}

\begin{proposition}\label{tenso6}
Let $A$ be a perfect $\F_p$-algebra. If $M$ and $N$ are perfect Frobenius modules over $A$, then the $\Tor$-groups $\Tor_{\ast}^{A}(M, N)$ are also perfect Frobenius modules over $A$.
\end{proposition}

\begin{proof}
This follows immediately from the description of the Frobenius structure on $\Tor_{\ast}^{A}(M, N)$ given in Remark \ref{tenso5}. Alternatively, we can show that
each $\Tor_{k}^{A}(M, N)$ is perfect using induction on $k$. When $k = 0$, the desired result follows from Remark \ref{tenso1}. For $k > 0$, we can choose
a short exact sequence $0 \rightarrow N' \rightarrow P \rightarrow N \rightarrow 0$, where $P$ is a free module over $A^{\perfection}[ F^{\pm 1} ]$ (see Example \ref{skriff}).
Then $N'$ is also a perfect Frobenius module over $A$. Moreover, since $A$ is perfect, the ring $A^{\perfection}[ F^{\pm 1} ]$ is free as a left $A$-module, so the groups
$\Tor_{\ast}^{A}(M, P)$ vanish for $\ast > 0$. We therefore have a short exact sequence
$$ 0 \rightarrow \Tor_{k}^{A}(M, N) \rightarrow \Tor_{k-1}^{A}(M, N') \rightarrow \Tor_{k-1}^{A}(M, P)$$
which exhibits $\Tor_{k}^{A}(M, N)$ as the kernel of a map between perfect Frobenius modules, so that $\Tor_{k}^{A}(M,N)$ is itself perfect.
\end{proof}

\begin{variant}
Let $A$ be an arbitrary $\F_p$-algebra, and let $M$ and $N$ be perfect Frobenius modules over $A$. Then we can regard $M$ and
$N$ as Frobenius modules over $A^{\perfection}$ in an essentially unique way (Proposition \ref{prop5}). Using Proposition \ref{tenso6}, we can
regard the $\Tor$-groups $\Tor_{\ast}^{ A^{\perfection} }( M, N)$ as perfect Frobenius modules over $A^{\perfection}$, and therefore also (by restriction of scalars)
as perfect Frobenius modules over $A$.
\end{variant}

\begin{proposition}\label{tenso11}
Let $A$ be an $\F_p$-algebra and let $M$ and $N$ be Frobenius modules over $A$. Then the canonical map 
$\Tor_{\ast}^{A}(M, N) \rightarrow \Tor_{\ast}^{A^{\perfection}}( M^{\perfection}, N^{\perfection} )$ induces an isomorphism of Frobenius modules
$$\rho_{\ast}: \Tor_{\ast}^{A}( M, N)^{\perfection} \simeq \Tor_{\ast}^{A^{\perfection}}( M^{\perfection}, N^{\perfection} ).$$
\end{proposition}

\begin{proof}
Let us regard $M$ as fixed. We will show that for every Frobenius module $N$ and every nonnegative integer $k$, the
map $$\rho_{k}: \Tor_{k}^{A}(M, N)^{\perfection} \rightarrow \Tor_{k}^{A^{\perfection}}( M^{\perfection}, N^{\perfection} )$$ is an isomorphism.
The proof proceeds by induction on $k$. When $k =0$, the desired result is the content of Remark \ref{tenso1}.
Assume that $k > 0$, and choose a short exact sequence of Frobenius modules $0 \rightarrow N' \rightarrow P \rightarrow N \rightarrow 0$
where $P$ is a free left module over $A[F]$. Then $P^{\perfection}$ is a free left module over $A^{\perfection}[ F^{\pm 1} ]$, and is therefore
also free as an $A^{\perfection}$-module. It follows that the groups $\Tor_{k}^{A}( M, P)$ and $\Tor_{k}^{A^{\perfection} }( M^{\perfection}, P^{\perfection} )$ both vanish.
Consequently, the map $\rho_{k}$ fits into a commutative diagram of exact sequences
$$ \xymatrix{ 0 \ar[d] & 0 \ar[d] \\
 \Tor_{k}^{A}(M, N)^{\perfection} \ar[r]^-{\rho_k} \ar[d] &  \Tor_{k}^{A^{\perfection}}( M^{\perfection}, N^{\perfection} ) \ar[d] \\
  \Tor_{k-1}^{A}(M, N')^{\perfection} \ar[r]^-{\rho'} \ar[d] & \Tor_{k-1}^{A^{\perfection}}( M^{\perfection}, N'^{\perfection} )  \ar[d] \\
  \Tor_{k-1}^{A}(M, P)^{\perfection} \ar[r]^-{\rho''} &  \Tor_{k-1}^{A^{\perfection}}( M^{\perfection}, P^{\perfection} ). }$$
The maps $\rho'$ and $\rho''$ are isomorphisms by our inductive hypothesis, so that $\rho_k$ is an isomorphism as well.
\end{proof}

\begin{remark}[Compatibility with Extension of Scalars]
Let $f: A \rightarrow B$ be a homomorphism of $\F_p$-algebras. Then the extension of scalars functor $$f^{\ast}_{\Frob}: \Mod_{A}^{\Frob} \rightarrow \Mod_{B}^{\Frob}$$
is right exact, having left derived functors $N \mapsto \Tor_{\ast}^{A}( B, N)$. Let $M$ be a Frobenius module over $B$. Then we can regard the functors $\{ \Tor_{k}^{A}(M, \bullet) \}_{k \geq 0}$
the the left derived functor of the construction $N \mapsto M \otimes_{B} (f^{\ast}_{\Frob} N)$. Since the functor $f^{\ast}_{\Frob}: \Mod_{A}^{\Frob} \rightarrow \Mod_B^{\Frob}$ carries
projective objects to projective objects, we have a Grothendieck spectral sequence (in the abelian category $\Mod_{A}^{\Frob}$)
$$ \Tor^{B}_{s}( M, \Tor^{A}_{t}(B, N) ) \Rightarrow \Tor^{A}_{s+t}( M, N).$$
If $M$ and $N$ are perfect, then we can apply the same reasoning to the induced map $A^{\perfection} \rightarrow B^{\perfection}$ to obtain a Grothendieck spectral sequence
$$ \Tor^{B^{\perfection}}_{s}( M, \Tor^{A^{\perfection}}_{t}(B^{\perfection}, N) ) \Rightarrow \Tor^{A^{\perfection}}_{s+t}( M, N).$$
\end{remark}

\subsection{Tensor Products of Holonomic Modules}\label{sec8sub3}

Our next goal is to prove the following variant of Theorem \ref{theoX9}:

\begin{theorem}\label{tenso10}
Let $A$ be an $\F_p$-algebra and let $M$ and $N$ be algebraic Frobenius modules over $A$. Then:
\begin{itemize}
\item[$(1)$] The tensor product $M \otimes_{ A^{\perfection} } N$ is algebraic.
\item[$(2)$] The $\Tor$-groups $\Tor^{A^{\perfection}}_{\ast}( M, N)$ vanish for $\ast > 0$.
\item[$(3)$] If $M$ and $N$ are holonomic, then $M \otimes_{A^{\perfection} } N$ is holonomic.
\end{itemize}
\end{theorem}

The proof of Theorem \ref{tenso10} will require some preliminaries.

\begin{lemma}\label{tenso20}
Let $A$ be a Noetherian $\F_p$-algebra, and let $M$ and $N$ be holonomic Frobenius modules over $A$.
Then:
\begin{itemize}
\item[$(1)$] The $\Tor$-groups $\Tor_{\ast}^{A^{\perfection}}(M, N)$ are also holonomic Frobenius modules over $A$.
\item[$(2)$] Let $f: A \rightarrow B$ be any homomorphism of commutative rings. Then the canonical map
$$ f^{\diamond} \Tor_{\ast}^{ A^{\perfection} }( M, N) \rightarrow \Tor_{\ast}^{ B^{\perfection} }( f^{\diamond} M, f^{\diamond} N)$$
is an isomorphism.
\end{itemize}
\end{lemma}

\begin{proof}
Since $M$ and $N$ are holonomic, we can write $M = M_0^{\perfection}$ and $N = N_0^{\perfection}$, where $M_0, N_0 \in \Mod_{A}^{\Frob}$ are finitely
generated as $A$-modules. The assumption that $A$ is Noetherian guarantees that the $\Tor$-groups $\Tor_{k}^{A}( M_0, N_0 )$ is finitely
generated as an $A$-module. Using the isomorphisms $\Tor_{k}^{A^{\perfection} }( M_0^{\perfection}, N_0^{\perfection} ) \simeq \Tor_{k}^{A}( M_0, N_0)^{\perfection}$ of
Proposition \ref{tenso11}, we conclude that each $\Tor_{k}^{A^{\perfection}}( M, N)$ is holonomic. This proves $(1)$.

We now prove $(2)$. Let $f: A \rightarrow B$ be a homomorphism of commutative rings. Let $P_{\ast}$ and $Q_{\ast}$ be resolutions of
$M$ and $N$ by projective objects of $\Mod_{A}^{\perf}$. Then $P_{\ast}$ and $Q_{\ast}$ are also resolutions of $M$ and $N$ by projective
$A^{\perfection}$-modules. It follows that the homology groups of the complexes $f^{\diamond} P_{\ast}$ and $f^{\diamond} Q_{\ast}$
can be identified with the groups $\Tor_{\ast}^{A^{\perfection}}( B^{\perfection}, M)$ and $\Tor_{\ast}^{A^{\perfection}}( B^{\perfection}, N)$, which vanish
for $\ast > 0$ by virtue of Theorem \ref{theoX9}. In other words, we can regard $f^{\diamond} P_{\ast}$ and $f^{\diamond} Q_{\ast}$
as projective resolutions of $f^{\diamond} M$ and $f^{\diamond} N$, respectively. It follows that
$\Tor_{\ast}^{ B^{\perfection} }( f^{\diamond} M, f^{\diamond} N)$ can be identified with the homology of the tensor product complex 
$$ (f^{\diamond} P_{\ast}) \otimes_{B^{\perfection}} ( f^{\diamond} Q_{\ast} ) \simeq B^{\perfection} \otimes_{ A^{\perfection} } ( P_{\ast} \otimes_{A^{\perfection} } Q_{\ast} ).$$
We therefore have a convergent spectral sequence
$$ E^{2}_{s,t} = \Tor_{s}^{A^{\perfection}}( B^{\perfection}, \Tor_{t}^{A^{\perfection}}( M, N) ) \Rightarrow \Tor_{s+t}^{ B^{\perfection} }( f^{\diamond} M, f^{\diamond} N).$$
To prove assertion $(2)$, it will suffice to show that the groups $E^{2}_{s,t}$ vanish for $s > 0$, which follows from assertion $(1)$ and Theorem \ref{theoX9}.
\end{proof}

\begin{proof}[Proof of Theorem \ref{tenso10}]
Let $M$ and $N$ be algebraic Frobenius modules over an $\F_p$-algebra $A$; we wish to prove that
the tensor product $M \otimes_{A^{\perfection} } N$ is algebraic and that the $\Tor$-groups $\Tor_{\ast}^{A^{\perfection}}( M, N)$ vanish for $\ast > 0$.
Using Theorem \ref{theoX54}, we can write $M$ as a filtered colimit of holonomic Frobenius modules and thereby reduce to the case where $M$ is
holonomic. Similarly, we can assume that $N$ is holonomic. Applying Proposition \ref{propX52}, we can assume that $M = \iota^{\diamond} M'$ and
$N = \iota^{\diamond} N'$, where $\iota: A' \hookrightarrow A$ is the inclusion of a finitely generated subalgebra and $M', N' \in
\Mod_{A}^{\hol}$. In this case, Lemma \ref{tenso20} supplies isomorphisms $\Tor_{\ast}^{A^{\perfection}}( M, N) \simeq \iota^{\diamond} \Tor_{\ast}^{A'^{\perfection}}( M', N' )$.
We may therefore replace $A$ by $A'$, and thereby reduce to the case where $A$ is Noetherian. It now follows from Lemma \ref{tenso20} that the $\Tor$-groups
$\Tor_{s}^{A^{\perfection}}(M, N)$ are holonomic for each $s \geq 0$; we wish to show that they vanish for $s > 0$. By virtue of Proposition \ref{theo78},
it will suffice to show that $f^{\diamond} \Tor_{s}^{A^{\perfection}}( M, N) \simeq 0$ for every homomorphism $f: A \rightarrow \kappa$, where $\kappa$ is a field.
Applying Lemma \ref{tenso20} again, we can reduce to the case where $A = \kappa$, in which case the vanishing is automatic.
\end{proof}

\subsection{Compatibility with the Riemann-Hilbert Correspondence}\label{sec8sub4}

Let $A$ be a commutative $\F_p$-algebra and let
$\Sol: \Mod_{A}^{\perf} \rightarrow \Shv_{\mathet}( \Spec(A), \F_p)$ be the solution sheaf functor (Construction \ref{solsheaf2}), given by the formula
$$ \Sol(M)(B) = \{ x \in (B \otimes_{A} M): \varphi_{ B \otimes_{A} M}(x) = x \}.$$
Note that if $x \in \Sol(M)(B)$ and $y \in \Sol(N)(B)$, then the tensor $x \otimes y$ can be regarded as an element of
$\Sol(M \otimes_{A^{\perfection}} N)(B)$. This observation determines a bilinear map
$$ \Sol(M)(B) \times \Sol(N)(B) \rightarrow \Sol(M \otimes_{ A^{\perfection} } N)(B)$$
which depends functorially on $B$, and therefore gives rise to a map of sheaves
$\Sol(M) \otimes_{ \F_p } \Sol(N) \rightarrow \Sol( M \otimes_{ A^{\perfection} } N )$.

\begin{theorem}\label{bijus}
Let $A$ be a commutative $\F_p$-algebra and suppose that $M$ and $N$ are algebraic $A$-modules. Then
the comparison map 
$$\theta: \Sol(M) \otimes_{ \F_p } \Sol(N) \rightarrow \Sol( M \otimes_{ A^{\perfection} } N )$$
is an isomorphism in the category $\Shv_{\mathet}( \Spec(A), \F_p)$.
\end{theorem}

\begin{proof}
It will suffice to show that for every algebraically closed field $\kappa$ and every homomorphism $f: A \rightarrow \kappa$,
the pullback $f^{\ast}(\theta)$ is an isomorphism in $\Shv_{\kappa}^{\mathet}$. Since
$M$, $N$, and $M \otimes_{ A^{\perfection}} N$ are algebraic (Theorem \ref{tenso10}), we can identify
$f^{\ast}( \theta)$ with the tautological map $$\Sol( f^{\diamond} M) \otimes_{ \F_p } \Sol( f^{\diamond} N)
\rightarrow \Sol( f^{\diamond} M \otimes_{ \kappa} f^{\diamond} N).$$
We may therefore replace $A$ by $\kappa$ and thereby reduce to the case where $A$ is an algebraically closed field.
In this case, Theorem \ref{theoX50} implies that $\Mod_{A}^{\alg}$ is equivalent to the category of vector spaces over $\F_p$.
Consequently, the Frobenius modules $M$ and $N$ can be decomposed as a direct sum of copies of $A^{\perfection} = \kappa$, and the desired result is obvious.
\end{proof}

\begin{corollary}
\label{tensoraffine}
Let $A$ be a commutative $\F_p$-algebra. Then the Riemann-Hilbert functor $\RH: \Shv_{\mathet}( \Spec(A), \F_p) \rightarrow \Mod_{A}^{\perf}$ admits the structure of a symmetric monoidal functor
(where the symmetric monoidal structure on $\Shv_{\mathet}( \Spec(A), \F_p)$ is given by the usual tensor product of sheaves,
and the symmetric monoidal structure on $\Mod_{A}^{\perf}$ is given by the tensor product $\otimes_{A^{\perfection}}$ of Remark \ref{tenso2}).
\end{corollary}

\begin{proof}
It follows from Theorem \ref{bijus} that the lax symmetric monoidal functor $\Sol: \Mod_{A}^{\perf} \rightarrow \Shv_{\mathet}( \Spec(A), \F_p)$ is
symmetric monoidal when restricted to $\Mod_{A}^{\alg}$. Combining this observation with Theorem \ref{theoX50}, we see that the functor
$\Sol|_{ \Mod_{A}^{\alg} }$ is an equivalence of symmetric monoidal categories. We conclude by observing that the functor $\RH$ can
be obtained by composing an inverse of $\Sol|_{ \Mod_{A}^{\alg} }$ with the inclusion functor $\Mod_{A}^{\alg} \hookrightarrow \Mod_{A}^{\perf}$.
\end{proof}

\newpage \section{The $p^{n}$-Torsion Case}\label{section.moretorsion}
\setcounter{subsection}{0}
\setcounter{theorem}{0}

Let $R$ be a commutative $\F_p$-algebra. Theorem \ref{maintheoXXX} supplies a fully faithful embedding from the category $\Shv_{\mathet}( \Spec(R), \F_p)$
$p$-torsion {\etale} sheaves on $\Spec(R)$ to the category of Frobenius modules over $R$. Our goal in this section is to prove a generalization
of Theorem \ref{maintheoXXX}, which gives an analogous realization for the category $\Shv_{\mathet}( \Spec(R), \Z / p^{n} \Z)$ of $\Z /p^{n} \Z$-torsion
{\etale} sheaves, for any nonnegative integer $n$. Our first step will be to introduce an analogous enlargement of the category
$\Mod_{R}^{\Frob}$ of Frobenius module over $R$. In \S \ref{sec9sub1}, we define a notion of Frobenius module over $W_n(R)$, where
$W_n(R)$ is the ring of length $n$ Witt vectors over $R$ (Definition \ref{definition.frobenius-module2}). The collection of such Frobenius modules
can be organized into a category $\Mod_{W_n(R)}^{\Frob}$. In \S \ref{sec9sub2}, we study the dependence of this category on the $\F_p$-algebra $R$
(emphasizing in particular the effect of replacing $R$ by its perfection $R^{\perfection}$, which makes Witt vectors much more pleasant to work with).
In \S \ref{sec9sub3}, we introduce a {\it solution functor}
$$\Sol: \Mod_{W_n(R)}^{\Frob} \rightarrow \Shv_{\mathet}( \Spec(R), \Z / p^{n} \Z)$$
connecting Frobenius modules over $W_n(R)$ to $p^{n}$-torsion {\etale} sheaves (which reduces to Construction \ref{construction.solsheaf} in the case $n=1$). 
Like its $p$-torsion counterpart, this solution functor is not exact. However, we show in \S \ref{sec9sub4} that it is {\em almost} exact when restricted to perfect Frobenius modules,
in the sense that it has only one nonvanishing derived functor (Proposition \ref{proposition.corX77gen}). We apply this result in \S \ref{sec9sub6} to show that
the functor $\Sol$ restricts to an equivalence of categories $\Mod_{ W_n(R)}^{\alg} \simeq \Shv_{\mathet}( \Spec(R), \Z/ p^{n} \Z)$ (Theorem \ref{theorem.generalizedRH}). 
Here $\Mod_{ W_n(R)}^{\alg}$ denotes the full subcategory of $\Mod_{ W_n(R)}^{\alg}$ spanned by the {\it algebraic} Frobenius modules over
$W_n(R)$, which we introduce in \S \ref{sec9sub5} (Definition \ref{defholgen}).

\subsection{Frobenius Modules over the Witt Vectors}\label{sec9sub1}

We begin by extending some of the notions introduced in \S \ref{section.frobenius-module}.
Let $R$ be a commutative $\F_p$-algebra. For every nonnegative integer $n$, we let $W_n(R)$ denote the ring of length $n$ Witt vectors of $R$.
The Frobenius map $\varphi_{R}: R \rightarrow R$ induces a ring homomorphism $F: W_n(R) \rightarrow W_n(R)$, which we will refer to as the
{\it Witt vector Frobenius}.

\begin{definition}\label{definition.frobenius-module2}
Let $R$ be a commutative $\F_p$-algebra and let $n \geq 0$ be an integer. A {\it Frobenius module over $W_n(R)$} is an
$W_n(R)$-module $M$ equipped with an additive map $\varphi_M: M \rightarrow M$ satisfying the identity $\varphi_M( \lambda x) = F(\lambda) \varphi_M(x)$ for $x \in M$, $\lambda \in W_n(R)$.
We will say that a Frobenius module $M$ is {\it perfect} if the map $\varphi_{M}: M \rightarrow M$ is an isomorphism of abelian groups. 

Let $(M, \varphi_M)$ and $(N, \varphi_N)$ be Frobenius modules over $W_n(R)$. A {\it morphism of Frobenius modules} from $(M, \varphi_M)$ to $(N, \varphi_N)$ is an
$W_n(R)$-module homomorphism $\rho: M \rightarrow N$ for which the diagram
$$ \xymatrix{ M \ar[r]^{\rho} \ar[d]^{\varphi_M} & N \ar[d]^{\varphi_N} \\
M \ar[r]^{\rho} & M' }$$
commutes. 
We let $\Mod_{W_n(R)}^{\Frob}$ denote the category whose objects are Frobenius modules $(M, \varphi_M)$ over $W_n(R)$, and whose morphisms are morphisms of Frobenius modules.
We let $\Mod_{W_n(R)}^{\perf}$ denote the full subcategory of $\Mod_{W_n(R)}^{\Frob}$ spanned by the perfect Frobenius modules over $W_n(R)$.
\end{definition}

\begin{remark}
In the special case $n = 1$, Definition \ref{definition.frobenius-module2} reduces to Definitions \ref{definition.frobenius-module} and \ref{definition.perfect}. In particular, we have an equivalence of categories $\Mod_{W_1(R)}^{\Frob} \simeq \Mod_{R}^{\Frob}$, which restricts to an equivalence $\Mod_{ W_1(R)}^{\perf} \simeq \Mod_{ R}^{\perf}$.
\end{remark}

\begin{remark}
Let $R$ be a commutative ring in which $p = 0$ and let $n \geq 0$. Then $\Mod_{ W_n(R)}^{\Frob}$ can be identified with the category of
modules over the noncommutative ring $W_n(R)[ F ]$ whose elements are finite sums 
$\sum_{i \geq 0} c_i F^{i}$, where each coefficient $c_i$ belongs to $W_n(R)$, with multiplication given by
$$ ( \sum_{i \geq 0} c_i F^{i} ) ( \sum_{ j \geq 0 } c'_{j} F^{j} ) = \sum_{ k \geq 0 } ( \sum_{ i + j = k } c_{i} F^{i}(c_{j})) F^{k}.$$
In particular, $\Mod_{ W_n(R)}^{\Frob}$ is an abelian category with enough projective objects and enough injective objects.
\end{remark}

\begin{remark}
In the situation of Definition \ref{definition.frobenius-module2}, the full subcategory $$\Mod_{W_n(R)}^{\perf} \subseteq \Mod_{W_n(R)}^{\Frob}$$ is closed under limits, colimits, and extensions. In particular,
$\Mod_{W_n(R)}^{\perf}$ is an abelian category, and the inclusion functor $\Mod_{W_n(R)}^{\perf} \hookrightarrow \Mod_{W_n(R)}^{\Frob}$ is exact. 
\end{remark}

\begin{remark}\label{remark.embeddings}
For each $n > 0$, we can identify $\Mod_{W_{n-1}(R)}^{\Frob}$ with the full subcategory of $\Mod_{ W_n(R)}^{\Frob}$ spanned by those objects $(M, \varphi_M)$ where
$M$ is annihilated by the kernel of the restriction map $W_{n}(R) \rightarrow W_{n-1}(R)$. We therefore obtain (exact) fully faithful embeddings
$$ \Mod_{R}^{\Frob} \simeq \Mod_{W_1(R)}^{\Frob} \hookrightarrow \Mod_{ W_2(R)}^{\Frob} \hookrightarrow \Mod_{ W_3(R)}^{\Frob} \hookrightarrow \cdots$$
Similarly, we have fully faithful embeddings
$$ \Mod_{R}^{\perf} \simeq \Mod_{W_1(R)}^{\perf} \hookrightarrow \Mod_{ W_2(R)}^{\perf} \hookrightarrow \Mod_{ W_3(R)}^{\perf} \hookrightarrow \cdots$$
\end{remark}

\begin{remark}
In the situation of Definition \ref{definition.frobenius-module2}, the inclusion $\Mod_{W_n(R)}^{\perf} \hookrightarrow \Mod_{W_n(R)}^{\Frob}$ admits a left adjoint. Concretely, this left adjoint carries a Frobenius module $M$ to the direct limit of the sequence
$$ M \xrightarrow{ \varphi_{M} } M^{ F^{-1} } \xrightarrow{ \varphi_{M} } 
M^{ F^{-2} } \rightarrow \cdots;$$
here $M^{ F^{-k} }$ denotes the $W_n(R)$-module obtained from $M$ by restriction of scalars along the ring homomorphism $F^{k}: W_n(R) \rightarrow W_n(R)$.
We will denote this direct limit by $M^{\perfection}$ and refer to it as the {\it perfection} of $M$. Note that when $n = 1$, this agrees with the construction of Notation \ref{biskar}.
\end{remark}

\begin{example}
Let $R$ be a commutative $\F_p$-algebra. For each $n \geq 0$, we can regard $M = W_n(R)$ as a Frobenius module over itself by taking
$\varphi_{M}$ to be the Witt vector Frobenius map $F: W_n(R) \rightarrow W_n(R)$. Then the perfection $M^{\perfection}$ can be identified with $W_n( R^{\perfection} )$.
\end{example}

\subsection{Functoriality}\label{sec9sub2}

If $f: A \rightarrow B$ is a homomorphism of $\F_p$-algebras, then there is an evident forgetful functor
$\Mod_{ W_n(B)}^{\Frob} \rightarrow \Mod_{ W_n(A) }^{\Frob}$. This functor admits a left adjoint $f^{\ast}_{\Frob}: \Mod_{ W_n(A)}^{\Frob} \rightarrow \Mod_{ W_n(B)}^{\Frob}$, given
by extension of scalars along the evident ring homomorphism $W_n(A)[F] \rightarrow W_n(B)[F]$. Since the natural map $W_n(B) \otimes_{ W_n(A) } W_n(A)[F] \rightarrow W_n(B)[F]$
is an isomorphism, we have canonical isomorphisms $f^{\ast}_{\Frob} M \simeq W_n(B) \otimes_{ W_n(A) } M$ in the category of $W_n(B)$-modules.

\begin{remark}\label{remX5gen}
Let $f: A \rightarrow B$ be a homomorphism of commutative $\F_p$-algebras and let $M$ be a Frobenius module over $W_n(B)$. Then $M$ is {perfect} as a Frobenius module over $W_n(B)$ if and only if it is {perfect} when regarded as a Frobenius module over $W_n(A)$.
Moreover, the perfection $M^{\perfection}$ does not depend on whether we regard $M$ as a Frobenius module over $W_n(B)$ or over $W_n(A)$.
It follows that the diagram of forgetful functors
$$ \xymatrix{ \Mod_{W_n(B)}^{\perf} \ar[r] \ar[d] & \Mod_{W_n(B)}^{\Frob} \ar[d] \\
\Mod_{W_n(A)}^{\perf} \ar[r] & \Mod_{W_n(A)}^{\Frob} }$$
commutes (up to canonical isomorphism). 
\end{remark}

The following result is a formal consequence of Remark \ref{remX5gen}:

\begin{proposition}\label{mallowgen}
Let $f: A \rightarrow B$ be a homomorphism of commutative $\F_p$-algebras. Then the forgetful functor
$\Mod_{W_n(B)}^{\perf} \rightarrow \Mod_{W_n(A)}^{\perf}$ admits a left adjoint $f^{\diamond}$. Moreover, the diagram of categories
$$ \xymatrix{ \Mod_{W_n(A)}^{\Frob} \ar[r]^-{(-)^{\perfection}} \ar[d]^{ f^{\ast}_{\Frob} } & \Mod_{W_n(A)}^{\perf} \ar[d]^{ f^{\diamond} } \\
\Mod_{W_n(B)}^{\Frob} \ar[r]^-{(-)^{\perfection}} & \Mod_{W_n(B)}^{\perf} }$$
commutes up to canonical isomorphism. More precisely, for every object $M \in \Mod_{A}^{\Frob}$, the canonical map
$f^{\diamond}(M^{\perfection}) \rightarrow ( f^{\ast}_{\Frob} M)^{\perfection}$ is an equivalence.
\end{proposition}

\begin{proposition}\label{prop5gen}
Let $R$ be a commutative $\F_p$-algebra. For each $n \geq 0$, the restriction of scalars functor
$\Mod_{ W_n( R^{\perfection} ) }^{\perf} \rightarrow \Mod_{W_n(R)}^{\perf}$ is an equivalence of categories.
\end{proposition}

\begin{proof}
Let $f: R \rightarrow R^{\perfection}$ be the tautological map.
Since the restriction of scalars functor is evidently conservative, it suffices to observe that for each object
$M \in \Mod_{ W_n(R)}^{\perf}$, the unit map $$M \rightarrow f^{\diamond}(M) = ( W_n(R^{\perfection} ) \otimes_{ W_n(R) } M )^{\perfection}$$
is an isomorphism of (perfect) Frobenius modules over $W_n(R)$.
\end{proof}

\begin{corollary}\label{corollary.describe-image}
Let $R$ be an $\F_p$-algebra and let $0 \leq m \leq n$. Then the essential image of the tautological map
$\Mod_{ W_m(R) }^{\perf} \hookrightarrow \Mod_{ W_n(R) }^{\perf}$ consists of those perfect Frobenius modules over
$W_n(R)$ which are annihilated by $p^{m}$.
\end{corollary}

\begin{proof}
By virtue of Proposition \ref{prop5gen}, we can assume without loss of generality that $R$ is perfect. In this case,
the desired result follows from Remark \ref{remark.embeddings}, since the kernel of the restriction map
$W_n(R) \rightarrow W_m(R)$ is the principal ideal $(p^{m} )$.
\end{proof}

\begin{proposition}\label{etalebasechangegen}
Let $f: A \rightarrow B$ be a homomorphism of perfect $\F_p$-algebras.
Then the extension of scalars functor $f^{\ast}_{\Frob}: \Mod_{W_n(A)}^{\Frob} \rightarrow \Mod_{W_n(A)}^{\Frob}$
carries perfect Frobenius modules over $W_n(A)$ to perfet Frobenius modules over $W_n(B)$.
\end{proposition}

\begin{proof}
Let $M$ be a {perfect} Frobenius module over $A$. Then the maps
$$ F_B: W_n(B) \rightarrow W_n(B) \quad \quad F_A: W_n(A) \rightarrow W_n(A) \quad \quad \varphi_{M}: M \rightarrow M$$
are isomorphisms, so the induced map $$\varphi_{ f^{\ast}_{\Frob} M}: W_n(B) \otimes_{W_n(A) } M \rightarrow W_n(B) \otimes_{ W_n(A) } M$$ is also an isomorphism.
\end{proof}

\begin{corollary}\label{corX14gen}
Let $f: A \rightarrow B$ be an {\etale} morphism of $\F_p$-algebras. Then the extension of scalars functor $f^{\ast}_{\Frob}: \Mod_{W_n(A)}^{\Frob} \rightarrow \Mod_{W_n(B)}^{\Frob}$
carries $\Mod_{W_n(A)}^{\perf}$ into $\Mod_{W_n(B)}^{\perf}$.
\end{corollary}

\begin{proof}
Let $M$ be a {perfect} Frobenius module over $A$. Then we can also regard $M$ as a Frobenius module over $A^{\perfection}$. Since $f$ is {\etale}, the diagram
of commutative rings
$$ \xymatrix{ W_n(A) \ar[r] \ar[d] & W_n(B) \ar[d] \\
W_n(A^{\perfection}) \ar[r] & W_n(B^{\perfection}) }$$
is a pushout square by the result \cite[2.4]{vdK} of van der Kallen. It follows that we can identify $f^{\ast}_{\Frob} M$ with the tensor product $W_n(B^{\perfection}) \otimes_{ W_n(A^{\perfection})} M$, which is perfect
by Proposition \ref{etalebasechangegen}. 
\end{proof}

\subsection{The Solution Functor}\label{sec9sub3}

We now adapt Construction \ref{construction.solsheaf} to the setting of Frobenius modules over the Witt vectors.

\begin{construction}\label{solsheaf2}
Let $R$ be a commutative $\F_p$-algebra and let $M$ be a Frobenius module over $W_n(R)$.
We let $\Sol(M)$ denote the functor $\CAlg_{R}^{\mathet} \rightarrow \Mod_{ \Z / p^{n} \Z}$ given by the formula
$$ \Sol(M)(R') = \{ x \in (W_n(R') \otimes_{W_n(R)} M): \varphi_{ W_n(R') \otimes_{W_n(R)} M}(x) = x \}.$$
We will refer to $\Sol(M)$ as the {\it solution sheaf} of $M$.
\end{construction}

\begin{remark}
In the situation of Construction \ref{solsheaf2}, suppose that the action of $W_n(R)$ on $M$ factors
through the restriction map $W_n(R) \rightarrow R$, so that $M$ can be regarded as a Frobenius module over $R$
(if $M$ is perfect, this is equivalent to the requirement that $pM = 0$, by virtue of Corollary \ref{corollary.describe-image}).
Then the functor $\Sol(M)$ of Construction \ref{solsheaf2} agrees with the functor $\Sol(M)$ of Construction \ref{construction.solsheaf}:
this follows from the fact that the diagram of commutative rings
$$ \xymatrix{ W_n(R) \ar[r] \ar[d] & W_n(R') \ar[d] \\
R \ar[r] & R' }$$
is a pushout square, for any {\etale} $R$-algebra $R'$.
\end{remark}

Our first goal is to show that the functor $\Sol(M)$ of Construction \ref{solsheaf2} is actually a sheaf with respect to the {\etale} topology
on $\Spec(R)$. To prove this, it will be convenient to consider the following variant of Example \ref{exX70}:

\begin{notation}\label{notation.M-tilde}
Let $R$ be a commutative $\F_p$-algebra and let $M$ be a module over $W_n(R)$. We let $\widetilde{M} \in \Shv_{\mathet}( \Spec(R), \Z / p^{n} \Z)$
denote the sheaf given by the formula $\widetilde{M}(R') = W_n(R') \otimes_{ W_n(R) } M$. Note that, when $M$ is annihilated by
the kernel of the restriction map $W_n(R) \rightarrow R$ (so that $M$ can be regarded as an $R$-module), this agrees with
the sheaf of $\Z/ p \Z$-modules introduced in Example \ref{exX70}.
\end{notation}

\begin{remark}
Let $R$ be a commutative $\F_p$-algebra. Then the construction $R' \mapsto W_n(R')$ induces an equivalence from
the category of {\etale} $R$-algebras to the category of {\etale} $W_n(R)$-algebras. In particular, the category of
{\etale} sheaves on $\Spec(R)$ is equivalent to the category of {\etale} sheaves on $\Spec( W_n(R) )$. If
$M$ is a module over $W_n(R)$, then it determines a quasi-coherent sheaf on $\Spec( W_n(R) )$, which corresponds
(under the preceding equivalence) to the functor $\widetilde{M}: \CAlg_{R}^{\mathet} \rightarrow \Mod_{ \Z / p^{n} \Z}$ of Notation \ref{notation.M-tilde}. 
In particular, the functor $\widetilde{M}$ is always a sheaf with respect to the {\etale} topology on $\CAlg_{R}^{\mathet}$.
\end{remark}

\begin{remark}\label{remark.technical}
In the situation of Notation \ref{notation.M-tilde}, suppose that $R$ is perfect and that $M$ is flat as a module over $\Z / p^{n} \Z$.
Then, for each $R' \in \CAlg_{R}^{\mathet}$, the abelian group $\widetilde{M}(R')$ is also flat as a $\Z / p^{n} \Z$-module.
In particular, the sheaf $\widetilde{M}$ is flat over $\Z / p^{n} \Z$.
\end{remark}

If $M$ is a Frobenius module over $W_n(R)$, then the Frobenius map $\varphi_{M}$ determines an endomorphism of
the associated {\etale} sheaf $\widetilde{M}$. By construction, we have an exact sequence of presheaves
$$ 0 \rightarrow \Sol(M) \rightarrow \widetilde{M} \xrightarrow{ \id - \varphi_M} \widetilde{M}.$$
It follows that $\Sol(M)$ is always a sheaf with respect to the {\etale} topology. We may therefore
regard the construction $M \mapsto \Sol(M)$ as a functor from the category of
Frobenius modules over $W_n(R)$ to the category of $p^{n}$-torsion sheaves on $\Spec(R)$. We will denote this functor by
$$\Sol: \Mod_{W_n(R)}^{\Frob} \rightarrow \Shv_{\mathet}( \Spec(R), \Z / p^{n} \Z)$$
and refer to it as the {\it solution sheaf functor}.

\begin{proposition}\label{propX74gen}
Let $R$ be a commutative $\F_p$-algebra and let $M$ be an injective object of the abelian category $\Mod_{ W_n(R)}^{\perf}$.
Then we have a short exact sequence 
$0 \rightarrow \Sol(M) \rightarrow \widetilde{M} \xrightarrow{\id - \varphi_{M} } \widetilde{M} \rightarrow 0$
in the category of abelian presheaves on $\CAlg_{R}^{\mathet}$.
\end{proposition}

The proof of Proposition \ref{propX74gen} is based on the following:

\begin{lemma}\label{lem.flatness}
Let $R$ be an $\F_p$-algebra and let $M$ be an injective object of the abelian category $\Mod_{ W_n(R)}^{\perf}$.
Then $M$ is free when regarded as a module over $\Z / p^{n} \Z$.
\end{lemma}

\begin{proof}
Choose a collection of elements $\{ x_i \}_{i \in I}$ of $M$, whose images form a basis for $M/pM$ as a vector space over $\F_p$.
Then the elements $x_{i}$ determine a map of $\Z / p^{n} \Z$-modules $f: \bigoplus_{i \in I} \Z / p^{n} \Z \rightarrow M$.
The map $f$ is surjective by virtue of Nakayama's lemma; we will complete the proof by showing that it is injective.
Assume otherwise: then there exists some nonzero element $\vec{c} \in \ker(f)$, which we can identify
with a collection of elements $\{ c_i \}_{i \in I}$ of $\Z / p^{n} \Z$ (almost all of which vanish). 
Let us assume that $\vec{c}$ has been chosen so that the ideal $( c_i )_{i \in I} \subseteq \Z / p^{n} \Z$ is
as large as possible. Since the elements $x_i$ have images in $M / p M$ which are linearly independent over $\F_p$, we must have
$(c_i)_{i \in I} \neq \Z / p^{n} \Z$. It follows that we can write $\vec{c} = p \vec{b}$ for some element $\vec{b} \in \bigoplus_{i \in I} \Z / p^{n} \Z$.
Then $p f( \vec{b} ) = f( \vec{c} ) = 0$, so there is a unique map of (perfect) Frobenius modules 
$g: R^{\perfection}[ F^{\pm 1} ] \rightarrow M$ satisfying $g(1) = f( \vec{b} )$. 

Let $W_{n}(R^{\perfection})[ F^{\pm 1} ]$ denote the perfection of the Frobenius module $W_n(R)[F]$. Note that multiplication
by $p^{n-1}$ induces a monomorphism $$R^{\perfection}[ F^{\pm 1} ] \rightarrow W_n(R^{\perfection})[ F^{\pm 1} ].$$ Invoking our assumption that
$M$ is injective, we conclude that $g$ factors as a composition
$R^{\perfection}[ F^{\pm 1} ] \xrightarrow{ p^{n-1} } W_n( R^{\perfection} )[ F^{\pm 1} ] \xrightarrow{h} M$.
Since $f$ is surjective, we can write $h(1) = f( \vec{a} )$ for some element $\vec{a} \in \bigoplus_{i \in I} \Z / p^{n} \Z$.
We then have
$$ f( \vec{b} - p^{n-1} \vec{a} ) = f( \vec{b} ) - p^{n-1} f( \vec{a} ) = g(1) - p^{n-1} h(1) = 0,$$
so that $\vec{b} - p^{n-1} \vec{a}$ belongs to $\ker(f)$. However, the ideal generated by the coefficients
of $\vec{b} - p^{n-1} \vec{a}$ is strictly larger than the ideal $(c_i)_{i \in I}$, which contradicts our choice of $\vec{c}$.
\end{proof}

\begin{proof}[Proof of Proposition \ref{propX74gen}]
Let $M$ be an injective object of the abelian category $\Mod_{ W_n(R)}^{\perf}$; we wish to show that the map
$$\id - \varphi_{M}: \widetilde{M} \rightarrow \widetilde{M}$$ is a epimorphism
of $\Z / p^{n} \Z$-valued presheaves. We proceed by induction on $n$. The case $n=0$ is vacuous and the case $n=1$ follows from Lemma \ref{lemma.propX74}, so we may
assume that $n \geq 2$. For each $k \geq 0$, let $M[ p^{k} ]$ denote the kernel of the map $p^{k}: M \rightarrow M$.
Write $n = i + j$, for some positive integers $i$ and $j$. Since $M$ is injective, Lemma \ref{lem.flatness} implies that we have a short
exact sequence of (perfect) Frobenius modules $$0 \rightarrow M[p^{i}] \rightarrow M \xrightarrow{ p^{i} } M[p^{j} ] \rightarrow 0.$$
Applying the construction $N \mapsto \widetilde{N}$, we obtain a commutative diagram of short exact sequences
$$ \xymatrix{ 0 \ar[r] & \widetilde{M[p^i]} \ar[r] \ar[d]^{ \id - \varphi_{ M[p^i] }} & \widetilde{M} \ar[d]^{ \id - \varphi_M } \ar[r]^-{p^i} & \widetilde{ M[p^j] } \ar[d]^{ \id - \varphi_{ M[p^{j} ] } } \ar[r] & 0 \\
0 \ar[r] & \widetilde{M[p^i]} \ar[r] & \widetilde{M} \ar[r]^-{p^i} & \widetilde{ M[p^j] } \ar[r] & 0 }$$
in the category of presheaves of abelian groups on $\CAlg_{R}^{\mathet}$. Since $M$ is an injective object of $\Mod_{ W_n(R)}^{\perf}$, the
submodules $M[p^{i}]$ and $M[p^{j}]$ are injective objects of $\Mod_{ W_i(R) }^{\perf}$ and $\Mod_{ W_j(R)}^{\perf}$, respectively
(this follows from Corollary \ref{corollary.describe-image}, since the inclusion functors $\Mod_{ W_i(R)}^{\perf} \hookrightarrow \Mod_{ W_n(R)}^{\perf} \hookleftarrow \Mod_{ W_j(R)}^{\perf}$
are exact). Applying our inductive hypothesis, we deduce that the outer vertical maps in the preceding diagram are epimorphisms, so that the middle vertical map is also an epimorphism (by the snake lemma).
\end{proof}

\subsection{Derived Solution Functors}\label{sec9sub4}

For every commutative $\F_p$-algebra $R$ and every integer $n \geq 0$, the solution functor $\Sol: \Mod_{W_n(R)}^{\perf} \rightarrow \Shv_{\mathet}( \Spec(R), \Z / p^{n} \Z)$
is left exact, and therefore admits right derived functors 
$$\Sol^{m}: \Mod_{ W_n(R)}^{\perf} \rightarrow \Shv_{\mathet}( \Spec(R), \Z / p^{n} \Z)$$
for $m \geq 0$. These functors are described by the following generalization of Proposition \ref{proposition.corX77}:

\begin{proposition}\label{proposition.corX77gen}
Let $R$ be a commutative $\F_p$-algebra and let $M$ be a perfect Frobenius module over $W_n(R)$. Then we have a canonical short exact sequence
$$ 0 \rightarrow \Sol^{0}(M) \rightarrow \widetilde{M} \xrightarrow{ \id - \varphi_{M}  } \widetilde{M} \rightarrow \Sol^{1}(M) \rightarrow 0,$$
and the sheaves $\Sol^{m}(M)$ vanish for $m \geq 2$.
\end{proposition}

\begin{proof}
Choose an injective resolution $0 \rightarrow M \rightarrow Q^{0} \rightarrow Q^{1} \rightarrow \cdots$ in the abelian category
$\Mod_{W_n(R)}^{\perf}$. Using Proposition \ref{propX74gen}, we obtain a short exact sequence of cochain complexes
$$0 \rightarrow \Sol( Q^{\ast} ) \rightarrow \widetilde{Q}^{\ast} \xrightarrow{ \id - \varphi } \widetilde{Q}^{\ast} \rightarrow 0.$$
Since the construction $N \mapsto \widetilde{N}$ is exact, the chain complex $\widetilde{Q}^{\ast}$ is an acyclic resolution of $\widetilde{M}$.
The associated long exact sequence now supplies the desired isomorphisms.
\end{proof}

\begin{remark}\label{remark.independence-n}
Let $R$ be a commutative $\F_p$-algebra and let $M$ be a perfect Frobenius module over $W_n(R)$. Then $M$ can also be regarded as a perfect
Frobenius module over $W_m(R)$ for $m \geq n$. The {\etale} sheaf $\Sol^{i}(M)$ depends {\it a priori} on whether we choose to regard $M$
as an object of the abelian category $\Mod_{ W_n(R)}^{\perf}$ (in which case $\Sol^{i}(M)$ is defined as sheaf of $\Z / p^{n} \Z$-modules on $\Spec(R)$),
or as an object of the larger abelian category $\Mod_{ W_m(R)}^{\perf}$ (in which case $\Sol^{i}(M)$ is defined as a sheaf of $\Z / p^{m} \Z$-modules on $\Spec(R)$). 
However, Proposition \ref{proposition.corX77gen} shows that the resulting {\etale} sheaves are canonically isomorphic.
\end{remark}

We will also need a generalization of Corollary \ref{corX75}:

\begin{proposition}\label{proposition.injective-preservation}
Let $R$ be a commutative $\F_p$-algebra and let $Q$ be an injective object of $\Mod_{ W_n(R)}^{\perf}$. Then $\Sol(Q)$ is an injective object
of $\Shv_{\mathet}( \Spec(R), \Z / p^{n} \Z)$. 
\end{proposition}

The proof of Proposition \ref{proposition.injective-preservation} will require the following:

\begin{lemma}\label{lemma.ext1}
Let $R$ be a commutative $\F_p$-algebra, let $M$ be a $W_n(R)$-module which is flat over $\Z / p^{n} \Z$, and let
$\sheafF \in \Shv_{\mathet}( \Spec(R), \F_p)$. Then the canonical map
$$\Ext^{1}_{ \underline{\F_p} }( \sheafF, \widetilde{ M[p] } ) \rightarrow
\Ext^{1}_{ \underline{\Z / p^{n} \Z} }( \sheafF, \widetilde{ M } )$$
is bijective.
\end{lemma}

\begin{proof}
Suppose we are given an extension $0 \rightarrow \widetilde{M} \rightarrow \sheafG \rightarrow \sheafF \rightarrow 0$
in the abelian category $\Shv_{\mathet}( \Spec(R), \Z / p^{n} \Z)$. We wish to show that there exists a commutative diagram of short exact sequences
$$ \xymatrix{ 0 \ar[r] & \widetilde{M[p] } \ar[r] \ar[d] & \sheafG' \ar[d] \ar[r] & \sheafF \ar[r] \ar[d]^{\id} & 0 \\
0 \ar[r] & \widetilde{M} \ar[r] & \sheafG \ar[r] & \sheafF \ar[r] & 0, }$$
where $\sheafG'$ is annihilated by $p$, and that the extension class of the upper exact sequence is uniquely determined. The uniqueness is clear:
note that if such a diagram exists, then it induces an isomorphism $\sheafG' \simeq \sheafG[p] = \ker(p: \sheafG \rightarrow \sheafG)$. To prove
existence, it will suffice to show that the composite map $\sheafG[p] \rightarrow \sheafG \rightarrow \sheafF$ is an epimorphism. To prove this,
we note that the commutative diagram
$$ \xymatrix{ 0 \ar[r] & \sheafG[p] \ar[r] \ar[d] & \sheafG \ar[r] \ar[d] & p \sheafG \ar[r] \ar[d] & 0 \\
0 \ar[r] & \sheafF \ar[r]^{\id} & \sheafF \ar[r] & 0 \ar[r] & 0 }$$
yields a long exact sequence
$$ \widetilde{M} \rightarrow \sheafG / \sheafG[p] \rightarrow \coker( \sheafG[p] \rightarrow \sheafF ) \rightarrow \coker( \sheafG \rightarrow \sheafF ),$$
where the last term vanishes (since the map $\sheafG \rightarrow \sheafF$ is an epimorphism). We are therefore reduced to showing
that the canonical map $\widetilde{M} \rightarrow \sheafG / \sheafG[p]$ is an epimorphism. Since $\sheafF$ is annihilated by $p$, the map
$p: \sheafG \rightarrow \sheafG$ induces a monomorphism $v: \sheafG/ \sheafG[p] \rightarrow \widetilde{M}$. It will therefore suffice to
show that the image of $v$ is contained in the image of the map $p: \widetilde{M} \rightarrow \widetilde{M}$. This follows
from Remark \ref{remark.technical}, since $\im(v)$ is annihilated by $p^{n-1}$.
\end{proof}

\begin{proof}[Proof of Proposition \ref{proposition.injective-preservation}]
Let $Q$ be an injective object of $\Mod_{ W_n(R)}^{\perf}$. We wish to show that $\Sol(Q)$ is an injective object of $\Shv_{\mathet}( \Spec(R), \Z / p^{n} \Z)$:
that is, that the group $\Ext^{1}_{ \underline{\Z / p^{n} \Z} }( \sheafF, \Sol(Q) )$ vanishes for every sheaf $\sheafF \in \Shv_{\mathet}( \Spec(R), \Z / p^{n} \Z)$.
Since the collection of those objects $\sheafF$ for which the group $\Ext^{1}_{ \underline{\Z / p^{n} \Z}}( \sheafF, \Sol(Q) )$ vanishes
is closed under extensions, we may assume without loss of generality that $\sheafF$ is annihilated by $p$.

By virtue of Proposition \ref{propX74gen}, we have short exact sequences of {\etale} sheaves
$$ 0 \rightarrow \Sol(Q[p] ) \rightarrow \widetilde{Q[p]} \xrightarrow{ \id - \varphi_{ Q[p] }} \widetilde{ Q[p] } \rightarrow 0$$
$$ 0 \rightarrow \Sol(Q) \rightarrow \widetilde{Q} \xrightarrow{ \id - \varphi_{Q} } \widetilde{Q} \rightarrow 0$$
which supply a commutative diagram of long exact sequences
$$ \xymatrix{ \Ext^{0}_{ \underline{\F_p}}( \sheafF, \widetilde{Q[p]} ) \ar[d]^{ \id - \varphi_{Q[p]} } \ar[r]^{\alpha} & 
\Ext^{0}_{ \underline{\Z / p^{n} \Z}}( \sheafF, \widetilde{Q} ) \ar[d]^{ \id - \varphi_{Q} } \\
\Ext^{0}_{ \underline{\F_p}}( \sheafF, \widetilde{Q[p]} ) \ar[d] \ar[r]^{\alpha} & 
\Ext^{0}_{\underline{\Z / p^{n} \Z}}( \sheafF, \widetilde{Q} ) \ar[d] \\
\Ext^{1}_{ \underline{\F_p}}( \sheafF, \Sol(Q[p]) ) \ar[r]^{\beta} \ar[d] & \Ext^{1}_{ \underline{\Z / p^{n}\Z}}( \sheafF, \Sol(Q) ) \ar[d] \\
\Ext^{1}_{ \underline{\F_p}}( \sheafF, \widetilde{Q[p]} ) \ar[d]^{ \id - \varphi_{Q[p]} } \ar[r]^{\gamma} & 
\Ext^{1}_{ \underline{\Z/p^{n}\Z}}( \sheafF, \widetilde{Q} ) \ar[d]^{ \id - \varphi_{Q} } \\
\Ext^{1}_{ \underline{\F_p}}( \sheafF, \widetilde{Q[p]} ) \ar[r]^{\gamma} & 
\Ext^{1}_{ \underline{\Z/p^{n}\Z}}( \sheafF, \widetilde{Q} ). }$$
The map $\alpha$ is obviously an isomorphism, and $\gamma$ is an isomorphism by virtue of Lemma \ref{lemma.ext1}.
It follows that $\beta$ is also an isomorphism. We are therefore reduced to proving that the group $\Ext^{1}_{ \underline{\F_p} }( \sheafF, \Sol(Q[p]) )$
vanishes. In fact, we claim that $\Sol( Q[p] )$ is an injective object of the abelian category $\Shv_{\mathet}( \Spec(R), \F_p)$: this is a special case
of Corollary \ref{corX75}, since $Q[p]$ is an injective object of the abelian category $\Mod_{ R }^{\perf}$.
\end{proof}

\subsection{Algebraic Frobenius Modules over $W_n(R)$}\label{sec9sub5}

Let $R$ be a commutative $\F_p$-algebra and let $M$ be a perfect Frobenius module over $W_n(R)$. We let $M[p]$ and $M/pM$ denote the kernel and cokernel
of the map $p: M \rightarrow M$. Then $M[p]$ and $M/pM$ are perfect Frobenius modules over $W_n(R)$ which are annihilated by $p$, and can therefore
be identified with perfect Frobenius modules over $R$ (Corollary \ref{corollary.describe-image}).

\begin{proposition}\label{proposition.charindhol}
Let $R$ be a commutative $\F_p$-algebra and let $M$ be a perfect Frobenius module over $W_n(R)$. The following conditions are equivalent:
\begin{itemize}
\item[$(1)$] The quotient $M/pM \in \Mod_{R}^{\perf}$ is algebraic, in the sense of Definition \ref{defhol1}.
\item[$(2)$] The submodule $M[p] \in \Mod_{R}^{\perf}$ is algebraic, in the sense of Definition \ref{defhol1}.
\item[$(3)$] Every element $x \in M$ satisfies an equation of the form
$$\varphi_{M}^{k}(x) + a_1 \varphi_{M}^{k-1}(x) + \cdots + a_k x = 0$$ for some coefficients $a_i \in W_n(R)$. 
\end{itemize}
\end{proposition}

\begin{definition}\label{defholgen}
Let $R$ be a commutative $\F_p$-algebra and let $M$ be a perfect Frobenius module over $W_n(R)$. We will say that
$M$ is {\it algebraic} if it satisfies the equivalent conditions of Proposition \ref{proposition.charindhol}.
We let $\Mod_{ W_n(R)}^{\alg}$ denote the full subcatgory of $\Mod_{ W_n(R)}^{\perf}$ spanned by the algebraic
Frobenius modules over $W_n(R)$.
\end{definition}

\begin{remark}
In the situation of Definition \ref{defholgen}, an object $M \in \Mod_{ W_n(R)}^{\Frob}$ is algebraic if and only if
it is algebraic when viewed as a Frobenius module over $W_m(R)$, for any $m \geq n$.
\end{remark}

\begin{remark}
In the situation of Definition \ref{defholgen}, let $M$ be a perfect Frobenius module over $W_n(R)$ which is annihilated by $p$.
Then $M$ can be regarded as a perfect Frobenius module over $R$ (Corollary \ref{corollary.describe-image}). Moreover,
$M$ is algebraic in the sense of Definition \ref{defholgen} if and only if it is algebraic in the sense of Definition \ref{defhol1}.
\end{remark}

\begin{proof}[Proof of Proposition \ref{proposition.charindhol}]
The implication $(3) \Rightarrow (2)$ is obvious. We now show that $(2) \Rightarrow (1)$. Assume that $M$ is a perfect
Frobenius module over $W_n(R)$ and that the $p$-torsion submodule $M[p]$ is algebraic (as a perfect Frobenius module over $R$).
For each integer $i \geq 0$, we have a short exact sequence
$$ 0 \rightarrow (M[p] \cap p^{i} M) / (M[p] \cap p^{i+1} M) \rightarrow p^{i} M / p^{i+1} M \xrightarrow{p} p^{i+1} M / p^{i+2} M \rightarrow 0$$
Since the collection of algebraic objects of $\Mod_{R}^{\perf}$ is closed under the formation of subobjects and quotient objects (Proposition \ref{propX53}),
condition $(2)$ guarantees that each $(M[p] \cap p^{i} M) / (M[p] \cap p^{i+1} M)$ is algebraic. Since the collection of algebraic objects of
$\Mod_{R}^{\perf}$ is closed under extensions (Proposition \ref{propX53}), it follows by descending induction on $i$ that each $p^{i} M / p^{i+1} M$
is algebraic. Taking $i = 0$, we deduce that $(1)$ is satisfied.

We now complete the proof by showing that $(1)$ implies $(3)$. We proceed by induction on $n$, the case $n = 0$ being trivial. 
Assume that $n > 0$ and let $x$ be an element of $M$ having image $\overline{x} \in M/pM$. Condition $(1)$ guarantees
that we can find an element $\overline{\mu} = F^{m} + \overline{a}_1 F^{m-1} + \cdots + \overline{a}_{m-1} F + \overline{a}_{m} \in R[F]$
such that $\overline{\mu} ( \overline{x} ) = 0$. Lift $\overline{\mu}$ to an element $\mu = F^{m} + a_1 F^{m-1} + \cdots + a_{m} \in W_n(R)[F]$,
so that $\mu(x) \in pM$. Note that $pM / p^2M$ is a quotient of $M / pM$, and is therefore algebraic by virtue of
Proposition \ref{propX53}. The Frobenius module $pM$ is annihilated by $p^{n-1}$, and can therefore be regarded
as a perfect Frobenius module over $W_{n-1}(R)$ by virtue of Corollary \ref{corollary.describe-image}. Applying our inductive hypothesis,
we deduce that there exists an expression $\nu = F^{m'} + b_1 F^{m' - 1} + \cdots + b_{m'-1} F + b_{m'} \in W_n(R)[F]$ such that
$\nu( \mu(x) ) = 0$, so that $x$ is annihilated by $\nu \mu \in W_n(R)[F]$.
\end{proof}

We have the following generalization of Proposition \ref{propX53}:

\begin{proposition}\label{propX53gen}
Let $R$ be a commutative $\F_p$-algebra and let $n \geq 0$. Then $\Mod_{W_n(R)}^{\alg}$ is a localizing subcategory of $\Mod_{W_n(R)}^{\perf}$. That is:
\begin{itemize}
\item[$(a)$] Given a short exact sequence $0 \rightarrow M' \rightarrow M \rightarrow M'' \rightarrow 0$ of {perfect} Frobenius modules over $R$,
$M$ is algebraic if and only if $M'$ and $M''$ are algebraic.

\item[$(b)$] The collection of algebraic Frobenius modules over $W_n(R)$ is closed under (possibly infinite) direct sums.
\end{itemize}
\end{proposition}

\begin{proof}
We will prove $(a)$; assertion $(b)$ is immediate from the definitions. Suppose we are given an exact sequence
$0 \rightarrow M' \rightarrow M \rightarrow M'' \rightarrow 0$ of perfect Frobenius modules over $W_n(R)$. Then
we also have an exact sequence $M' / p M' \rightarrow M / pM \rightarrow M'' / pM'' \rightarrow 0$.
If $M' / pM'$ is algebraic, then Proposition \ref{propX53} implies that $M / pM$ is algebraic
if and only if $M'' / pM''$ is algebraic. Using characterization $(1)$ of Proposition \ref{proposition.charindhol}, we
conclude that if $M'$ is algebraic, then $M$ is algebraic if and only if $M''$ is algebraic.
Applying the same argument to the exact sequence $0 \rightarrow M'[p] \rightarrow M[p] \rightarrow M''[p]$
(and using characterization $(2)$ of Proposition \ref{proposition.charindhol}), we deduce that if
$M''$ is algebraic, then $M$ is algebraic if and only if $M'$ is algebraic.
\end{proof}

\begin{proposition}\label{proposition.derived-vanishing}
Let $R$ be a commutative $\F_p$-algebra, let $n \geq 0$, and let $M$ be an algebraic Frobenius module over $W_n(R)$. Then
the {\etale} sheaves $\Sol^{i}(M) \in \Shv_{ \mathet}( \Spec(R), \Z / p^{n} \Z)$ vanish for $i \neq 0$. 
\end{proposition}

\begin{proof}
We prove the following assertion by induction on $m$:
\begin{itemize}
\item[$(\ast_m)$] Let $M$ be an algebraic Frobenius module over $W_n(R)$ which is annihilated by $p^{m}$. Then
$\Sol^{i}(M) = 0$.
\end{itemize}
Note that assertion $(\ast_0)$ is trivial, and assertion $(\ast_n)$ implies Proposition \ref{proposition.derived-vanishing}. 
It will therefore suffice to show that $(\ast_m)$ implies $(\ast_{m+1})$. Note that if $M$ is an algebraic Frobenius module which is annihilated by $p^{m+1}$, then
the the short exact sequence $0 \rightarrow M[p] \rightarrow M \rightarrow pM \rightarrow 0$ yields an exact sequence of sheaves
$\Sol^{i}( M[p] ) \rightarrow \Sol^{i}( M) \rightarrow \Sol^i( pM )$. Here $pM$ and $M[p]$ are also algebraic (Proposition \ref{propX53gen}), and
$pM$ is annihilated by $p^{m}$. Our inductive hypothesis then guarantees that $\Sol^{i}(pM) = 0$. To complete the proof, it will suffice to show that
$\Sol^{i}(M[p]) = 0$. To prove this, we can replace $W_n(R)$ by $R$ (Remark \ref{remark.independence-n}), in which case the desired result follows
from Corollary \ref{corollary.derived-vanishing}.
\end{proof}

\subsection{The Riemann-Hilbert Correspondence for $\Z / p^{n} \Z$-Sheaves}\label{sec9sub6}

We can now formulate the main result of this section:

\begin{theorem}[Riemann-Hilbert Correspondence]\label{theorem.generalizedRH}
Let $R$ be a commutative $\F_p$-algebra and let $n \geq 0$. Then the functor $\Sol: \Mod_{ W_n(R)}^{\perf} \rightarrow \Shv_{\mathet}( \Spec(R), \Z / p^{n} \Z )$
induces an equivalence of categories $\Mod_{ W_n(R) }^{\alg} \rightarrow \Shv_{\mathet}( \Spec(R), \Z / p^{n} \Z )$.
\end{theorem}

We will deduce Theorem \ref{theorem.generalizedRH} from the following comparison result:

\begin{proposition}\label{prop:ext-comparison}
Let $R$ be a commutative $\F_p$-algebra and let $M$ and $N$ be perfect Frobenius modules over $W_n(R)$. Assume
that $M$ is algebraic and that $\Sol^{1}(N) \simeq 0$. Then the canonical map
$$ \Ext^{i}_{ W_n(R)[F] }( M, N) \rightarrow \Ext^{i}_{ \underline{\Z / p^{n} \Z} }( \Sol(M), \Sol(N) )$$
is an isomorphism for $i \geq 0$.
\end{proposition}

\begin{proof}[Proof of Theorem \ref{theorem.generalizedRH} from Proposition \ref{prop:ext-comparison}]
We first claim that the composite functor
$$\Mod_{ W_n(R) }^{\alg} \hookrightarrow  \Mod_{ W_n(R)}^{\perf} \xrightarrow{ \Sol } \Shv_{\mathet}( \Spec(R), \Z / p^{n} \Z )$$
is fully faithful. Let $M$ and $N$ be algebraic Frobenius modules over $W_n(R)$; we wish to show that
the canonical map
$$ \Hom_{ W_n(R)[F] }( M, N) \rightarrow \Hom_{ \underline{\F_p} }( \Sol(M), \Sol(N) )$$
is an isomorphism. This is a special case of Proposition \ref{prop:ext-comparison}, since $\Sol^{1}(N) \simeq 0$
by virtue of Proposition \ref{proposition.derived-vanishing}.

Let $\calC \subseteq  \Shv_{\mathet}( \Spec(R), \Z / p^{n} \Z )$ denote the full subcategory spanned by those sheaves of the form
$\Sol(M)$, where $M$ is an algebraic Frobenius module over $W_n(R)$. To complete the proof of Theorem \ref{theorem.generalizedRH}, it will
suffice to show that every object of $\Shv_{\mathet}( \Spec(R), \Z / p^{n} \Z)$ belongs to $\calC$. Note that Theorem \ref{theoX50} guarantees
$\calC$ contains every sheaf of $\Z / p\Z$-modules on $\Spec(R)$. We will complete the proof by showing that $\calC$ is closed under the formation of extensions.
Suppose we are given a short exact sequence of {\etale} sheaves $$0 \rightarrow \sheafF' \rightarrow \sheafF \rightarrow \sheafF'' \rightarrow 0,$$ where
$\sheafF'$ and $\sheafF''$ belong to $\calC$; we wish to show that $\sheafF$ also belongs to $\calC$. Without loss of generality, we may assume
that $\sheafF' = \Sol(M')$ and $\sheafF'' = \Sol(M'')$ for some algebraic Frobenius modules $M'$ and $M''$ over $W_n(R)$. In this case,
the preceding exact sequence is classified by an element $\eta \in \Ext^{1}_{ \underline{ \Z / p^{n} \Z} }( \Sol(M''), \Sol(M') )$.
Invoking Proposition \ref{prop:ext-comparison} again, we deduce that $\eta$ can be lifted (uniquely) to an element
$\overline{\eta} \in \Ext^{1}_{ W_n(R)[F] }(M'', M')$, which classifies a short exact sequence of Frobenius modules
$0 \rightarrow M' \rightarrow M \rightarrow M'' \rightarrow 0$. Proposition \ref{propX53gen} guarantees that $M$ is algebraic, so that
$\sheafF \simeq \Sol(M)$ also belongs to the category $\calC$.
\end{proof}

We now turn to the proof of Proposition \ref{prop:ext-comparison}. We begin with some special cases.

\begin{lemma}\label{lemma.hom-comparison}
Let $R$ be an $\F_p$-algebra and let $M$ be an algebraic Frobenius module over $R$,
and let $N$ be any object of $\Mod_{ W_n(R)}^{\perf}$. Then the canonical map
$$ \theta: \Hom_{W_n(R)[F]}( M, N) \rightarrow \Hom_{ \underline{ \Z / p^{n} \Z} }( \Sol(M), \Sol(N) )$$
is an isomorphism.
\end{lemma}

\begin{proof}
Since the functor $\Sol$ is left exact, we have an isomorphism $\Sol(N)[p] \simeq \Sol(N[p] )$. Since $M$ and $\Sol(M)$ are annihilated by $p$, we can
identify $\theta$ with the canonical map 
$$ \Hom_{ R[F] }( M, N[p] ) \rightarrow \Hom_{ \underline{\F_p} }( \Sol(M), \Sol(N[p] ) ).$$
We may therefore replace $N$ by $N[p]$ and thereby reduce to the case $n=1$.
Using Theorem \ref{theoX50}, we can choose an isomorphism $M \simeq \RH(\sheafF)$ for some object $\sheafF \in \Shv_{\mathet}( \Spec(R), \F_p )$. In this case,
we $\theta$ has a left inverse, given by the map
$$ \Hom_{ \underline{\F_p} }( \Sol( \RH(\sheafF)), \Sol(N) ) \rightarrow  \Hom_{ \underline{\F_p} }( \sheafF, \Sol(N) )$$
given by precomposition with the unit map $u: \sheafF \rightarrow \Sol( \RH( \sheafF ) )$. This map is an isomorphism by virtue of Proposition \ref{prop75}.
\end{proof}

\begin{lemma}\label{lemma.ext-comparison}
Let $R$ be an $\F_p$-algebra and let $M$ be an algebraic Frobenius module over $R$. 
Let $N$ be any perfect Frobenius module over $W_n(R)$. If $\Sol^1(N) \simeq 0$, then
the canonical map
$$ \Ext^{i}_{ W_n(R)[F] }( M, N) \rightarrow \Ext^{i}_{\underline{\Z / p^{n} \Z} }( \Sol(M), \Sol(N) )$$
is an isomorphism for $i \geq 0$.
\end{lemma}

\begin{proof}
Choose an injective resolution
$0 \rightarrow N \rightarrow Q^0 \rightarrow Q^1 \rightarrow \cdots$
in the abelian category $\Mod_{ W_n(R)}^{\perf}$. Our hypothesis that $\Sol^1(N)$ vanishes guarantees
that the complex $0 \rightarrow \Sol(N) \rightarrow \Sol(Q^0) \rightarrow \Sol(Q^1) \rightarrow \cdots$ is exact in the abelian category $\Shv_{\mathet}( \Spec(R), \Z / p^n \Z)$ (Proposition \ref{proposition.corX77gen}). Moreover,
each $\Sol(Q^i)$ is an injective object of $\Shv_{\mathet}( \Spec(R), \Z / p^n \Z)$ (Proposition \ref{proposition.injective-preservation}).
It will therefore suffice to show that the canonical map
$$\Hom_{ W_n(R)[F] }( M, Q^{\ast} ) \rightarrow \Hom_{ \underline{\Z / p^n \Z} }( \Sol(M), \Sol(Q^{\ast}) )$$
is a quasi-isomorphism of chain complexes. In fact, this map is an isomorphism of chain complexes: this is a special case of
Lemma \ref{lemma.hom-comparison}.
\end{proof}

\begin{proof}[Proof of Proposition \ref{prop:ext-comparison}]
Let $N$ be a perfect Frobenius module over $W_n(R)$, and suppose that $\Sol^1(N) \simeq 0$. Let us say that
an object $M \in \Mod_{ W_n(R)}^{\alg}$ is {\it good} if the canonical map
$$\rho_i: \Ext^{i}_{ W_n(R)[F] }( M, N) \rightarrow \Ext^{i}_{ \underline{\Z / p^n \Z} }( \Sol(M), \Sol(N) )$$
is an isomorphism for $i \geq 0$. It follows from Lemma \ref{lemma.ext-comparison} that if $M \in \Mod_{ W_n(R)}^{\alg}$ is annihilated by
$p$, then $M$ is good. We wish to show that every object of $\Mod_{ W_n(R)}^{\alg}$ is good. For this, it will suffice to establish the following:
\begin{itemize}
\item[$(\ast)$] Let $0 \rightarrow M' \rightarrow M \rightarrow M'' \rightarrow 0$ be a short exact sequence of algebraic Frobenius modules over $W_n(R)$.
If $M'$ and $M''$ are good, then $M$ is also good.
\end{itemize}
To prove $(\ast)$, we note that the vanishing of $\Sol^{1}(M')$ (Proposition \ref{proposition.derived-vanishing}) guarantees
the exactness of the sequence $0 \rightarrow \Sol(M') \rightarrow \Sol(M) \rightarrow \Sol(M'') \rightarrow 0$.
It follows that each $\rho_i$ fits into a commutative diagram of exact sequences
$$ \xymatrix{  \Ext^{i-1}_{ W_n(R)[F] }( M', N) \ar[r]^-{\rho'_{i-1} } \ar[d] & \Ext^{i-1}_{ \underline{\Z / p^{n} \Z}}( \Sol(M'), \Sol(N) ) \ar[d] \\
\Ext^{i}_{ W_n(R)[F] }( M'', N) \ar[r]^-{\rho''_{i} } \ar[d] & \Ext^{i}_{ \underline{\Z / p^{n} \Z} }( \Sol(M''), \Sol(N) ) \ar[d] \\
\Ext^{i}_{ W_n(R)[F] }( M, N) \ar[r]^-{\rho_{i} } \ar[d] & \Ext^{i}_{ \underline{\Z / p^{n} \Z} }( \Sol(M), \Sol(N) ) \ar[d] \\
\Ext^{i}_{ W_n(R)[F] }( M', N) \ar[r]^-{\rho'_{i} } \ar[d] & \Ext^{i}_{ \underline{\Z / p^{n} \Z} }( \Sol(M'), \Sol(N) ) \ar[d] \\
\Ext^{i+1}_{ W_n(R)[F] }( M'', N) \ar[r]^-{\rho''_{i+1} }& \Ext^{i+1}_{\underline{\Z / p^{n} \Z} }( \Sol(M''), \Sol(N) ). }$$
Our hypothesis that $M'$ and $M''$ are good guarantees that the maps $\rho'_{i-1}$, $\rho''_{i}$, $\rho'_{i}$, and $\rho''_{i+1}$ are isomorphisms,
so that $\rho_i$ is also an isomorphism.
\end{proof}

\newpage \section{Globalization}\label{section.global}
\setcounter{subsection}{0}
\setcounter{theorem}{0}

For any commutative $\F_p$-algebra $R$, the Riemann-Hilbert correspondence of Theorem \ref{maintheoXXX} supplies a description of the category of $p$-torsion {\etale} sheaves
on the affine $\F_p$-scheme $X = \Spec(R)$ in terms of Frobenius modules over $R$. Our goal in this section is to extend the Riemann-Hilbert correspondence to the case of an arbitrary $\F_p$-scheme $X$.
We begin in \S \ref{sec12sub1} by introducing the notion of a {\it Frobenius sheaf} on $X$: that is, a quasi-coherent sheaf $\qE$ on $X$ equipped with a Frobenius-semilinear endomorphism $\varphi_{\qE}$ (Definition \ref{definition.frobenius-sheaf}). The collection of Frobenius sheaves on $X$ forms a category, which we will denote by $\QCoh_{X}^{\Frob}$. 
In \S \ref{sec12sub2} we construct an equivalence $\RH$ from the category $\Shv_{\mathet}(X; \F_p)$ of $p$-torsion {\etale} sheaves on $X$ to a full subcategory
$\QCoh_{X}^{\alg} \subseteq \QCoh_{X}^{\Frob}$ (Theorem \ref{globalRH} and Notation \ref{notation.RH}). This is essentially a formal exercise (given the earlier results of this paper): roughly speaking,
the Riemann-Hilbert functor $\RH$ is constructed by amalgamating the equivalences $\Shv_{\mathet}(U; \F_p) \simeq \QCoh_{U}^{\alg}$ where $U$ ranges over affine open subsets of $X$.
Consequently, any {\em local} question about the the functor $\RH$ can be reduced to the affine case: we use this observation in \S \ref{sec12sub3} to argue that the Riemann-Hilbert
correspondence is compatible with the formation of pullbacks along an arbitrary morphism of $\F_p$-schemes $f: X \rightarrow Y$ (Variant \ref{variant.compatibility}).
However, we do encounter a genuinely new {\em global} phenomenon: the Riemann-Hilbert correspondence is also compatible with direct images (and higher direct images)
along a morphism $f: X \rightarrow Y$ which is proper and of finite presentation (Theorem \ref{properRH1}). We prove this in \S \ref{sec12sub5} using a global characterization for
holonomic Frobenius sheaves (Theorem \ref{silphil}), which we establish in \S \ref{sec12sub4}. In \S \ref{sec12sub6}, we apply these ideas to give a proof of the proper base change theorem
in {\etale} cohomology (in the special case of $p$-torsion sheaves on $\F_p$-schemes; see Corollary \ref{corollary.basechange}).

\begin{remark}
Throughout this section, we confine our study of Frobenius sheaves on $X$ to the case where $X$ is an $\F_p$-scheme. However, the results of this section can be extended to more general geometric objects, such as algebraic spaces over $\F_p$. Similarly, the results can also be extended to have ``coefficients in $\Z/p^n$'' in the sense of \S \ref{section.moretorsion}. We leave such extensions to the reader. 
\end{remark}

\subsection{Frobenius Sheaves on a Scheme}\label{sec12sub1}

We begin by introducing some terminology.

\begin{notation}
For any scheme $X$, we let $\QCoh_X$ denote the category of quasi-coherent sheaves on $X$. If $X$ is an $\F_p$-scheme, we let $\varphi_X: X \rightarrow X$
denote the absolute Frobenius morphism from $X$ to itself.
\end{notation}

\begin{definition}\label{definition.frobenius-sheaf}
Let $X$ be an $\F_p$-scheme. A {\it Frobenius sheaf on $X$} is a pair $( \qE, \varphi_{\qE} )$, where
$\qE$ is a quasi-coherent sheaf on $X$ and $\varphi_{\qE}: \qE \rightarrow \varphi_{X \ast} \qE$ is a morphism of quasi-coherent sheaves.
If $(\qE, \varphi_{\qE})$ and $(\qF, \varphi_{\qF} )$ are Frobenius sheaves on $X$, then we will say that a $\calO_{X}$-module map $f: \qE \rightarrow \qF$ is a {\it morphism of Frobenius sheaves}
if the diagram
$$ \xymatrix{ \qE \ar[r]^{f} \ar[d]^-{ \varphi_{\qE} } & \qF \ar[d]^{ \varphi_{\qF} } \\
\varphi_{X \ast} \qE \ar[r]^-{ \varphi_{X \ast}(f) } & \varphi_{X \ast} \qF }$$
commutes. We let $\QCoh_X^{\Frob}$ denote the category whose objects are Frobenius sheaves on $X$ and whose morphisms are morphisms of Frobenius sheaves.
\end{definition}

We will generally abuse terminology by identifying a Frobenius sheaf $(\qE, \varphi_{\qE} )$ with its underlying quasi-coherent sheaf $\qE$, and simply referring to 
$\qE$ as a {\it Frobenius sheaf on $X$}.

\begin{example}\label{example.semble}
Let $X = \Spec(R)$ be an affine $\F_p$-scheme. Then the the global sections functor $\qE \mapsto \Gamma(X, \qE)$ induces an equivalence of
categories $\QCoh_{X}^{\Frob} \rightarrow \Mod_{R}^{\Frob}$.
\end{example}

\begin{remark}
Let $X$ be an $\F_p$-scheme. Then the category $\QCoh_{X}^{\Frob}$ is abelian. Moreover, the forgetful functor $\QCoh_{X}^{\Frob} \rightarrow \QCoh_X$ is exact.
\end{remark}

\begin{variant}\label{FStabSub1}
Let $X$ be an $\F_p$-scheme. Using the adjointness of the functors $\varphi_{X \ast}$ and $\varphi_{X}^{\ast}$, we can obtain a slightly different
description of the category $\QCoh_{X}^{\Frob}$ of Frobenius sheaves:
\begin{itemize}
\item The objects of $\QCoh_{X}^{\Frob}$ can be identified with pairs $( \qE, \psi_{\qE} )$, where $\qE$ is a quasi-coherent sheaf on $X$
and $\psi_{\qE}: \varphi_{X}^{\ast} \qE \rightarrow \qE$ is a morphism of quasi-coherent sheaves.

\item A morphism from $(\qE, \psi_{\qE})$ to $(\qF, \psi_{\qF} )$ in the category $\QCoh_{X}^{\Frob}$ is a morphism of quasi-coherent
sheaves $f: \qE \rightarrow \qF$ for which the diagram
$$ \xymatrix{ \varphi_{X}^{\ast} \qE \ar[r]^-{ \varphi_{X}^{\ast}(f) }  \ar[d]^{ \psi_{\qE} } & \varphi_{X}^{\ast} \qF \ar[d]^{ \psi_{\qF} } \\
\qE \ar[r]^{f} & \qF}$$
commutes.
\end{itemize}
\end{variant}

In what follows, we will regard quasi-coherent sheaves on a scheme $X$ as sheaves on the {\etale} site of $X$ (see Example \ref{exX70}). Given a
quasi-coherent sheaf $\qE \in \QCoh_X$ and an {\etale} morphism $f: U \rightarrow X$, we let $\qE(U)$ denote the abelian group of global sections
$\Gamma( U, f^{\ast} \qE )$. Note that if $U = \Spec(R)$ is affine, then $\qE(U)$ has the structure of an $R$-module; if $X$ is an $\F_p$-scheme
and $\qE$ is a Frobenius sheaf, then $\qE(U)$ inherits the structure of a Frobenius module over $R$.

\begin{proposition}\label{proposition.local-properties}
Let $X$ be an $\F_p$-scheme and let $\qE$ be a Frobenius sheaf on $X$. The following conditions are equivalent:
\begin{itemize}
\item[$(1)$] For every {\etale} morphism $f: U \rightarrow X$ where $U \simeq \Spec(R)$ is affine, the
group of sections $\qE(U)$ is perfect (respectively algebraic, holonomic) when regarded as a Frobenius module over $R$.

\item[$(2)$] For every open subset $U \subseteq X$ where $U \simeq \Spec(R)$ is affine, the group of sections
$\qE$ is perfect (respectively algebraic, holonomic) when regarded as a Frobenius module over $R$.

\item[$(3)$] There exists an {\etale} covering $\{ U_{\alpha} \rightarrow X \}$ where each $U_{\alpha} \simeq \Spec(R_{\alpha} )$
is affine, and each $\qE( U_{\alpha} )$ is perfect (respectively algebraic, holonomic) when regarded as a Frobenius module over $R_{\alpha}$.
\end{itemize}
\end{proposition}

\begin{proof}
The implications $(1) \Rightarrow (2) \Rightarrow (3)$ are obvious. The implication $(3) \Rightarrow (1)$ follows from 
Corollary \ref{corX14} (respectively Lemma \ref{ulroc}, Corollary \ref{corlocal}).
\end{proof}

\begin{definition}
Let $X$ be an $\F_p$-scheme and let $\qE$ be a Frobenius sheaf on $X$. We will say that $\qE$ is {\it perfect} (respectively
{\it algebraic}, {\it holonomic}) if it satisfies the equivalent conditions of Proposition \ref{proposition.local-properties}.
We let $\QCoh_{X}^{\perf}$ (respectively $\QCoh_{X}^{\alg}$, $\QCoh_{X}^{\hol}$) denote the full subcategory of $\QCoh_X^{\Frob}$
spanned by those Frobenius sheaves which are perfect (respectively algebraic, holonomic), so that we have inclusions
$$ \QCoh_{X}^{\hol} \subseteq \QCoh_{X}^{\alg} \subseteq \QCoh_{X}^{\perf} \subseteq \QCoh_{X}^{\Frob}.$$
\end{definition}

\begin{example}
Let $X = \Spec(R)$ be an affine $\F_p$-scheme. Then a Frobenius sheaf $\qE \in \QCoh_{X}^{\Frob}$ is
perfect (respectively algebraic, holonomic) if and only if $\Gamma(X, \qE)$ is perfect (respectively
algebraic, holonomic) when regarded as a Frobenius module over $R$.
\end{example}

\begin{remark}
Let $X$ be an $\F_p$-scheme. Then the subcategories 
$$\QCoh_{X}^{\hol} \subseteq \QCoh_{X}^{\alg} \subseteq \QCoh_{X}^{\perf} \subseteq \QCoh_{X}^{\Frob}$$
are closed under the formation of kernels, cokernels, and extensions. In particular, they are abelian subcategories of $\QCoh_{X}^{\Frob}$.
Moreover, the subcategories $\QCoh_{X}^{\alg} \subseteq \QCoh_{X}^{\perf} \subseteq \QCoh_{X}$ are closed under (possibly infinite) direct sums
(and therefore under all colimits). To prove these assertions, we can work locally and thereby reduce to the case where $X$ is affine: in this case,
the desired results follow from Remark \ref{remark.perfectextension}, Proposition \ref{propX53}, and Corollary \ref{corX30}.
\end{remark}




\begin{remark}[Descent]
Let $X$ be an $\F_p$-scheme. Then the theory of Frobenius sheaves satisfies effective descent with respect to the {\etale} topology on $X$, and is therefore
determined (in some sense) by its behavior when $X$ is affine. In other words, the construction $(U \rightarrow X) \mapsto \QCoh_{U}^{\Frob}$ determines a stack on the {\etale} site of $X$. The same remark applies to 
the subcategories $\QCoh^{\hol}_{U}$,  $\QCoh^{\alg}_{U}$, and $\QCoh^{\perf}_{U}$ (by virtue of Proposition \ref{proposition.local-properties}).
\end{remark}

\begin{remark}[Perfection]
Let $X$ be an $\F_p$-scheme and let $\qE$ be a Frobenius sheaf on $X$. We let $\qE^{\perfection}$ denote the direct limit of the diagram
$$ \qE \xrightarrow{ \varphi_{\qE} } \varphi_{X \ast} \qE \xrightarrow{ \varphi_{X \ast}(\varphi_{\qE})} \varphi_{X \ast}^{2} \qE \rightarrow \cdots$$
Then we have a canonical isomorphism $\qE^{\perfection} \simeq \varphi_{X \ast} \qE^{\perfection}$ which endows $\qE^{\perfection}$ with
the structure of a perfect Frobenius sheaf on $X$. Moreover, the canonical map $u: \qE \rightarrow \qE^{\perfection}$ is a morphism of Frobenius sheaves
with the following universal property: for any perfect Frobenius sheaf $\qF$ on $X$, composition with $u$ induces a bijection
$$ \Hom_{ \QCoh_{X}^{\perf} }( \qE^{\perfection}, \qF) \rightarrow \Hom_{ \QCoh_{X}^{\Frob} }( \qE, \qF ).$$
In other words, we can regard the construction $\qE \mapsto \qE^{\perfection}$ as a left adjoint to the inclusion functor
$\QCoh_X^{\perf} \subseteq \QCoh_{X}^{\Frob}$. Note that the the perfection functor $\qE \mapsto \qE^{\perfection}$ is exact (since filtered direct limits in $\QCoh_X$ are exact; see  see \cite[Tag 077K]{Stacks}).
\end{remark}

\subsection{The Riemann-Hilbert Correspondence}\label{sec12sub2}

We now extend the Riemann-Hilbert correspondence of Theorem \ref{maintheoXXX} to the case of a general $\F_p$-scheme.

\begin{notation}
For any scheme $X$, we let $\Shv_{\mathet}( X, \F_p )$ denote the abelian category of $p$-torsion sheaves on the {\etale} site of $X$. If $X$ is an $\F_p$-scheme, then
we have a forgetful functor $\QCoh_{X} \rightarrow \Shv_{\mathet}( X, \F_p )$ which carries a sheaf of $\calO_{X}$-modules to its underlying sheaf of $\F_p$-modules.
We will generally abuse notation by not distinguishing between a quasi-coherent sheaf $\qE$ and its image under this functor. Moreover, we will also abuse notation
by identify $\qE$ with its direct image $\varphi_{X \ast} \qE$ under the absolute Frobenius map $\varphi_{X}: X \rightarrow X$: note that there is a canonical isomorphism $\qE \simeq \varphi_{X \ast} \qE$
in the category $\Shv_{\mathet}( X, \F_p )$, though this isomorphism is not $\calO_{X}$-linear.
\end{notation}

\begin{construction}[The Solution Functor]\label{construction.solution2}
Let $X$ be an $\F_p$-scheme and let $( \qE, \varphi_{ \qE} )$ be a Frobenius sheaf on $X$. We let
$\Sol(\qE)$ denote the kernel of the map $( \id - \varphi_{\qE} ): \qE \rightarrow \qE$, formed in the abelian category $\Shv_{\mathet}( X, \F_p )$. The construction
$( \qE, \varphi_{\qE}) \mapsto \Sol( \qE )$ determines a functor $\Sol:\QCoh_X^{\Frob} \to \Shv_{\mathet}(X,\F_p)$, which we will refer to as {\it the solution functor}. 
\end{construction}

\begin{remark}
In the special case where $X = \Spec(R)$ is affine, the solution functor of Construction \ref{construction.solution2} agrees with the solution functor of Construction \ref{construction.solsheaf}. More precisely,
for any Frobenius sheaf $\qE$ on $X$, we have a canonical isomorphism $\Sol( \qE ) \simeq \Sol( \Gamma(X, \qE) )$ in the category $\Shv_{\mathet}( X, \F_p)$.
\end{remark}

\begin{remark}
Construction \ref{construction.solution2} is local with respect to the {\etale} topology. More precisely, if $f: U \rightarrow X$ is an {\etale} morphism of $\F_p$-schemes,
then we have a canonical isomorphism $f^{\ast} \Sol( \qE ) \simeq \Sol( f^{\ast} \qE )$ for every Frobenius sheaf $\qE$ on $X$.
\end{remark}

\begin{remark}
Let $X$ be an $\F_p$-scheme and let $\qE$ be a Frobenius sheaf on $X$. Then the canonical map $\qE \rightarrow \qE^{\perfection}$ induces an isomorphism of
{\etale} sheaves $\Sol( \qE ) \rightarrow \Sol( \qE^{\perfection} )$. To prove this, we can reduce to the case where $X$ is affine, in which case the desired result follows
from Proposition \ref{proposition.solperfect}.
\end{remark}

\begin{remark}\label{solcomplex}
Let $X$ be an $\F_p$-scheme and let $\qE$ be an algebraic Frobenius sheaf on $X$. Then the sequence
$$ 0 \rightarrow \Sol( \qE ) \rightarrow \qE \xrightarrow{ \id - \varphi_{ \qE} } \qE \rightarrow 0$$
is exact (in the abelian category $\Shv_{\mathet}( X, \F_p )$. To prove this, we can work locally
on $X$ and thereby reduce to the case where $X$ is affine, in which case the desired result follows from Propositions \ref{proposition.corX77} and \ref{prop75}.
\end{remark}

\begin{theorem}\label{globalRH}
Let $X$ be an $\F_p$-scheme. Then the solution functor $\Sol$ induces equivalences of abelian categories
$$\QCoh_{X}^{\alg} \simeq \Shv_{\mathet}(X, \F_p) \quad \quad \QCoh_{X}^{\hol} \rightarrow \Shv_{\mathet}^{c}( X, \F_p ).$$
Here $\Shv_{\mathet}^{c}( X, \F_p )$ denotes the full subcategory of $\Shv_{\mathet}(X, \F_p)$ spanned by those $p$-torsion {\etale} sheaves $\sheafF$
which are locally constructible (that is, for which the restriction $\sheafF|_{U} \in \Shv_{\mathet}(U, \F_p)$ is constructible for each affine open subset $U \subseteq X$).
\end{theorem}

\begin{remark}
If the scheme $X$ is quasi-compact and quasi-separated, then a sheaf $\sheafF \in \Shv_{\mathet}(X, \F_p )$ belongs to the subcategory $\Shv_{\mathet}^{c}( X, \F_p )$ if and only if
it is constructible: that is, if and only if it becomes locally constant along some constructible stratification of $X$.
\end{remark}

\begin{proof}[Proof of Theorem \ref{globalRH}]
Since the constructions 
$$ (U \subseteq X) \mapsto \QCoh_U^{\hol}, \QCoh_{U}^{\alg}, \Shv_{\mathet}(U, \F_p), \Shv_{\mathet}^{c}( U, \F_p )$$
satisfy effective descent with respect to the Zariski topology (or even the {\etale} topology), we can reduce to the case where $X = \Spec(R)$ is affine. In this case, the desired equivalences
follow from Theorems \ref{maintheoXXX} and \ref{companion}.
\end{proof}

\begin{corollary}\label{corollary.describe-indh}
Let $X$ be an $\F_p$-scheme which is quasi-compact and quasi-separated. Then the inclusion functor $\QCoh_{X}^{\hol} \hookrightarrow \QCoh_{X}^{\alg}$
extends to an equivalence of categories $\Ind( \QCoh_{X}^{\hol} ) \simeq \QCoh_{X}^{\alg}$.
\end{corollary}

\begin{proof}
By virtue of Theorem \ref{globalRH}, it will suffice to show that the inclusion functor $\Shv_{\mathet}^{c}( X, \F_p) \hookrightarrow \Shv_{\mathet}( X, \F_p )$
extends to an equivalence $\Ind( \Shv_{\mathet}^{c}(X, \F_p) ) \simeq \Shv_{\mathet}( X, \F_p )$, which follows from \cite[Tag 03SA]{Stacks}.
\end{proof}

\begin{notation}\label{notation.RH}
Let $X$ be an $\F_p$-scheme. We let $\RH: \Shv_{\mathet}( X, \F_p ) \rightarrow \QCoh_{X}^{\alg}$ denote an inverse of the solution functor. We will refer to
$\RH$ as the {\it Riemann-Hilbert functor}. 
\end{notation}

\begin{remark}\label{remark.adjoin}
Let $X$ be an $\F_p$-scheme and let $\sheafF \in \Shv_{\mathet}(X, \F_p)$ be a $p$-torsion {\etale} sheaf on $X$. Then the Frobenius sheaf
$\RH( \sheafF )$ is characterized by the following universal property: for every perfect Frobenius sheaf $\qE$ on $X$, the canonical map
$$ \Hom_{ \QCoh_{X}^{\perf} }( \RH(\sheafF), \qE ) \rightarrow \Hom_{ \underline{\F_p} }( \Sol( \RH(\sheafF) ), \Sol(\qE) ) 
\simeq \Hom_{ \underline{\F_p} }( \sheafF, \Sol(\qE) )$$
is a bijection. To prove this, we can reduce to the case where $X$ is affine, in which case the desired result follows from the properties of the Riemann-Hilbert functor given
in Theorem \ref{theorem.RHexist}. 

We can summarize the situation as follows: when regarded as a functor from $\Shv_{\mathet}( X, \F_p )$ to $\QCoh_{X}^{\perf}$, the Riemann-Hilbert functor of Notation \ref{notation.RH}
is left adjoint to the solution functor $\Sol: \QCoh_{X}^{\perf} \rightarrow \Shv_{\mathet}(X, \F_p)$.
\end{remark}

\subsection{Functoriality}\label{sec12sub3}

We now consider the behavior of Frobenius sheaves as the $\F_p$-scheme $X$ varies.

\begin{construction}[Pullback of Frobenius Sheaves]\label{construction.pullback}
Let $f: X \rightarrow Y$ be a morphism of $\F_p$-schemes, so that we have a commutative diagram of schemes
$$ \xymatrix{ X \ar[r]^{f} \ar[d]^{\varphi_X} & Y \ar[d]^{ \varphi_{Y}} \\
X \ar[r]^{f} & Y }$$
and therefore a canonical isomorphism $f^{\ast} \circ \varphi_{Y}^{\ast} \simeq \varphi_{X}^{\ast} \circ f^{\ast}$
in the category of functors from $\QCoh_Y$ to $\QCoh_X$.

Let $\qE$ be a Frobenius sheaf on $Y$, and let $\psi_{ \qE}: \varphi_{Y}^{\ast} \qE \rightarrow \qE$
be as in Variant \ref{FStabSub1}. We let $\psi_{ f^{\ast} \qE }$ denote the composite map
$$ \varphi_{X}^{\ast} f^{\ast} \qE \simeq f^{\ast} \varphi_{Y}^{\ast} \qE \xrightarrow{ f^{\ast} \psi_{\qE} } f^{\ast} \qE.$$
The construction $(\qE, \psi_{ \qE} ) \mapsto ( f^{\ast} \qE, \psi_{ f^{\ast} \qE} )$ determines a functor
$\QCoh_{Y}^{\Frob} \rightarrow \QCoh_{X}^{\Frob}$. We will denote this functor also by $f^{\ast}$, and
refer to it as the functor of {\it pullback along $f$}.
\end{construction}

\begin{remark}
In the special case where $X$ and $Y$ are affine, the pullback functor of Construction \ref{construction.pullback} 
agrees with the extension of scalars functor of Construction \ref{construction.extension}.
\end{remark}

Under some mild assumptions, the pullback functor $f^{\ast}$ of Construction \ref{construction.pullback} admits a right adjoint:

\begin{proposition}\label{proposition.construct-push}
Let $f: X \rightarrow Y$ be a morphism of schemes which is quasi-compact and quasi-separated. Then the pullback functor
$f^{\ast}: \QCoh_{Y}^{\Frob} \rightarrow \QCoh_{X}^{\Frob}$ admits a right adjoint $f_{\ast}: \QCoh_{X}^{\Frob} \rightarrow \QCoh_{Y}^{\Frob}$.
Moreover, the functor $f_{\ast}$ is compatible with the usual direct image functor on quasi-coherent sheaves: that is, the diagram
$$ \xymatrix{ \QCoh_{X}^{\Frob} \ar[r]^{f_{\ast} } \ar[d] & \QCoh_{Y}^{\Frob} \ar[d] \\
\QCoh_{X} \ar[r]^{ f_{\ast} } & \QCoh_{Y} }$$
commutes up to canonical isomorphism.
\end{proposition}

\begin{proof}
The assumption that $f$ is quasi-compact and quasi-separated guarantees that the pullback functor $f^{\ast}: \QCoh_Y \rightarrow \QCoh_X$
admits a right adjoint $f_{\ast}: \QCoh_{X} \rightarrow \QCoh_{Y}$. If $\qE$ is a Frobenius sheaf on $X$, we can equip the direct image
$f_{\ast} \qE$ with the structure of a Frobenius sheaf on $Y$ by defining $\varphi_{ f_{\ast} \qE }$ to be the composition
$$ f_{\ast} \qE \xrightarrow{ f_{\ast} \varphi_{\qE} } f_{\ast} \varphi_{X \ast} \qE \simeq \varphi_{Y \ast} f_{\ast} \qE.$$
We leave it to the reader to verify that the construction $( \qE, \varphi_{ \qE } ) \mapsto ( f_{\ast} \qE, \varphi_{ f_{\ast} \qE} )$
determines a functor from $\QCoh_{X}^{\Frob}$ to $\QCoh_Y^{\Frob}$ which is right adjoint to the pullback functor of
Construction \ref{construction.pullback}.
\end{proof}

\begin{remark}
In the situation of Proposition \ref{proposition.construct-push}, if $\qE \in \QCoh_{X}^{\Frob}$ has the property that $\varphi_{ \qE }$ is an isomorphism,
then $\varphi_{ f_{\ast} \qE}$ is also an isomorphism. In other words, the direct image functor $f_{\ast}$ carries $\QCoh_{X}^{\perf}$ into
$\QCoh_{Y}^{\perf}$.
\end{remark}

The pullback functor of Construction \ref{construction.pullback} generally does not carry perfect Frobenius sheaves to perfect Frobenius sheaves. To remedy this, we consider the following
variant:

\begin{construction}\label{construction.diamond}
Let $f: X \rightarrow Y$ be a morphism of $\F_p$-schemes. We define a functor
$f^{\diamond}: \QCoh_{Y}^{\perf} \rightarrow \QCoh_{X}^{\perf}$ by the formula
$f^{\diamond}( \qE ) = ( f^{\ast} \qE)^{\perfection}$.
\end{construction}

\begin{remark}
If $f: X \rightarrow Y$ is a quasi-compact, quasi-separated morphism of $\F_p$-schemes, then the functor
$f^{\diamond}:  \QCoh_{Y}^{\perf} \rightarrow \QCoh_{X}^{\perf}$ is left adjoint to the direct image functor
$f_{\ast}: \QCoh_X^{\perf} \rightarrow \QCoh_{Y}^{\perf}$.
\end{remark}

\begin{remark}
In the situation where $X$ and $Y$ are affine, the functor $f^{\diamond}: \QCoh_{Y}^{\perf} \rightarrow \QCoh_{X}^{\perf}$
agrees (using the identification of Example \ref{example.semble}) with the functor described in Proposition \ref{mallow}.
\end{remark}

In some cases, there is no difference between the functors $f^{\ast}$ and $f^{\diamond}$:

\begin{proposition}
Let $f: X \rightarrow Y$ be a morphism of $\F_p$-schemes. Assume either that $f$ is {\etale}, or that both $X$ and $Y$ are perfect (that is, the Frobenius maps
$\varphi_{X}: X \rightarrow X$ and $\varphi_{Y}: Y \rightarrow Y$ are isomorphisms). Then the pullback functor $f^{\ast}: \QCoh_{Y}^{\Frob} \rightarrow \QCoh_{X}^{\Frob}$
carries $\QCoh_{Y}^{\perf}$ into $\QCoh_{X}^{\perf}$. Consequently, the functors $f^{\ast}$ and $f^{\diamond}$ coincide on $\QCoh_{Y}^{\perf}$.
\end{proposition}

\begin{proof}
The assertion is local on both $X$ and $Y$, and therefore follows from Corollary \ref{corX14} (in the case where $f$ is {\etale}) and Proposition \ref{prop5prime} (in the case where $X$ and $Y$ are perfect).
\end{proof}

\begin{proposition}\label{proposition.obvo}
Let $f: X \rightarrow Y$ be a morphism of $\F_p$-schemes and let $\qE$ be an algebraic Frobenius sheaf on $Y$. Then $f^{\diamond} \qE$ is an algebraic Frobenius sheaf on $X$.
If $\qE$ is holonomic, then $f^{\diamond} \qE$ is also holonomic. 
\end{proposition}

\begin{proof}
Both assertions are local on $X$ and $Y$. We may therefore assume that $X$ and $Y$ are affine, in which case the desired results follow from Corollary \ref{corX62} and Proposition \ref{propX51}.
\end{proof}

\begin{proposition}
Let $X$ be an $\F_p$-scheme and let $X^{\perf}$ denote the perfection of $X$ (so that $\calO_{ X^{\perf} } = \calO_{X}^{\perfection}$). Then the canonical map
$f: X^{\perf} \rightarrow X$ induces an equivalence of categories $f_{\ast}: \QCoh^{\perf}_{X^{\perf}} \rightarrow \QCoh^{\perf}_{X}$.
\end{proposition}

\begin{proof}
The assertion is local on $X$ and we may therefore assume that $X$ is affine, in which case the desired conclusion follows from Proposition \ref{prop5}.
\end{proof}

We now consider behavior of direct and inverse image functors under the Riemann-Hilbert correspondence. We first observe that any morphism of schemes $f: X \rightarrow Y$ induces
a left exact functor $f_{\ast}: \Shv_{\mathet}(X, \F_p ) \rightarrow \Shv_{\mathet}( Y, \F_p )$, which is compatible with the direct image functor on quasi-coherent sheaves when
$f$ is quasi-compact and quasi-separated. We therefore obtain the following:

\begin{proposition}\label{proposition.compatibility}
Let $f: X \rightarrow Y$ be a morphism of $\F_p$-schemes which is quasi-compact and quasi-separated. Then the diagram of functors
$$ \xymatrix{ \QCoh^{\Frob}_{X} \ar[r]^-{\Sol} \ar[d]^{f_{\ast}} & \Shv_{\mathet}( X, \F_p) \ar[d]^{ f_{\ast}} \\
\QCoh_{Y}^{\Frob} \ar[r]^-{ \Sol} & \Shv_{\mathet}( Y, \F_p ) }$$
commutes (up to canonical isomorphism).
\end{proposition}

\begin{variant}\label{variant.compatibility}
Let $f: X \rightarrow Y$ be any morphism of $\F_p$-schemes. Then the diagram of functors
$$ \xymatrix{ \Shv_{\mathet}(Y, \F_p) \ar[r]^{\RH} \ar[d]^{f^{\ast}} & \Shv_{\mathet}(X,\F_p) \ar[d] \\
\QCoh^{\perf}_{Y} \ar[r]^{ f^{\diamond} } & \QCoh_{X}^{\perf} }$$
commutes (up to canonical isomorphism). In the case where $f$ is quasi-compact and quasi-separated,
this follows formally from Proposition \ref{proposition.compatibility} (by passing to left adjoints; see Remark \ref{remark.adjoin}).
The general case can be handled by working locally on $X$ and $Y$ (which reduces us to the situation of Proposition \ref{prop70}).
\end{variant}

\begin{construction}[{\Etale} Compactly Supported Direct Images]\label{construction.f-shriek}
Let $f: X \rightarrow Y$ be an {\etale} morphism of $\F_p$-schemes. Then the functor $f^{\ast}: \Shv_{\mathet}(Y, \F_p) \rightarrow \Shv_{\mathet}(X, \F_p)$ admits a left adjoint
$f_{!}: \Shv_{\mathet}(X, \F_p) \rightarrow \Shv_{\mathet}( Y, \F_p )$. Using Theorem \ref{globalRH}, we deduce that there is an essentially unique functor
$f_{!}: \QCoh_{X}^{\alg} \rightarrow \QCoh_{Y}^{\alg}$ for which the diagram
$$ \xymatrix{ \Shv_{\mathet}(X, \F_p) \ar[r]^-{\RH} \ar[d]^{f_{!} } & \QCoh_{X}^{\alg} \ar[d]^{ f_{!} } \\
\Shv_{\mathet}( Y, \F_p ) \ar[r]^-{ \RH} & \QCoh_{Y}^{\alg}. }$$
commutes up to isomorphism. We will refer to $f_{!}$ as the functor of {\it compactly supported direct image along $f$}.
\end{construction}

\begin{example}
In the situation of Construction \ref{construction.f-shriek}, if $X$ and $Y$ are {\etale}, then $f_{!}: \QCoh_{X}^{\alg} \rightarrow \QCoh_{Y}^{\alg}$ agrees with the functor constructed in Section \ref{section.compactimage}.
\end{example}

In the situation of Construction \ref{construction.f-shriek}, the functor $f_{!}: \QCoh_{X}^{\alg} \rightarrow \QCoh_{Y}^{\alg}$ can be characterized a left adjoint
to the pullback functor $f^{\ast} \simeq f^{\diamond}: \QCoh_{Y}^{\alg} \rightarrow \QCoh_X^{\alg}$. However, it has a slightly stronger property:

\begin{proposition}
Let $f: X \rightarrow Y$ be an {\etale} morphism of $\F_p$-schemes and let $\qE$ be an algebraic Frobenius sheaf on $X$.
Then, for any perfect Frobenius sheaf $\sheafF$ on $Y$, the canonical map
$$ \theta: \Hom_{ \QCoh_{Y}^{\perf} }( f_{!} \qE, \qF) \rightarrow
\Hom_{ \QCoh_{X}^{\perf} }( f^{\ast} f_{!} \qE, f^{\diamond} \qF ) \rightarrow \Hom_{ \QCoh_{X}^{\perf} }( \qE, f^{\ast} \qF)$$
is a bijection.
\end{proposition}

\begin{proof}
By virtue of Theorem \ref{globalRH}, we can assume that $\qE = \RH( \sheafE )$ for some $p$-torsion {\etale} sheaf $\sheafE$ on $X$.
In this case, the map $\theta$ fits into a commutative diagram
$$ \xymatrix{ \Hom_{ \QCoh_{Y}^{\perf} }( f_{!} \RH( \sheafE), \qF ) \ar[r]^{\theta} \ar[d] & \Hom_{ \QCoh_{X}^{\perf} }( \RH(\sheafE ), f^{\ast} \qF ) \ar[d] \\
\Hom_{ \underline{\F_p} }( f_{!} \sheafE, \Sol(\qF) ) \ar[r] & \Hom_{ \underline{\F_p}}( \sheafE, \Sol( f^{\ast} \qF) ) }$$
where the bottom horizontal map is an isomorphism because the formation of solution sheaves is local for the {\etale} topology, and
the vertical maps are isomorphisms by virtue of Remark \ref{remark.adjoin}.
\end{proof}

\begin{example}
Let $j:U \rightarrow X$ be a quasi-compact open immersion of $\F_p$-schemes. Then the functor $j_{!}: \QCoh_{U}^{\alg} \rightarrow \QCoh_{X}^{\alg}$
can be described explicitly as follows: if $\mathcal{I} \subseteq \calO_{X^{\perf}}$ denotes the (necessarily quasi-coherent) radical ideal sheaf defining $(X-U)^{\perf}$, then the Frobenius automorphism of $\calO_{X^{\perf}}$ endows $\mathcal{I}$ with the structure of a holonomic Frobenius module on $X$, and we have $j_!^{\alg} \calO_{U^{\perf}}) \simeq \mathcal{I}$. More generally, for any $E \in \QCoh_X^{\alg}$, we have $j_!^{\alg}(j^{\diamond} E) \simeq \mathcal{I} \otimes E$.
\end{example}

\subsection{Holonomic Frobenius Sheaves}\label{sec12sub4}

Recall that a Frobenius module $M$ over a commutative $\F_p$-algebra $R$ is holonomic if and only if there exists an isomorphism $M \simeq M_0^{\perfection}$, where $M_0 \in \Mod_{R}^{\Frob}$
is finitely presented as an $R$-module. We now show that holonomic Frobenius sheaves admit an analogous characterization:

\begin{theorem}\label{silphil}
Let $X$ be a Noetherian $\F_p$-scheme and let $\qE$ be a Frobenius sheaf on $X$. The following conditions are equivalent:
\begin{itemize}
\item[$(1)$] There exists a Frobenius subsheaf $\qE_0 \subseteq \qE$ such that $\qE_0$ is coherent as an $\calO_{X}$-module
and the inclusion $\qE_0 \hookrightarrow \qE$ induces an isomorphism $\qE_0^{\perfection} \simeq \qE$.

\item[$(2)$] There exists an isomorphism $\qE \simeq \qE_0^{\perfection}$ for some $\qE_0 \in \QCoh_{X}^{\perf}$ which is 
coherent as a $\calO_X$-module.

\item[$(3)$] The Frobenius sheaf $\qE$ is holonomic.
\end{itemize}
\end{theorem}

The proof of Theorem \ref{silphil} will require some preliminaries.

\begin{remark}\label{selit}
Let $X$ and $\qE$ be as in Theorem \ref{silphil}, and suppose that we are given Frobenius subsheaves $\qE_0 \subseteq \qE_1 \subseteq \qE$.
If the inclusion $\qE_0 \hookrightarrow \qE$ induces an isomorphism $\qE_0^{\perfection} \simeq \qE$, then the inclusion $\qE_1 \hookrightarrow \qE$
has the same property. This follows immediately from the exactness of the perfection construction $\qF \mapsto \qF^{\perfection}$.
\end{remark}

\begin{lemma}\label{subset}
Let $X$ be an $\F_p$-scheme, let $\qE$ be a Frobenius sheaf on $X$, and let $\qE_0 \subseteq \qE$ be a quasi-coherent $\calO_{X}$-submodule of $\qE$.
Then there exists a smallest Frobenius subsheaf $\qE' \subseteq \qE$ which contains $\qE_0$.
\end{lemma}

\begin{proof}
Take $\qE'$ to be the image of the composite map
$$\oplus_{n \geq 0} (\varphi_X^n)^{\ast} \qE_0 \rightarrow \oplus_{n \geq 0} (\varphi_X^n)^{\ast} \qE \xrightarrow{\psi^n_{\qE}} \qE.$$
\end{proof}

\begin{remark}\label{remark.restrictor}
In the situation of Lemma \ref{subset}, the construction $\qE_0 \mapsto \qE'$ is compatible with pullbacks along flat morphisms; in particular, it is compatible
with restrictions to open sets.
\end{remark}

\begin{lemma}\label{lemma.coh}
Let $X$ be a Noetherian $\F_p$-scheme, let $\qE$ be an algebraic Frobenius sheaf on $X$, and let $\qE_0 \subseteq \qE$ be a coherent $\calO_{X}$-submodule of $\qE$.
Then the Frobenius subsheaf $\qE' \subseteq \qE$ of Lemma \ref{subset} is also coherent as a $\calO_X$-module.
\end{lemma}

\begin{proof}
By virtue of Remark \ref{remark.restrictor}, we can assume without loss of generality that $X = \Spec(R)$ is affine. In this case, the desired result follows from
Remark \ref{rmk:IndHolCriterion}.
\end{proof}

\begin{proof}[Proof of Theorem \ref{silphil}]
The implications $(1) \Rightarrow (2) \Rightarrow (3)$ are obvious. We will prove that $(3)$ implies $(1)$. Let $\qE$ be a holonomic Frobenius sheaf on $X$.
Choose a finite cover $\{U_i\}$ of $X$ by affine open sets $U_i \simeq \Spec(R_i)$, and set $M_i = \qE(U_i)$. Then each $M_i$ is a holonomic Frobenius module over $R_i$.
We can therefore choose isomorphisms $M_i \simeq N_i^{\perfection}$, where each $N_i \in \Mod_{R_i}^{\Frob}$ is finitely generated as a module over $R_i$.
Replacing each $N_i$ with its image in $M_i$, we can assume that $N_i$ corresponds to a Frobenius-stable subsheaf $\overline{\qF}_i \subseteq \qE|_{U_i}$. Applying \cite[Tag 01PF]{Stacks}, we can find choose a coherent subsheaf 
$\qF_i \subseteq \qE$ satisfying $\qF_i|_{U} = \overline{\qF}_i$. Let $\qF$ denote the smallest Frobenius subsheaf of $\qE$ which contains each $\qF_i$ (Lemma \ref{subset}). It follows from Lemma \ref{lemma.coh}
that $\qF$ is coherent as a $\calO_X$-module. We claim that the inclusion $\qF \hookrightarrow \qE$ induces an isomorphism $\qF^{\perfection} \simeq \qE$. To prove this, it suffices to show that
each restriction $\qE|_{ U_i }$ is the perfection of $\qF|_{U_i}$. This follows from Remark \ref{selit}, since $\qF|_{U_i}$ contains $\overline{\qF}_i$ by construction.
\end{proof}

\begin{remark}
Theorem \ref{silphil} can be generalized to the non-Noetherian case. If $X$ is an $\F_p$-scheme which is quasi-compact and quasi-separated and
$\qE$ is a holonomic Frobenius sheaf on $X$, then there exists an isomorphism $\qE \simeq \qE_0^{\perfection}$, where $\qE_0 \in \QCoh_{X}^{\Frob}$ is
locally finitely presented as a $\calO_X$-module. To prove this, we first apply Theorem \ref{globalRH} to choose an isomorphism $\qE = \RH( \sheafF)$ for some constructible $p$-torsion
{\etale} sheaf $\sheafF$ on $X$. Using a Noetherian approximation argument, we can choose a map $f: X \rightarrow Y$ and an isomorphism $\sheafF \simeq f^{\ast} \sheafF'$,
where $Y$ is a Noetherian $\F_p$-scheme and $\sheafF'$ is a constructible $p$-torsion {\etale} sheaf on $Y$. Applying Lemma \ref{silphil}, we can choose an isomorphism
$\RH( \sheafF' ) \simeq \qE'^{\perfection}_0$ for some $\qE'_0 \in \QCoh_{Y}^{\Frob}$ which is coherent as an $\calO_Y$-module. Then 
$$\qE \simeq \RH( \sheafF ) \simeq \RH( f^{\ast} \sheafF' ) \simeq f^{\diamond}( \RH( \sheafF' ) ) \simeq f^{\diamond}( \qE'^{\perfection}_0 ) \simeq (f^{\ast} \qE'_0)^{\perfection},$$
where $f^{\ast} \qE'_0$ is locally finitely presented as a $\calO_{X}$-module.
\end{remark}



\subsection{Proper Direct Images}\label{sec12sub5}

In \S \ref{finimage}, we proved that the Riemann-Hilbert equivalence $\Sol:\Mod_{R}^{\alg} \simeq \Shv_{\mathet}( \Spec(R), \F_p)$ is compatible with direct images
along ring homomorphisms which are finite and of finite presentation (Theorem \ref{theo76}). In this section, we prove a generalization of this result: the global Riemann-Hilbert correspondence
of Theorem \ref{globalRH} is compatible with direct images along morphisms of $\F_p$-schemes $f: X \rightarrow Y$ which are proper and of finite presentation. In the global setting, there is more to the story,
since neither of the direct image functors 
$$f_{\ast}: \QCoh_{X}^{\Frob} \rightarrow \QCoh_{Y}^{\Frob} \quad \quad f_{\ast}: \Shv_{\mathet}(X, \F_p) \rightarrow \Shv_{\mathet}(Y, \F_p)$$
is necessarily exact. In this case, we also have a comparison result for higher direct images (see Theorem \ref{properRH1} below).

We begin with some general remarks. Let $f:X \rightarrow Y$ be a quasi-compact and quasi-separated morphism of schemes. Then we have higher direct image functors $R^i f_{\ast}:\QCoh_X \to \QCoh_Y$ (see \cite[Tag 01XJ]{Stacks}). These functors are equipped with canonical isomorphisms $\varphi_{Y,\ast} \circ R^i f_{\ast} \simeq R^i f_{\ast} \circ \varphi_{X,\ast}$, and therefore carry
(perfect) Frobenius sheaves on $X$ to (perfect) Frobenius sheaves on $Y$. The central observation is the following:

\begin{theorem}\label{properindh}
Let $f: X \rightarrow Y$ be a morphism of $\F_p$-schemes which is proper and of finite presentation. If $\qE$ is an algebraic Frobenius sheaf on $X$,
then the higher direct images $R^{i} f_{\ast} \qE$ are algebraic Frobenius sheaves on $Y$.
\end{theorem}

We begin by studying the Noetherian case.

\begin{lemma}\label{lemma.pdimage1}
Let $f: X \rightarrow Y$ be a proper morphism of Noetherian $\F_p$-schemes. If $\qE$ is a holonomic Frobenius sheaf on $X$,
then the higher direct images $R^{i} f_{\ast} \qE$ are holonomic Frobenius sheaves on $Y$.
\end{lemma}

\begin{proof}
Invoking Theorem \ref{silphil}, we can write $\qE = \qE_0^{\perfection}$, where $\qE_0$ is a Frobenius sheaf on $X$ which is coherent
as an $\calO_{X}$-module. It follows from the direct image theorem \cite[Tag 02O3]{Stacks} that the higher direct images
$R^{i} f_{\ast} \qE_0$ are coherent $\calO_{Y}$-modules. Since the functors $R^{i} f_{\ast}$ commute with filtered direct limits, we have canonical
isomorphisms
$$ R^{i} f_{\ast} \qE \simeq R^{i} f_{\ast} (\qE_0^{\perfection}) \simeq ( R^{i} f_{\ast} \qE_0 )^{\perfection}.$$
Applying Theorem \ref{silphil} again, we see that each $R^{i} f_{\ast} \qE$ is holonomic.
\end{proof}

\begin{lemma}\label{lemma.pdimage2}
Let $R$ be a commutative $\F_p$-algebra, let $f: X \rightarrow \Spec(R)$ be a morphism of schemes which is proper and of finite presentation,
and let $\qE$ be an holonomic Frobenius sheaf on $X$. Then, for every integer $i$, the cohomology group $\mathrm{H}^{i}( X, \qE )$
is an algebraic Frobenius module over $R$.
\end{lemma}

\begin{remark}
In the situation of Lemma \ref{lemma.pdimage2}, one can say more: the cohomology groups $\mathrm{H}^{i}(X, \qE)$ are actually holonomic
Frobenius modules over $R$ (see Corollary \ref{properhol} below).
\end{remark}

\begin{proof}[Proof of Lemma \ref{lemma.pdimage2}]
Applying Theorem \ref{globalRH}, we can choose an isomorphism $\qE \simeq \RH( \sheafF )$ for
some constructible sheaf $\sheafF \in \Shv_{\mathet}^{c}( X, \F_p )$. 
Using Noetherian approximation \cite[Tags 01ZM and 081F]{Stacks}, we can choose a finitely generated $\F_p$-subalgebra $R_0 \subseteq R$
and a pullback diagram of schemes
$$ \xymatrix{ X \ar[d]^{f} \ar[r]^{\pi} & X_0 \ar[d]^{f_0} \\
\Spec(R) \ar[r] & \Spec(R_0), }$$
where $f_0$ is proper. Enlarging $R_0$ if necessary, we can further arrange that $\sheafF = \pi^{\ast} \sheafF_0$, where $\sheafF_0$ is a constructible
$p$-torsion {\etale} sheaf on $X_0$ (see \cite[\S 1, Proposition 4.17]{FK}). Set $\qE_0 = \RH( \sheafF_0 )$, so that
$\qE \simeq \pi^{\diamond} \qE_0$ (see Variant \ref{variant.compatibility}).

Write $R$ as a filtered direct limit of finitely generated subrings $R_{\alpha} \subseteq R$ containing $R_0$.
For each index $\alpha$, set $X_{\alpha} = \Spec(R_{\alpha} ) \times_{ \Spec(R_0) } X_0$, let
$\pi_{\alpha}: X_{\alpha} \rightarrow X_0$ be the projection onto the second factor, and set
$\qE_{\alpha} = \pi_{\alpha}^{\diamond} \qE_0$. Then each $\qE_{\alpha}$ is a holonomic Frobenius sheaf on
$X_{\alpha}$ (Proposition \ref{proposition.obvo}). Invoking Lemma \ref{lemma.pdimage1}, we see that
the cohomology group $\mathrm{H}^{i}( X_{\alpha}, \qE_{\alpha} )$ is a holonomic Frobenius module over
$R_{\alpha}$, so that the tensor product $R^{\perfection} \otimes_{ R_{\alpha}^{\perfection} } \qE_{\alpha}$
is a holonomic Frobenius module over $R$ (Proposition \ref{propX51}). We now compute
\begin{eqnarray*}
\mathrm{H}^{i}( X, \qE ) & \simeq & \varinjlim \mathrm{H}^{i}( X_{\alpha}, \qE_{\alpha} )  \\
& \simeq & R^{\perfection} \otimes_{ R^{\perfection} } \varinjlim \mathrm{H}^{i}( X_{\alpha}, \qE_{\alpha} ) \\
& \simeq & \varinjlim R^{\perfection} \otimes_{ R_{\alpha}^{\perfection}} \mathrm{H}^{i}( X_{\alpha} , \qE_{\alpha} ).
\end{eqnarray*}
Since the collection of holonomic Frobenius modules over $R$ is closed under direct limits (Proposition \ref{propX53}), it follows
that $\mathrm{H}^{i}(X, \qE)$ is algebraic. 
\end{proof}

\begin{proof}[Proof of Theorem \ref{properindh}]
Let $f: X \rightarrow Y$ be a morphism of $\F_p$-schemes which is proper and of finite presentation, and let
$\qE$ be an algebraic Frobenius sheaf on $X$. We wish to show that each higher direct image
$R^{i} f_{\ast} \qE$ is an algebraic Frobenius sheaf on $Y$. This assertion is local on $Y$, so we may
assume without loss of generality that $Y = \Spec(R)$ is affine. In this case, $X$ is quasi-compact and quasi-separated,
so Corollary \ref{corollary.describe-indh} guarantees that we can write $\qE$ as a filtered direct limit $\varinjlim \qE_{\alpha}$, where each $\qE_{\alpha}$
is holonomic. Since the functor $\qE \mapsto \mathrm{H}^{i}(X, \qE)$ commutes with filtered direct limits and the collection of
algebraic $R$-modules is closed under direct limits (Proposition \ref{propX53}), we may replace $\qE$ by $\qE_{\alpha}$
and thereby reduce to the case where $\qE$ is holonomic. In this case, the desired result follows from Lemma \ref{lemma.pdimage2}.
\end{proof}

We now apply Theorem \ref{properindh} to the study of our Riemann-Hilbert correspondence.

\begin{theorem}\label{properRH1}
Let $f: X \rightarrow Y$ be a morphism of $\F_p$-schemes which is proper and of finite presentation. For every algebraic Frobenius sheaf $\qE$ on $X$, we have
canonical isomorphisms $\Sol( R^{i} f_{\ast} \qE ) \simeq R^i f_{\ast} \Sol( \qE )$ in the category $\Shv_{\mathet}( Y, \F_p)$.
\end{theorem}

\begin{proof}
Since $\qE$ is algebraic, we have an exact sequence 
$$ 0 \rightarrow \Sol( \qE ) \rightarrow \qE \xrightarrow{ \id - \varphi_{ \qE } } \qE \rightarrow 0$$
in the category of {\etale} sheaves on $X$ (Remark \ref{solcomplex}). This gives rise to a long exact sequence of higher direct images
$$ R^{i-1} f_{\ast} \qE \xrightarrow{ \id - \varphi } R^{i-1} f_{\ast} \qE
\rightarrow R^{i} f_{\ast} \Sol( \qE ) \rightarrow R^i f_{\ast} \qE \xrightarrow{ \id - \varphi} R^{i} f_{\ast} \qE$$
which gives rise to short exact sequence of {\etale} sheaves
$$ 0 \rightarrow \sheafF \rightarrow R^{i} f_{\ast} \Sol( \qE ) \rightarrow \Sol( R^{i} f_{\ast} \qE ) \rightarrow 0,$$
where $\sheafF$ denotes the cokernel of the map $(\id - \varphi): R^{i-1} f_{\ast} \qE \rightarrow R^{i-1} f_{\ast} \qE$.
It will therefore suffice to show that the map $(\id - \varphi): R^{i-1} f_{\ast} \qE \rightarrow R^{i-1} f_{\ast} \qE$
is an epimorphism of {\etale} sheaves on $Y$. This follows from Remark \ref{solcomplex}, since the
Frobenius sheaf $R^{i-1} f_{\ast} \qE$ is algebraic by virtue of Theorem \ref{properindh}.
\end{proof}

\begin{corollary}\label{properindh2}
Let $f: X \rightarrow Y$ be a morphism of $\F_p$-schemes which is proper and of finite presentation. Then, for any
$p$-torsion {\etale} sheaf $\sheafF$ on $X$, we have canonical isomorphisms
$\RH( R^{i} f_{\ast} \sheafF ) \simeq R^{i} f_{\ast} \RH( \sheafF )$.
\end{corollary}

\begin{proof}
Using Theorem \ref{properRH1}, we a canonical isomorphism
$$ R^{i} f_{\ast} \sheafF \simeq R^{i} f_{\ast} \Sol( \RH(\sheafF) )  \simeq 
\Sol( R^{i} f_{\ast} \RH( \sheafF) ),$$
which is adjoint to a comparison map $\gamma: \RH( R^{i} f_{\ast} \sheafF ) \rightarrow R^{i} f_{\ast} \RH( \sheafF )$
(Remark \ref{remark.adjoin}). Since the Frobenius sheaf $R^{i} f_{\ast} \RH(\sheafF )$ is
algebraic (Theorem \ref{properindh}), the map $\gamma$ is an isomorphism.
\end{proof}

\subsection{Application: The Proper Base Change Theorem}\label{sec12sub6}

Suppose we are given a pullback diagram of schemes $\sigma:$
$$ \xymatrix{ X' \ar[r]^{g'} \ar[d]^{f'} & X \ar[d]^{f} \\
Y' \ar[r]^{g} & Y. }$$
For every {\etale} sheaf $\sheafF$ on $X$ and every integer $n \geq 0$, we have a natural comparison map
$\alpha: g^{\ast} R^{n} f_{\ast} \sheafF \rightarrow R^{n} f'_{\ast} g'^{\ast} \sheafF$
in the category of {\etale} sheaves on $Y'$. The {\it proper base change theorem} in {\etale} cohomology
asserts that, if the morphism $f$ is proper and $\sheafF$ is a torsion sheaf, then the map $\alpha$ is an isomorphism \cite[Tag 095S]{Stacks}.
Our goal in this section is to show that, in special case where $\sigma$ is a diagram of $\F_p$-schemes and $\sheafF$ is a $p$-torsion sheaf, the
proper base change theorem can be deduced from the results of this paper in an essentially formal way.

We begin with some general remarks. Let $\sigma$ be as above, and suppose that the morphisms
$f$ and $f'$ are quasi-compact and quasi-separated. In this case, we can associate to every quasi-coherent sheaf $\qE$ on $X$ a comparison map
$$ \beta: g^{\ast} R^{n} f_{\ast} \qE \rightarrow R^{n} f'_{\ast} g'^{\ast} \qE$$
in the category $\QCoh_{Y}$ of quasi-coherent sheaves on $Y'$. Moreover, if $\sigma$ is a diagram of $\F_p$-schemes and $\qE$ is a Frobenius sheaf on $X$, then $\beta$ is a morphism of Frobenius sheaves.
If, in addition, the Frobenius sheaf $\qE$ is perfect, then the perfection of $\beta$ supplies a comparison map
$\gamma: g^{\diamond} R^{n} f_{\ast} \qE \rightarrow R^n f'_{\ast} g'^{\diamond} \qE$ in the category $\QCoh_{Y'}^{\perf}$.

\begin{proposition}\label{ebly}
Suppose we are given a pullback diagram of $\F_p$-schemes
$$ \xymatrix{ X' \ar[r]^{g'} \ar[d]^{f'} & X \ar[d]^{f} \\
Y' \ar[r]^{g} & Y, }$$
where $f$ is proper and of finite presentation. Then, for any algebraic Frobenius sheaf $\qE$ on $X$, the comparison maps
$\gamma: g^{\diamond} R^{n} f_{\ast} \qE \rightarrow R^n f'_{\ast} g'^{\diamond} \qE$ are isomorphisms.
\end{proposition}

\begin{proof}
The assertion is local on $Y$ and $Y'$; we may therefore assume without loss of generality that $Y = \Spec(R)$ and $Y' = \Spec(S)$ are affine. In this case, we wish to show that
the canonical map $S^{\perfection} \otimes_{ R^{\perfection} } \mathrm{H}^{\ast}( X, \qE ) \rightarrow \mathrm{H}^{\ast}(X', g'^{\diamond} \qE )$ is an isomorphism.
Choose a finite covering $\{ U_i \}$ of $X$ by affine open subsets and let $\{ U'_i \}$ denote the open covering of $X'$ given by $U'_i = g'^{-1} U_i$.
Let $C^{\ast}$ denote the \v{C}ech complex of $\{ U_i \}$ with coefficients in the sheaf $\qE$, so that we can identify
$\mathrm{H}^{\ast}( X, \qE)$ with the cohomology of the cochain complex $C^{\ast}$. Note that for any affine open subset $U \subseteq X$
having inverse image $U' \subseteq X'$, we have a canonical isomorphism $(g'^{\diamond} \qE)(U') = S^{\perfection} \otimes_{ R^{\perfection} } \qE(U)$,
so that $S^{\perfection} \otimes_{ R^{\perfection} } C^{\ast}$ is the \v{C}ech complex of the open covering $\{ U'_i \}$ with coefficients in the sheaf $g'^{\diamond} \qE$.
We can therefore identify $\gamma$ with the canonical map $$S^{\perfection} \otimes_{ R^{\perfection} } \mathrm{H}^{n}( C^{\ast} ) \rightarrow
\mathrm{H}^{n}( S^{\perfection} \otimes_{ R^{\perfection} } C^{\ast} ).$$

For every affine open subset $U \subseteq X$ with inverse image $U' \subseteq X'$, Remark \ref{TorVanishingRings} supplies isomorphisms
$$\Tor_{m}^{ R^{\perfection} }( S^{\perfection}, \calO^{\perfection}_X(U) ) = \begin{cases} \calO^{\perfection}_{X'}(U') & \text{ if } m = 0 \\
0 & \text{otherwise. } \end{cases}$$
It follows that the canonical maps $$\Tor_{k}^{ R^{\perfection} }( S^{\perfection}, \qE(U) )
\rightarrow \Tor_{k}^{ \calO_{X}^{\perfection}(U) }( \calO_{X'}^{\perfection}(U'), \qE(U) )$$
are isomorphisms. Our assumption that $\qE$ is algebraic guarantees that $\qE(U)$ is an
algebraic Frobenius module over $\calO_{X}(U)$, so that the groups 
$$\Tor_{k}^{ \calO_{X}^{\perfection}(U) }( \calO_{X'}^{\perfection}(U'), \qE(U) )$$ vanish for $k > 0$.
We therefore also have $\Tor_{k}^{ R^{\perfection} }( S^{\perfection}, \qE(U) ) \simeq 0$ for $k > 0$. Allowing $U$ to vary,
we conclude that the tensor product $S^{\perfection} \otimes_{ R^{\perfection} } C^{\ast}$ is equivalent to the left derived tensor product
$S^{\perfection} \otimes_{ R^{\perfection} }^{L} C^{\ast}$. We therefore have a convergent spectral sequence
$$E_2^{s,t}: \Tor_{s}^{R^{\perfection}}( S^{\perfection}, \mathrm{H}^{t}( C^{\ast} ) ) \Rightarrow \mathrm{H}^{t-s}( R^{\perfection} \otimes_{ S^{\perfection} } C^{\ast} ),$$
in which $\gamma$ appears as an edge map. To show that $\gamma$ is an isomorphism, it suffices to show that the groups $E_{2}^{s,t}$ vanish for $s > 0$.
This follows from Theorem \ref{theoX9}, since each $\mathrm{H}^{t}( C^{\ast} ) \simeq \mathrm{H}^{t}( X, \qE )$ is an algebraic Frobenius module over $R$
by virtue of Proposition \ref{properindh}.
\end{proof}

\begin{corollary}[Proper Base Change]\label{corollary.basechange}
Suppose we are given a pullback diagram of $\F_p$-schemes
$$ \xymatrix{ X' \ar[r]^{g'} \ar[d]^{f'} & X \ar[d]^{f} \\
Y' \ar[r]^{g} & Y, }$$
where $f$ is proper and of finite presentation. Then, for every $p$-torsion {\etale} sheaf $\sheafF$ on $X$, the comparison map
$\alpha: g^{\ast} R^{n} f_{\ast} \sheafF \rightarrow R^n f'_{\ast} g'^{\ast} \sheafF$ is an isomorphism
\end{corollary}

\begin{proof}
Using Corollary \ref{properindh2}, we can identify the image of $\alpha$ under the Riemann-Hilbert correspondence
$\RH: \Shv_{\mathet}(Y', \F_p) \rightarrow \QCoh_{Y'}^{\perf}$ with the comparison map
$\gamma: g^{\diamond} R^{n} f_{\ast} \RH(\sheafF) \rightarrow R^n f'_{\ast} g'^{\diamond} \RH(\sheafF)$
of Proposition \ref{ebly}. Since $\RH( \sheafF)$ is algebraic, the map $\gamma$ is an isomorphism, so that
$\alpha$ is also an isomorphism.
\end{proof}

We can use Proposition \ref{ebly} to show that Lemma \ref{lemma.pdimage1} holds in the non-Noetherian case:

\begin{corollary}\label{properhol}
Let $f: X \rightarrow Y$ be a morphism of $\F_p$-schemes which is proper and of finite presentation. If
$\qE$ is a holonomic Frobenius sheaf on $X$, then the higher direct images $R^{n} f_{\ast} \qE$ are holonomic
Frobenius sheaves on $Y$.
\end{corollary}

\begin{proof}[Proof of Corollary \ref{properhol}]
The assertion is local on $Y$, so we may assume without loss of generality that $Y = \Spec(R)$ is affine.
Proceeding as in the proof of Lemma \ref{lemma.pdimage2}, we can choose a pullback diagram
$$ \xymatrix{ X \ar[d]^{f} \ar[r]^{\pi} & X_0 \ar[d]^{f_0} \\
\Spec(R) \ar[r]^{\pi'} & \Spec(R_0) }$$
where $f_0$ is proper, $R_0 \subseteq R$ is a finitely generated $\F_p$-subalgebra, and
$\qE \simeq \pi^{\diamond} \qE_0$ for some holonomic Frobenius module $\qE_0$ on $X_0$.
Lemma \ref{lemma.pdimage1} guarantees that $R^{n} f_{0 \ast} \qE_0$ is holonomic, so that
$\pi'^{\diamond} R^{n} f_{0 \ast} \qE_0$ is also holonomic (Proposition \ref{propX51}). 
Proposition \ref{ebly} supplies an isomorphism
$$\gamma: \pi'^{\diamond} R^{n} f_{0 \ast} \qE_0 \rightarrow R^{n} f_{\ast} \pi^{\diamond} \qE_0 \simeq R^{n} f_{\ast} \qE,$$
so that $R^{n} f_{\ast} \qE$ is holonomic as well.
\end{proof}

\begin{corollary}
Let $f: X \rightarrow Y$ be a morphism of $\F_p$-schemes which is proper and of finite presentation. Then the higher direct image
functors $R^{n} f_{\ast}: \Shv_{\mathet}( X, \F_p ) \rightarrow \Shv_{\mathet}( Y, \F_p )$ carry constructible sheaves to constructible sheaves.
\end{corollary}

\begin{proof}
Combine Corollary \ref{properhol} with Theorem \ref{globalRH} and Corollary \ref{properindh2}.
\end{proof}

\begin{remark}
Let $f: R \rightarrow S$ be a morphism of $\F_p$-algebras, let $X$ be an $R$-scheme which is proper and of finite presentation, and set
$X_{S} = X \times_{ \Spec(R)} \Spec(S)$. In this situation, we have a comparison map
$$ \beta: \mathrm{H}^{\ast}( X, \calO_{X} ) \otimes_{R} S \rightarrow \mathrm{H}^{\ast}( X_S, \calO_{X_S} ).$$
In general, this map need not be an isomorphism, even if $X$ is assumed to be smooth and projective over $R$ (see Example~\ref{ex:BCFails}). However,
the domain and codomain of $\beta$ can be regarded as Frobenius modules over $S$, and Proposition \ref{ebly} implies that $\beta^{\perfection}$ is an isomorphism:
in other words, every element of $\ker(\beta)$ or $\coker(\beta)$ is annihilated by some power of the Frobenius. In other words, the proper base change theorem
holds in the setting of coherent cohomology, provided that we work ``up to perfection.''
\end{remark}

\begin{example}
\label{ex:BCFails}
Let $k$ be a field of characteristic $p$ and set $R = k \llbracket t \rrbracket$. Let $G \to \mathrm{Spec}(R)$ be a finite flat group scheme with generic fibre $\mu_p$ and special fibre $\alpha_p$. For each $k \geq 2$, we can approximate the stack $BG \to \mathrm{Spec}(R)$ by a smooth projective $R$-scheme $X$ with geometrically connected fibres, i.e., the $\calO$-cohomology of the generic fibre $X_\eta$ agrees with that of $B(\mu_p)$ in degrees $\leq k$, while that for the special fibre $X_s$ agrees with that of $B(\alpha_p)$ in degrees $\leq k$; an explicit example is provided when $p=k=2$ by degenerating a ``classical'' Enriques surface $X_\eta$ to a ``supersingular'' Enriques surface $X_s$. Assume now that $k=2$ for simplicity. As $\mu_p$ is linearly reductive, it follows that $H^i(X_\eta, \calO_{X_\eta}) = 0$ for $i \in \{1,2\}$. On the other hand, $H^i(X_s, \calO_{X_s}) \neq 0$ for $i=1,2$. Now consider the $R$-module $H^1(X, \calO_X)$. Since $H^0(X_s, \calO_{X_s}) = 0$, it is easy to see that $H^1(X, \calO_X)$ is $t$-torsionfree. But $H^1(X, \calO_X)[\frac{1}{t}] = H^1(X_\eta, \calO_{X_\eta}) = 0$, so it follows that $H^1(X, \calO_X) = 0$. On the other hand, $H^1(X_s, \calO_{X_s}) \neq 0$, so we have constructed an example where the base change map
\[ H^1(X, \calO_X) \otimes_R k \to H^1(X_s, \calO_{X_s})\]
is not an isomorphism.
\end{example}

\newpage \section{The Contravariant Riemann-Hilbert Correspondence}\label{sec:ContravariantRH}
\setcounter{subsection}{0}
\setcounter{theorem}{0}

Let $R$ be a smooth algebra over a field $k$ of characteristic $p$. In \cite{EK}, Emerton and Kisin construct an equivalence of triangulated categories
$$ \RSol_{\EK}: D^{b}_{\fgu}( R[F] )^{\op} \simeq D^{b}_{c}( \Spec(R), \F_p),$$
where $D^{b}_{\fgu}( R[F] )$ denotes the full subcategory of $D( R[F] )$ spanned by the cohomologically bounded chain complexes whose cohomology groups
{\it finitely generated unit} Frobenius modules and $D^{b}_{c}( \Spec(R), \F_p)$ the constructible derived category of $\Spec(R)$: that is,
the full subcategory of the derived category of $\Shv_{\mathet}( \Spec(R), F_p)$ spanned by those chain complexes which are cohomologically bounded with
constructible cohomology. 

Our goal in this section is to review (and generalize) the construction of the functor $\RSol_{\EK}$. We begin in \S \ref{sec10sub1} by reviewing the notion of a {\it finitely generated unit} 
Frobenius module over a commutative $\F_p$-algebra $R$ (Definition \ref{definition.fgu}), following \cite{Ly} and \cite{EK}. The collection of finitely generated unit modules is always closed under the formation of cokernels and extensions (Propositions \ref{in1} and \ref{in2}). In \S \ref{sec10sub2} we show that, when $R$ is a regular Noetherian $\F_p$-algebra, it is also closed under the formation of kernels (Proposition \ref{in3}). In this case,
we let $D^{b}_{\fgu}( R[F] )$ denote the full subcategory of the derived category $D( R[F] )$ spanned by those cochain complexes $M = M^{\ast}$ whose cohomology groups
$\mathrm{H}^{n}(M)$ are locally finitely generated unit Frobenius modules which vanish for all but finitely many values of $n$. In \S \ref{sec10sub4}, we show that there is a sensible way to define
the definition of the subcategory $D^{b}_{\fgu}( R[F] ) \subseteq D( R[F] )$ for an arbitrary $\F_p$-algebra $R$, by restricting our attention to cochain complexes which satisfy suitable ``derived'' versions of the requirements defining finitely generated unit modules (see Definition \ref{definition.derived-fgu} and Proposition \ref{proposition.char-fgu}). 
In \S \ref{sec10sub5} we define a solution functor $\RSol_{\EK}: D^{b}_{\fgu}( \Mod_{R}^{\Frob} )^{\op} \rightarrow D( \Shv_{\mathet}( \Spec(R), \F_p) )$ and assert that it restricts to an equivalence of
categories $D^b_{\fgu}( \Mod_{R}^{\Frob} )^{\op} \simeq D^{b}_{c}(\Spec(R),\F_p)$ (Theorem \ref{strongEK}). Taking $R$ to be a smooth algebra of finite type over a field $k$, this recovers the main result of
\cite{EK} in the case of the affine scheme $X = \Spec(R)$. However, our equivalence is a bit more general, since we allow $R$ to be an arbitrary $\F_p$-algebra.
The proof of Theorem \ref{strongEK} will be given in \S \ref{section.duality} by comparing the functor $\RSol_{\EK}$ with the solution functor $\Sol$ of Construction \ref{construction.solsheaf} (and its derived functors).

\subsection{Finitely Generated Unit Frobenius Modules}\label{sec10sub1}

We now introduce the class of finitely generated unit Frobenius modules, following \cite{EK}.

\begin{notation}\label{notation.psiM}
Let $R$ be a commutative $\F_p$-algebra and let $M$ be an $R$-module. We let $\varphi_{R}^{\ast} M$ denote the $R$-module obtained from $M$ by
extending scalars along the Frobenius homomorphism $\varphi_{R}: R \rightarrow R$. If $M$ is a Frobenius module over $R$, then
the Frobenius map $\varphi_{M}: M \rightarrow M^{1/p}$ determines an $R$-module homomorphism $\varphi_{R}^{\ast} M \rightarrow M$, which
we will denote by $\psi_M$. 
\end{notation}

\begin{remark}\label{remark.obvious}
In the situation of Notation \ref{notation.psiM}, we can regard $\varphi_{R}^{\ast} M$ as a Frobenius module over $R$ (Construction \ref{construction.extension}),
and $\psi_{M}$ is a morphism of Frobenius modules over $R$. Moreover, the morphism $\psi_M$ induces an isomorphism of perfections
$(\varphi_{R}^{\ast} M)^{\perfection} \rightarrow M^{\perfection}$. To prove this, we can extend scalars to the perfection $R^{\perfection}$ and thereby reduce to the case where
$R$ is perfect. In this case, the morphism $\psi_M$ coincides with (the Frobenius pullback of) the map $\varphi_{M}: M \rightarrow M^{1/p}$, which is evidently an isomorphism of perfections.
\end{remark}

\begin{definition}\label{definition.fgu}
Let $R$ be a commutative $\F_p$-algebra and let $M$ be a Frobenius module over $R$. We will say that $M$ is {\it finitely generated unit} if it satisfies the following pair of conditions:
\begin{itemize}
\item[$(a)$] The module $M$ is finitely generated as a left module over the noncommutative ring $R[F]$ of Notation \ref{not2}.
\item[$(b)$] The map $\psi_{M}: \varphi_{R}^{\ast} M \rightarrow M$ of Notation \ref{notation.psiM} is an isomorphism. 
\end{itemize}
\end{definition}

We now record some easy closure properties of the class finitely generated unit Frobenius modules.

\begin{proposition}\label{in1}
Let $R$ be a commutative $\F_p$-algebra and let $f: M \rightarrow N$ be a morphism of Frobenius modules over $R$. If
$M$ and $N$ are finitely generated unit, then the cokernel $\coker(f)$ is finitely generated unit.
\end{proposition}

\begin{proof}
Since $N$ is finitely generated as a left module over $R[F]$, the quotient $\coker(f)$ is also finitely generated as a left module over $R[F]$.
We have a commutative diagram of exact sequences
$$ \xymatrix{ \varphi_{R}^{\ast} M \ar[d]^{\psi_M} \ar[r]^-{ \varphi_{R}^{\ast}(f) } & \varphi_{R}^{\ast} N \ar[d]^{ \psi_N } \ar[r] & \varphi_{R}^{\ast}( \coker(f) ) \ar[d]^{ \psi_{ \coker(f) } }  \ar[r] & 0 \\
M \ar[r]^-{ f} & N \ar[r] & \coker(f) \ar[r] & 0. }$$
Since $\psi_M$ and $\psi_{N}$ are isomorphisms, it follows that $\psi_{ \coker(f) }$ is also an isomorphism.
\end{proof}

\begin{proposition}\label{in2}
Let $R$ be a commutative $\F_p$-algebra and suppose we are given an exact sequence of Frobenius modules
$0 \rightarrow M' \rightarrow M \rightarrow M'' \rightarrow 0$.
If $M'$ and $M''$ are finitely generated unit, then $M$ is finitely generated unit.
\end{proposition}

\begin{proof}
Since the collection of finitely generated left $R[F]$-modules is closed under extensions, the module $M$ is finitely generated over $R[F]$.
It will therefore suffice to show that the map $\psi_M: \varphi_{R}^{\ast} M \rightarrow M$ is an isomorphism. Let $K$ denote the kernel of the map
$\varphi_{R}^{\ast} M \rightarrow \varphi_{R}^{\ast} M''$. The morphism $\psi_M$ fits into a commutative diagram of exact sequences
$$ \xymatrix{ 0 \ar[r] & K \ar[r] \ar[d]^{f} & \varphi_{R}^{\ast} M \ar[r] \ar[d]^{ \psi_M} & \varphi_{R}^{\ast} M'' \ar[r] \ar[d]^{ \psi_{M''} } & 0 \\
0 \ar[r] & M' \ar[r] & M \ar[r] & M'' \ar[r] & 0. }$$
Note that the map $\psi_{M'}$ factors as a composition $\varphi_{R}^{\ast} M' \xrightarrow{g} K \xrightarrow{f} M'$, where
$g$ is surjective. Since $M'$ is finitely generated unit, the map $\psi_{M'}$ is an isomorphism. It follows that $f$ is also an isomorphism.
Applying the five lemma to the preceding diagram, we conclude that $\psi_M$ is also an isomorphism.
\end{proof}

\subsection{Existence of Kernels}\label{sec10sub2}

Our next goal is to prove a counterpart of Proposition \ref{in1} for {\em kernels} of morphisms between finitely generated unit Frobenius modules.
This will require a stronger assumption on $R$:

\begin{proposition}\label{in3}
Let $R$ be a regular Noetherian $\F_p$-algebra and let $f: M \rightarrow N$ be a morphism of Frobenius modules over $R$.
If $M$ and $N$ are finitely generated unit, then $K = \ker(f)$ is also finitely generated unit.
\end{proposition}

The proof of Proposition \ref{in3} is essentially contained in \cite{Ly} (see also \cite{EK}). We include a proof here for the convenience of the reader, and because the proof uses
an auxiliary construction which will play a central role in \S \ref{section.duality}.

\begin{construction}[Unitalization]\label{unitalize}
Let $R$ be a commutative $\F_p$-algebra and let $M$ be an $R$-module equipped with an $R$-linear map $\alpha_M: M \rightarrow \varphi_{R}^{\ast} M$. 
We let $M^{u}$ denote the direct limit of the diagram
$$ M \xrightarrow{\alpha_M} \varphi_{R}^{\ast} M \xrightarrow{ \varphi_{R}^{\ast} \alpha_M} \varphi_{R}^{2 \ast} M \xrightarrow{ \varphi_{R}^{2 \ast} \alpha_M} \varphi_{R}^{3 \ast} M \rightarrow \cdots$$
We will refer to $M^{u}$ as the {\it unitalization} of the pair $(M, \alpha_M)$. Note that there is a canonical isomorphism $M^{u} \simeq \varphi_{R}^{\ast} M^{u}$,
whose inverse endows $M^{u}$ with the structure of a Frobenius module over $R$.
\end{construction}

\begin{example}\label{rexamp}
Let $R$ be a commutative $\F_p$-algebra and let $M$ be a Frobenius module over $R$ for which the map $\psi_{M}: \varphi_{R}^{\ast} M \rightarrow M$
of Notation \ref{notation.psiM} is an isomorphism. Then the unitalization of the pair $(M, \psi_{M}^{-1} )$ can be identified with $M$.
\end{example}

\begin{remark}[Functoriality]\label{pilax}
Let $R$ be a commutative $\F_p$-algebra and suppose we are given a commutative diagram of $R$-modules
$$ \xymatrix{ M \ar[r]^{f} \ar[d]^{ \alpha_M} & N \ar[d]^{ \alpha_N } \\
\varphi_{R}^{\ast} M \ar[r]^{ \varphi_{R}^{\ast}(f) } & \varphi_{R}^{\ast} N. }$$
Then $f$ induces a map of unitalizations $f^{u}: M^{u} \rightarrow N^{u}$. Moreover:
\begin{itemize}
\item The cokernel of $f^{u}$ can be identified with the unitalization of $\coker(f)$ (with respect to the induced map
$\alpha_{ \coker(f) }: \coker(f) \rightarrow \coker( \varphi_{R}^{\ast} f) \simeq \varphi_{R}^{\ast} \coker(f)$).

\item If the Frobenius map $\varphi_{R}: R \rightarrow R$ is flat (for example, if $R$ is regular and Noetherian),
then the kernel of $f^{u}$ can be identified with the unitalization of $\ker(f)$ (with respect to the map
$\ker(f) \rightarrow \ker( \varphi_{R}^{\ast} f) \simeq \varphi_{R}^{\ast} \ker(f)$).
\end{itemize}
\end{remark}

Let $R$ be a commutative $\F_p$-algebra and let $R[F]$ denote the noncommutative ring of Notation \ref{not2}. For any
$R$-module $M$, we have a canonical isomorphism 
$$ R[F] \otimes_{R} M \simeq M \oplus \varphi_{R}^{\ast} M \oplus \varphi_{R}^{2 \ast} M \oplus \cdots.$$
Suppose that $M$ is equipped with a map $\alpha_M: M \rightarrow \varphi_{R}^{\ast} M$. Then the construction
$x \mapsto (x,- \alpha_M(x))$ determines an $R$-linear map 
$M \rightarrow M \oplus \varphi_{R}^{\ast} M \subseteq R[F] \otimes_{R} M$,
which extends to an $R[F]$-linear map $\alpha': R[F] \otimes_R M \rightarrow R[F] \otimes_R M$.
A simple calculation shows that the map $\alpha'$ is a monomorphism with cokernel $M^{u}$. We therefore obtain the following:

\begin{proposition}\label{stopmake}
Let $R$ be a commutative $\F_p$-algebra and let $M$ be an $R$-module equipped with an $R$-linear map $\psi_M: M \rightarrow \varphi_{R}^{\ast} M$.
Then the preceding construction determines an exact sequence of Frobenius modules
$0 \rightarrow R[F] \otimes_{R} M \rightarrow R[F] \otimes_{R} M \rightarrow M^{u} \rightarrow 0$.
\end{proposition}

\begin{corollary}\label{corollary.sevex} 
Let $R$ be a commutative $\F_p$-algebra and let $M$ be a finitely generated $R$-module equipped with an $R$-linear map
$\alpha_M: M \rightarrow \varphi_{R}^{\ast} M$. Then the unitalization $M^{u}$ is a finitely generated unit Frobenius module.
\end{corollary}

\begin{proof}
Condition $(a)$ of Definition \ref{definition.fgu} follows from Proposition \ref{stopmake}, and condition $(b)$ is immediate from the construction.
\end{proof}

We will be primarily interested in the following special case of Construction \ref{unitalize}:

\begin{construction}\label{construction.dualunit}
Let $R$ be a commutative $\F_p$-algebra and let $M$ be a Frobenius module over $R$, which we regard as an $R$-module equipped with an $R$-linear map
$\psi_{M}: \varphi_{R}^{\ast} M \rightarrow M$ (Notation \ref{notation.psiM}). Suppose that $M$ is finitely generated and projective as an $R$-module, with $R$-linear dual
$M^{\vee} = \Hom_{R}(M, R)$. Then the dual of $\psi_{M}$ is an $R$-linear map $\psi_{M}^{\vee}: M^{\vee} \rightarrow \varphi_{R}^{\ast} M^{\vee}$. We let $\mathbb{D}(M)$ denote the unitalization
of the pair $(M^{\vee}, \psi_{M}^{\vee} )$.
\end{construction}

\begin{example}\label{example.makedual}
Let $R$ be a commutative $\F_p$-algebra and let $M$ be a Frobenius module over $R$. Suppose that $M$ is a projective $R$-module of finite rank and that the map
$\psi_{M}: \varphi_{R}^{\ast} M \rightarrow M$ is an isomorphism. In this case, the Frobenius module $\mathbb{D}(M)$ of Construction \ref{construction.dualunit} can be identified
with the $R$-linear dual $M^{\vee}$, endowed with the Frobenius structure characterized by the formula $\psi_{M^{\vee} } = ( \psi_{M}^{\vee} )^{-1}$ (Example \ref{rexamp}).
\end{example}

\begin{proposition}\label{stopmakecor}
Let $R$ be a commutative $\F_p$-algebra and let $M$ be a Frobenius module over $R$ which is finitely generated and projective as an $R$-module.
Then $\mathbb{D}(M)$ has projective dimension $\leq 1$ as a left $R[F]$-module.
\end{proposition}

\begin{proof}
Use the exact sequence $0 \rightarrow R[F] \otimes_{R} M^{\vee} \rightarrow R[F] \otimes_{R} M^{\vee} \rightarrow \mathbb{D}(M) \rightarrow 0$
supplied by Proposition \ref{stopmake}.
\end{proof}

\begin{remark}\label{remark.flatness}
In the situation of Construction \ref{construction.dualunit}, the $R$-module $\mathbb{D}(M)$ is presented as a filtered direct limit of
projective $R$-modules of finite rank, and is therefore flat over $R$.
\end{remark}

In the case where $R$ is a smooth algebra over a field $k$, Emerton and Kisin prove a converse to Corollary \ref{corollary.sevex}: every finitely generated unit
Frobenius module arises as the unitalization of a finitely generated $R$-module $M$, equipped with some map $\alpha_M: M \rightarrow \varphi_{R}^{\ast} M$. 
The proof given in \cite{EK} applies more generally whenever $R$ is a regular Noetherian $\F_p$-algebra (Corollary \ref{corollary.sevey}). We begin with an observation which
is valid for {\em any} $\F_p$-algebra $R$:

\begin{proposition}\label{proposition.key2}
Let $R$ be a commutative $\F_p$-algebra and let $M$ be a finitely generated unit Frobenius module over $R$. 
Then there exists a Frobenius module $N$ over $R$ which is finitely generated and free as an $R$-module
and a surjective map of Frobenius modules $f: \mathbb{D}(N) \rightarrow M$ (here $\mathbb{D}(N)$ is the Frobenius
module given by Construction \ref{construction.dualunit}).
\end{proposition}

\begin{proof}
Choose a finite collection of elements $\{ x_i \}_{ i \in I}$ of $M$ which generate $M$ as a left $R[F]$-module.
Invoking the assumption that the map $\psi_M: \varphi_{R}^{\ast} M \rightarrow M$ is an isomorphism, we conclude that
$M$ is generated as an $R$-module by the elements $F^{k} x_j$ for $k > 0$ and $j \in I$. We may therefore
choose some integer $n > 0$ such that each $x_{i}$ belongs to the $R$-submodule of $M$ generated by the elements
$\{ F^{k} x_j \}_{j \in J, 1 \leq k \leq n}$. Replacing the set $\{ x_i \}_{i \in I}$ by the finite set $\{ F^{k} x_i \}_{i \in I, k < n}$,
we can reduce to the case $n = 1$: that is, we can arrange that there are relations
$x_{i} = \sum_{j \in I} a_{i,j} \varphi_M( x_j )$ for some coefficients $a_{i,j} \in R$.
Let $N = R^{I}$ be the free $R$-module on generators $y_i^{\vee}$ for $i \in I$, and equip $N$ with the structure of a Frobenius module by setting
$\varphi_{N}( y_{i}^{\vee} ) = \sum_{j \in I} a_{j,i} y_{j}^{\vee}$. Using Proposition \ref{stopmake} (or by inspection), we see that
$\mathbb{D}(N)$ can be identified with the left $R[F]$-module generated by symbols $\{ y_i \}_{i \in I}$, subject to the relations
$y_i = \sum_{j \in I} a_{i,j} F x_j$. It follows that there is a unique morphism of Frobenius modules
$f: \mathbb{D}(N) \rightarrow M$ satisfying $f( y_i ) = x_i$. Since the elements $x_i$ generate $M$ as an $R[F]$-module, the morphism
$f$ is surjective.
\end{proof}

\begin{proof}[Proof of Proposition \ref{in3}]
Let $R$ be a regular Noetherian $\F_p$-algebra and let $f: M \rightarrow N$ be a morphism of finitely generated unit Frobenius modules over $R$.
We wish to show that the kernel $K = \ker(f)$ is also finitely generated unit.
The regularity of $R$ guarantees that the Frobenius morphism $\varphi_{R}: R \rightarrow R$ is flat. It follows that
we can identify the pullback $\varphi_{R}^{\ast} K$ with the kernel of the induced map $\varphi_{R}^{\ast}(f): \varphi_{R}^{\ast} M \rightarrow \varphi_{R}^{\ast} N$.
We therefore have a commutative diagram of short exact sequences
$$ \xymatrix{ 0 \ar[r] & \varphi_{R}^{\ast} K \ar[d]^{ \psi_K } \ar[r] & \varphi_{R}^{\ast} M \ar[d]^{ \psi_M } \ar[r]^-{ \varphi_{R}^{\ast}(f) } & \varphi_{R}^{\ast} N \ar[d]^{\psi_N } \\
0 \ar[r] & K \ar[r] & M \ar[r]^{f} & N. } $$
Since $\psi_M$ and $\psi_N$ are isomorphisms, it follows that $\psi_K$ is also an isomorphism.

We now complete the proof by showing that $K$ is finitely generated as a left $R[F]$-module. Using Proposition \ref{proposition.key2}, we can
choose a finitely generated projective $R$-module $M_0$, a map $\alpha_{M_0}: M_0 \rightarrow \varphi_{R}^{\ast} M_0$, and a surjection
of Frobenius modules $g: M_0^{u} \rightarrow M$. It follows that the induced map $\ker(g \circ f) \rightarrow K$ is also surjective. We may therefore replace $f$ by $g \circ f$ and thereby reduce to the case
$M = M_0^{u}$. Let $f_0$ denote the composition of $f$ with the tautological map $M_0 \rightarrow M_0^{u} \simeq M$.
Applying Remark \ref{pilax} to the commutative diagram
$$ \xymatrix{ M_0 \ar[d]^{ \alpha_{M_0} } \ar[r]^{f_0} & N \ar[d]^{ \psi_N^{-1} } \\
\varphi_{R}^{\ast} M_0 \ar[r] & \varphi_{R}^{\ast} N, }$$
we deduce that the kernel of $f$ can be identified with the unitalization of $\ker(f_0)$. Since $R$ is Noetherian, the kernel $\ker(f_0)$ is finitely generated as an $R$-module,
so that $\ker(f)$ is finitely generated as an $R[F]$-module by virtue of Proposition \ref{stopmake}.
\end{proof}

\begin{corollary}[\cite{EK}]\label{corollary.sevey}
Let $R$ be a regular Noetherian $\F_p$-algebra and let $M$ be a finitely generated unit Frobenius module over $R$. Then there
exists an isomorphism $M \simeq M_0^{u}$, where $M_0$ is a finitely generated $R$-module equipped with a map $\alpha: M_0 \rightarrow \varphi_{R}^{\ast} M_0$.
\end{corollary}

\begin{proof}
Using Proposition \ref{proposition.key2}, we can choose a surjection of Frobenius modules $f: N^{u} \rightarrow M$, where $N$ is a free $R$-module of finite
rank equipped with a map $\alpha_N: N \rightarrow \varphi_{R}^{\ast} N$. As in the proof of Proposition \ref{in3}, we can write $\ker(f) \simeq K^{u}$,
where $K$ denotes the kernel of the composite map $N \rightarrow N^{u} \xrightarrow{f} M$ and $\alpha_{K}: K \rightarrow \varphi_{R}^{\ast} K$ is the restriction of
$\alpha_N$. Applying Remark \ref{pilax} to the diagram
$$ \xymatrix{ K \ar[d]^{ \alpha_K } \ar[r] & N \ar[d]^{ \alpha_N} \\
\varphi_{R}^{\ast} K \ar[r] & \varphi_{R}^{\ast} N, }$$
we deduce that $M \simeq \coker( K^{u} \rightarrow N^{u} )$ can be identified with the unitalization of the quotient $N/K$ 
(with respect to the map $\alpha_{N/K}: N/K \rightarrow \varphi_{R}^{\ast}(N/K)$ induced by $\alpha_N$).
\end{proof}

\subsection{Finitely Generated Unit Complexes}\label{sec10sub4}

Our next goal is to introduce an analogue of Definition \ref{definition.fgu} for {\em cochain complexes} $M = M^{\ast}$ of Frobenius modules over a commutative $\F_p$-algebra $R$.
When $R$ is a regular Noetherian $\F_p$-algebra, the collection of finitely generated unit Frobenius modules span an abelian subcategory of $\Mod_{R}^{\Frob}$ which is closed under
extensions (Propositions \ref{in1}, \ref{in2}, and \ref{in3}), so we obtain a sensible finiteness condition on cochain complexes by requiring that the cohomology groups
$\mathrm{H}^{n}(M)$ are finitely generated unit. However, to get a theory which works well for {\em arbitrary} $\F_p$-algebras, we must abandon the idea of having a finiteness condition
that can be tested at the level of individual cohomology groups: instead, we will require that the entire cochain complex $M^{\ast}$ satisfies suitable analogues of conditions $(a)$
and $(b)$ of Definition \ref{definition.fgu}, when regarded as an object of a suitable derived category.

\begin{notation}
For every associative ring $A$, we let $D(A)$ denote the derived category of the abelian category of left $A$-modules. 
We will be particularly interested in the case where $A = R[F]$ for some commutative $\F_p$-algebra $R$; in this case, we refer
to $D( R[F] )$ as the {\it derived category of Frobenius modules over $R$}. We will generally abuse notation by identifying
$\Mod_{R}^{\Frob}$ with its essential image in $D( R[F] )$ (by regarding every Frobenius module over $R$ as a chain complex concentrated in degree zero).
\end{notation}

\begin{remark}
Let $R$ be a commutative $\F_p$-algebra and let $M$ be an object of $D( R[F] )$. We will generally abuse notation by identifying
$M$ with its image under the forgetful functor $D(R[F]) \rightarrow D(R)$. Note that we have a canonical map
$M \rightarrow M^{1/p}$ in $D(R)$ (where $M^{1/p}$ denotes the image of $M$ under the functor $D(R) \rightarrow D(R)$ given
by restriction of scalars along the Frobenius). We will denote this map by $\varphi_{M}$ and refer to it as {\it the Frobenius morphism} of $M$.
\end{remark}

\begin{remark}[Comparison with $D(R)$]\label{remark.computeExt}
Let $R$ be a commutative $\F_p$-algebra. Recall
that the forgetful functor $\Mod_{R}^{\Frob} \rightarrow \Mod_{R}$ has an exact right adjoint, given by the functor
$M \mapsto M^{\dagger}$ of Construction \ref{conX72}. Passing to derived categories,
we see that the forgetful functor $D(R[F]) \rightarrow D(R)$ also has a right adjoint, given by applying
the functor $M \mapsto M^{\dagger}$ levelwise. For any cochain complex $N = N^{\bullet}$ of Frobenius modules over $R$,
Construction \ref{conX7} produces a short exact sequence of cochain complexes
$0 \rightarrow N^{\bullet} \rightarrow N^{\bullet \dagger} \rightarrow (N^{\bullet})^{1/p \dagger} \rightarrow 0$,
which we can regard as a distinguished triangle in the derived category $D(R[F])$. It follows that for any object $M \in D(R[F])$, we have
a long exact sequence
$$ \cdots \rightarrow \Hom_{ D(R) }( M, N^{1/p}[-1] ) \rightarrow
\Hom_{ D(R[F])}( M, N) \rightarrow \Hom_{ D(R) }( M, N) \rightarrow \cdots,$$
which specializes to the exact sequence of Construction \ref{conX7} in the special case where $M$ and $N$ are concentrated in a single degree.
\end{remark}

We now introduce a ``derived'' analogue of Definiton \ref{definition.fgu}:

\begin{definition}\label{definition.derived-fgu}
Let $R$ be a commutative $\F_p$-algebra and let $M$ be an object of $D(R[F])$. We will say that $M$ is {\it derived finitely generated unit} if it satisfies
the following pair of conditions:
\begin{itemize}
\item[$(a)$] The module $M$ is a compact object of the triangulated category $D(R[F])$: that is, it is quasi-isomorphic to a bounded chain complex
of finitely generated projective left $R[F]$-modules.

\item[$(b)$] The Frobenius map $\varphi_{M}: M \rightarrow M^{1/p}$ induces an isomorphism $R^{1/p} \otimes_{R}^{L} M \rightarrow M$ in the derived category $D(R)$.
\end{itemize}
We let $D^{b}_{\fgu}( R[F] )$ denote the full subcategory of $D(R[F])$ spanned by the derived finitely generated unit objects.
\end{definition}

\begin{remark}
Let $R$ be a commutative $\F_p$-algebra. Then $D^{b}_{\fgu}( R[F] )$ is a triangulated subcategory of $D(R[F] )$. In other words,
for any distinguished triangle $M' \rightarrow M \rightarrow M'' \rightarrow M'[1]$ in $D( R[F] )$, if any two of the objects
$M$, $M'$, and $M''$ are derived finitely generated unit, then so it the third.
\end{remark}

\begin{example}\label{sevex2}
Let $R$ be a commutative $\F_p$-algebra and let $M$ be a Frobenius module over $R$ which is finitely generated and projective as an $R$-module.
Then the Frobenius module $\mathbb{D}(M)$ of Construction \ref{construction.dualunit} belongs to $D^{b}_{\fgu}( R[F] )$. Condition
$(b)$ of Definition \ref{definition.derived-fgu} follows from Corollary \ref{corollary.sevex} (note that the derived pullback
$R^{1/p} \otimes_{R}^{L} \mathbb{D}(M)$ agrees with $\varphi_{R}^{\ast} \mathbb{D}(M)$, since $\mathbb{D}(M)$ is flat over $R$
by virtue of Remark \ref{remark.flatness}). Condition $(a)$ of Definition \ref{definition.derived-fgu} follows from
the exact sequence $0 \rightarrow R[F] \otimes_{R} M^{\vee} \rightarrow R[F] \otimes_{R} M^{\vee} \rightarrow \mathbb{D}(M) \rightarrow 0$
of Proposition \ref{stopmake}.
\end{example}

We now study the relationship between the collection of derived finitely generated unit objects of $D(R[F])$ and the collection of finitely
generated unit objects of $\Mod_{R}^{\Frob}$. We begin with a simple observation which is valid for any $\F_p$-algebra $R$:

\begin{proposition}\label{proposition.key1}
Let $R$ be a commutative $\F_p$-algebra and let $M$ be a nonzero object of $D^{b}_{\fgu}( R[F] )$. Then:
\begin{itemize}
\item[$(1)$] There exists a largest integer $n$ for which the Frobenius module $\mathrm{H}^{n}(M)$ is nonzero. 
\item[$(2)$] For the integer $n$ of $(1)$, the Frobenius module $\mathrm{H}^{n}(M)$ is finitely generated unit.
\end{itemize}
\end{proposition}

\begin{proof}
Without loss of generality, we can assume that $M$ is a bounded cochain complex of finitely generated projective left $R[F]$-modules.
Assertion $(1)$ is immediate. To prove $(2)$, we first note that we can arrange (replacing $M$ by a quasi-isomorphic complex if necessary)
that $M^{m} = 0$ for $m > n$; in this case, we have $\mathrm{H}^{n}(M) = \coker( M^{n-1} \rightarrow M^{n} )$, which guarantees
that $\mathrm{H}^{n}(M)$ is finitely generated as a left $R[F]$-module. The spectral sequence
$$ \Tor^{R}_{s}( R^{1/p}, \mathrm{H}^{t}(M) ) \Rightarrow \mathrm{H}^{t-s}( R^{1/p} \otimes_{R}^{L} M )$$
supplies an isomorphism $\mathrm{H}^{n}( R^{1/p} \otimes_{R}^{L} M ) \simeq \varphi_{R}^{\ast} \mathrm{H}^{n}(M)$, so that
condition $(b)$ of Definition \ref{definition.fgu} follows from condition $(b)$ of Definition \ref{definition.derived-fgu}.
\end{proof}

\begin{corollary}\label{corollary.stepmap}
Let $R$ be a commutative $\F_p$-algebra and let $M$ be a nonzero object of $D^{b}_{\fgu}( R[F] )$,
and let $n$ be an integer for which the cohomology groups $\mathrm{H}^{m}(M)$ vanish for $m > n$.
Then there exists an object $N \in \Mod_{R}^{\Frob}$ which is finitely generated and projective as
an $R$-module and a map $f: \mathbb{D}(N)[-n] \rightarrow M$ in $D(R[F])$ for which
the induced map $\mathbb{D}(N) \rightarrow \mathrm{H}^{n}(M)$ is surjective.
\end{corollary}

\begin{proof}
Applying Proposition \ref{proposition.key1}, we conclude that $\mathrm{H}^{n}(M)$ a finitely generated unit,
Using Proposition \ref{proposition.key2},
we can choose an object $N \in \Mod_{R}^{\Frob}$ which is finitely generated and projective as an $R$-module and a surjection
of Frobenius modules $g_0: \mathbb{D}(N) \rightarrow \mathrm{H}^{n}(M)$. It follows from Proposition
\ref{stopmakecor} that the map $\Hom_{ D(R[F])}( \mathbb{D}(N)[-n], M) \rightarrow
\Hom_{ R[F] }( \mathbb{D}(N), \mathrm{H}^{n}(M) )$ is surjective, so we can lift $g_0$ to a morphism
$g: \mathbb{D}(N)[-n] \rightarrow M$ in the derived category $D(R[F])$. 
\end{proof}

When the $\F_p$-algebra $R$ is sufficiently nice, there is a very close relationship between Definitions \ref{definition.fgu} and \ref{definition.derived-fgu}:

\begin{proposition}\label{proposition.char-fgu}
Let $R$ be a regular Noetherian $\F_p$-algebra. Then an object $M \in D( R[F] )$ belongs to $D^{b}_{\fgu}( R[F] )$ if and only if it satisfies the following conditions:
\begin{itemize}
\item[$(1)$] For every integer $n$, the cohomology group $\mathrm{H}^{n}(M)$ is a finitely generated unit Frobenius module (in the sense of Definition \ref{definition.fgu}).
\item[$(2)$] The cohomology groups $\mathrm{H}^{n}(M)$ vanish for $n \ll 0$ and for $n \gg 0$.
\end{itemize}
\end{proposition}

The proof of Proposition \ref{proposition.char-fgu} will require a few preliminary remarks.

\begin{lemma}\label{lemma.fpd}
Let $R$ be a regular Noetherian ring and let $M$ be a finitely generated $R$-module. Then $M$ has finite projective dimension as an $R$-module.
\end{lemma}

\begin{remark}
In the situation of Lemma \ref{lemma.fpd}, it is not necessarily true that {\em every} $R$-module has finite projective dimension: this holds if and only if $R$
has finite Krull dimension.  
\end{remark}

\begin{proof}[Proof of Lemma \ref{lemma.fpd}]
We define finitely generated $R$-modules $\{ M(n) \}_{n \geq 0}$ by recursion as follows: set $M(0) = M$, and for $n > 0$ let
$M(n)$ denote the kernel of some surjection $R^{k} \rightarrow M(n-1)$. For each $n \geq 0$, let $U(n) \subseteq \Spec(R)$ denote the set of prime ideals
$\mathfrak{p} \subseteq R$ for which the localization $M(n)_{\mathfrak{p} }$ is a projective $R$-module. Then $U(n)$ is an open subset of $\Spec(R)$:
more precisely, it is the largest open subset on which the coherent sheaf associated to $M(n)$ is locally free. Note that a point $\mathfrak{p}$ belongs
to $U(n)$ if and only if the localization $M_{\mathfrak{p}}$ has projective dimension $\leq n$ as an $R_{\mathfrak{p} }$-module. Since $R$ is regular, the set $U(n)$ contains every prime ideal of height $\leq n$. 
We therefore have $\bigcup_{n \geq 0} U(n) = \Spec(R)$. Since the spectrum $\Spec(R)$ is quasi-compact, we must have $U(n) = \Spec(R)$ for some $n \gg 0$, which guarantees that
$M$ has projective dimension $\leq n$.
\end{proof}

\begin{lemma}\label{lemma.fpd2}
Let $R$ be a regular Noetherian $\F_p$-algebra and let $M$ be a finitely generated unit Frobenius module over $R$. Then $M$ has finite projective dimension
as a left $R[F]$-module.
\end{lemma}

\begin{proof}
Using Corollary \ref{corollary.sevey}, we can assume that $M$ is the unitalization $M_0^{u}$, where $M_0$ is a finitely generated
$R$-module equipped with a map $\alpha: M_0 \rightarrow \varphi_{R}^{\ast} M_0$. Invoking Lemma \ref{lemma.fpd}, we deduce that there exists an integer $n \gg 0$ such that $M_0$ has finite projective dimension $\leq n$ as an $R$-module. 
Note that $R[F]$ is isomorphic to the direct sum $\bigoplus_{m \geq 0} R^{1/p^m}$ as a {\em right} $R$-module, and is therefore flat over $R$ since the Frobenius map
$\varphi_{R}$ is flat. It follows that the tensor product $R[F] \otimes_{R} M_0$ has projective dimension $\leq n$ as a left $R[F]$-module. Using the exact sequence
$0 \rightarrow R[F] \otimes_{R} M_0 \rightarrow R[F] \otimes_{R} M_0 \rightarrow M_0^{u} \rightarrow 0$ of Proposition \ref{stopmake}, we see that
$M_0^{u} \simeq M$ has projective dimension $\leq n+1$ as a left $R[F]$-module. 
\end{proof}

\begin{proof}[Proof of Proposition \ref{proposition.char-fgu}]
Let $R$ be a commutative $\F_p$-algebra and let $M \in D( R[F] )$. We first show that if $M$ satisfies conditions $(1)$ and $(2)$, then $M$ is derived finitely generated unit. 
Since $D_{\fgu}^{b}( R[F] )$ is a triangulated subcategory of $D(R[F])$, we may assume without loss of generality that $M$ is a finitely generated unit Frobenius module, regarded as a cochain complex concentrated in a single degree.
Then the map $\psi_M: \varphi_{R}^{\ast} M \rightarrow M$ is an isomorphism. Since $R$ is a regular Noetherian $\F_p$-algebra, the Frobenius morphism $\varphi_{R}: R \rightarrow R$ is flat;
we may therefore identify $\varphi_{R}^{\ast} M$ with the derived pullback $R^{1/p} \otimes_{R}^{L} M$. It follows that $M$ satisfies condition $(b)$ of Definition \ref{definition.derived-fgu}). We now verify
$(a)$. Using Lemma \ref{lemma.fpd2}, we see that there exists an integer $n \geq 1$ that $M$ has projective dimension $\leq n$ as a left module over $R[F]$. We proceed by induction on $n$.
Assume first that $n > 1$. Using Proposition \ref{proposition.key2}, we can choose a short exact sequence of Frobenius modules
$$ 0 \rightarrow K \rightarrow \mathbb{D}(N) \rightarrow M \rightarrow 0$$
where $N \in \Mod_{R}^{\Frob}$ is finitely generated and projective as an $R$-module. Since $\mathbb{D}(N)$ has projective dimension $\leq 1$ over $R[F]$ (Proposition \ref{stopmakecor}), it follows that 
$K$ has projective dimension $\leq n-1$ as a left module over $R[F]$. Using Proposition \ref{in3}, we
deduce that $K$ is a finitely generated unit Frobenius module. Applying our inductive hypothesis, we conclude that $K$ belongs to $D_{\fgu}^{b}( R[F] )$.
Since $\mathbb{D}(N)$ also belongs to $D_{\fgu}^{b}( R[F] )$ (Example \ref{sevex2}), we conclude that $M$ belongs to $D_{\fgu}^{b}( R[F] )$ as desired.

We now treat the case where $M$ has projective dimension $\leq 1$ over $R[F]$. Choose an exact sequence of Frobenius modules $0 \rightarrow Q \rightarrow P \rightarrow M \rightarrow 0$,
where $P$ is a finitely generated free left $R[F]$-module. Our assumption that $M$ has projective dimension $\leq 1$ guarantees that $Q$ is a projective $R[F]$-module.
Consequently, to verify condition $(a)$ of Definition \ref{definition.derived-fgu}, it will suffice to show that $Q$ is finitely generated as an $R[F]$-module. Equivalently, it will suffice
to show that the module $M$ is finitely presented as an $R[F]$-module. This follows from the exact sequence $0 \rightarrow K \rightarrow \mathbb{D}(N) \rightarrow M \rightarrow 0$
above, since $\mathbb{D}(N)$ is a finitely presented left $R[F]$-module (Proposition \ref{stopmake}) and $K$ is a finitely presented left $R[F]$-module (Proposition \ref{in3}).

We now prove the converse. Suppose that $M$ is an object of $D_{\fgu}^{b}( R[F] )$; we wish to show that $M$ satisfies conditions $(1)$ and $(2)$.
Condition $(2)$ is obvious (since $M$ is quasi-isomorphic to a bounded chain complex of projective left modules over $R[F]$). Suppose that condition $(1)$ fails:
then there exists some largest integer $n$ such that $\mathrm{H}^{n}(M)$ is not a finitely generated unit Frobenius module. Form a distinguished triangle
$M' \xrightarrow{f} M \rightarrow M'' \rightarrow M'[1]$, where $f$ induces an isomorphism $\mathrm{H}^{k}(M') \rightarrow \mathrm{H}^{k}(M)$ for $k \leq n$, and the groups $\mathrm{H}^{k}(M')$ vanish for $k > n$.
Then $M''$ satisfies conditions $(1)$ and $(2)$, and therefore belongs to $D_{\fgu}^{b}( R[F] )$. It follows that $M'$ also belongs to
$D_{\fgu}^{b}( R[F] )$. This contradicts Proposition \ref{proposition.key1}, since the top cohomology group
$\mathrm{H}^{n}(M') \simeq \mathrm{H}^{n}(M)$ is not a finitely generated unit Frobenius module.
\end{proof}

\subsection{The Emerton-Kisin Correspondence}\label{sec10sub5}

We now introduce a variant of Construction \ref{construction.solsheaf}.

\begin{construction}\label{solsheafEK}
Let $R$ be a commutative $\F_p$-algebra and let $M$ be a Frobenius module over $R$. We define a functor
$$ \Sol_{\EK}( M ): \CAlg_{R}^{\mathet} \rightarrow \Mod_{\F_p}$$
by the formula $\Sol_{\EK}(M)(A) = \Hom_{ R[F] }( M, A)$. It is not difficult to see that the functor $\Sol_{\EK}(M)$ is a sheaf for the
{\etale} topology, which is contravariantly functorial in $M$. We can therefore regard the functor $M \mapsto \Sol_{\EK}(M)$ as a functor
of abelian categories $\Sol_{\EK}: ( \Mod_{R}^{\Frob} )^{\op} \rightarrow \Shv_{\mathet}( \Spec(R), \F_p)$. We will refer to
$\Sol_{\EK}$ as the {\it Emerton-Kisin solution functor}.
\end{construction}

\begin{example}
Let $R$ be a commutative $\F_p$-algebra. Then the Emerton-Kisin solution functor $\Sol_{\EK}$ carries the Frobenius module $R[F]$
to the structure sheaf of $\Spec(R)$: that is, to the quasi-coherent sheaf $\widetilde{R}$ of Example \ref{exX70}.
\end{example}

\begin{construction}\label{construction.SolEK}
Let $R$ be a commutative $\F_p$-algebra. We let $D^{-}( R[F] )$ denote the subcategory of $D( R[F] )$ spanned by those cochain complexes
$M$ which are cohomologically bounded above: that is, which satisfy $\mathrm{H}^{n}(M) \simeq 0$ for $n \gg 0$. Note that
$D^{-}( R[F] )$ contains the subcategory $D^{b}_{\fgu}( R[F] ) \subseteq D(R[F] )$ of Definition \ref{definition.derived-fgu}.

Let $D_{\mathet}( \Spec(R), \F_p )$ denote the derived category of the abelian category $\Shv_{\mathet}( \Spec(R), \F_p)$.
It follows immediately from the definitions that the Emerton-Kisin solution functor
$\Sol_{\EK}: ( \Mod_{R}^{\Frob} )^{\op} \rightarrow \Shv_{\mathet}( \Spec(R), \F_p)$ is left exact.
It therefore admits a right derived functor $$\RSol_{\EK}: D^{-}( R[F] )^{\op} \rightarrow D_{\mathet}( \Spec(R), \F_p).$$
\end{construction}

We can now formulate the main result:

\begin{theorem}\label{strongEK}
Let $R$ be a commutative $\F_p$-algebra. Then the functor
$\RSol_{\EK}: D^{-}( R[F] )^{\op} \rightarrow D_{\mathet}( \Spec(R), \F_p)$ induces a fully faithful embedding
$$ D^{b}_{\fgu}( R[F] )^{\op} \rightarrow D_{\mathet}( \Spec(R), \F_p ),$$
whose essential image is the subcategory $D_{c}^{b}( \Spec(R), \F_p ) \subseteq D_{\mathet}( \Spec(R), \F_p )$
spanned by those complexes of sheaves which are cohomologically bounded with constructible cohomology sheaves.
\end{theorem}

\begin{remark}\label{remark.normalization}
In the special case where $R$ is a smooth algebra of finite type over a field $k$, Theorem \ref{strongEK} essentially follows from Theorem 11.3 of \cite{EK}, applied to the affine scheme
$X = \Spec(R)$. Beware that the functor $\RSol_{\EK}$ is not quite the same as the functor appearing in \cite{EK}: they differ by a cohomological shift by the dimension
of $X$. Of course, this is an issue of normalization and has no effect on the conclusion of Theorem \ref{strongEK}.
\end{remark}

We will give a proof of Theorem \ref{strongEK} in \S \ref{section.duality} by developing a theory of duality for Frobenius modules, which will allow us to relate 
$\RSol_{\EK}$ to the solution functor $\Sol$ studied earlier in this paper (see \S \ref{section.proofhappens}).

\newpage

\section{Duality for Frobenius Modules}\label{section.duality}

Let $R$ be a commutative $\F_p$-algebra. Our goal in this section is to prove Theorem \ref{strongEK} by showing that the functor
$$ \RSol_{\EK}: D^{b}_{\fgu}( R[F] )^{\op} \rightarrow D^{b}_{c}(\Spec(R); \F_p)$$
is an equivalence of triangulated categories. Our strategy is to construct a commutative diagram of triangulated categories
$\sigma:$
$$ \xymatrix{ & D^{b}_{\hol}( R[F] ) \ar[dl]_{\mathbb{D}} \ar[dr]^{ \RSol} & \\
D^{b}_{\fgu}( R[F] )^{\op} \ar[rr]^{ \RSol_{\EK} } & & D^{b}_{c}( \Spec(R); \F_p ), }$$
where $D^{b}_{\hol}( R[F] )$ is the {\em holonomic} derived category of Frobenius modules, $\RSol$ is a derived version of the solution functor of Construction \ref{solsheaf}, and 
$\mathbb{D}$ is a form of $R$-linear duality.

We begin in \S \ref{section.derivedRH} with a general discussion of the derived category of Frobenius modules $D( R[F] )$; in particular, we define
the subcategory $D^{b}_{\hol}( R[F] )$, the functor $\RSol$, and show that it is an equivalence of categories (Corollary \ref{corollary.RHDerived}). 
This is essentially a formal consequence of the analogous equivalence at the level of abelian categories (Theorem \ref{companion}), since we have
already shown that the Riemann-Hilbert correspondence is compatible with the formation of $\Ext$-groups (Corollary \ref{corX80}). 

Most of this section is devoted to the study of the duality functor $\mathbb{D}$. In \S \ref{section.weakdual}, we introduce the notion of
a {\it weak dual} for an object of the derived category $D( R[F] )$ (Definition \ref{definition.weakdual}). The weak dual of an object
$M \in D( R[F] )$ depends functorially on $M$, provided that it exists: in other words, the formation of weak duals determines
a {\em partially defined} (contravariant) functor from the derived category $D( R[F] )$ to itself. We have already met this functor in a special case:
if $M$ is a Frobenius module which is finitely generated and projective as an $R$-module, then the weak dual of $M$ coincides with the
Frobenius module $\mathbb{D}(M)$ given by Construction \ref{construction.dualunit}. This follows from a universal property of
Construction \ref{construction.dualunit}, which we establish in \S \ref{subsection.dual-properties} (Proposition \ref{proposition.easydual}). 
In \S \ref{biduality}, we exploit this fact to show that every object of $D^{b}_{\hol}( R[F] )$ admits a weak dual (Proposition \ref{proposition.holdual});
the proof uses a characterization of the holonomic derived category which we establish in \S \ref{subsection.presentations} (Theorem \ref{theorem.holochar}).
It follows that the construction $M \mapsto \mathbb{D}(M)$ determines a functor $\mathbb{D}(M): D^{b}_{\hol}( R[F] ) \rightarrow D( R[F] )^{\op}$,
which we prove to be fully faithful with essential image $D_{\fgu}^{b}( R[F] )$ (Theorem \ref{theorem.dual-image}. In \S \ref{section.proofhappens} we show that
the diagram $\sigma$ commutes (up to canonical isomorphism), thereby completing the proof of Theorem \ref{strongEK}.







\subsection{The Derived Riemann-Hilbert Correspondence}\label{section.derivedRH}

Our first goal is to to extend the equivalence $\Sol: \Mod_{R}^{\hol} \rightarrow \Shv_{\mathet}( \Spec(R), \F_p )$ of Theorem \ref{companion} to the level of derived categories. We begin by establishing some notation.

\begin{notation}\label{notation.derived}
Let $R$ be a commutative $\F_p$-algebra. 
We define subcategories
$$ D_{\hol}( R[F] ) \subseteq D_{ \alg}( R[F] ) \subseteq D_{ \perf}( R[F] ) \subseteq D(R[F])$$
as follows: an object $M \in D(R[F])$ belongs to the subcategory $D_{\perf}( R[F] )$ (respectively $D_{\alg}(\Mod_R^{\Frob} ), D_{\hol}( R[F] )$)
if each cohomology group $\mathrm{H}^{i}(M)$ is perfect (respectively algebraic, holonomic) when regarded as a Frobenius module over $R$.
It follows from Remark \ref{remark.perfectextension}, Proposition \ref{propX53}, and Corollary \ref{corX30}, we see that $D_{\perf}( R[F] )$,
$D_{\alg}( R[F] )$, and $D_{\hol}( R[F] )$ are triangulated subcategories of $D(R[F])$.

We let $D^{+}( R[F] )$ denote the full subcategory of $D(R[F])$ spanned by those objects $M$ which are (cohomologically) bounded below: that is,
for which the cohomology groups $\mathrm{H}^{i}(M)$ vanish for $i \ll 0$. We let $D^{b}( R[F] )$ denote the full subcategory of $D(R[F])$ spanned by those
objects which are bounded above and below: that is, for which the cohomology groups $\mathrm{H}^{i}(M)$ vanish both for $i \ll 0$ and $i \gg 0$. Similarly, we have full subcategories
$$ D_{\hol}^{+}( R[F] ) \subseteq D^{+}_{ \alg}( R[F] ) \subseteq D^{+}_{ \perf}( R[F] ) \subseteq D^{+}( R[F] )$$
$$ D_{\hol}^{b}( R[F] ) \subseteq D^{b}_{ \alg}( R[F] ) \subseteq D^{b}_{ \perf}( R[F] ) \subseteq D^{b}( R[F] )$$
which are defined in the obvious way.
\end{notation}

For any commutative $\F_p$-algebra $R$, the inclusion functor $\Mod_{R}^{\perf} \hookrightarrow \Mod_{R}^{\Frob}$ is exact, and therefore
extends to a functor of derived categories $D( \Mod_{R}^{\perf} ) \rightarrow D(R[F])$. 

\begin{proposition}\label{sekos}
Let $R$ be a commutative $\F_p$-algebra. Then the forgetful functor
$D( \Mod_{R}^{\perf} ) \rightarrow D(R[F])$ is a fully faithful embedding, whose essential image 
is the full subcategory $D_{\perf}( R[F] ) \subseteq D(R[F])$. 
\end{proposition}

\begin{proof}
The inclusion functor $\Mod_{R}^{\perf} \hookrightarrow \Mod_{R}^{\Frob}$ has a left adjoint, given by the perfection functor $M \mapsto M^{\perfection}$
of Notation \ref{biskar}. This functor is exact, and therefore extends to a functor of derived categories $F: D(R[F]) \rightarrow D( \Mod_{R}^{\perf} )$
which is left adjoint to the forgetful functor. It now suffices to observe that the counit map $(F \circ G)(M) \rightarrow M$ is an isomorphism for every object $M \in D( \Mod_{R}^{\perf} )$,
and that the unit map $N \rightarrow (G \circ F)(N)$ is an isomorphism precisely when $N$ belongs to the full subcategory $D_{\perf}( R[F] ) \subseteq D(R[F])$
(since both of these assertions can be checked at the level of cohomology).
\end{proof}

We now wish to compare the derived categories of Notation \ref{notation.derived} with suitable derived categories of {\etale} sheaves.

\begin{notation}\label{notation.derived2}
Let $R$ be any commutative ring. We let $D_{\mathet}( \Spec(R), \F_p )$ denote the derived category of the abelian category
$\Shv_{\mathet}( \Spec(R), \F_p)$ of $p$-torsion {\etale} sheaves on $\Spec(R)$. We define full subcategories
$$ D^{b}_{c}( \Spec(R), \F_p ) \subseteq D^{+}_{\mathet}( \Spec(R), \F_p ) \subseteq D_{\mathet}( \Spec(R), \F_p )$$
as follows:
\begin{itemize}
\item An object $\sheafF \in D_{\mathet}( \Spec(R), \F_p )$ belongs to $D^{+}_{\mathet}( \Spec(R), \F_p )$ if and only if
the cohomology sheaves $\mathrm{H}^{n}( \sheafF)$ vanish for $n \ll 0$.
\item An object $\sheafF \in D_{\mathet}( \Spec(R), \F_p )$ belongs to $D^{b}_{c}( \Spec(R), \F_p )$ if and only if
the cohomology sheaves $\mathrm{H}^{n}( \sheafF)$ are constructible for all $n$ and vanish for $|n | \gg 0$.
\end{itemize}
\end{notation}

If $R$ is a commutative $\F_p$-algebra, then the Riemann-Hilbert functor
$\RH: \Shv_{\mathet}( \Spec(R), \F_p ) \rightarrow \Mod_{R}^{\perf} \subseteq \Mod_{R}^{\Frob}$
is exact, and therefore extends to a functor of derived categories $$\RH: D_{\mathet}( \Spec(R), \F_p) \rightarrow D( \Mod_{R}^{\perf} ) \simeq D_{\perf}( R[F] ).$$
This functor is t-exact, and therefore restricts to a functor $D^{+}_{\mathet}( \Spec(R), \F_p) \rightarrow D^{+}( \Mod_{R}^{\perf} ) \simeq D^{+}_{\perf}( R[F] )$.
This restriction admits a right adjoint $$\RSol: D^{+}_{\perf}( R[F] ) \simeq D^{+}( \Mod_{R}^{\perf} ) \rightarrow D^{+}_{\mathet}( \Spec(R), \F_p),$$ given
by the total right derived functor of $\Sol: \Mod_{R}^{\perf} \rightarrow \Shv_{\mathet}( \Spec(R), \F_p)$.

\begin{remark}\label{remark.computeRSol}
For every object $M \in D^{+}_{\perf}( R[F] )$, we have a hypercohomology spectral sequence
$\Sol^{s}( \mathrm{H}^{t}(M) ) \Rightarrow \mathrm{H}^{s+t}( \RSol(M) )$.
Note that the groups $\Sol^{s}( \mathrm{H}^{t}(M) )$ vanish for $s \geq 2$ (Proposition \ref{prop75} and Theorem \ref{theoX50}, or
Proposition \ref{proposition.corX77gen}), so
this spectral sequence degenerates to yield short exact sequences
$$ 0 \rightarrow \Sol^{1}( \mathrm{H}^{n-1}(M) ) \rightarrow \mathrm{H}^{n} \RSol(M) \rightarrow \Sol( \mathrm{H}^{n}(M) ) \rightarrow 0.$$
If $M$ belongs to the subcategory $D^{+}_{\alg}( R[F] ) \subseteq D^{+}_{\perf}( R[F] )$, then the sheaves
$\Sol^{1}( \mathrm{H}^{n-1}(M) )$ vanish (Proposition \ref{proposition.derived-vanishing}); we therefore obtain isomorphisms
$\mathrm{H}^{\ast}( \RSol(M) ) \simeq \Sol( \mathrm{H}^{\ast}(M) )$.
\end{remark}

\begin{theorem}\label{RHDerived}
Let $R$ be a commutative $\F_p$-algebra. Then the functor 
$\RH: D^{+}_{\mathet}( \Spec(R), \F_p) \rightarrow D^{+}( R[F] )$ is a fully faithful embedding,
whose essential image is the full subcategory $D^{+}_{\alg}( R[F] ) \subseteq D^{+}( R[F] )$.
\end{theorem}

\begin{proof}
Since the Riemann-Hilbert functor $\RH: \Shv_{\mathet}( \Spec(R), \F_p) \rightarrow \Mod_{R}^{\perf}$ is exact at the level of abelian categories,
its extension to the level of derived categories is t-exact: that is, we have canonical isomorphisms $\mathrm{H}^{\ast}( \RH(\sheafF) ) \simeq \RH( \mathrm{H}^{\ast}(\sheafF) )$
for each $\sheafF \in D^{+}_{\mathet}( \Spec(R), \F_p)$. It follows from Theorem \ref{theorem.RHexist} that the functor $\RH$ carries $D^{+}_{\mathet}( \Spec(R), \F_p)$ into $D^{+}_{\alg}( R[F] )$. Combining
this observation with Remark \ref{remark.computeRSol}, we obtain isomorphisms 
$$ \mathrm{H}^{\ast}( (\RSol \circ \RH)(\sheafF) ) \simeq \Sol( \mathrm{H}^{\ast}( \RH(\sheafF) ) ) \simeq (\Sol \circ \RH)( \mathrm{H}^{\ast}(\sheafF) ).$$
It follows from Proposition \ref{prop75} that the unit map $\sheafF \rightarrow (\RSol \circ \RH)(\sheafF)$ is an isomorphism: that is, the derived Riemann-Hilbert functor
is fully faithful. To complete the proof, it will suffice to show that for every object $M \in D^{+}_{\alg}( R[F] )$, the counit map
$(\RH \circ \RSol)(M) \rightarrow M$ is an isomorphism. Applying Remark \ref{remark.computeRSol} again, we obtain isomorphisms
$$ \mathrm{H}^{\ast}( (\RH \circ \RSol)(M) ) \simeq \RH( \mathrm{H}^{\ast}( \RSol(M) ) \simeq (\RH \circ \Sol)( \mathrm{H}^{\ast}(M) ),$$
so that the desired result follows from Theorem \ref{theoX50}.
\end{proof}

\begin{remark}
If $R$ is a Noetherian $\F_p$-algebra of finite Krull dimension, then one can show that the category of {\etale} sheaves $\Shv_{\mathet}( \Spec(R), \F_p)$ has finite injective dimension.
In this case, it is not hard to see that Theorem \ref{RHDerived} can be extended to yield an equivalence $\RH: D_{\mathet}( \Spec(R), \F_p) \rightarrow D(R[F])$
of {\em unbounded} derived categories. We do not know if this holds in general.
\end{remark}

Combining Theorem \ref{RHDerived} with Theorem \ref{companion}, we obtain the following:

\begin{corollary}\label{corollary.RHDerived}
Let $R$ be a commutative $\F_p$-algebra. Then the Riemann-Hilbert functor $\RH: \Shv_{\mathet}( \Spec(R), \F_p) \rightarrow \Mod_{R}^{\Frob}$ induces an equivalence
of triangulated categories $D_{c}^{b}( \Spec(R), \F_p) \rightarrow D^{b}_{\hol}( R[F] )$; an inverse equivalence is given by applying the derived solution functor $\RSol$.
\end{corollary}

Theorem \ref{RHDerived} also implies that a slightly weaker version Proposition \ref{sekos} holds for algebraic Frobenius modules:

\begin{corollary}\label{corollary.reform}
Let $R$ be a commutative $\F_p$-algebra. Then the inclusion functor $\Mod_{R}^{\alg} \hookrightarrow \Mod_{R}^{\Frob}$ extends to a fully faithful
embedding of derived categories $D^{+}( \Mod_{R}^{\alg} ) \rightarrow D^{+}( R[F] )$, whose essential image is
the full subcategory $D^{+}_{\alg}( \Mod_{R}^{\Frob}) \subseteq D^{+}( R[F] )$.
\end{corollary}

\begin{proof}
Since the Riemann-Hilbert functor $\RH: \Shv_{\mathet}( \Spec(R), \F_p) \rightarrow \Mod_{R}^{\alg}$ is an equivalence of categories (Theorem \ref{maintheoXXX}),
Corollary \ref{corollary.reform} is a reformulation of Theorem \ref{RHDerived}.
\end{proof}

\subsection{Duality for $R$-Projective Frobenius Modules}\label{subsection.dual-properties}

Let $R$ be a commutative $\F_p$-algebra and let $M \in \Mod_{R}^{\Frob}$ be finitely generated and projective as an $R$-module. In \S \ref{sec10sub2}, we introduced a Frobenius module
$\mathbb{D}(M)$, given by the direct limit of the sequence 
$$ M^{\vee} \xrightarrow{ \psi_M^{\vee}} \varphi_{R}^{\ast} M^{\vee} \xrightarrow{ \varphi_{R}^{\ast} \psi_M^{\vee} } \varphi_{R}^{2 \ast} M^{\vee} \rightarrow \cdots.$$
As the notation suggests, we can think of $\mathbb{D}(M)$ as a kind of dual of $M$ in the setting of Frobenius modules. Our goal in this section is to make this idea precise.
We begin by observing that there is a canonical map
$$ c: R \rightarrow M \otimes_{R} M^{\vee} \rightarrow M \otimes_{R} \mathbb{D}(M).$$
It is not hard to see that $c$ is a map of Frobenius modules (where we regard the tensor product $M \otimes_{R} \mathbb{D}(M)$ as a Frobenius module via
Construction \ref{tenso0}). Moreover, it enjoys the following universal property:

\begin{proposition}\label{proposition.easydual}
Let $R$ be a commutative $\F_p$-algebra and let $M$ and $N$ be Frobenius modules over $R$, where $M$ is finitely generated and projective as an $R$-module.
Then composition with the map $c: R \rightarrow M \otimes_{R} \mathbb{D}(M)$ induces isomorphisms
$\Ext^{n}_{ R[F] }( \mathbb{D}(M), N ) \rightarrow \Ext^{n}_{R[F]}( R, M \otimes_{R} N )$.
\end{proposition}

\begin{proof}
Using Proposition \ref{stopmake} and Remark \ref{remark.computeExt}, we see that both sides can be computed as the cohomology groups
of the two-term chain complex
$$ M \otimes_{R} N \xrightarrow{ \id - \varphi_M \otimes \varphi_N } M \otimes_{R} N.$$
\end{proof}

\begin{remark}
The abelian groups $\Ext^{n}_{ R[F] }( \mathbb{D}(M), N ) \simeq \Ext^{n}_{R[F]}( R, M \otimes_{R} N )$ of Proposition \ref{proposition.easydual}
vanish for $n \geq 2$.
\end{remark}

We also have the following dual version of Proposition \ref{proposition.easydual}:

\begin{proposition}\label{proposition.harderdual}
Let $R$ be a commutative $\F_p$-algebra and let $M$ and $N$ be Frobenius modules over $R$, where $M$ is finitely generated and projective as an $R$-module.
If $N$ is perfect, then composition with the map $c: R \rightarrow M \otimes_{R} \mathbb{D}(M)$ induces isomorphisms
$\Ext^{n}_{ R[F] }( M, N ) \rightarrow \Ext^{n}_{R[F]}( R, N \otimes_{R} \mathbb{D}(M) )$.
\end{proposition}

\begin{proof}
Using Remark \ref{remark.computeExt}, we can identify $\Ext^{\ast}_{R[F]}( R, N \otimes_{R} \mathbb{D}(M) )$ with the direct limit
of the diagram
$$ \Ext^{\ast}_{R[F]}( M, N ) \rightarrow \Ext^{\ast}_{ R[F] }( \varphi_{R}^{\ast} M, N ) \rightarrow \Ext^{\ast}_{ R[F] }( \varphi_{R}^{2 \ast} M, N ) \rightarrow \cdots,$$
where the transition maps are given by precomposition with the map $\psi_M: \varphi_{R}^{\ast} M \rightarrow M$ of Notation \ref{notation.psiM}. It will therefore suffice to show
that each of the transition maps $\Ext^{\ast}_{ R[F] }(\varphi_{R}^{k \ast} M, N) \rightarrow \Ext^{\ast}_{ R[F] }( \varphi_{R}^{(k+1) \ast} M, N)$ is an isomorphism.
This follows from the assumption that $N$ is perfect, since the map $\psi_{M}$ induces an isomorphism of perfections $(\varphi_{R}^{k \ast} M)^{\perfection} \rightarrow (\varphi_{R}^{(k+1) \ast} M)^{\perfection}$
(Remark \ref{remark.obvious}). 
\end{proof}

\subsection{Weak Duality in $D(R[F])$}\label{section.weakdual}

We now introduce some language to place Proposition \ref{proposition.easydual} in a more general context. First, we need a bit of notation.

\begin{construction}[Derived Tensor Products]\label{construction.tensoderived}
Let $R$ be a commutative $\F_p$-algebra. Then we can identify $D(R[F])$ with the category whose objects
are $K$-projective cochain complexes of left $R[F]$-modules (in the sense of Spaltenstein, see \cite[Tag 070G]{Stacks} for a summary), and whose morphisms
are homotopy classes of chain maps. Using Example \ref{example.tensorfree}, it is not hard to show that if $M^{\bullet}$ and $N^{\bullet}$ are $K$-projective cochain
complexes, then the tensor product $M^{\bullet} \otimes_{R} N^{\bullet}$ is also $K$-projective (where we regard the tensor product as a chain complex of left $R[F]$-modules
via Construction \ref{tenso0}). This construction gives rise to a functor
$$ \otimes_{R}^{L}: D(R[F]) \times D(R[F]) \rightarrow D(R[F])$$
which we will refer to as the {\it derived tensor product}.
\end{construction}

\begin{remark}
Let $R$ be a commutative $\F_p$-algebra. Then the forgetful functor $D(R[F]) \rightarrow D(R)$ is compatible with derived tensor products.
\end{remark}

\begin{remark}
Let $R$ be a commutative $\F_p$-algebra and let $M$ and $N$ be Frobenius modules over $R$, which we regard as objects of $D( R[F] )$. 
Then we have canonical isomorphisms $\mathrm{H}^{n}( M \otimes_{R}^{L} N ) \simeq \Tor_{-n}^{R}( M, N )$ in the category of Frobenius modules.
More generally, if $M$ and $N$ are arbitrary object of $D(R[F])$, we have a convergent spectral sequence
$$ \bigoplus_{i + j = t} \Tor_{s}^{R}( \mathrm{H}^{i}(M), \mathrm{H}^{j}(N) ) \Rightarrow \mathrm{H}^{t-s}( M \otimes_{R}^{L} N).$$
\end{remark}

\begin{definition}\label{definition.weakdual}
Let $R$ be a commutative $\F_p$-algebras and let $M$ and $M'$ be objects of the derived category $D(R[F])$. We will say that a morphism
$c: R \rightarrow M \otimes_{R}^{L} M'$ {\it exhibits $M'$ as a weak dual of $M$} if, for every object $N \in D(R[F])$, composition with
$c$ induces a bijection
$$ \Hom_{ D(R[F]) }( M', N ) \rightarrow \Hom_{ D(R[F]) }( R, M \otimes_{R}^{L} N).$$
\end{definition}

\begin{proposition}\label{proposition.makedual}
Let $R$ be a commutative $\F_p$-algebra and let $M \in \Mod_{R}^{\Frob}$ be a projective $R$-module of finite rank. Then the map
$c: R \rightarrow M \otimes_{R} \mathbb{D}(M)$ of Proposition \ref{proposition.easydual} exhibits $\mathbb{D}(M)$ as a weak dual of $M$.
\end{proposition}

\begin{proof}
We first observe that $M \otimes_{R} \mathbb{D}(M)$ can be identified with the derived tensor product $M \otimes_{R}^{L} \mathbb{D}(M)$ 
(since both $M$ and $\mathbb{D}(M)$ are flat $R$-modules). We wish to show that, for every object $N \in D( R[F] )$,
composition with $c$ induces an isomorphism $\Hom_{ D(R[F]) }( \mathbb{D}(M), N ) \rightarrow \Hom_{ D(R[F]) }( R, M \otimes_{R}^{L} N)$. Using
the fact that $\mathbb{D}(M)$ has finite projective dimension as an $R[F]$-module (Proposition \ref{stopmakecor}), we can reduce to the case where
$N$ is concentrated in a single degree, in which case the desired result is a translation of Proposition \ref{proposition.easydual}.
\end{proof}

\begin{notation}\label{notation.weakdual}
Let $R$ be a commutative $\F_p$-algebra and let $M \in D(R[F])$. It follows immediately from the definitions that
if there exists a morphism $c: R \rightarrow M \otimes_{R}^{L} M'$ which exhibits $M'$ as a weak dual of $M$, then the object
$M'$ (and the morphism $c$) are well-defined up to unique isomorphism (in the derived category $D(R[F])$). In this case, we will say that
$M$ is {\it weakly dualizable} and denote its weak dual $M'$ by $\mathbb{D}(M)$. Note that, by virtue of Proposition \ref{proposition.makedual},
this notation is consistent with that of Construction \ref{construction.dualunit}.
\end{notation}


\begin{warning}\label{warning.nonsym}
In the situation Definition \ref{definition.weakdual}, the roles of $M$ and $M'$ are not symmetric. A morphism $c: R \rightarrow M \otimes^{L}_{R} M'$ which exhibits $M'$ as a weak dual of $M$
generally does not exhibit $M$ as a weak dual of $M'$ (see Example \ref{sofo}). This asymmetry already appeared in \S \ref{subsection.dual-properties}: note that in the statement of Proposition \ref{proposition.harderdual}
we required the Frobenius module $N$ to be perfect, but no corresponding hypothesis was needed in the statement of Proposition \ref{proposition.easydual}.
\end{warning}



\begin{proposition}\label{dualperfection}
Let $R$ be a commutative $\F_p$-algebra and let $c: R \rightarrow M \otimes_{R}^{L} M'$ be a morphism in $D(R[F])$ which exhibits $M'$ as a weak dual of $M$.
Then the composite map
$$ R \xrightarrow{c} M \otimes_{R}^{L} M' \rightarrow M^{\perfection} \otimes_{R}^{L} M'$$
exhibits $M'$ as a weak dual of the perfection $M^{\perfection}$.
\end{proposition}

We will deduce Proposition \ref{dualperfection} from the following variant of Proposition \ref{proposition.solperfect}:

\begin{lemma}\label{SolRootPerf}
Let $R$ be a commutative $\F_p$-algebra and let $f: M \rightarrow M'$ be a morphism in
$D(R[F])$ which induces an isomorphism $M^{\perfection} \rightarrow M'^{\perfection}$.
Then the induced map
$$ \Hom_{ D(R[F]) }( R, M ) \rightarrow \Hom_{ D(R[F]) }( R, M')$$
is an isomorphism. 
\end{lemma}

\begin{proof}
Let $N$ denote the cone of the morphism $M \rightarrow M'$; we will show that $\Hom_{ D(R[F]) }( R, N[k] ) $ vanishes
for every integer $k$. By virtue of Remark \ref{remark.computeExt}, it will suffice to show that the map
$\id - \varphi_{N}: \mathrm{H}^{\ast}( N ) \rightarrow \mathrm{H}^{\ast}(N)$
is an isomorphism. This is clear: the assumption that $f$ induces an equivalence $M^{\perfection} \simeq M'^{\perfection}$
guarantees that $N^{\perfection}$ vanishes, so that the action of $\varphi_{N}$ is locally nilpotent on $\mathrm{H}^{\ast}(N)$.
\end{proof}

\begin{proof}[Proof of Proposition \ref{dualperfection}]
Let $c: R \rightarrow M \otimes_{R}^{L} M'$ be a morphism in $D(R[F])$ which exhibits $M'$ as a weak dual of $M$, and 
let $N$ be any object of $D(R[F])$. 
Then the composite map
$$ \Hom_{ D(R[F]) }( M', N ) \rightarrow \Hom_{ D(R[F]) }( R, M \otimes_{R}^{L} N) 
\rightarrow \Hom_{ D(R[F]) }( R, M^{\perfection} \otimes_{R}^{L} N)$$
is an isomorphism, since the left map is an isomorphism (by virtue of our assumption that $M'$ is a weak dual of $M$)
and the right map is an isomorphism (Lemma \ref{SolRootPerf}). Allowing $N$ to vary, we deduce that $M'$ is also
a weak dual of $M^{\perfection}$.
\end{proof}

\begin{example}\label{sofo}
Let $R$ be a commutative $\F_p$-algebra. Then the canonical isomorphism $R \simeq R \otimes_{R}^{L} R$ exhibits $R$ as a weak dual of itself. It follows from Proposition \ref{dualperfection} that unit map
$u: R \rightarrow R^{\perfection} \simeq R^{\perfection} \otimes_{R}^{L} R$ also exhibits $R$ as a weak dual of $R^{\perfection}$. However,
$u$ cannot exhibit $R^{\perfection}$ as a weak dual of $R$ (unless $R$ is perfect), since the weak dual of $R$ is determined uniquely up to isomorphism.
\end{example}

We conclude this section with another application of Lemma \ref{SolRootPerf}:

\begin{proposition}\label{autofgu}
Let $R$ be a commutative $\F_p$-algebra and let $c: R \rightarrow M \otimes_{R}^{L} M'$ be a morphism in $D(R[F] )$ which exhibits $M'$ as a weak dual of $M$. Then $M'$ belongs
to $D_{\fgu}^{b}( R[F ] )$.
\end{proposition}

\begin{proof}
From the isomorphism $\Hom_{ D(R[F]) }( M', \bullet ) \simeq \Hom_{ D(R[F] ) }( R, M \otimes_{R}^{L} \bullet )$ (and the compactness of $R$ as an object of $D( R[F] )$), we conclude
that $M'$ is a compact object of $D(R[F] )$. It will therefore suffice to show that the canonical map $\psi_{M'}: R^{1/p} \otimes_{R}^{L} M' \rightarrow M'$ is an isomorphism.
Note that $\psi_{M'}$ can be regarded as a morphism in $D(R[F] )$; it will therefore suffice to show that for each $N \in D( R[F] )$, composition with $\psi_{M'}$
induces an isomorphism $$\theta: \Hom_{ D(R[F] ) }( M', N) \rightarrow \Hom_{ D(R[F] ) }( R^{1/p} \otimes_{R}^{L} M', N ) \simeq \Hom_{ D(R[F] )}( M', N^{1/p} ).$$ Invoking the universal property of $M'$,
we can identify $\theta$ with the natural map $\Hom_{ D(R[F] ) }( R, M \otimes_{R}^{L} N ) \rightarrow \Hom_{ D(R[F] )}( R, M \otimes_{R} N^{1/p} )$ (induced by the Frobenius map
$\varphi_{N}: N \rightarrow N^{1/p}$). This map is an isomorphism, since the induced map $M \otimes_{R}^{L} N \rightarrow M \otimes_{R}^{L} N^{1/p}$ induces an isomorphism of
perfections (Lemma \ref{SolRootPerf}).
\end{proof}

\subsection{Presentations of Holonomic Complexes}\label{subsection.presentations}

Let $R$ be a commutative $\F_p$-algebra and let $M$ be a Frobenius module over $R$. By definition, $M$ is holonomic if and only if there
exists an isomorphism $M \simeq M_0^{\perfection}$, where $M_0 \in \Mod_{R}^{\Frob}$ is finitely presented as an $R$-module.
Our goal in this section is to prove an analogous statement for objects of the derived category $D(R[F] )$:

\begin{theorem}\label{theorem.holochar}
Let $R$ be a commutative $\F_p$-algebra and let $M$ be an object of $D( R[F] )$. The following conditions are equivalent:
\begin{itemize}
\item[$(1)$] The complex $M$ belongs to the subcategory $D^{b}_{\hol}( R[F] ) \subseteq D(R[F] )$: that is,
it is cohomologically bounded with holonomic cohomologies.

\item[$(2)$] There exists an isomorphism $M \simeq M_0^{\perfection}$ in the category $D( R[F] )$, where
$M_0 \in D( R[F] )$ has the property that its image in $D(R)$ is compact.
\end{itemize}
\end{theorem}

The proof of Theorem \ref{theorem.holochar} will require some preliminaries. We first study condition $(2)$ of Theorem \ref{theorem.holochar}.
Note that an object $M \in D( R[F] )$ has compact image in $D(R)$ if and only if it is quasi-isomorphic to a bounded cochain complex $N^{\ast}$ of finitely generated projective $R$-modules.
We now show that, in this situation, we can arrange that $N^{\ast}$ is also a cochain complex of Frobenius modules:

\begin{lemma}\label{lemma.charroot}
Let $R$ be a commutative $\F_p$-algebra and let $M$ be an object of the derived category $D(R[F])$. The following conditions are equivalent:
\begin{itemize}
\item[$(1)$] The object $M$ is isomorphic (in $D(R[F])$) to a bounded cochain complex of Frobenius modules, each of which is projective of finite rank as an $R$-module.
\item[$(2)$] The image of $M$ in $D(R)$ is compact: that is, it is isomorphic to a bounded cochain complex of projective $R$-modules of finite rank.
\end{itemize}
\end{lemma}

\begin{proof}
The implication $(1) \Rightarrow (2)$ is clear. Conversely, suppose that $(2)$ is satisfied; we will prove $(1)$. Assume that, as an object of $D(R)$,
the complex $M$ is quasi-isomorphic to a finite cochain complex of finitely generated projective $R$-modules concentrated in degrees $\{ a-n, a-n+1, \ldots, a \}$.
Replacing $M$ by a shift if necessary, we can assume $a=0$. We claim that, as an object of $D(R[F])$, the complex
$M$ is quasi-isomorphic to a finite cochain complex of Frobenius modules, which are finitely generated and projective over $R$, also concentrated
in degrees $\{ -n, -n+1, \ldots, 0 \}$.  We proceed by induction on $n$. If $n=0$, then the cohomology groups $\mathrm{H}^{i}(M)$ vanish for
$i \neq 0$ and $\mathrm{H}^{0}(M)$ is a projective $R$-module of finite rank. In this case, the desired result follows from the observation that $M$ is isomorphic to 
$\mathrm{H}^{0}(M)$ as an object of $D(R[F])$. Let us therefore suppose that $n > 0$, and set $N = \mathrm{H}^{0}(M)$. Then 
$N$ is finitely presented as an $R$-module, so that $N^{\perfection}$ is a holonomic Frobenius module over $R$. Choose elements
$x_1, x_2, \ldots, x_k$ which generate $N$ as an $R$-module. 
It follows from Proposition \ref{propX58} that the image of each $x_i$ in $N^{\perfection}$ is annihilated by some element
$P_i \in R[F]$ of the form $F^{m_i} + c_{1,i} F^{m_i-1} + \cdots + c_{m_i,i}$. Replacing $P_i$ by $F^{a} P_i$ for $a \gg 0$, we may assume
that $P_i(x_i) = 0$. Choose a cocycle $\overline{x}_i \in M^{0}$ representing $x_i$, so that we can write $P_i( \overline{x}_i ) = d y_i$ for
some elements $y_i \in M^{-1}$. The elements $\overline{x}_i$ and $y_i$ determine a map of cochain complexes $f: M' \rightarrow M$,
where $M'$ is the two-term complex
$$ \cdots \rightarrow 0 \rightarrow R[F]^{k} \xrightarrow{ (P_1, \ldots, P_k )} R[F]^{k} \rightarrow 0 \rightarrow \cdots.$$
Note that $M'$ is isomorphic, as an object of $D(R[F])$, to the Frobenius module
$K = \bigoplus_{i} R[F] / R[F] P_i$, which is projective of finite rank as an $R$-module. Extend $f$ to a distinguished triangle
$Q \xrightarrow{g} M' \xrightarrow{f} M \rightarrow Q[1]$
in $D(R[F])$. Then, as an object of $D(R)$, the complex $Q$ is quasi-isomorphic to a chain complex of finitely generated projective
$R$-modules concentrated in degrees $\{ 1-n, \ldots, 0 \}$. Applying our inductive hypothesis, we may assume that
each $Q^{i}$ is a projective $R$-module of finite rank and that $Q^i$ vanishes unless $-n < i \leq 0$. Then $g$ determines a map of Frobenius modules
$Q^{0} \rightarrow K$, and $M$ is quasi-isomorphic to the cochain complex of Frobenius modules
$$ \cdots \rightarrow 0 \rightarrow Q^{-n+1} \rightarrow Q^{-n+2} \rightarrow \cdots \rightarrow Q^{0} \rightarrow K \rightarrow 0 \rightarrow \cdots$$
\end{proof}

\begin{remark}\label{morecompact}
Let $R$ be a commutative $\F_p$-algebra and let $M \in D( R[F] )$ be an object whose image in $D(R)$ is compact.
Then $M$ is also compact as an object of $D(R[F])$: this follows immediately from Remark \ref{remark.computeExt}.
However, the converse is false: the Frobenius module $R[F]$ is compact as an object of $D( R[F] )$, but its image in $D(R)$ is not compact
unless $R \simeq 0$.
\end{remark}

We now apply Lemma \ref{lemma.charroot} to give a simple characterization of the holonomic derived category $D^{b}_{\hol}( R[F] )$.

\begin{lemma}\label{proposition.holonomicderived}
Let $R$ be a commutative $\F_p$-algebra. Then $D^{b}_{\hol}( R[F] )$ is the smallest triangulated subcategory of
$D(R[F])$ which contains every object of the form $M^{\perfection}$, where $M \in \Mod_{R}^{\Frob}$ is finitely
generated and projective as an $R$-module.
\end{lemma}

\begin{proof}
Let $\calC$ be a triangulated subcategory of $D(R[F])$ which contains every object of the form $M^{\perfection}$,
where $M \in \Mod_{R}^{\Frob}$ is finitely generated and projective over $R$. We wish to show that $\calC$ contains every object of
$D^{b}_{\hol}( R[F] )$. Using our assumption that $\calC$ is a triangulated subcategory, we are reduced to showing that
$\calC$ contains every holonomic Frobenius module $N$ over $R$ (regarded as a chain complex concentrated in degree zero).
Using Proposition \ref{propX52}, we can assume that $N$ has the form $(R \otimes_{R_0} N' )^{\perfection}$, where $R_0$ is a finitely generated subring of $R$ and
$N'$ is a holonomic Frobenius module over $R_0$. Choose a surjection $A \rightarrow R_0$, where $A$ is a polynomial ring over $\F_p$.
Then $N'$ is also holonomic when regarded as a Frobenius module over $A$ (Remark \ref{elos}). Choose an isomorphism
$N' \simeq N'^{\perfection}_0$, where $N'_0$ is a Frobenius module over $A$ which is finitely generated as an $A$-module.
Since $A$ is a regular Noetherian ring, the $A$-module $N'_0$ admits a finite resolution by projective $A$-modules of finite rank.
It follows from Lemma \ref{lemma.charroot} that $N'_0$ admits a finite resolution
$$ \cdots \rightarrow P_3 \rightarrow P_2 \rightarrow P_1 \rightarrow P_0 \rightarrow N'_0 \rightarrow 0$$
in the category of Frobenius modules over $A$, where each $P_k$ is projective of finite rank as an $A$-module.
Then $(R \otimes_{A} P_{\bullet})^{\perfection}$ is a finite resolution of $N$ by objects of $\Mod_{R}^{\Frob}$ which belong to $\calC$.
Since $\calC$ is a triangulated subcategory of $D(R[F])$, we deduce that $N$ also belongs to $\calC$.
\end{proof}

\begin{proof}[Proof of Theorem \ref{theorem.holochar}]
Let $R$ be a commutative $\F_p$-algebra and let $\calC$ denote the full subcategory of $D( R[F] )$ spanned by those objects
which are isomorphic to $M_0^{\perfection}$, for some $M_0 \in D( R[F] )$ having compact image in $D(R)$.
We wish to show that $\calC = D^{b}_{\hol}( R[F] )$. We first show that $\calC$ is contained in $D^{b}_{\hol}(R[F] )$.
Let $M_0 \in D( R[F] )$ have compact image in $D(R)$; we wish to show that $M_0^{\perfection}$ belongs to $D^{b}_{\hol}( R[F] )$.
Using Lemma \ref{lemma.charroot}, we can assume that $M_0$ is a bounded cochain complex consisting of Frobenius modules which
are finitely generated and projective over $R$. Since $D^{b}_{\hol}( R[F] )$ is a triangulated subcategory of $D( R[F] )$, we can
reduce to the case where $M_0$ is a finitely generated projective $R$-module, concentrated in degree zero. In this case,
the inclusion is clear (since $M_0^{\perfection}$ is a holonomic Frobenius module over $R$).

We now show that $D^{b}_{\hol}( R[F] )$ is contained in $\calC$. By virtue of Lemma \ref{proposition.holonomicderived}, it will suffice to show that $\calC$ is a triangulated subcategory of
$D( R[F] )$. It is clear that $\calC$ contains zero objects of $D( R[F] )$ and is closed under shifts; it will therefore suffice to show that it contains the cone of any morphism
$f: M \rightarrow N$ where $M$ and $N$ belong to $\calC$. Write $M = M_0^{\perfection}$ and $N = N_0^{\perfection}$, where $M_0$ and $N_0$ are objects of
$D( R[F] )$ having compact image in $D(R)$. Using Lemma \ref{lemma.charroot}, we can further assume that $M_0$ is a cochain complex of Frobenius modules
which are finitely generated and projective over $R$. Note that $N$ can be identified with the homotopy colimit of the diagram
$$ N \xrightarrow{ \varphi_N} N^{1/p} \xrightarrow{ \varphi_N } N^{1/p^2} \rightarrow \cdots.$$
Since $M_0$ is a compact object of $D(R[F] )$ (Remark \ref{morecompact}), the composite map
$M_0 \rightarrow M \xrightarrow{f} N$ factors through some map $f': M_0 \rightarrow N_0^{1/p^n}$ for $n \gg 0$. Then $f'$ is adjoint to a map
$f'': \varphi_{R}^{n \ast} M_0 \rightarrow N_0$, where $\varphi_{R}^{n \ast} M_0$ is the cochain complex obtained from $M_0$ by applying the pullback functor $\varphi_{R}^{n \ast}$ degreewise.
Note that $\varphi_{R}^{n \ast} M_0$ is also a bounded cochain complex of finitely generated projective $R$-modules, and therefore has compact image in $D(R)$.
Let $C_0$ be a cone of $f''$. Using Remark \ref{remark.obvious} (and the exactness of the functor $K \mapsto K^{\perfection}$), we see that
the cone of $f$ can be identified with $C_0^{\perfection}$, and therefore belongs to $\calC$ as desired.
\end{proof}

\begin{remark}\label{remark.Verdier}
With a bit more effort, one can prove the following stronger version of Theorem \ref{theorem.holochar}: the construction $M \mapsto M^{\perfection}$
induces an equivalence of triangulated categories $\calC / \calC_0 \simeq D^{b}_{\hol}( R[F] )$, where
$\calC$ denotes the triangulated subcategory of $D( R[F] )$ spanned by those objects having compact image in $D(R)$,
$\calC_0 \subseteq \calC$ is the triangulated subcategory spanned by those objects $M \in \calC$ satisfying $M^{\perfection} \simeq 0$,
and $\calC / \calC_0$ denotes the Verdier quotient of $\calC$ by $\calC_0$. Since we will not need this fact, the proof is left to the reader.
\end{remark}

\subsection{The Duality Functor}\label{biduality}

We now return to the study of the duality construction $M \mapsto \mathbb{D}(M)$ of \S \ref{section.weakdual}.

\begin{proposition}\label{proposition.holdual}
Let $R$ be a commutative $\F_p$-algebra and let $M$ be an object of $D^{b}_{\hol}(R[F] )$.
Then $M$ is weakly dualizable (in the sense of Notation \ref{notation.weakdual}).
\end{proposition}

\begin{proof}
Using Theorem \ref{theorem.holochar}, we can assume $M = M_0^{\perfection}$, where
$M_0 \in D( R[F] )$ has compact image in $D(R)$. By virtue of Lemma \ref{lemma.charroot}, we may assume
that $M_0$ is a bounded cochain complex of finitely generated projective $R$-modules. Let
$M_0^{\vee}$ denote its $R$-linear dual (which we also regard as a cochain complex of finitely generated projective $R$-modules)
and let $\mathbb{D}(M_0)$ denote the cochain complex of Frobenius modules obtained by applying Construction
\ref{unitalize} termwise. Let $c$ denote the composite map $$R \rightarrow M_0 \otimes_{R} M_0^{\vee} \rightarrow
M_0 \otimes_{R} \mathbb{D}(M_0).$$
A simple calculation shows that $c$ is a morphism of (cochain complexes of) Frobenius modules. Note that
the tensor product $M_0 \otimes_{R} \mathbb{D}(M_0)$ is equivalent to the derived tensor product
$M_0 \otimes_{R}^{L} \mathbb{D}(M_0)$ (since both $M_0$ and $\mathbb{D}(M_0)$ are bounded cochain complexes of flat $R$-modules).
We claim that $c$ exhibits $\mathbb{D}(M_0)$ as a weak dual of $M_0$ in the derived category $D(R[F] )$. In other words,
we claim that for every object $N \in D( R[F] )$, composition with $c$ induces a bijection
$$ \Hom_{ D(R[F]) }( \mathbb{D}(M_0), N) \rightarrow \Hom_{ D(R[F]) }( R, M_0 \otimes_{R}^{L} N).$$
To prove this, we can proceed by induction on the length of the cochain complex $M_0$ and thereby
reduce to the case where $M_0$ is concentrated in a single degree, which follows from Proposition \ref{proposition.makedual}.
Applying Proposition \ref{dualperfection}, we deduce that the composite map
$$ R \xrightarrow{c} M_0 \otimes_{R}^{L} \mathbb{D}(M_0) \rightarrow M_0^{\perfection} \otimes_{R}^{L} \mathbb{D}(M_0) \simeq
M \otimes_{R}^{L} \mathbb{D}(M_0)$$
exhibits $\mathbb{D}(M_0)$ as a weak dual of $M$, so that $M$ is weakly dualizable as desired.
\end{proof}

Recall that a morphism $c: R \rightarrow M \otimes_{R}^{L} M'$ which exhibits $M'$ as a weak dual of $M$ need not exhibit $M$ as a weak dual of $M'$. However,
holonomic Frobenius complexes do enjoy the following weak form of biduality.

\begin{notation}
Let $R$ be a commutative $\F_p$-algebra. We let $D_{\perf}( R[F] )$ denote the full subcategory of $D( R[F] )$ spanned by those cochain complexes $M$
whose cohomology groups $\mathrm{H}^{\ast}(M)$ are perfect Frobenius modules.
\end{notation}

\begin{proposition}\label{proposition.bidual}
Let $R$ be a commutative $\F_p$-algebra and let $M$ be an object of $D^{b}_{\hol}( R[F] )$ with weak dual $\mathbb{D}(M)$. Then, for every object $N \in D_{\perf}( R[F] )$, composition with 
the canonical map $c: R \rightarrow M \otimes^{L}_{R} \mathbb{D}(M)$ induces an isomorphism
$$ \Hom_{ D(R[F]) }( M, N ) \rightarrow \Hom_{ D(R[F]) }( R, N \otimes_{R}^{L} \mathbb{D}(M) ).$$
\end{proposition}

\begin{proof}
Let us say that an object $M \in D^{b}_{\hol}( R[F] )$ is {\it good} if, for every object $N \in D_{\perf}( R[F] )$, the canonical map
$\Hom_{ D(R[F]) }( M, N ) \rightarrow \Hom_{ D(R[F]) }( R, N \otimes_{R}^{L} \mathbb{D}(M) )$ is an isomorphism. We wish to show that
every object of $M \in D^{b}_{\hol}( R[F] )$ is good. It is easy to see that the good objects of $D^{b}_{\hol}(R[F] )$ span a triangulated
subcategory. By virtue of Lemma \ref{proposition.holonomicderived}, it will suffice to show that every object of the form
$M_0^{\perfection}$ is good, where $M_0 \in \Mod_{R}^{\Frob}$ is finitely generated and projective as an $R$-module. In this case, for each $N \in D_{\perf}( R[F] )$, we have a commutative diagram
$$ \xymatrix{ \Hom_{ D(R[F]) }( M, N ) \ar[r] \ar[d] & \Hom_{ D(R[F]) }( R, N \otimes_{R}^{L} \mathbb{D}(M) ) \ar[d] \\
\Hom_{ D(R[F]) }( M_0, N ) \ar[r]^-{\theta_N} & \Hom_{ D(R[F]) }( R, N \otimes_{R}^{L} \mathbb{D}(M_0) ); }$$
here the right vertical map is bijective by virtue of Proposition \ref{dualperfection}, and the left vertical map is bijective by virtue of our assumption that $N$ is perfect.
It will therefore suffice to show that the map $\theta_N$ is an isomorphism for every perfect object $N \in D_{\perf}( R[F] )$. Using the fact that
$M_0$ and $R$ admit finite resolutions by projective left $R[F]$-modules (Remark \ref{morecompact}), we can reduce to the situation where
$N$ is concentrated in a single degree. In this case, the desired result follows from Proposition \ref{proposition.harderdual}.
\end{proof}

We are now ready to prove the main result of this section:

\begin{theorem}\label{theorem.dual-image}
Let $R$ be a commutative $\F_p$-algebra. Then the construction $M \mapsto \mathbb{D}(M)$ induces
an equivalence of categories $D^{b}_{\hol}( R[F] ) \rightarrow D_{\fgu}^{b}( R[F] )^{\op}$.
\end{theorem}

\begin{proof} 
It follows from Propositions \ref{proposition.holdual} and \ref{autofgu} that the duality functor
$\mathbb{D}: D^{b}_{\hol}( R[F] ) \rightarrow D_{\fgu}^{b}( R[F] )^{\op}$ is well-defined.
We next claim that it is fully faithful.
Let $M$ and $N$ be objects of $D^{b}_{\hol}( R[F] )$; we wish to show that the canonical map
$$\theta: \Hom_{ D(R[F]) }( M, N ) \rightarrow \Hom_{ D(R[F]) }( \mathbb{D}(N), \mathbb{D}(M) ).$$
Using the definition of $\mathbb{D}(N)$, we can identify the codomain of $\theta$ with the set
$\Hom_{ D(R[F]) }( R, N \otimes_{R}^{L} \mathbb{D}(M) )$. Under this identification, $\theta$ corresponds
to the comparison map of Proposition \ref{proposition.bidual}, which is an isomorphism because $N$ is perfect.

Let $\calC$ denote the essential image of the weak duality functor $\mathbb{D}: D^b_{\hol}( R[F] ) \rightarrow D_{\fgu}^{b}( R[F] )^{\op}$,
so that $\calC$ is a triangulated subcategory of $D_{\fgu}^{b}( R[F] )$. We will complete the proof
by showing that every object $N \in D_{\fgu}^{b}( R[F] )$ belongs to $\calC$. We will deduce this from the following assertion:
\begin{itemize}
\item[$(\ast)$] There exists a diagram
$$ \cdots \rightarrow N(2) \rightarrow N(1) \rightarrow N(0) \rightarrow N(-1) \rightarrow N(-2) \rightarrow \cdots \rightarrow N$$
in the derived category $D(R[F])$, where each $N(k)$ belongs to $\calC$ and each of the maps
$\mathrm{H}^{n}( N(k) ) \rightarrow \mathrm{H}^{n}(N)$ is an isomorphism for $n > k$ and a surjection for $n = k$.
\end{itemize}
Assume $(\ast)$ for the moment. Then $N$ can be identified with the homotopy colimit of the diagram $\{ N(k) \}_{k \in \Z}$
in the triangulated category $D(R[F])$. Since $N$ is a compact object of $D(R[F])$, it follows
that the identity map $\id_{N}: N \rightarrow N$ factors through $N(k)$ for some integer $k$: that is, $N$ is a direct summand
of $N(k)$. Consequently, to prove that $N$ belongs to $\calC$, it will suffice to show that the category $\calC$ is idempotent complete.
Using the equivalence $\mathbb{D}: D^{b}_{\hol}( R[F] ) \rightarrow \calC^{\op}$, we are reduced to proving that
the category $D^{b}_{\hol}( R[F] )$ is idempotent complete, which is clear (since any direct summand of a holonomic Frobenius module over $R$ is itself holonomic).

It remains to prove $(\ast)$. We will construct the objects $N(k)$ by descending induction on $k$, taking $N(k) = 0$ for $k \gg 0$. To carry out the induction,
it will suffice to prove the following:
\begin{itemize}
\item[$(\ast')$] Let $f: N(k+1) \rightarrow N$ be a morphism in $D(R[F])$, where $N(k+1) \in \calC$ and the induced map
$\mathrm{H}^{n}( N(k+1) ) \rightarrow \mathrm{H}^{n}(N)$ is an isomorphism for $n > k+1$ and a surjection for $n = k+1$. Then the morphism
$f$ factors as a composition $N(k+1) \xrightarrow{f'} N(k) \xrightarrow{f''} N$, where $N(k) \in \calC$ and
the map $\mathrm{H}^{n}( N(k) ) \rightarrow \mathrm{H}^{n}(N)$ is an isomorphism for $n > k$ and a surjection for $n = k$.
\end{itemize}
To prove $(\ast')$, let $C$ denote the cone of $f$, so that $C$ belongs to $D_{\fgu}^{b}( R[F] )$ and
the cohomology groups $\mathrm{H}^{n}(C)$ vanish for $n > k$. Using Corollary \ref{corollary.stepmap},
we can choose an object $M \in \Mod_{R}^{\Frob}$ which is finitely generated and projective as an $R$-module
and a map $g: \mathbb{D}(M)[-n] \rightarrow C$ which induces a surjection $\mathbb{D}(M) \rightarrow \mathrm{H}^{n}(C)$.
Invoking the octahedral axiom, we conclude that $f$ factors as a composition $N(k+1) \xrightarrow{f'} N(k) \xrightarrow{f''} N$,
where the cone of $f'$ is isomorphic to $\mathbb{D}(M)[-k]$ (which guarantees that $N(k)$ belongs to $\calC$)
and the cone of $f''$ is isomorphic to the cone of $g$ (and therefore has vanishing cohomology in degrees $\geq k$).
\end{proof}

\subsection{Comparison of Solution Functors}\label{section.proofhappens}

We will deduce Theorem \ref{strongEK} from the following comparison result:

\begin{theorem}\label{theorem.compareduality}
Let $R$ be a commutative $\F_p$-algebra. Then the diagram of categories
$$ \xymatrix{ & D^{b}_{\hol}( R[F] ) \ar[dr]^{ \RSol} \ar[dl]_{ \mathbb{D} } & \\
D^{b}_{\fgu}(R[F])^{\op} \ar[rr]^{ \RSol_{\EK} } & & D_{\mathet}( \Spec(R), \F_p) }$$
commutes up canonical isomorphism. Here $\RSol$ denotes the derived solution
functor of \S \ref{section.derivedRH}, $\mathbb{D}$ is the duality functor of Theorem \ref{theorem.dual-image}, and
$\RSol_{\EK}$ is the derived Emerton-Kisin solution functor of Construction \ref{construction.SolEK}.
\end{theorem}

\begin{proof}[Proof of Theorem \ref{strongEK} from Theorem \ref{theorem.compareduality}]
Theorem \ref{theorem.dual-image} asserts that the functor $\mathbb{D}: D^{b}_{\hol}( R[F] ) \rightarrow D^{b}_{\fgu}(R[F] )^{\op}$ is an equivalence of
categories, and Corollary \ref{corollary.RHDerived} asserts that the functor $\RSol: D^{b}_{\hol}( R[F] ) \rightarrow D_{\mathet}( \Spec(R), \F_p)$ is a fully
faithful embedding whose essential image is the constructible derived category $D_{c}^{b}( \Spec(R), \F_p) \subseteq D_{\mathet}( \Spec(R), \F_p)$.
Using the commutative diagram of Theorem \ref{theorem.compareduality}, we deduce that
$\RSol_{\EK}: D^{b}_{\fgu}(R[F])^{\op} \rightarrow D_{\mathet}( \Spec(R), \F_p)$ is also a fully faithful embedding
with essential image $D_{c}^{b}( \Spec(R), \F_p)$.
\end{proof}

The proof of Theorem \ref{theorem.compareduality} will require some auxiliary constructions. We begin by introducing a slight modification of the derived solution functor
$\RSol$.

\begin{construction}
Let $R$ be a commutative $\F_p$-algebra and let $M = M^{\ast}$ be a cochain complex of Frobenius modules.
We let $\widetilde{M}$ denote the associated cochain complex of quasi-coherent sheaves on $\Spec(R)$ (Example \ref{exX70}),
so that the Frobenius morphism $\varphi_{M}$ determines an endomorphism of $\widetilde{M}$, which we will denote by
$\varphi_{ \widetilde{M} }$. We let $\Sol'( M )$ denote the cochain complex of {\etale} sheaves on $\Spec(R)$ given by
the shifted mapping cone $\cn( \id - \varphi_{ \widetilde{M} })[-1]$. It is clear that the construction $M \mapsto \Sol'(M)$ respects quasi-isomorphisms and therefore
determines a functor of derived categories $\Sol': D( R[F] ) \rightarrow D_{\mathet}( \Spec(R), \F_p)$. By construction, we have a distinguished triangle
$$ \Sol'(M) \rightarrow \widetilde{M} \xrightarrow{\id - \varphi_{\widetilde{M} } } \widetilde{M} \rightarrow \Sol'(M)[1],$$
depending functorially on $M$.
\end{construction}

\begin{remark}\label{remark.aux1}
In the special case where $M^{\ast}$ is a bounded below cochain complex of injective objects of $\Mod_{R}^{\perf}$, we can identify
$\RSol(M)$ with the kernel (formed in the category of chain complexes of {\etale} sheaves) of the map $\id - \varphi_{ \widetilde{M} }: \widetilde{M} \rightarrow \widetilde{M}$.
We therefore obtain a canonical map $\RSol(M) \rightarrow \Sol'(M)$, and Lemma \ref{lemma.propX74} guarantees that this map is a quasi-isomorphism 
(even at the level of presheaves). It follows that the functor $\RSol: D^{+}_{\perf}( R[F] ) \rightarrow D_{\mathet}( \Spec(R), \F_p)$ is canonically isomorphic
to the restriction $\Sol'|_{ D^{+}_{\perf}( R[F] )}$. 
\end{remark}

\begin{construction}\label{construction.aux2}
Let $R$ be a commutative $\F_p$-algebra and let $P$ denote the two-term cochain complex
$$ \cdots \rightarrow 0 \rightarrow R[F] \xrightarrow{ 1 - F} R[F] \rightarrow 0 \rightarrow \cdots,$$
which we regard as a projective representative for $R$ in the derived category $D( R[F] )$.
Let $M'$ be a bounded above cochain complex of projective left $R[F]$-modules, let
$M$ be an arbitrary cochain complex of left $R[F]$-modules, and suppose we are given a morphism of cochain complexes
$\overline{c}: P \rightarrow M \otimes_{R} M'$, which represents a morphism $c$ from $R$ to $M \otimes_{R}^{L} M'$ in the
derived category $D( R[F] )$. Note that we can identify $\RSol_{\EK}( M' )$ and $\Sol'(M)$ with the cochain complexes of {\etale} sheaves given concretely
by the formulae
$$ \RSol_{\EK}(M')( A) = \Hom_{R[F]}( M', A) \quad \quad \Sol'(M)(A) = \Hom_{R[F]}( P, M \otimes_{R} A).$$
It follows that $\overline{c}$ determines a map of cochain complexes
\begin{eqnarray*}
\RSol_{\EK}(M') & = & \Hom_{R[F]}( M', \bullet) \\
& \rightarrow & \Hom_{R[F]}( M \otimes_{R} M', M \otimes_{R} \bullet ) \\
& \xrightarrow{\circ \overline{c}} & \Hom_{R[F] }( P, M \otimes_{R} \bullet) \\
& = & \Sol'(M).
\end{eqnarray*}
Note that the chain homotopy class of this map depends only on the chain homotopy class of $\overline{c}$.
We therefore obtain a morphism $\gamma_{c}: \RSol_{\EK}( M' ) \rightarrow \Sol'(M)$ in the derived category
$D_{\mathet}( \Spec(R), \F_p )$ which depends only the map $c: R \rightarrow R \rightarrow M \otimes_{R}^{L} M'$ in
$D( R[F] )$.
\end{construction}

\begin{proof}[Proof of Theorem \ref{theorem.compareduality} ]
By virtue of Remark \ref{remark.aux1}, it will suffice to show that the functors
$$ \Sol', \RSol_{\EK} \circ \mathbb{D}: D^{b}_{\hol}( R[F] ) \rightarrow D_{\mathet}( \Spec(R), \F_p )$$
are naturally isomorphic. Fix an object $M \in D^{b}_{\hol}(R[F] )$ and let $c: R \rightarrow M \otimes_{R}^{L} \mathbb{D}(M)$
be a morphism in $D(R[F] )$ which exhibits $\mathbb{D}(M)$ as a weak dual of $M$. Applying Construction
\ref{construction.aux2}, we obtain a morphism
$\gamma_{c}: (\RSol_{\EK} \circ \mathbb{D} )(M) \rightarrow \Sol'(M)$ in the derived category $D_{\mathet}( \Spec(R), \F_p)$.
It is not difficult to see that this morphism depends functorially on $M$, and therefore determines a natural transformation of functors
$\gamma: \RSol_{\EK} \circ \mathbb{D} \rightarrow \Sol'$. To complete the proof, it will suffice to show that
this natural transformation is invertible: that is, $\gamma_c$ is a quasi-isomorphism for each $M \in D^{b}_{\hol}( R[F] )$.
To prove this, we may assume without loss of generality that $\mathbb{D}(M)$ is represented by a bounded above cochain complex of projective left $R[F]$-modules and that $c$ is represented
by a morphism of cochain complexes $\overline{c}: P \rightarrow M \otimes_{R} \mathbb{D}(M)$, so that $\gamma_c$
is represented by the map of cochain complexes of {\etale} sheaves
$$ \Hom_{R[F]}( \mathbb{D}(M), \bullet ) \rightarrow \Hom_{ R[F] }( P, M \otimes_{R} \bullet )$$
appearing in Construction \ref{construction.aux2}. We wish to show that this map is a quasi-isomorphism of {\etale} sheaves.
In fact, we claim that it is already a quasi-isomorphism of presheaves: that is, for every {\etale} $R$-algebra $A$, the map of
complexes $\Hom_{ R[F] }( \mathbb{D}(M), A) \rightarrow \Hom_{ R[F] }( P, M \otimes_{R} A)$ is a quasi-isomorphism.
This is a special case of our assumption that $c$ exhibits $\mathbb{D}(M)$ as a weak dual of $M$.
\end{proof}


\newpage 

\begin{thebibliography}{99}



\bibitem{BS} Bhatt, B., and P. Scholze. {\it Projectivity of the Witt vector affine Grassmannian.}

\bibitem{BB1} Blickle, M. and B\"{o}ckle, G. {\it Cartier Modules: finiteness results}, arXiv:0909.2531

\bibitem{BB2} Blickle, M. and B\"{o}ckle, G. {\it Cartier crystals}, arXiv:1309.1035

\bibitem{BP} B\"{o}ckle, G. and R. Pink. {\it Cohomological Theory of Crystals over Function Fields.}

\bibitem{CL} Chambert-Loir, A. {\em Cohomologie cristalline: un survol}.

\bibitem{EK} Emerton, M. and M. Kisin. {\it The Riemann-Hilbert Correspondence for Unit $F$-Crystals.}

\bibitem{FK} Freitag, E. and R. Kiehl. {\it Etale cohomology and the Weil conjecture.}

\bibitem{Gabber} Gabber, O. {\it Notes on some $t$-structures},  Geometric aspects of Dwork work, vol. II.

\bibitem{HottaEtAl} Hotta, R., Takeuchi, K. and Tanisaki, T. {\em $D$-modules, perverse sheaves and representation theory.}

\bibitem{KatzPadic} Katz, N. {\it $p$-adic properties of modular forms and moduli schemes.}

\bibitem{Ly} Lyubeznik, G. {\it $F$-modules: applications to local cohomology and D-modules in characteristic $p > 0$}, J. Reine Angew. Math. 491 (1997), 65-130. 

\bibitem{Ohkawa} Ohkawa, S. {\it Riemann-Hilbert correspondence for unit F-crystals on embeddable algebraic varieties}, arXiv:1601.01525.

\bibitem{Sched} Schedlmeier, T. {\it  Cartier crystals and perverse constructible \'etale p-torsion sheaves}, arXiv:1603.07696.

\bibitem{Stacks} The Stacks Project.

\bibitem{vdK} van der Kallen, W. {\it Descent for the K-theory of polynomial rings.} Math. Z. 191:3 (1986), 405-415.

\bibitem{Verdier}  Verdier, J.L. {\it Des cat\'egories d\'eriv\'ees des cat\'egories ab\'eliennes.} Ast\'erisque, vol. 239. 

\end{thebibliography}
\end{document}